\documentclass[11pt]{amsart}
\usepackage{fullpage}
\usepackage[utf8]{inputenc}
\usepackage{graphicx}				
\usepackage{amssymb,amsmath,amsthm}
\usepackage{enumerate}
\usepackage{xcolor}
\usepackage{fancyhdr}
\usepackage{tikz-cd}
\usepackage{float}
\usepackage{listings}
\usepackage{colonequals}
\usepackage{versions}
\usepackage[all]{xy}
\usepackage{tabularx}
\usepackage{url}
\usepackage{rotating} 

\newcommand{\pre}[1]{%
  \begin{tabular}{@{}c@{}} #1 \end{tabular}%
}
\newcommand{\middlexcolumn}{%
  \renewcommand{\tabularxcolumn}[1]{m{##1}}%
}
\newcommand{\heading}[1]{\raggedright\arraybackslash #1}

\definecolor{codegreen}{rgb}{0,0.6,0}
\definecolor{codegray}{rgb}{0.5,0.5,0.5}
\definecolor{codepurple}{rgb}{0.58,0,0.82}
\definecolor{backcolour}{rgb}{0.95,0.95,0.92}

\lstdefinestyle{mystyle}{
    backgroundcolor=\color{backcolour},
    commentstyle=\color{codegreen},
    keywordstyle=\color{magenta},
    numberstyle=\tiny\color{codegray},
    stringstyle=\color{codepurple},
    basicstyle=\ttfamily\footnotesize,
    breakatwhitespace=false,
    breaklines=true,
    captionpos=b,
    keepspaces=true,
    numbers=left,
    numbersep=5pt,
    showspaces=false,
    showstringspaces=false,
    showtabs=false,
    tabsize=2
}
\newcommand{\Z}{\mathbb Z}
\newcommand{\Q}{\mathbb Q}

\newcommand{\F}{\mathbb F}

\newcommand{\M}{\mathcal{M}}
\newcommand{\A}{\mathcal{A}}
\newcommand{\bbA}{\mathbb{A}}
\newcommand{\bbP}{\mathbb{P}}
\newcommand{\fA}{\mathfrak{A}}

\newcommand{\Pic}{\operatorname{Pic}}
\newcommand{\PGL}{\operatorname{PGL}}
\newcommand{\Gal}{\operatorname{Gal}}
\newcommand{\Aut}{\operatorname{Aut}}

\newcommand{\Rat}{\operatorname{Rat}}
\newcommand{\SL}{\operatorname{SL}}
\newcommand{\PrePer}{\operatorname{PrePer}}

\theoremstyle{definition}
\newtheorem{theorem}{Theorem}[section]
\newtheorem*{theorem*}{Theorem}
\newtheorem{lemma}[theorem]{Lemma}

\newtheorem{cor}[theorem]{Corollary}
\newtheorem{conj}[theorem]{Conjecture}
\newtheorem{prop}[theorem]{Proposition}
\newtheorem*{fact*}{Fact}
\newtheorem{remark}[theorem]{Remark}

\providecommand{\new}[1]{\textcolor{blue}{#1}}

\title{Automorphism loci for degree 3 and degree 4 endomorphisms of the projective line}
\author{Brandon Gontmacher}
\address{Department of Mathematics,
         University of Maryland,
         College Park, MD 20742}
\email{bgontmac@umd.edu}

\author{Benjamin Hutz}
\address{Department of Mathematics and Statistics,
         Saint Louis University,
         St.~Louis, MO 63103}
\email{benjamin.hutz@slu.edu}

\author{Grayson Jorgenson}
\address{Department of Mathematics,
         Florida State University,
         Tallahassee, FL 32306}
\email{gjorgens@math.fsu.edu}

\author{Srinjoy Srimani}
\address{Department of Mathematics,
         Brown University,
        Providence, RI 02912}
\email{srinjoy\_srimani@brown.edu}

\author{Simon Xu}
\address{Department of Mathematics and Statistics,
         Colby College,
        Waterville, Maine 04901}
\email{sxu20@colby.edu}


\includeversion{code}

\thanks{The authors thank the Institute for Computational and Experimental Research in Mathematics which oversaw the REU during the summer of 2019, where the majority of this project was completed. Also the Brown Center for Computation and Visualization which provided computational resources for some of the computations. Several other individuals provide helpful discussions and expertise on various specific rational points problems in this paper: Brendan Hassett, Michael Stoll, Andrew Sutherland, Jan Tuitman, Jennifer Balakrishnan, and an anonymous referee.}

\begin{document}
\begin{abstract}
Let $f$ be an endomorphism of the projective line. There is a natural conjugation action
on the space of such morphisms by elements of the projective linear group.
The group of automorphisms, or stabilizer group, of a given $f$ for this action is known
to be a finite group. We determine explicit families that parameterize all
endomorphisms defined over $\bar{\mathbb{Q}}$ of degree $3$ and degree $4$ that have a nontrivial automorphism, the \textit{automorphism locus} of the moduli space of dynamical systems.
We analyze the geometry of these loci in the appropriate moduli space of dynamical systems. Further, for each family of maps, we study the possible structures of $\mathbb{Q}$-rational preperiodic points which occur under specialization.
\end{abstract}

\maketitle

\section{Introduction}
    Let $K$ be a field and $\bbP^1$ the projective line. Throughout, $K$ is a finite extension of $\Q$. An endomorphism of $\bbP^1$ of degree $d$ can be represented as a pair of homogeneous polynomials of degree $d$ with no common factors. The space of all such maps is denoted as $\Rat_d$. There is a natural conjugation action on $\Rat_d$ by $\PGL_2$, the automorphisms of $\bbP^1$, given as
    \begin{equation*}
        f^{\alpha} \colonequals \alpha^{-1} \circ f \circ \alpha \quad \text{for } f \in \Rat_d \text{ and } \alpha \in \PGL_2.
    \end{equation*}
    The quotient by this action is a geometric quotient in terms of geometric invariant theory and forms the moduli space of degree $d$ dynamical systems on $\bbP^1$, $\M_d \colonequals \Rat_d/\PGL_2$ \cite{Silverman9}. We denote a conjugacy class as $[f] \in \M_d$ and a representation of a conjugacy class as $f \in \Rat_d$. Our primary objects of study are those conjugacy classes $[f] \in \M_d$ for which there is a nontrivial $\alpha \in \PGL_2$ so that $f^\alpha = f$. Such an $\alpha$ is called an \emph{automorphism} of $f$, and the set of all such automorphisms forms a group
    \begin{equation*}
        \Aut(f) \colonequals \{ \alpha \in \PGL_2 : f^{\alpha} = f\}.
    \end{equation*}
    In additional to being special from the existence of these extra symmetries, conjugacy classes with nontrivial automorphisms are exactly the singular points of the moduli space $\M_d$ for $d \geq 3$ \cite{Miasnikov}.

    Since every automorphism must leave certain sets of points invariant (such as the set of periodic points of a given period), the automorphism group of a given $f$ must be a finite subgroup of $\PGL_2$. Sharper bounds than this permutation bound on the size of an automorphism group in terms of $d$ can be obtained but do not concern us here \cite{Levy}. Key to this work is the (classical) classification in characteristic zero of the finite subgroups of $\PGL_2$ restated in modern notation by Silverman \cite{Silverman12}.

    It is important to note that $\Aut(f)$ is well defined on conjugacy classes. In particular, given $\alpha \in \PGL_2$, the action on $\Aut(f)$ defined by $\sigma \mapsto \alpha^{-1} \circ \sigma \circ \alpha$ provides a group isomorphism $\Aut(f) \cong \Aut(f^{\alpha}).$ The conjugacy class of $\Aut(f)$ in $\PGL_2$ is, thus, an invariant of $[f]$ rather than just $f$.
    Denote by $\A_d \subset \M_d$ the set of all conjugacy classes with a nontrivial automorphism. Let $\Gamma \subset \PGL_2$ be a finite subgroup with representation $\rho: \Gamma \to \SL_2$. Denote the set of conjugacy classes whose automorphism group contains a subgroup isomorphic to $\Gamma$ by $\A_d(\Gamma)$.
    Similarly denote $A_d \subset \Rat_d$ as the set of rational maps with nontrivial automorphism group with $\rho(\Gamma)$ as a subgroup, i.e., $A_d(\Gamma) = \{f \in A_d : \Aut(f) \supseteq \rho(\Gamma)\}$. It is important to note that while every finite subgroup of $\PGL_2$ has only one inequivalent representation, the choice of group representation affects the representation in homogeneous coordinates of the map. In particular, questions about the field of definition for elements of $A_d(\Gamma)$ are heavily dependent on the choice of representation of $\Gamma$, e.g., \cite{Hutz15, Silverman12}.

    The goals of this work are two-fold. The first is to give explicit parameterizations for all maps in $\A_d$ for $3 \leq d \leq 4$. The case $d=2$ is well known \cite{Fujimura}. The case $d=3$ can be derived from the unpublished work of West \cite{West}, but he focuses on the parameterizations of $\M_3$ as a whole and a side result is the ability to determine which elements lie in $\A_3$. However, it is nontrivial to move between West's parameterization of $\M_3$ and elements of $\Rat_3$. Further, his methods are not easily applicable in degree $d > 3$. The methods here can be used for any degree, and we produce explicit families in $\Rat_3$ and $\Rat_4$, ($\M_3$ and $\M_4$). The second goal is to study other arithmetical dynamical properties of these families with nontrivial automorphisms focusing on the structure of the set of rational preperiodic points. The motivation for this portion is the uniform boundedness conjecture of Morton and Silverman.
    \begin{conj}[\cite{Silverman7}]
        Fix integers $d\geq 2$, $N\geq 1$, and $D \geq 1$. There is a constant $C(d, N, D)$ such that for all number fields $K/\Q$ of degree at most $D$ and all morphisms $f:\bbP^N \to \bbP^N$ of degree $d$ defined over $K$,
        \begin{equation*}
            \#\PrePer(f,K) \leq C(d, N, D),
        \end{equation*}
        where $\PrePer(f,K)$ is the set of preperiodic points for $f$ defined over $K$.
    \end{conj}
    This conjecture is equivalent to a uniform bound on the number of $K$-rational preperiodic graph structures, where the vertices are $K$-rational preperiodic points and edges connect a point $Q$ to its forward image $f(Q)$ \cite{DoyleSilverman}. While an unconditional bound remains out of reach, we use Poonen \cite{Poonen} and Manes \cite{Manes2} as our model and classify graph structures assuming an upper bound on the period of a rational periodic point.

    We now give a summary of the main results and an outline of the article.
    Section \ref{sect_M3} gives parameterizations of the automorphism locus $\A_3 \subset \M_3$. The classification of finite subgroups of $\PGL_2$ is given in modern notation in Silverman \cite{Silverman12} as
    \begin{itemize}
        \item Cyclic group of order $n$, denoted as $C_n$.
        \item Dihedral group of order $2n$, denoted as $D_{n}$.
        \item Tetrahedral group $\fA_4$ (or alternating group on $4$ elements).
        \item Octahedral group $S_4$ (or symmetric group on $4$ elements).
        \item Icosahedral group $\fA_5$ (or alternating group on $5$ elements).
    \end{itemize}
    Combining this classification with the known dimensions of $\A_3(\Gamma)$ from Miasnikov-Stout-Williams \cite{Miasnikov}, we find families in the parameter space $\Rat_3$ that map finite-to-one onto families in $\M_3$ of the appropriate dimension.
    \begin{theorem} \label{theorem_1}
        \hfill
        \begin{enumerate}
            \item The locus $\A_3(C_4) = \A_3(D_4)$ is a single conjugacy class in $\M_3$ given by $f(z) = \frac{1}{z^3}$. (Corollary \ref{cor_A3C4}).
            \item The locus $\A_3(\fA_4)$ is a single conjugacy class in $\M_3$ given by $f(z) = \frac{z^3-3}{-3z^2}$. This single conjugacy class is exactly the intersection of $\A_3(C_2)$ and $\A_3(C_3)$ (Proposition \ref{M3A4}).
            \item The locus $\A_3(C_3)$ is an irreducible curve in $\M_3$ given by $f_a(z) = \frac{z^3+a}{az^2}$, $a \neq 0$ (Proposition \ref{prop_A3C3}).
            \item The locus $\A_3(D_2)$ consists of two irreducible curves given by
            \begin{equation*}
                f_a(z) = \frac{az^2+1}{z^3+az}, a \neq \pm 1 \quad \text{and} \quad g_a(z) = \frac{az^2-1}{z^3-az}, a \neq \pm 1,
            \end{equation*}
            which intersect at the single point $\A(D_4)=\A(C_4)$ (Proposition \ref{prop_A3D2}).
            \item The locus $\A_3(C_2)$ is the union of two irreducible surfaces given by
            \begin{equation*}
                f_{a,b}(z) = \frac{z^3+az}{bz^2+1}, ab \neq 1 \quad \text{and} \quad g_{a,b}(z) = \frac{az^2+1}{z^3+bz}, ab \neq 1
            \end{equation*}
            which intersect in a curve that is the $f_a$ component of $\A_3(D_2)$ (Proposition \ref{prop_A3C2}).
        \end{enumerate}
    \end{theorem}
    The methods are a combination of invariant theory as utilized in deFaria-Hutz \cite{Hutz15}, explicit forms for maps with cyclic or dihedral automorphism groups from Silverman \cite{Silverman12}, and explicit calculation using the generators of the finite subgroups.

    In the process of studying $\A_3$, we were able to complete the rational realization problem over $\Q$ started in \cite{Hutz16}.
    \begin{theorem}
        Every finite subgroup of $\PGL_2$ can be realized as the automorphism group of a map defined over $\Q$.
    \end{theorem}
    Section \ref{sect_geom_A3} studies these parameterizations as families in moduli space giving explicit maps to sets of periodic point multiplier invariants that are finite-to-one. To construct the multiplier invariants, recall that to each fixed point $Q$ we can compute an algebraic number called the multiplier $\lambda_{Q} = \tilde{f}'(Q)$, where $\tilde{f}$ is a dehomogenization and $'$ represents the derivative. The multiplier is conjugation invariant and the set of fixed points is invariant (as a set) under conjugation. So, taking the elementary symmetric polynomials evaluated on the set of multipliers produces invariants of the moduli space \cite{Silverman9}.
    We can similarly construct invariants from the set of periodic points (or formal periodic points) of any period. We denote these invariants $\sigma_i^{(n)}$, where $n$ denotes the period of the points used and $1 \leq i \leq (\deg(f))^n+1$. These invariants were first studied by Milnor \cite{Milnor} to construct an isomorphism $\M_2 \cong \bbA^2$. In higher degrees, we no longer get an isomorphism to an affine space, but utilizing enough multiplier invariants does produce a finite-to-one map producing an embedding of $\M_3$ into affine space \cite{McMullen}. We can then examine the image of $\A_3(\Gamma)$ as a variety in this affine space.
    In all families but one component of $\A_3(D_2)$, using the fixed points multipliers is sufficient; in that component the multiplier invariants of the 2-periodic points are needed. Theorem \ref{thm_A3_geometry} summarizes the embeddings of the familes $A_3(\Gamma)$ into the moduli space $\M_3$ and the geometric results on $\A_3(\Gamma)$. The details of the embeddings into affine space via the multiplier invariants that are used to analyze the geometry can be found within the referenced propositions.

    \begin{theorem}\label{thm_A3_geometry}
     \hfill
     \begin{enumerate}
         \item \label{thm_A3_geo_item1} The map
         \begin{equation*}
             \bbA^1\setminus\{0\} \to \M_3, a \mapsto \left[ \frac{z^3+a}{az^2} \right]
         \end{equation*}
         is one-to-one. The locus $\A_3(C_3)$ is an irreducible curve of genus zero with one singular point corresponding to $\A_3(\fA_4)$.  (Proposition \ref{prop_c3})
         \item The map
         \begin{equation*}
             \bbA^1\setminus\{\pm 1\} \to \M_3, a \mapsto \left[ \frac{az^2+1}{z^3+az}\right]
         \end{equation*}
         is two-to-one. This component of $\A_3(D_2)$ is a smooth irreducible curve of genus zero. (Proposition \ref{prop_D2_1})
         \item The map
         \begin{equation*}
             \bbA^1\setminus\{\pm 1\} \to \M_3, a \mapsto \left[ \frac{az^2-1}{z^3-az}\right]
         \end{equation*}
         is six-to-one. This component of $\A_3(D_2)$ described by the image is a smooth irreducible curve of genus zero. (Lemma \ref{lem_D2_2}, Proposition \ref{prop_D2_2})
         \item The map
         \begin{equation*}
             \bbA^2\setminus\{ab=1\} \to \M_3, (a,b) \mapsto  \left[ \frac{z^3+az}{bz^2+1}\right]
         \end{equation*}
         is two-to-one. This component of $\A_3(C_2)$ is an irreducible rational singular surface. (Proposition \ref{prop_a3_c2_1_geo})
          \item The map
          \begin{equation*}
             \bbA^2\setminus\{ab=1\} \to \M_3, (a,b) \mapsto  \left[ \frac{az^2+1}{z^3+bz}\right]
         \end{equation*}
        is four-to-one. This component of $\A_3(C_2)$ is an irreducible singular surface. (Lemma \ref{lem_a3_c2_2_geo}, Proposition \ref{prop_a3_c2_2_geo})
     \end{enumerate}
    \end{theorem}

    Section \ref{sect_M4} gives parameterizations of the automorphism locus $\A_4 \subset \M_4$. The methods are similar to Section \ref{sect_M3}.
    \begin{theorem} \label{theorem_2}
        \hfill
        \begin{enumerate}
            \item The locus $\A_4(C_5)=\A_4(D_5)$ is the single conjugacy class given by $f(z) = \frac{1}{z^4}$. (Proposition \ref{prop_A4_C5})
            \item The locus $\A_4(C_4)$ is an irreducible curve in $\M_4$ given by the 1-parameter family $f_k(z) = \frac{z^4 + 1}{kz^3}$ for $k \neq 0$. (Proposition \ref{prop_A4_C4})
            \item The locus $\A_4(D_3)$ is an irredcible curve in $\M_4$ given by the family $f_k(z) = \frac{z^4 + kz}{kz^3+1}$ for $k \neq \pm 1$. (Proposition \ref{prop_A4_D3})
            \item The locus $\A_4(C_3)$ is an irreducible surface in $\M_4$ given by the family $f_{k_1,k_2}(z)=\frac{z^4+k_1z}{k_2z^3+1}$ for $k_1k_2 \neq 1$. (Proposition \ref{prop_A4_C3})
            \item The locus $\A_4(C_2)$ is given by the 3-parameter family $f_{k_1,k_2,k_3}(z) = \frac{z^4+k_1z^2+1}{k_2z^3+k_3z}$ for $k_2^2 + k_3^2 \neq k_1k_2k_3$. (Proposition \ref{prop_A4_C2}).
        \end{enumerate}
    \end{theorem}

    Section \ref{sect_geom_A4} studies these parameterizations as families in moduli space giving explicit maps to sets of periodic point multiplier invariants that are finite-to-one. In all families but $\A_4(C_2)$, using the fixed point multipliers is sufficient; in that family the multiplier invariants of the formal 2-periodic points are needed. Theorem \ref{theorem_5} summarizes the embeddings of the familes $A_4(\Gamma)$ into the moduli space $\M_4$ and the geometric results on $\A_4(\Gamma)$. The details of the embeddings into affine space via the multiplier invariants that are used to analyze the geometry can be found within the referenced propositions.
    \begin{theorem} \label{theorem_5}
     \hfill
     \begin{enumerate}
         \item \label{thm5_item1} The map
         \begin{equation*}
             \bbA^1\setminus\{0\} \to \M_4, k \mapsto \left[ \frac{z^4+1}{kz^3} \right]
         \end{equation*}
         is one-to-one. The locus $\A_4(C_4)$ is an irreducible curve of genus zero with one singular point.  (Lemma \ref{geometry of C_4}, Proposition \ref{prop_A4_C4_geo})
         \item \label{thm5_item2} The map
         \begin{equation*}
             \bbA^1\setminus\{\pm 1\} \to \M_4, k \mapsto \left[ \frac{z^4+kz}{kz^3+1} \right]
         \end{equation*}
         is one-to-one. The locus $\A_4(D_3)$ is an irreducible curve of genus zero with one singular point.  (Lemma \ref{lem_A4_D3_geo}, Proposition \ref{prop_A4_D3_geo})
         \item The map
         \begin{equation*}
             \bbA^2\setminus\{k_1k_2=1\} \to \M_4, (k_1,k_2) \mapsto \left[ \frac{z^4+k_1z}{k_2z^3+1} \right]
         \end{equation*}
         is one-to-one. The locus $\A_4(C_3)$ is a singular irreducible surface.  (Proposition \ref{prop_A4_C3_geo})
        \item The map
         \begin{equation*}
             \bbA^3\setminus\{k_2^2 + k_3^2 = k_1k_2k_3\} \to \M_4, (k_1,k_2,k_3) \mapsto \left[ \frac{z^4+k_1z^2 + 1}{k_2z^3+k_3z} \right]
         \end{equation*}
         is generically two-to-one. (Proposition \ref{lem_A4_C2})
     \end{enumerate}
    \end{theorem}
    Note that the singular point in the curves in parts \eqref{thm5_item1} and \eqref{thm5_item2} still have the appropriate automorphism group, so we are not seeing the phenomenon of changing automorphism group we saw in the singular point in Theorem \ref{thm_A3_geometry} \eqref{thm_A3_geo_item1} (and in the cuspidal cubic which is $\A_2$). However, these two points do correspond to non-conjugate maps with the same set of fixed point multiplier invariants. Since the embeddings in affine space for these families use only the fixed points invariants, in some sense this point is the point of intersection of the image of these two families in $\bbA^4$. This common point is the only (nondegenerate) such point of intersection.

    Having given parameterizations of $\A_3$ and $\A_4$, we turn to studying rational preperiodic structures. From the set of preperiodic points $S$ for $f$, we create a directed graph whose vertices are the points in $S$ and for $P,Q \in S$ there is an edge from $P$ to $Q$ if and only if $f(P) = Q$.
    An important consideration for rationality questions on $\A_d$ is which representative of the conjugacy class is used. Studying $\Q$-rational preperiodic points, conjugaing by an element of $\PGL_2(\Q)$ would not affect the rational preperiodic structure. However, conjugating by an element over a field extension would affect the rational preperiodic structure. This effect on rationality still occurs even when the new map is still defined over the original field of definition, the \emph{rational twist} case; see \cite{Levy2} for uniform boundedness for families of twists. Consequently, we study rational preperiodic structures for the paramater space locus $A_d$.

    Section \ref{sect_M3_preperiodic} classifies the $\Q$-rational preperiodic structures that occur for maps in the locus $A_3$ given by these parameterizations with $\Q$-rational parameter values. The graph structures are specified in the theorems of Section \ref{sect_M3_preperiodic} and summarized in the following theorem.
    \begin{theorem}\label{theorem_3}
        \hfill
        \begin{enumerate}
            \item Single conjugacy classes $A_3(C_4)=A_3(D_4)$, and $A_3(\fA_4)$: Each has one possible rational preperiodic graph structure. (Theorem \ref{prop_rational_single})
            \item $A_3(C_3)$: Assuming there are no points of minimal period $4$ or higher, then there are four possible $\Q$-rational preperiodic graph structures for $f_a(z) = \frac{z^3+a}{az^2}$. Each of these graph structures occurs infinitely often. (Theorem \ref{thm_M3_C3_preperiodic}).
            \item $A_3(D_2)$:
                \begin{enumerate}
                    \item Assuming there are no points of minimal period $4$ or higher, then there are four possible $\Q$-rational preperiodic graph structures for $f_a(z) = \frac{az^2+1}{z^3+az}$. Each of these graph structures occurs infinitely often. (Theorem \ref{thm_A3_D2_1_preperiodic}).
                    \item Assuming there are no points of minimal period $4$ or higher, then there are four possible $\Q$-rational preperiodic graph structures for $f_a(z) = \frac{az^2-1}{z^3-az}$. Each of these graph structures occurs infinitely often. (Theorem \ref{thm_A3_D2_2_preperiodic}).
                \end{enumerate}
            \item $A_3(C_2)$:
            \begin{enumerate}
                \item The family $f_{a,b}(z) = \frac{z^3+az}{bz^2+1}$ has at least 33 different $\Q$-rational preperiodic graph structures. (Section \ref{sect_A3_C2_2_preperiodic}).                \item The family $g_{a,b}(z) = \frac{az^2+1}{z^3+bz}$ has at least 23 different $\Q$-rational preperiodic graph structures. (Section \ref{sect_A3_C2_1_preperiodic}).
            \end{enumerate}
        \end{enumerate}
    \end{theorem}
    The methods of this section are to first classify the occurrence of $\Q$-rational periodic points up to some bound, then to classify the rational preimages of the possible cycle structures. Further, we must classify when distinct connected components can occur for the same choice of parameter. In all these cases, the problems come down to finding all the rational points on a curve. In the case of rational or elliptic curves, this set can be infinite. In the case of higher genus curves, there can only be finitely many rational points (Faltings' Theorem). For the two $A_3(C_2)$ components that are dimension two, the task moves to finding all rational points on surfaces. The tools available for such problems are much more limited, so instead we perform a census of $\Q$-rational preperiodic graph structures with $\Q$-rational parameter values.

    In Section \ref{sect_M4_preperiodic} we study the $\Q$-rational preperiodic structures that occur for maps in the locus $A_4$ given by these parameterizations with $\Q$-rational parameter values. The graph structures are specified in the Theorems of Section \ref{sect_M4_preperiodic} and summarized in the following theorem.
    \begin{theorem}\label{theorem_4}
        \hfill
        \begin{enumerate}
            \item Single conjugacy classes $A_4(C_5)=A_4(D_5)$ have one possible rational preperiodic graph structure. (Theorem \ref{M4_rational_single})
            \item $A_4(C_4)$: Assuming there are no points of minimal period $3$ or higher, then there are four possible $\Q$-rational preperiodic graph structures for $f_k(z) = \frac{z^4+1}{kz^3}$. Each of these graph structures occurs infinitely often. (Theorem \ref{theorem_A4_C4_preperiodic}).
            \item $A_4(D_3)$: Assuming there are no points of minimal period $3$ or higher, then there are five possible $\Q$-rational preperiodic graph structures for $f_k(z) = \frac{z^4+kz}{kz^3+1}$ with possibly finitely many exceptional values of the parameter $k$. Each of these graph structures occurs infinitely often. (Theorem \ref{thm_A4_D3_preperiodic}).
            \item $A_4(C_3)$: The family $f_{k_1,k_2}(z) = \frac{z^4+k_1z}{k_2z^3+1}$ has at least 13 different $\Q$-rational preperiodic graph structures. (Section \ref{sect_A4_C3_preperiodic}).
            \item $A_4(C_2)$: The family $f_{k_1,k_2,k_3}(z) = \frac{z^4+k_1z^2+1}{k_2z^3 + k_3z}$ has at least 55 different $\Q$-rational preperiodic graph structures. (Section \ref{sect_A4_C2_preperiodic}).
        \end{enumerate}
    \end{theorem}
    The methods parallel those in Section \ref{sect_M3_preperiodic}, but the increase in degree causes an increase in the complexity of the curves. The finitely many possible exceptions for $A_4(D_3)$ correspond to the rational points of two explicitly given genus 6 curves.

    Finally, in Section \ref{sect_concluding} we make some remarks on further problems.

    The computations were done primarily in Magma \cite{magma} and Sage \cite{sage}. Code used for the computations can be found in the arXiv version\footnote{\url{https://arxiv.org/abs/2007.15483}} of this article.

\section{Automorphism loci in $\M_3$} \label{sect_M3}

    In this section we determine $\A_3$, the conjugacy classes of degree $3$ endomorphisms of $\bbP^1$ that have a nontrivial automorphism. We will always be working over $\overline{\Q}$.

    We know that if $\Gamma \subset \Aut(f)$ for some $f \in \Rat_d$, then $\Gamma$ must be isomorphic to $C_n$, $D_n$, $\fA_4$, $S_4$, or $\fA_5$; see \cite{Silverman12}.
    Also, from the classification of finite subgroups of $\PGL_2$, all isomorphic subgroups are in fact conjugate to each other; see \cite[Lemma 2.3]{Miasnikov}. Thus, given a group, we can simply fix an $\SL_2$ representation and work only with this representation.

    From previous work of Miasnikov, Stout, and Williams \cite{Miasnikov} (henceforth abbreviated as MSW), we are able to determine the dimension of $\A_3(\Gamma)$ for any $\Gamma$.
    \begin{lemma} \label{lem_containment}
        When $\deg(f)=3$, $\Aut(f)$ must be $C_2$, $C_3$, $C_4$, $D_2$, $D_4$ or $\fA_4$. The dimensions of $\A_3(\Gamma)$ for these groups are given by
        \begin{enumerate}
            \item $\dim \A_3(C_2)=2, \dim \A_3(C_3)=1$, and $\dim\A_3(C_4)=0$.
            \item $\dim \A_3(D_2)=1$, and $\dim \A_3(D_4)=0$.
            \item $\dim \A_3(\fA_4)=0$.
        \end{enumerate}
        Moreover,
        \begin{enumerate}
        \setcounter{enumi}{3}
            \item $\A_3(D_4)\subset \A_3(C_4)\subset \A_3(C_2)$.
            \item $\A_3(D_2)\subset \A_3(C_2)$.
            \item $\A_3(\fA_4)\subset \A_3(C_3)$ and $\A_3(\fA_4)\subset \A_3(D_2) \subset \A_3(C_2)$.
        \end{enumerate}
    \end{lemma}
    \begin{proof}
        The groups that occur and the dimensions are calculated in Section 2 of MSW \cite{Miasnikov}. The second half follows from the observation that if $G$ is a subgroup of $H$, then $\A_d(H)\subseteq \A_d(G)$.
    \end{proof}
    To determine $\A_3(\Gamma)$, we proceed in two steps. First we find a family in $A_3 \subset \Rat_3$ which parameterizes all maps that have $\Gamma$ contained in their automorphism group. Then we find a normal form for elements in this family. By normal form, we mean that the projection $\pi: \Rat_3 \to \M_3$ is an isomorphism on this family, i.e., every member of the family represents a distinct conjugacy class. However, in practice we typically end up with a finite-to-one projection so that the dimension of the family in parameter space is the same as the dimension in moduli space. Stated more precisely, given a group $\Gamma$, we want to find a $k$-parameter family in $\Rat_3$ such that $k=\dim \A_3(\Gamma)$ and for all maps $f \in A_3(\Gamma)$, $f$ is conjugate to some member of this $k$-parameter family.

\subsection{$\A_3(C_4)$ and $\A_3(D_4)$}
    From Silverman's classification of maps with cyclic automorphism groups \cite[Proposition 7.3]{Silverman12}, we know that if a degree 3 map $f$ has a $C_4$ symmetry generated by $z \mapsto iz$, then $f$ must be of the form $\frac{1}{az^3}$ for some $a\in \overline{\Q}$. This family of maps is actually a family of twists corresponding to a single conjugacy class in $\M_3$.
    \begin{prop}
        Let $f \in \Rat_3$. If $C_4 \subseteq \Aut(f)$, then $f$ is conjugate to $\frac{1}{z^3}$.
    \end{prop}
    \begin{proof}
        Let $\varphi(z)=\frac{1}{z^3}$ and $f_a(z) = \frac{1}{az^3}$ for any nonzero $a \in \overline{\Q}$. Consider the M\"obius transform $\alpha(z)=pz$, where $p^4=a$. We can compute
        \begin{equation*}
            \varphi^\alpha=(\alpha^{-1}\circ\varphi\circ\alpha)(z)=\frac{\frac{1}{(pz)^3}}{p}=\frac{1}{p^4z^3}=\frac{1}{az^3} = f_a.\qedhere
        \end{equation*}
    \end{proof}
    Note that the automorphism group of $\frac{1}{z^3}$ is actually all of $D_4$ with the extra symmetries coming from $z \mapsto \frac{1}{z}$. By the containments from Lemma \ref{lem_containment}, we have the following corollary.
    \begin{cor}\label{cor_A3C4}
        The automorphism locus $\A_3(C_4) = \A_3(D_4)$ is a single point in $\M_3$ given by the conjugacy class of $\frac{1}{z^3}$.
    \end{cor}

\subsection{$\A_3(\fA_4)$}
    We know that $C_2$ and $C_3$ are subgroups of $\fA_4$ so that $\A_3(\fA_4) \subseteq \A_3(C_2) \cap \A_3(C_3)$. This intersection turns out to be a single conjugacy class giving $\A_3(\fA_4)$.
    \begin{prop}\label{M3A4}
        The locus $\A_3(\fA_4)$ is a single point in moduli space represented by $f(z) = \frac{\sqrt{-3}z^2 - 1}{z^3-\sqrt{-3}z}$. This point is exactly the intersection of $\A_3(C_2)$ and $\A_3(C_3)$ in $\M_3$.
    \end{prop}
    \begin{proof}
        We start by computing the intersection $\A_3(C_2) \cap \A_3(C_3)$. To determine this intersection, we use a Groebner basis calculation similar to the end of the proof of Proposition \ref{prop_A3C2}. Recall that $\A_3(C_3)$ is given by the family $g_a(z) = \frac{z^3+a}{az^2}$.

        Consider the component of $\A_3(C_2)$ given by $f_{k_1,k_2}(z) = \frac{z^3 + k_1z}{k_2z^2+1}$. Taking a generic element of $\PGL_2$, $\alpha = \begin{pmatrix}a&b\\c&d\end{pmatrix}$, we consider the system of equations obtained from the coefficients of
        \begin{equation*}
            f_{k_1,k_2}^{\alpha} = tg_a
        \end{equation*}
        for some nonzero constant $t$. Note that $f^{\alpha}$ is degree $3$, so that if $f^{\alpha} = g$, then its numerator and denominator differ from those of $g$ only by a constant multiple. The resulting equations have no solutions.

        Now consider the component given by $f_{k_1,k_2} = \frac{k_1z^2 + 1}{z^3+k_2z}$. Again we set $f^{\alpha} = tg$ and solve the resulting system of equations. In this case, there are two solutions
        \begin{equation*}
            a=-3, \quad (k_1,k_2) = (\pm \sqrt{3}, \mp \sqrt{3}).
        \end{equation*}
        These two choices of $(k_1,k_2)$ are in fact conjugate to each other. Furthermore, the automorphism group of $f_{k_1,k_2}$ has order $12$ at these values. The only group with order $12$ in our list of possible automorphism groups is $\fA_4$. Thus, $\A_3(C_2)$ and $\A_3(C_3)$ intersect at a single point in the moduli space, and the point of intersection is $\A_3(\fA_4)$.
    \end{proof}
\begin{code}
\begin{lstlisting}[language=Python]
# Intersection of C_2, C_3 and D_2
# Intersection between (z^3+k1z)/(k2z^2+1) and (z^3+k)/(kz^2)
# empty intersection
R.<a,b,c,d,t,s,k1,k2,k3>=PolynomialRing(QQ, order='lex')
P.<x,y>=ProjectiveSpace(FractionField(R),1)
f=DynamicalSystem_projective([x^3 + k1*x*y^2, k2*x^2*y+y^3])
g=DynamicalSystem_projective([x^3 + k3*y^3, k3*x^2*y])
m=matrix(FractionField(R),2,[a,b,c,d])
f2=f.conjugate(m)
f2.scale_by(a*d-b*c)
gens=(f2[0] - t*g[0]).coefficients() + (f2[1] - t*g[1]).coefficients() + [1-s*t]
I=R.ideal(gens)
G=I.groebner_basis()
G

# Intersection between (k1z^2+1)/(z^3+k2z) and (z^3+k)/(kz^2)
# two points in intersection
R.<a,b,c,d,t,s,k1,k2,k3>=PolynomialRing(QQ, order='lex')
P.<x,y>=ProjectiveSpace(FractionField(R),1)
f=DynamicalSystem_projective([k1*x^2*y + y^3, x^3+k2*x*y^2])
g=DynamicalSystem_projective([x^3 + k3*y^3, k3*x^2*y])
m=matrix(FractionField(R),2,[a,b,c,d])
f2=f.conjugate(m)
f2.scale_by(a*d-b*c)
gens=(f2[0] - t*g[0]).coefficients() + (f2[1] - t*g[1]).coefficients() + [1-s*t]
I=R.ideal(gens)
G=I.groebner_basis()
G

# the two intersection points are conjugate
A.<z>=AffineSpace(QQbar,1)
f1=DynamicalSystem_affine((QQbar(-sqrt(3))*z^2+1)/(z^3+(QQbar(sqrt(3))*z)))
f2=DynamicalSystem_affine((QQbar(sqrt(3))*z^2+1)/(z^3+(QQbar(-sqrt(3))*z)))
F1=f1.homogenize(1)
F2=f2.homogenize(1)
print(len(F1.automorphism_group()))
F2.is_conjugate(F1)
\end{lstlisting}
\end{code}
    As mentioned in the introduction, the choice of representation affects the field of definition and it is natural to ask whether a map defined over $\Q$ has $\fA_4$ as automorphism group. Hutz-deFaria \cite{Hutz15} proved that the ``standard'' representation of $\fA_4$ given by Silverman \cite{Silverman12} does not have a map defined over $\Q$ with tetrahedral automorphism group. However, when computing the intersection of $\A(C_2)$ and $\A(C_3)$ in Proposition \ref{M3A4}, we discovered that the map
    \begin{equation*}
        f(z) = \frac{z^3-3}{-3z^2}
    \end{equation*}
    has tetrahedral automorphism group (and is conjugate to the maps in Proposition \ref{M3A4}), which can be verified with direct computation by the algorithm of Faber-Manes-Viray \cite{FMV} as implemented in Sage. This example completes the construction started in Hutz-deFaria to show that every finite subgroup of $\PGL_2$ can be realized as the automorphism group of a map defined over $\Q$.
    \begin{theorem}
        Every conjugacy class of a finite subgroup of $\PGL_2$ can be realized as the automorphism group of a map defined over $\Q$.
    \end{theorem}

\subsection{$\A_3(C_3)$}
    From Silverman's classification of maps with cyclic automorphism groups \cite[Proposition 7.3]{Silverman12}, we know that if $f \in \Rat_3$ has a $C_3$ symmetry, then $f$ belongs to one of the following two families:
    \begin{equation*}
        f_1(z)= \frac{az^3+b}{cz^2} \quad \text{or} \quad
        f_2(z) = \frac{az}{bz^3+c}.
    \end{equation*}
    Since all parameters must be nonzero to be a degree $3$ map, we can divide through by $a$ and $b$, respectively, to get
    \begin{equation} \label{eq1}
        f_1(z) = \frac{z^3+a_1}{a_2z^2} \quad \text{or} \quad
        f_2(z) = \frac{b_1z}{z^3+b_2},
    \end{equation}
    where $a_1,a_2,b_1,b_2 \in \overline{\Q} \backslash \{0\}$. These two families in parameter space are exactly $A_3(C_3)$, and we next show they project onto the same irreducible curve in moduli space.

    Now that we are working with positive dimensional components of $\A_3$, we set some terminology for parameterized families. A \emph{$k$-parameter family} of rational degree $d$ maps is a set of rational maps $\bbP^1\rightarrow \bbP^1$ of the form
    \begin{equation*}
        f(x,y) = (F_0x^d + F_1x^{d - 1}y + \ldots + F_d y^d: F_{d+1}x^d + \ldots + F_{2d + 1}y^d)
    \end{equation*}
    for some polynomials $F_0,\ldots,F_{2d + 1}$ in indeterminants $u_1,\ldots,u_k$, where we may require that some of the $F_j$ never vanish and the two defining polynomials of $f$ have no common factors for all $(u_1,\ldots,u_k)$ considered. In particular, a $k$-parameter family is parameterized by an open subset of $\bbA^k$.
    \begin{lemma} \label{lem_k_param_image}
        A $k$-parameter family in $\Rat_d$ induces a map
        \begin{equation*}
            F:U\to \M_d
        \end{equation*}
        for an open subset $U\subseteq \bbA^k$ and this map is a morphism of varieties. In particular, the image is irreducible and so is its closure.
    \end{lemma}
    \begin{proof}
        A $k$-parameter family is represented in the form
        \begin{equation*}
            f(x,y) = (F_0x^d + F_1x^{d - 1}y + \ldots + F_d y^d: F_{d+1}x^d + \ldots + F_{2d + 1}y^d)
        \end{equation*}
        where the $F_j$ are polynomials in parameters $(u_1,\ldots,u_k)$.
        These $F_j$ determine a morphism of varieties $\phi: \bbA^k \rightarrow \bbA^{2d + 2}$ defined as
        \begin{equation*}
            (u_1,\ldots,u_k)\mapsto (F_0(u_1,\ldots,u_k), \ldots, F_{2d+1}(u_1,\ldots,u_k)).
        \end{equation*}

        The complement of the origin $\bbA^{2d + 2}\setminus \{0\}$ is an open subset and, thus, $$U^\prime := \phi^{-1}(\bbA^{2d + 2}\setminus \{0\})$$ is also open in $\bbA^k$. So restricting $\phi$ gives a morphism $U^\prime\rightarrow \bbA^{2d + 2}\setminus\{0\}$.

        There is a projection morphism $\bbA^{2d + 2}\setminus\{0\}\rightarrow \bbP^{2d + 1}$ defined \begin{equation*}
            (a_0,\ldots, a_{2d + 1}) \mapsto (a_0:\ldots : a_{2d + 1}).
        \end{equation*}
        Composing the restriction of $\phi$  with this projection yields a morphism $U^\prime\rightarrow \bbP^{2d + 1}$.

        By definition, $\Rat_d$ is the complement of a hypersurface in $\bbP^{2d + 1}$ and is open itself. Therefore, its inverse image by the composite map from $U^\prime$ is also an open subset $U$ of $\bbA^k$.

        One may further refine $U$ as needed by intersecting it with the inverse image of the complement of the hyperplane in $\bbP^{2d + 1}$ defined by the vanishing of a given $a_j$. This amounts to requiring the corresponding coefficient of the rational maps in the $k$-parameter family to be nonzero.

        Finally, geometric invariant theory can be used to show the quotient map $\Rat_d\rightarrow \M_d$ \cite[Theorem 4.36]{Silverman10} is a morphism; thus, altogether we obtain the desired map $F: U\rightarrow \M_d$.

        From point set topology, nonempty open subsets of varieties are irreducible, the closure of an irreducible subset of a larger topological space is irreducible, and the continuous image of an irreducible set is irreducible. Therefore, the image of $F$ and its closure are both irreducible.
    \end{proof}

    For example, with the notation from Lemma \ref{lem_k_param_image}, the $1$-parameter family of rational maps of the form $f_a(z) = \frac{az}{az^3 + 1}$ written projectively as
    \begin{equation*}
        (x:y)\mapsto (axy^2:ax^3 + y^3),
    \end{equation*}
    where $a\neq 0$ corresponds to taking
    \begin{align*}
        F_0&=F_1=F_3=F_5=F_6=0\\
        F_2&=F_4=a\\
        F_7&=1.
    \end{align*}
    We have refined the open subset $U$ from the proof by intersecting it with the inverse image of the complement of the hyperplane defined by $a_2=0$ (or, equivalently, that defined by $a_4=0$), if the coordinates for $\bbP^7$ are $(a_0:a_1:a_2:\ldots :a_7)$. Lemma \ref{lem_k_param_image} then shows that the set of all conjugacy classes of maps of the form $f_a(z) = \frac{az}{az^3 + 1}$ is an irreducible subset of $\M_3$.

    Let $f \in \Rat_d$. To each periodic point $z$ of period $n$ for $f$, we can compute the multiplier as $\lambda_z := (\tilde{f}^n)'(z)$ for a dehomogenization $\tilde{f}$ of $f$. The set of all multipliers of a given period is invariant under conjugation (but may be permuted) so applying the elementary symmetric polynomials to this collection of multipliers produces a set of complex numbers that are invariants of the conjugacy class---these \emph{multiplier invariants} have also been called \emph{Milnor parameters}. See Silverman \cite[\S 4.5]{Silverman10} for more details on dimension 1 and Hutz \cite{Hutz24} for dimension $>1$. The key fact we need is that they are invariants of the conjugacy class. In particular, if $f,g \in \Rat_d$ have different multiplier invariants for any fixed $n$, then they cannot be conjugate. However, the converse is not true; having the same multiplier invariants for some (or all) $n$ does not make two functions conjugate.

    \begin{prop}\label{prop_A3C3}
        $\A_3(C_3)$ is an irreducible curve in $\M_3$ given by either one-parameter family,
        \begin{equation*}
            f_a(z) = \frac{z^3 + a}{az^2} \quad \text{or} \quad g_b(z) = \frac{bz}{bz^3+1}.
        \end{equation*}
    \end{prop}
    \begin{proof}
        We first show that we can conjugate the two-parameter families in equation \eqref{eq1} to the one-parameter families given in the statement. We first consider $f_1$. Let $\beta^3=\frac{a_2}{a_1}$, and consider the M\"obius transform $\alpha(z)=\frac{1}{\beta}z$ so that $\alpha^{-1}(z)=\beta z$. We can compute
        \begin{equation*}
            f_1^\alpha(z)=\left(\frac{\left(\frac{1}{\beta}z\right)^3+a_1}{a_2\left(\frac{1}{\beta}z\right)^2}\right)\beta=\left(\frac{\frac{z^3}{\beta^3}+a_1}{a_2\frac{z^2}{\beta^2}}\right)\beta=\frac{\beta z^3+a_1\beta^4}{a_2\beta z^2}=\frac{z^3+a_1\beta^3}{a_2 z^2}=\frac{z^3+a_2}{a_2z^2}.
        \end{equation*}
        Thus, the two-parameter family $\frac{z^3+a_1}{a_2z^2}$ is conjugate to the one-parameter family $f_{a_2}$. Similarly for $f_2$, we can set $\gamma^3=b_1$, and consider the M\"obius transform $\alpha(z)=\gamma z$ and its inverse $\alpha^{-1}(z)=\frac{1}{\gamma}z$. We can compute
        \begin{equation*}
            f_2^\alpha(z)=\left(\frac{b_1\gamma z}{\gamma^3 z^3+b_2}\right)\frac{1}{\gamma}=\frac{b_1z}{\gamma^3z^3+b_2}
            =\frac{\frac{b_1}{b_2}z}{\frac{b_1}{b_2}z^3+1}=\frac{bz}{bz^3+1},
        \end{equation*}
        where $b=\frac{b_1}{b_2}$. Thus, the two-parameter family $\frac{b_1z}{z^3+b_2}$ is conjugate to the one-parameter family $g_b$.

        Next, consider the M\"obius transform $\alpha(z)=\frac{1}{z}$. We can compute
        \begin{equation*}
            f_b^\alpha(z)=\frac{1}{\frac{\frac{1}{z^3}+b}{b\frac{1}{z^2}}}=\frac{1}{\frac{bz^3+1}{bz}}=\frac{bz}{bz^3+1}=g_b(z).
        \end{equation*}
        Thus, the two one-parameter families $f_a$ and $g_b$ are conjugate and we can consider either one. In particular, every $f$ with a $C_3$ symmetry is conjugate to some member of the family $f_a(z) = \frac{z^3+a}{az^2}$.
    \end{proof}

\subsection{$\A_3(D_2)$}
    The representation of $D_2$ we will be working with is
    \begin{equation*}
        \left\{\begin{bmatrix}1&0\\0&1\end{bmatrix},\quad \begin{bmatrix}-i&0\\0&i\end{bmatrix},\quad\begin{bmatrix}0&i\\i&0\end{bmatrix},\quad \begin{bmatrix}0&-1\\1&0\end{bmatrix}\right\},
    \end{equation*}
    where $\begin{bmatrix}-i&0\\0&i\end{bmatrix}$ and $\begin{bmatrix}0&i\\i&0\end{bmatrix}$ are the two generators.
    They correspond to M\"obius transforms $\alpha_1=-z$ and $\alpha_2=\frac{1}{z}$.
    \begin{prop}\label{prop_A3D2}
        $\A_3(D_2)$ consists of two irreducible curves which intersect at the single point $\A_3(D_4)=\A_3(C_4)$.
    \end{prop}
    \begin{proof}
        Let $f(z)=\frac{a_1z^3+a_2z^2+a_3z+a_4}{b_1z^3+b_2z^2+b_3z+b_4}$ be a degree $3$ map. If $f$ has a $D_2$ symmetry, we can obtain restrictions on the coefficients from the equations $f=f^{\alpha_1}$ and $f=f^{\alpha_2}$.

        We get 14 equations and compute the irreducible components of the variety generated by these equations as a subvariety of $\bbP^7$. There are $8$ irreducible components, but only $4$ of them correspond to degree $3$ maps. Thus, $f\in \Rat_3$ has a $D_2$ symmetry (with this representation) if and only if it has one of the following four forms:
        \begin{align*}
            &f_1(z)=\frac{k_1z^2-1}{z^3-k_1z}\quad\quad\quad f_2(z)=\frac{k_2z^2+1}{z^3+k_2z}\\
            &f_3(z)=\frac{z^3+k_3z}{k_3z^2+1}\quad\quad \quad f_4(z)=\frac{z^3-k_4z}{k_4z^2-1},
        \end{align*}
        where none of the parameters can be $1$. We will show that $f_2$, $f_3$, and $f_4$ are all conjugate to each other, whereas $f_1$ is not in general conjugate to the rest. First, consider the families $f_3$ and $f_2$.
        Conjugating $f_2$ by the element $\alpha=\begin{bmatrix}1&-i\\-i&1\end{bmatrix}$, we get
        \begin{equation*}
            f_2^\alpha(z)=\frac{z^3+\frac{k_2+3}{k_2-1}z}{\frac{k_2+3}{k_2-1}z^2+1}.
        \end{equation*}
        Thus, $f_2$ is conjugate to the map $f_3(z)=\frac{z^3+k_3z}{k_3z^2+1}$, where $k_3=\frac{k_2+3}{k_2-1}$. Now we show that $f_4$ is conjugate to $f_2$. For every $k_4\neq 1$, set $k_2=\frac{k_4+3}{k_4-1}$ and work with $f_2(z)=\frac{\frac{k_4+3}{k_4-1}z^2+1}{z^3+\frac{k_4+3}{k_4-1}}$.
        Conjugating this map by $\alpha=\begin{bmatrix}i&i\\-1&1\end{bmatrix}$, we get $f_2^\alpha(z)=\frac{z^3 - k_4z}{k_4z^2 - 1}=f_4(z)$.
        Thus, $f_2$, $f_3$, and $f_4$ are all conjugate to each other.

        To see that $f_1$ is not conjugate to the rest, we compute the multiplier invariants for the fixed points of $f_1$. Interestingly, they are independent of the parameter $k_1$:
        \begin{equation*}
            \sigma_1=-12 \quad \sigma_2=54\quad \sigma_3=-108\quad \sigma_4=81.
        \end{equation*}
        However, the multiplier invariants for the fixed points of $f_4$ are dependent on the parameter $k_4$:
        \begin{align*}
            \sigma_1(k_4)&=\frac{2k_4^2 + 6}{k_4 + 1}\\
            \sigma_2(k_4)&=\frac{k_4^4 - 2k_4^3 + 10k_4^2 + 6k_4 + 9}{k_4^2 + 2k_4 + 1}\\
            \sigma_3(k_4)&=\frac{-2k_4^4 + 6k_4^3 - 6k_4^2 + 18k_4}{k_4^2 + 2k_4 + 1}\\
            \sigma_4(k_4)&=\frac{k_4^4 - 6k_4^3 + 9k_4^2}{k_4^2 + 2k_4 + 1}.
        \end{align*}
        Solving the equations
        \begin{equation*}
            \sigma_1(k_4) = -12,\quad \sigma_2(k_4) = 54, \quad \sigma_3(k_4)=-108, \quad\sigma_4(k_4) = 81,
        \end{equation*}
        we see that the only possible conjugacy occurs for $k_4=-3$. Setting $k_4=-3$, we compute the multiplier invariants for the periodic points of period $2$ for $f_1$ and $f_4$, denoted $\sigma_i^{(2)}$ for $1 \leq i \leq 10$. For $f_1$ the invariants depend on $k_1$, and for $f_4$ at $k_4=-3$ they are constants. Solving the equations $\sigma_i^{(2)}(f_1) = \sigma_i^{(2)}(f_4)$, we see that the three values $k_1 \in \{3,0,-3\}$ are all conjugate to $f_4$ with $k_4=-3$. Further, note that the three elements of the $f_1$ family with $k_1 \in \{3,0,-3\}$ are all in the same conjugacy class as $\frac{1}{z^3}$. These are the only conjugacies in the families $f_1$ and $f_4$.

        In summary, every degree $3$ map $f$ with $D_2$ symmetry is conjugate to some member of the family $f_1(z) = \frac{k_1z^2-1}{z^3-k_1z}$ or the family $f_2(z)=\frac{k_2z^2+1}{z^3+k_2z}$, and we can conclude that $\A_3(D_2)$ has two irreducible components which intersect at the calculated conjugacy class.
    \end{proof}
\begin{code}
\begin{lstlisting}[language=Python]
#Irreducible components of the subscheme for D2 symmetry
A.<a1,a2,a3,a4,b1,b2,b3,b4>=ProjectiveSpace(QQ,7)
R=A.coordinate_ring()
I=R.ideal([a1*a4-b1*b4,
           (a1*a3 + a2*a4)-(b1*b3 + b2*b4),
           (a1*a2 + a2*a3 + a3*a4)-(b1*b2 + b2*b3 + b3*b4),
           (a1^2 + a2^2 + a3^2 + a4^2)-(b1^2 + b2^2 + b3^2 + b4^2),
           (a1*a2 + a2*a3 + a3*a4)-(b1*b2 + b2*b3 + b3*b4),
           (a1*a3 + a2*a4)-(b1*b3 + b2*b4),
           a1*a4-b1*b4,
           a1*b1-(-a1*b1),
           (-a2*b1 + a1*b2)-(-a2*b1 + a1*b2),
           (a3*b1 - a2*b2 + a1*b3)-(-a3*b1 + a2*b2 - a1*b3),
           (-a4*b1 + a3*b2 - a2*b3 + a1*b4)-(-a4*b1 + a3*b2 - a2*b3 + a1*b4),
           (-a4*b2 + a3*b3 - a2*b4)-(a4*b2 - a3*b3 + a2*b4),
           (-a4*b3 + a3*b4)-(-a4*b3 + a3*b4),
           -a4*b4-a4*b4])
print(A.subscheme(I).dimension())
irreds = A.subscheme(I).irreducible_components()
for i in irreds:
    print(i)
    print(i.dimension())
    print("")

#f_2 and f_3 are conjugate to each other
R.<i>=QuadraticField(-1)
FF.<k2> = PolynomialRing(R,1)
A.<z> = AffineSpace(FF,1)
f3 = DynamicalSystem_affine((z^3+(k2+3)/(k2-1)*z)/((k2+3)/(k2-1)*z^2+1))
f2 = DynamicalSystem_affine((k2*z^2+1)/(z^3+k2*z))
m = matrix([[1,-1*i],[-1*i,1]])
f2.conjugate(m)==f3

#f_2 and f_4 are conjugate to each other
RR.<i>=QuadraticField(-1)
FF.<k4> = PolynomialRing(RR,1)
P.<x,y> = ProjectiveSpace(FractionField(FF),1)
f2 = DynamicalSystem_projective([(((k4+3)/(k4-1))*x^2*y+y^3),(x^3+((k4+3)/(k4-1))*x*y^2)])
f4 = DynamicalSystem_projective([(x^3-k4*x*y^2),(k4*x^2*y-y^3)])
m = matrix([[i,i],[-1,1]])
f2.conjugate(m)==f4

#sigma invariant calculation
R.<a> = QQ[]
P.<x> = R[]
P = FractionField(P)
A = AffineSpace(P,1)
f1 = DynamicalSystem_affine((a*x^2-1)/(x^3-a*x))
F1 = f1.homogenize(1)
print(F1.sigma_invariants(1))
print("")
R.<k4> = QQ[]
P.<z> = R[]
P = FractionField(P)
A = AffineSpace(P,1)
f2 = DynamicalSystem_affine((z^3-k4*z)/(k4*z^2-1))
F2 = f2.homogenize(1)
print(F2.sigma_invariants(1))


#f_4 and f_2 conjugacy
f1s1=F1.sigma_invariants(1)
f2s1=F2.sigma_invariants(1)
I=R.ideal([f2s1[i].numerator() - R(f1s1[i])*f2s1[i].denominator() for i in range(len(f1s1))])
I.gens()[0].factor()


R.<k1>=QQ[]
A.<z> = AffineSpace(R,1)
f2 = DynamicalSystem_affine((k1*z^2-1)/(z^3-k1*z))
F2 = f2.homogenize(1)
f2s2=F2.sigma_invariants(2)
k=-3
f4 = DynamicalSystem_affine((z^3-k*z)/(k*z^2-1))
F4 = f4.homogenize(1)
f4s2=F4.sigma_invariants(2)
I=R.ideal([f4s2[i].numerator()*f2s2[i].denominator() - f2s2[i].numerator()*f4s2[i].denominator() for i in range(len(f2s1))])
I.gens()[0].factor()

#check is_conjugate:
P.<x,y>=ProjectiveSpace(QQ, 1)
f4=DynamicalSystem([x^3+3*x*y^2,-3*x^2*y-y^3])
R.<k>=QQ[]
P.<x,y>=ProjectiveSpace(R, 1)
f2=DynamicalSystem([k*x^2*y-y^3,x^3-k*x*y^2])
#check is_conjugate:
for t in [3,0,-3]:
    f2t = f2.specialization({k:t})
    print(f4.change_ring(QQbar).is_conjugate(f2t.change_ring(QQbar)))


#Intersection locus of the two irreducible components of D_2
R.<a> = QQ[]
P.<z> = R[]
P = FractionField(P)
A = AffineSpace(P,1)
f2 = DynamicalSystem_affine((z^3-a*z)/(a*z^2-1))
F2 = f2.homogenize(1)
print(F2.sigma_invariants(1))

A.<a>=AffineSpace(QQ,1)
R=A.coordinate_ring()
I=R.ideal([(2*a^2 + 6)-(a + 1)*(-12),
           (a^4 - 2*a^3 + 10*a^2 + 6*a + 9)-(a^2 + 2*a + 1)*(54),
           (-2*a^4 + 6*a^3 - 6*a^2 + 18*a)-(a^2 + 2*a + 1)*(-108),
           (a^4 - 6*a^3 + 9*a^2)-(a^2 + 2*a + 1)*(81)])
A.subscheme(I).dimension()
A.subscheme(I).irreducible_components()

A.<x>=AffineSpace(QQbar,1)
f=DynamicalSystem_affine((x^3+3*x)/(-3*x^2-1))
F=f.homogenize(1)
g=DynamicalSystem_affine(1/x^3)
G=g.homogenize(1)
F.is_conjugate(G)
\end{lstlisting}
\end{code}

\subsection{$\A_3(C_2)$}
    \begin{prop} \label{prop_A3C2}
        The locus $\A_3(C_2)$ is the union of two irreducible surfaces which intersect in the $f_a(z)$ component of $\A_3(D_2)$.
    \end{prop}
    \begin{proof}
        From Silverman's classification of maps with cyclic automorphism groups \cite[Proposition 7.3]{Silverman12}, we know that if $f$ is degree $3$ with a $C_2$ symmetry; then $f$ is conjugate to one of the forms
        \begin{equation*}
            f(z) = \frac{k_1z^3+k_2z}{k_3z^2+k_4} \quad \text{or} \quad
            f(z) = \frac{l_1z^2+l_2}{l_3z^3+l_4z},
        \end{equation*}
        where $k_1k_4\neq 0$ and $l_2l_3\neq 0$. Thus, we can divide through by $k_1$ and $l_3$, respectively, to obtain:
        \begin{equation*}
            f(z)= \frac{z^3+a_1z}{a_2z^2+a_3} \quad \text{or} \quad
            f(z) = \frac{b_1z^2+b_2}{z^3+b_3z}.
        \end{equation*}
        We first reduce these two three-parameter families to two-parameter families. Consider the first family $f_1(z)=\frac{z^3+a_1z}{a_2z^2+a_3}$. Conjugating it by the M\"obius transform $\alpha(z)=\sqrt{a_3}z$, we get
        \begin{equation*}
            f_1^\alpha(z)=\frac{1}{\sqrt{a_3}}\left(\frac{(\sqrt{a_3}z)^3+a_1\sqrt{a_3}z}{a_2a_3z^2+a_3}\right)=\frac{a_3z^3+a_1z}{a_2a_3z^2+a_3}=\frac{z^3+\frac{a_1}{a_3}z}{a_2z^2+1}.
        \end{equation*}
        Thus, for every choice of $a_1,a_2,a_3$, we can set $k_1=\frac{a_1}{a_3}$ and $k_2=a_2$; note that $a_3 \neq 0$ since $k_4 \neq 0$.  Thus, the family $f_1(z)=\frac{z^3+a_1z}{a_2z^2+a_3}$ is conjugate to the family $\varphi(z)=\frac{z^3+k_1z}{k_2z^2+1}$.

        Similarly, given $f_2(z)=\frac{b_1z^2+b_2}{z^3+b_3z}$, we can set $\beta^4=\frac{1}{b_2}$ and consider the M\"obius transform $\alpha(z)=\frac{1}{\beta}z$. We can compute
        \begin{equation*}
            f_2^\alpha(z)=\beta\left(\frac{b_1\frac{z^2}{\beta^2}+b_2}{\frac{z^3}{\beta^3}+b_3\frac{z}{\beta}}\right)=\frac{b_1\beta^2 z^2+b_2\beta^4}{z^3+b_3\beta^2 z}=\frac{b_1\beta^2z^2+1}{z^3+b_3\beta^2z}.
        \end{equation*}
        Thus, for any choice of $b_1,b_2,b_3$, we can set $l_1=b_1\beta^2$ and $l_2=b_3\beta^2$. So the family $f_2(z)=\frac{b_1z^2+b_2}{z^3+b_3z}$ is conjugate to the family $\phi(z)=\frac{l_1z^2+1}{z^3+l_2z}$.

        Finally, we need to determine if these two-parameter families are ever conjugate to each other. Taking a generic $\PGL_2$ element $\alpha = \begin{pmatrix}a&b\\c&d\end{pmatrix}$ and conjugating, we can setup a system of equations for $\varphi^{\alpha}$ to be of the form $\phi$ by equating the known coefficients (up to scalar multiple). This produces a variety with three irreducible components:
        \begin{align*}
            X_1&=V(l_2 + 3, l_1 + 3, k_1 + k_2, k_2^2 + 9),\\
            X_2&=V(l_1 - l_2, k_1 - k_2,  k_2l_2 - k_2 - l_2 - 3),\\
            X_3&=V(l_1 - l_2, k_1 - k_2,  k_2l_2 - k_2 + l-2 + 3).
        \end{align*}
        The component $X_1$ is dimension $0$ containing the three pairs of maps with
        \begin{equation*}
            (l_1,l_2,k_1,k_2) \in \{(-3,-3,0,0), (-3,-3,3i,-3i),(-3,-3,-3i,3i)\}.
        \end{equation*}
        All four such maps are conjugate to $\frac{1}{z^3}$. The components $X_2$ and $X_3$ we now show are the component $f_a(z)$ of $\A_3(D_2)$. Note that when $l_1=l_2$, this is exactly the family $f_a(z)$. The last defining equation for each variety solves as
        \begin{equation*}
            k_1 = k_2 = \frac{l_2+3}{l_2-1} \quad \text{and} \quad k_1=k_2 = \frac{l_2-3}{l_2+1},
        \end{equation*}
        respectively. Recall that $l_2 = \pm 1$ is degenerate, so every (nondegenerate) map with pair $(k_1,k_2)$ is conjugate to a corresponding map with pair $(l_1,l_2)$ and vice versa. In other words, the two forms represent the same family $f_a(z)$ as families in moduli space. Since $\frac{1}{z^3}$ is in the $f_a(z)$ component of $\A_3(D_2)$, we have proven the intersection statement.
    \end{proof}

\begin{code}
The code we ran is attached below.
\begin{lstlisting}[language=Python]
#Are the two two-parameter family conjugate to each other?
R.<a,b,c,d,t,k1,k2,l1,l2>=PolynomialRing(QQ, order='lex')
P.<x,y>=ProjectiveSpace(FractionField(R),1)
f=DynamicalSystem_projective([x^3 + l1*x*y^2, l2*x^2*y+y^3])
g=DynamicalSystem_projective([k1*x^2*y + y^3, x^3+k2*x*y^2])
m=matrix(FractionField(R),2,[a,b,c,d])
f2=f.conjugate(m)
f2.scale_by(a*d-b*c)

gens=[a*d-b*c- t] #scale the conjugation
poly=f2[0]*g[1] - f2[1]*g[0] #equate the maps
gens += poly.coefficients()
I=R.ideal(gens)
I=I.saturation(k1*k2-1)[0] #remove degenerate
I=I.saturation(l1*l2-1)[0] #remove degenerate
I=I.saturation(t)[0] #remove degenerate
G=I.groebner_basis()
J=I.elimination_ideal([a,b,c,d,t])
A.<k1,k2,l1,l2>=AffineSpace(QQ,4)
S=A.coordinate_ring()
C=A.subscheme(J.gens())
C.irreducible_components()
\end{lstlisting}
\end{code}

    This completes the proof of Theorem \ref{theorem_1}.

    \begin{remark}
        It is worth noting that for each of the cyclic automorphism groups, the same procedure was able to produce a family in moduli space. Fix an integer $m \geq 2$. The following algorithm will produce a $k$-parameter family in $A_d(C_m)$ where $k = \dim(\A_d(C_m))$.
        \begin{enumerate}
            \item Start with the normal form $z\psi(z^m)$ from Silverman \cite[Proposition 7.3]{Silverman12}. The number of parameters of this normal form is $\dim\A(C_m)+2$. Moreover, there are at least two parameters that are nonzero.
            \item Divide through by one of the nonzero parameters.
            \item Apply a matrix of the form $\begin{bmatrix}\ast&0\\0&1\end{bmatrix}$ to eliminate another parameter.
        \end{enumerate}
    \end{remark}

\section{Geometry of certain automorphism loci in $\M_3$} \label{sect_geom_A3}
    In this section, we examine the geometry of automorphism loci $\A_3(\Gamma) \subset \M_3$, such as smoothness and genus. The method in general is to embed $\M_3$ into an affine space with a collection of multiplier invariants, i.e.,
    \begin{align*}
        \tau_n:\M_3 &\to \bbA^{3^n+1}\\
         [f] &\mapsto (\sigma^{(n)}_1,\ldots, \sigma^{(n)}_{3^n+1}).
    \end{align*}
    We can then examine the image $\tau_n(\A_3(\Gamma)) \subseteq \bbA^{3^n+1}$ and talk about the geometry of the resulting variety.

\subsection{$\A_3(C_3)$}
Recall that every map with a $C_3$ symmetry is conjugate to some map of the form $f_a(z)=\frac{z^3+a}{az^2}$

\begin{prop} \label{prop_c3}
    Define $f_a(z) = \frac{z^3 + a}{az^2}$. The map
    \begin{align*}
        \bbA^1\setminus \{0\} &\to \M_3\\
        a &\mapsto [f_a]
    \end{align*}
is one-to-one.The locus $\A_3(C_3)$ is an irreducible curve of genus zero with one singular point corresponding to $\A_3(\fA_4)$.
\end{prop}
\begin{proof}
    We compute the multiplier invariants associated to the fixed points for the family $f_a(z)$ as functions of $a$ as
    \begin{align*}
        \sigma_1(a)&=\frac{a^2 - 6a + 9}{a}\\
        \sigma_2(a)&=\frac{-6a^3 + 21a^2 - 36a + 27}{a^2}\\
        \sigma_3(a)&=\frac{12a^4 - 44a^3 + 63a^2 - 54a + 27}{a^3}\\
        \sigma_4(a)&=\frac{-8a^3 + 36a^2 - 54a + 27}{a^2}.
    \end{align*}
    Then, to see which choices of parameters $a$ and $b$ have $f_a$ and $f_b$ with the same invariants, we solve the system of equations
    \begin{equation*}
        \sigma_i(a) = \sigma_i(b) \qquad 1 \leq i \leq 4.
    \end{equation*}
    The only solutions occur with $a=b$. In particular, since for each choice of $a$, $f_a$ has distinct multiplier invariants, each choice of $a$ corresponds to a distinct conjugacy class.

    Because every choice of $a$ provides a distinct set of fixed point multiplier invariants and, hence, a distinct conjugacy class in $\M_3$, we can use the fixed point multiplier invariants to parameterize this curve in $\M_3$. In particular, we have a map $\tau_1:\A_3(C_3)\to \bbA^3$ defined by $[f_a]\mapsto (\sigma_1,\sigma_2,\sigma_3).$ We can omit $\sigma_4$ because it is dependent on $(\sigma_1$, $\sigma_2$, $\sigma_3)$ through the standard index relation (see Hutz-Tepper \cite{Hutz10} or Fujimura-Nishizawa \cite[Theorem 1]{Fujimura}).
    Then we can consider the ideal generated by
    \begin{align*}
        &a^2 - 6a + 9 - a\sigma_1,\\
        &-6a^3 + 21a^2 - 36a + 27 -a^2\sigma_2,\\
        &(12a^4 - 44a^3 + 63a^2 - 54a + 27)-a^3\sigma_3.
    \end{align*}
    We compute the saturation with the ideal $(a)$ to avoid the vanishing of $a$ and compute the Gr\"{o}bner basis of the resulting ideal using lexicographic ordering to eliminate $a$. We obtain the following relations among the multiplier invariants:
    \begin{align*}
        0&=6\sigma_1^2 + \sigma_1\sigma_2 + 15\sigma_1 - 18\sigma_2 - 9\sigma_3 - 36\\
        0&=\sigma_2^2 - 3\sigma_1\sigma_3 -24\sigma_1 +12\sigma_2 +36.
    \end{align*}
    These relations define a curve in $\bbA^3$. This curve is irreducible over $\Q$ and has genus $0$. It is not smooth, and the only singular point is $(-12, 54, -108)$. This singular point corresponds to $a=-3$ and, thus, to the rational map $f_{-3}(z)=\frac{z^3-3}{-3z^2}$. This rational map has a tetrahedral automorphism group.
\end{proof}

\begin{code}
The code we used to compute sigma invariants and checking the irreducible components of the subscheme is attached below.
\begin{lstlisting}[language=Python]
#Sigma invariants for the family (z^3+a)/(az^2)
R.<a> = QQ[]
P.<z> = R[]
P = FractionField(P)
A = AffineSpace(P,1)
f = DynamicalSystem_affine((z^3 + a)/(a*z^2))
F = f.homogenize(1)
F.sigma_invariants(1)

#Irreducible component of the subscheme
A.<a,b>=AffineSpace(QQ,2)
R=A.coordinate_ring()
I=R.ideal([(a^2 - 6*a + 9)*b-(b^2 - 6*b + 9)*a,
           (-6*a^3 + 21*a^2 - 36*a + 27)*(b^2)-(-6*b^3 + 21*b^2 - 36*b + 27)*(a^2),
           (12*a^4 - 44*a^3 + 63*a^2 - 54*a + 27)*(b^3)-(12*b^4 - 44*b^3 + 63*b^2 - 54*b + 27)*(a^3),
           (-8*a^3 + 36*a^2 - 54*a + 27)*b^2-(-8*b^3 + 36*b^2 - 54*b + 27)*a^2])
A.subscheme(I).irreducible_components()
\end{lstlisting}

Here are the codes we ran:
\begin{lstlisting}[language=Python]
#first sigma invariants
R.<a> = QQ[]
P.<z> = R[]
P = FractionField(P)
A = AffineSpace(P,1)
g = DynamicalSystem_affine((z^3+a)/(a*z^2))
G = g.homogenize(1)
G.sigma_invariants(1)

#Curve
R.<a,s1,s2,s3>=QQ[]
A=AffineSpace(R)
g1=(a^2 - 6*a + 9) - a*s1
g2=(-6*a^3 + 21*a^2 - 36*a + 27) -a^2*s2
g3=(12*a^4 - 44*a^3 + 63*a^2 - 54*a + 27)-a^3*s3
I=R.ideal(g1,g2,g3)
J=R.ideal(a)
IJ=I.saturation(J)[0]
A.<s1,s2,s3>=AffineSpace(QQ,3)
phi=R.hom([0,s1,s2,s3],A.coordinate_ring())
C=A.curve([phi(F) for F in IJ.elimination_ideal(a).gens()])
print(C.defining_polynomials())
print(C.genus())
print(C.is_smooth())
print(C.singular_points())
CP=C.projective_closure()
CP.degree()

#check singular point
R.<a> = QQ[]
P.<z> = R[]
P = FractionField(P)
A = AffineSpace(P,1)
g = DynamicalSystem_affine((z^3+a)/(a*z^2))
G = g.homogenize(1)
G=G.specialization({a:-3})
G.sigma_invariants(1)
\end{lstlisting}
\end{code}

\subsection{$\A_3(D_2)$}
There are two components in the locus $\A_3(D_2)$. We examine each separately.

\begin{prop} \label{prop_D2_1}
    Define $f_a(z) = \frac{az^2+1}{z^3+az}$. The map
    \begin{align*}
        \bbA^1 \setminus \{\pm 1\} \to M_3\\
        a \mapsto [f_a]
    \end{align*}
    is two-to-one. The component of the locus $\A_3(D_2)$ described by the image in $\M_3$ of the family $f_a$ is a smooth irreducible curve of genus zero.
\end{prop}
\begin{proof}
    We first compute the fixed point multiplier invariants for this family as functions of $a$ as
    \begin{align*}
        \sigma_1(a)&=\frac{4a^2 + 12}{a^2 - 1}\\
        \sigma_2(a)&=\frac{6a^4 + 4a^2 + 54}{a^4 - 2a^2 + 1}\\
        \sigma_3(a)&=\frac{4a^4 - 24a^2 - 108}{a^4 - 2a^2 + 1}.
    \end{align*}
    To show the map is two-to-one, we fix $a$ and determine how many $b$ satisfy
    \begin{equation}\label{eq2}
        (\sigma_1(a), \sigma_2(a), \sigma_3(a)) = (\sigma_1(b), \sigma_2(b), \sigma_3(b))
    \end{equation}
    and then show that both solutions are conjugate. We start with the ideal defined by equation \eqref{eq2} and exclude the cases $a^2=1$ and $b^2=1$ through saturation. The resulting ideal is $(a^2-b^2)$.
    We conclude that $f_a$ and $f_b$ have the same fixed point multiplier invariants if and only if $a = \pm b$. We know that $f_a(z)=\frac{az^2+1}{z^3+az}$ and $f_{-a}(z)=\frac{-az^2+1}{z^3-az}$ are conjugate to each other via the M\"obius transform $\alpha(z)=-iz$:
    \begin{equation*}
        f_a^{\alpha}(z) = -\frac{1}{i}\cdot\frac{a(-iz)^2+1}{(-iz)^3+a(-iz)}=-\frac{1}{i}\cdot\frac{-az^2+1}{iz^3-aiz}=\frac{-az^2+1}{z^3-az} = f_{-a}(z).
    \end{equation*}
    Furthermore, this shows we can parameterize the image of the family $f_a$ in $\M_3$ by the fixed point multiplier invariants. To study its geometry, consider the ideal generated by
    \begin{align*}
        &(4a^2 + 12)-(a^2 - 1)\sigma_1,\\
        &(6a^4 + 4a^2 + 54)-(a^4 - 2a^2 + 1)\sigma_2,\\
        &(4a^4 - 24a^2 - 108)-(a^4 - 2a^2 + 1)\sigma_3.
    \end{align*}
    After saturation by $(a^2-1)$, we eliminate $a$ to obtain the relations
    \begin{align*}
        0&=\sigma_1 - 2\sigma_2 - \sigma_3 + 12\\
        0&=4\sigma_2^2 + 4\sigma_2\sigma_3 - 60\sigma_2 + \sigma_3^2 - 28\sigma_3 + 216.
    \end{align*}
    These define a curve in $\bbA^3$. The curve is smooth, irreducible, and has genus $0$.
\end{proof}

\begin{code}
The code we ran is attached below:
\begin{lstlisting}[language=Python]
#first sigma invariants
R.<a> = QQ[]
P.<x> = R[]
P = FractionField(P)
A = AffineSpace(P,1)
g = DynamicalSystem_affine((a*x^2+1)/(x^3+a*x))
G = g.homogenize(1)
G.sigma_invariants(1)

#Irreducible components
A.<a,b>=AffineSpace(QQ,2)
R=A.coordinate_ring()
I=R.ideal([(4*a^2 + 12)*(b^2 - 1)-(4*b^2 + 12)*(a^2 - 1),
           (6*a^4 + 4*a^2 + 54)*(b^4 - 2*b^2 + 1)-(6*b^4 + 4*b^2 + 54)*(a^4 - 2*a^2 + 1),
           (4*a^4 - 24*a^2 - 108)*(b^4 - 2*b^2 + 1)-(4*b^4 - 24*b^2 - 108)*(a^4 - 2*a^2 + 1),
           (a^4 - 18*a^2 + 81)*(b^4 - 2*b^2 + 1)-(b^4 - 18*b^2 + 81)*(a^4 - 2*a^2 + 1)])
J=R.ideal([a^2-1])
I=I.saturation(J)[0]
J=R.ideal([b^2-1])
I=I.saturation(J)[0]
A.subscheme(I).irreducible_components()

#The curve information
R.<a,s1,s2,s3>=QQ[]
g1=(4*a^2 + 12)-(a^2 - 1)*s1
g2=(6*a^4 + 4*a^2 + 54)-(a^4 - 2*a^2 + 1)*s2
g3=(4*a^4 - 24*a^2 - 108)-(a^4 - 2*a^2 + 1)*s3
I=R.ideal(g1,g2,g3)
J=R.ideal(a^2-1)
I=I.saturation(J)[0]
A.<s1,s2,s3>=AffineSpace(QQ,3)
phi=R.hom([0,s1,s2,s3],A.coordinate_ring())
C=A.curve([phi(F) for F in I.elimination_ideal(a).gens()])
print(C.genus())
print(C.is_smooth())
CP=C.projective_closure()
CP.degree()
\end{lstlisting}
\end{code}

The other component of $\A_3(D_2)$ is given by the family $g_a(z) = \frac{az^2-1}{z^3-az}$. Recall that this family has fixed point multiplier invariants that are independent of the parameter $a$. We consider the multiplier invariants of the periodic points of period $2$:
\begin{align*}
    \sigma_1^{(2)}&=\frac{2a^6 + 36a^4 + 18a^2 + 72}{a^4 - 2a^2 + 1}\\
    \sigma_2^{(2)}&=\frac{a^{12} + 76a^{10} + 514a^8 + 1228a^6 + 9a^4 + 2268}{a^8 - 4a^6 + 6a^4 - 4a^2 + 1}\\
    \sigma_3^{(2)}&=\frac{40a^{12} + 1208a^{10} + 7304a^8 + 32528a^6 + 30744a^4 + 51192a^2 + 40824}{a^8 - 4a^6 + 6a^4 - 4a^2 + 1}\\
    \sigma_4^{(2)}&=\frac{636a^{12} + 11232a^{10} + 85806a^8 + 335448a^6 + 785376a^4 + 927288a^2 + 459270}{a^8 - 4a^6 + 6a^4 - 4a^2 + 1}\\
    \sigma_5^{(2)}&=\frac{5080a^{12} + 74700a^{10} + 624024a^8 + 2354184a^6 + 6805944a^4 + 7637004a^2 + 3306744}{a^8 - 4a^6 + 6a^4 - 4a^2 + 1}\\
    \sigma_6^{(2)}&=\frac{21286a^{12} + 365112a^{10} + 2597184a^8 + 11914776a^6 + 25286094a^4 + 32122656a^2 + 14880348}{a^8 - 4a^6 + 6a^4 - 4a^2 + 1}\\
    \sigma_7^{(2)}&=\frac{45720a^{12} + 1030968a^{10} + 6945912a^8 + 29498256a^6 + 47711592a^4 + 63772920a^2 + 38263752}{a^8 - 4a^6 + 6a^4 - 4a^2 + 1}\\
    \sigma_8^{(2)}&=\frac{51516a^{12} + 979776a^{10} + 11396457a^8 + 12045996a^6 + 86093442a^4 + 57395628a^2 + 43046721}{a^8 - 4a^6 + 6a^4 - 4a^2 + 1}\\
    \sigma_9^{(2)}&=\frac{29160a^{12} + 13122a^{10} - 7085880a^8 + 40389516a^6 + 86093442a^2}{a^8 - 4a^6 + 6a^4 - 4a^2 + 1}\\
    \sigma_{10}^{(2)}&=\frac{6561a^{12} - 236196a^{10} + 3188646a^8 - 19131876a^6 + 43046721a^4}{a^8 - 4a^6 + 6a^4 - 4a^2 + 1}.
\end{align*}
\begin{lemma}\label{lem_D2_2}
     Define $g_a(z) = \frac{az^2-1}{z^3-az}$. The map
    \begin{align*}
        \varphi:\bbA^1\setminus \{\pm 1\} &\to \M_3\\
        a &\mapsto [g_a]
    \end{align*}
    is six-to-one. The map
    \begin{align*}
        \tilde{\tau}_2: \varphi(\bbA^1) &\to \bbA^3\\
        [g_a] &\mapsto (\sigma_1^{(2)}, \sigma_2^{(2)}, \sigma_3^{(2)})
    \end{align*}
    is injective.
\end{lemma}
\begin{proof}
    The ideal generated by $(\sigma_1^{(2)},\ldots, \sigma_{10}^{(2)})$ is, in fact, generated by $(\sigma_1^{(2)},\sigma_2^{(2)}, \sigma_3^{(2)})$, so we just need to focus on these three invariants. Thinking of the invariants as functions of $a$, we need to find all $b$ so that
    \begin{equation*}
        \left(\sigma_1^{(2)}(a), \sigma_2^{(2)}(a), \sigma_3^{(2)}(a)\right) = \left(\sigma_1^{(2)}(b), \sigma_2^{(2)}(b), \sigma_3^{(2)}(b)\right).
    \end{equation*}
    Taking the ideal generated by these equations and saturating by the ideals $(a^2-1)$ and $(b^2-1)$, we have the principal ideal generated by the polynomial
    \begin{equation*}
        (-a + b)(a + b)(ab - a - b - 3)(ab - a + b + 3)(ab + a - b + 3)(ab + a + b - 3).
    \end{equation*}
    We check that the six parameter values
    \begin{equation*}
        b \in \left\{ \pm a, \pm \frac{a+3}{a-1}, \pm \frac{a-3}{a+1}\right\}
    \end{equation*}
    produce functions that are all conjugate.

    If $b=-a$, $g_a$ is, in fact, conjugate to $g_{b}$ via the M\"obius transform $\alpha(z) = iz$
    \begin{equation*}
        g_a^{\alpha} = \frac{1}{i}\cdot\frac{a(iz)^2-1}{(iz)^3-a(iz)}=\frac{1}{i}\cdot\frac{-az^2-1}{-iz^3-aiz}=\frac{-az^2-1}{z^3+az} = g_{-a}.
    \end{equation*}
    If $b=\frac{a+3}{a-1}$, then $g_a$ is conjugate to $g_b$ via the M\"obius transform $\alpha=\begin{pmatrix}-1&1\\1&1\end{pmatrix}$. This calculation is straightforward, but omitted.

    If $b=\frac{a-3}{a+1}=-\frac{-a+3}{-a-1}$, then $g_{-a}$ is conjugate to $g_b$ since $\frac{-a+3}{-a-1}=\frac{(-a)+3}{(-a)-1}$. Thus, this choice of $b$ is also conjugate to $a$.

    If $b=\frac{-a-3}{a-1}=-\frac{a+3}{a-1}$, then $g_b$ is conjugate to $g_{\frac{a+3}{a-1}}$ and $g_a$.

    If $b=\frac{-a+3}{a+1}=\frac{a-3}{-a-1}=-\frac{a-3}{a+1}$, $g_b$ is conjugate to $g_{\frac{a-3}{a+1}}$ and $g_a$ as well.

    The six parameter values that produce the same triple $(\sigma_1^{(2)}, \sigma_2^{(2)}, \sigma_3^{(2)})$ are all conjugate, so $\varphi$ is six-to-one and $\tilde{\tau}_2$ is injective on the image $\varphi(\bbA^1)$.
\end{proof}

\begin{prop}\label{prop_D2_2}
    The curve given by the component  $g_a(z) = \frac{az^2-1}{z^3-az}$ of $\A_3(D_2)$ is a smooth irreducible curve of genus $0$.
\end{prop}
\begin{proof}
    By Lemma \ref{lem_D2_2} we can parameterize the curve with the coordinates $(\sigma_1^{(2)}, \sigma_2^{(2)}, \sigma_3^{(2)})$ in $\bbA^3$. In terms of the parameter $a$, this gives the equations
    \begin{align*}
        &(2a^6 + 36a^4 + 18a^2 + 72)-(a^4 - 2a^2 + 1)\sigma_1=0,\\
        &(a^{12} + 76a^{10} + 514a^8 + 1228a^6 + 9a^4 + 2268)-(a^8 - 4a^6 + 6a^4 - 4a^2 + 1)\sigma_2=0,\\
        &(40a^{12} + 1208a^{10} + 7304a^8 + 32528a^6 + 30744a^4 + 51192a^2 + 40824)\\
        &-(a^8 - 4a^6 + 6a^4 - 4a^2 + 1)\sigma_3=0.
    \end{align*}
    Saturating with respect to $(a^2-1)$ and eliminating $a$ yields the relations
    \begin{align*}
        0&=916\sigma_1 - 40\sigma_2 + \sigma_3 - 16056\\
        0&=1600\sigma_2^2 - 80\sigma_2\sigma_3 + \sigma_3^2 + 859456\sigma_2 - 105392\sigma_3 - 136334016.
    \end{align*}
    These relations define a smooth, irreducible curve of genus $0$ in $\bbA^3$.
\end{proof}

\begin{code}
The code we ran is attached below:
\begin{lstlisting}[language=Python]
#Sigma Invariants
R.<a> = QQ[]
P.<x> = R[]
P = FractionField(P)
A = AffineSpace(P,1)
f = DynamicalSystem_affine((a*x^2-1)/(x^3-(a*x)))
F = f.homogenize(1)
F.sigma_invariants(2)

#Checking redundancy of sigma invariants
R.<a>=QQ[]
A=AffineSpace(R)
g2=a^12 + 76*a^10 + 514*a^8 + 1228*a^6 + 9*a^4 + 2268
g3=40*a^12 + 1208*a^10 + 7304*a^8 + 32528*a^6 + 30744*a^4 + 51192*a^2 + 40824
I=R.ideal(g2,g3)

g4=636*a^12 + 11232*a^10 + 85806*a^8 + 335448*a^6 + 785376*a^4 + 927288*a^2 + 459270
g5=(5080*a^12 + 74700*a^10 + 624024*a^8 + 2354184*a^6 + 6805944*a^4 + 7637004*a^2 + 3306744)
g6=(21286*a^12 + 365112*a^10 + 2597184*a^8 + 11914776*a^6 + 25286094*a^4 + 32122656*a^2 + 14880348)
g7=(45720*a^12 + 1030968*a^10 + 6945912*a^8 + 29498256*a^6 + 47711592*a^4 + 63772920*a^2 + 38263752)
g8=(51516*a^12 + 979776*a^10 + 11396457*a^8 + 12045996*a^6 + 86093442*a^4 + 57395628*a^2 + 43046721)
g9=(29160*a^12 + 13122*a^10 - 7085880*a^8 + 40389516*a^6 + 86093442*a^2)
g0=(6561*a^12 - 236196*a^10 + 3188646*a^8 - 19131876*a^6 + 43046721*a^4)
print(g4 in I)
print(g5 in I)
print(g6 in I)
print(g7 in I)
print(g8 in I)
print(g9 in I)
print(g0 in I)

#Irreducible components
A.<a,b>=AffineSpace(QQ,2)
R=A.coordinate_ring()
I=R.ideal([(2*a^6 + 36*a^4 + 18*a^2 + 72)*(b^4 - 2*b^2 + 1)-(2*b^6 + 36*b^4 + 18*b^2 + 72)*(a^4 - 2*a^2 + 1),
          (a^12 + 76*a^10 + 514*a^8 + 1228*a^6 + 9*a^4 + 2268)*(b^8 - 4*b^6 + 6*b^4 - 4*b^2 + 1)-(b^12 + 76*b^10 + 514*b^8 + 1228*b^6 + 9*b^4 + 2268)*(a^8 - 4*a^6 + 6*a^4 - 4*a^2 + 1),
          (40*a^12 + 1208*a^10 + 7304*a^8 + 32528*a^6 + 30744*a^4 + 51192*a^2 + 40824)*(b^8 - 4*b^6 + 6*b^4 - 4*b^2 + 1)-(40*b^12 + 1208*b^10 + 7304*b^8 + 32528*b^6 + 30744*b^4 + 51192*b^2 + 40824)*(a^8 - 4*a^6 + 6*a^4 - 4*a^2 + 1)])
A.<a,b>=AffineSpace(QQ,2)
R=A.coordinate_ring()
I=R.ideal([(2*a^6 + 36*a^4 + 18*a^2 + 72)*(b^4 - 2*b^2 + 1)-(2*b^6 + 36*b^4 + 18*b^2 + 72)*(a^4 - 2*a^2 + 1),
          (a^12 + 76*a^10 + 514*a^8 + 1228*a^6 + 9*a^4 + 2268)*(b^8 - 4*b^6 + 6*b^4 - 4*b^2 + 1)-(b^12 + 76*b^10 + 514*b^8 + 1228*b^6 + 9*b^4 + 2268)*(a^8 - 4*a^6 + 6*a^4 - 4*a^2 + 1),
          (40*a^12 + 1208*a^10 + 7304*a^8 + 32528*a^6 + 30744*a^4 + 51192*a^2 + 40824)*(b^8 - 4*b^6 + 6*b^4 - 4*b^2 + 1)-(40*b^12 + 1208*b^10 + 7304*b^8 + 32528*b^6 + 30744*b^4 + 51192*b^2 + 40824)*(a^8 - 4*a^6 + 6*a^4 - 4*a^2 + 1)])
Ja=R.ideal([a^2-1])
Jb=R.ideal([b^2-1])
I=I.saturation(Ja)[0]
I=I.saturation(Jb)[0]
I.gens()[0].factor()


#Checking conjugacy between a and (a+3)/(a-1)
R.<a,b,c,d,k,t>=QQ[]
P.<x,y>=ProjectiveSpace(FractionField(R),1)
f=DynamicalSystem_projective([k*x^2*y - y^3, x^3-k*x*y^2])
m=matrix(FractionField(R),2,[-1,1,1,1])
f.conjugate(m)

#Curve information
R.<a,s1,s2,s3>=QQ[]
A=AffineSpace(R)
g1=(2*a^6 + 36*a^4 + 18*a^2 + 72)-(a^4 - 2*a^2 + 1)*s1
g2=(a^12 + 76*a^10 + 514*a^8 + 1228*a^6 + 9*a^4 + 2268)-(a^8 - 4*a^6 + 6*a^4 - 4*a^2 + 1)*s2
g3=(40*a^12 + 1208*a^10 + 7304*a^8 + 32528*a^6 + 30744*a^4 + 51192*a^2 + 40824)-(a^8 - 4*a^6 + 6*a^4 - 4*a^2 + 1)*s3

I=R.ideal(g1,g2,g3)
J=R.ideal(a^2-1)
I=I.saturation(J)[0]
print(I.elimination_ideal(a))
A.<s1,s2,s3>=AffineSpace(QQ,3)
phi=R.hom([0,s1,s2,s3],A.coordinate_ring())
C=A.curve([phi(F) for F in I.elimination_ideal(a).gens()])
print(C.genus())
print(C.is_smooth())
CP=C.projective_closure()
CP.degree()
\end{lstlisting}
\end{code}

\subsection{$\A_3(C_2)$}
    We start with the family $f_{a,b}(z) = \frac{z^3+az}{bz^2+1}$.
    \begin{lemma}
        The map
        \begin{align*}
            \varphi:\bbA^2\setminus \{ab=1\} &\to \M_3\\
            (a,b) &\mapsto [f_{a,b}]
        \end{align*}
        is two-to-one. The map
        \begin{align*}
            \tau_1: \varphi(\bbA^2\setminus \{ab=1\}) &\to \bbA^3\\
            [f_{a,b}] &\mapsto (\sigma_1, \sigma_2, \sigma_3)
        \end{align*}
        is injective.
    \end{lemma}
    \begin{proof}
        We calculate the fixed point multiplier invariants as functions of $(a,b)$ as
        \begin{align*}
            \sigma_1(a,b)&=\frac{a^2b + ab^2 - 2ab + 3a + 3b - 6}{ab - 1}\\
            \sigma_2(a,b)&=\frac{a^3b^3 - 2a^3b^2 - 2a^2b^3 + 4a^3b + 7a^2b^2 + 4ab^3 - 8a^2b - 8ab^2 + 7ab - 6a - 6b + 9}{(ab-1)^2}\\
            \sigma_3(a,b)&=\frac{1}{(ab-1)^2}\Big(-2a^3b^3 + 5a^3b^2 + 5a^2b^3 - 4a^3b - 12a^2b^2 - 4ab^3 + 4a^3\\
            &+ 14a^2b + 14ab^2 + 4b^3 - 12a^2 - 18ab - 12b^2 + 9a + 9b\Big),
        \end{align*}
        where we have omitted $\sigma_4$ as usual due to its dependence from the standard index relation. To determine the degree of $\varphi$, we consider the equations
        \begin{equation*}
            (\sigma_3(a,b), \sigma_2(a,b), \sigma_3(a,b)) =
            (\sigma_1(c,d), \sigma_2(c,d), \sigma_3(c,d)).
        \end{equation*}
        We saturate the resulting ideal with respect to $(ab-1)$ and $(cd-1)$, the parameter values where the family is degenerate, and find the irreducible components of the resulting variety in $\bbA^3$ as
        \begin{equation*}
            (c,d) =(a,b) \quad \text{or} \quad (c,d) = (b,a).
        \end{equation*}
        The maps $f_{a,b}$ and $f_{b,a}$ are conjugate by $\alpha(z) = \frac{1}{z}$.

        In no other situation are the fixed point multiplier invariants equal, so $\tau_1$ is injective on the image of $\varphi$.
    \end{proof}
\begin{code}
The code we ran is attached below:
\begin{lstlisting}[language=Python]
#find the invariants
R.<a,b,c,d> = QQ[]
P.<z> = R[]
P = FractionField(P)
A = AffineSpace(P,1)
f = DynamicalSystem_affine((z^3+a*z)/(b*z^2+1))
g = DynamicalSystem_affine((z^3+c*z)/(d*z^2+1))
F = f.homogenize(1)
G = g.homogenize(1)
sF=F.sigma_invariants(1)
sG=G.sigma_invariants(1)

#create the ideal and saturate
I=R.ideal([sF[i].numerator()*sG[i].denominator()-sG[i].numerator()*sF[i].denominator() for i in range(len(sF))])
Ja=R.ideal(a*b-1)
Jb=R.ideal(c*d-1)
I=I.saturation(Ja)[0]
I=I.saturation(Jb)[0]

#get the irreducible components
A=AffineSpace(R)
X=A.subscheme(I)
IC=X.irreducible_components()

#check conjugation
F=f.homogenize(1)
m=matrix(QQ,2,2,[0,1,1,0])
g=F.conjugate(m)
F,g
\end{lstlisting}
\end{code}

\begin{prop}\label{prop_a3_c2_1_geo}
    The component of $\A_3(C_2)$ given by $f_{a,b} = \frac{z^3+az}{bz^2+1}$ is an irreducible rational surface defined by
    \begin{align*}
        0&=36\sigma_1^5 - 12\sigma_1^4\sigma_2 + \sigma_1^3\sigma_2^2 - 2\sigma_1^2\sigma_2^3 - 12\sigma_1^4\sigma_3 + 8\sigma_1^3\sigma_2\sigma_3 - \sigma_1^2\sigma_2^2\sigma_3 + 4\sigma_1^3\sigma_3^2 - 60\sigma_1^4 - 194\sigma_1^3\sigma_2\\
        &+ 64\sigma_1^2\sigma_2^2 - 4\sigma_1\sigma_2^3 + 8\sigma_2^4 + 4\sigma_1^3\sigma_3 + 60\sigma_1^2\sigma_2\sigma_3 - 36\sigma_1\sigma_2^2\sigma_3 + 4\sigma_2^3\sigma_3 - 18\sigma_1\sigma_2\sigma_3^2 + \sigma_1^3 + 318\sigma_1^2\sigma_2\\
        &+ 180\sigma_1\sigma_2^2 - 56\sigma_2^3 + 261\sigma_1^2\sigma_3 - 108\sigma_1\sigma_2\sigma_3 - 36\sigma_2^2\sigma_3 - 81\sigma_1\sigma_3^2 + 54\sigma_2\sigma_3^2 + 27\sigma_3^3 + 684\sigma_1^2\\
        &- 576\sigma_1\sigma_2 - 144\sigma_2^2 - 648\sigma_1\sigma_3 + 108\sigma_3^2 - 2160\sigma_1 + 864\sigma_2 + 432\sigma_3 + 1728
    \end{align*}
    whose projective closure is parameterized by the map
    \begin{align*}
        \phi: \bbP^2 &\to \overline{S} \subset\bbP^3\\
        (x:y:z) &\mapsto (\phi_1:\phi_2:\phi_3:\phi_4)
    \end{align*}
    for
    \begin{align*}
        \phi_1 &= \frac{1}{12}x^5y + 6x^5z + 2x^4yz + 24x^4z^2 + 18x^3yz^2 + 72x^2yz^3 + 288x^2z^4 + 108xyz^4 - 864xz^5\\
        \phi_2&=\frac{13}{24}x^5y + 9x^5z + \frac{37}{4}x^4yz - 27x^4z^2 + 51x^3yz^2 + 216x^3z^3 + 90x^2yz^3 - 504x^2z^4 + 54xyz^4\\
        &- 432xz^5 + 324yz^5 + 1296z^6\\
        \phi_3&=x^5y + \frac{23}{2}x^4yz + 36x^3yz^2 + 36x^2yz^3 - 648yz^5\\
        \phi_4&=x^5z + \frac{23}{2}x^4z^2 + 36x^3z^3 + 36x^2z^4 - 648z^6.
    \end{align*}
    This surface is singular with singular locus given by the conjugacy classes described by
    \begin{equation*}
        \{(a,b) : a=b\} \cup \{(a,b) : a+b=-6\} \cup \{(a,b) : b=\frac{3}{a+2}\}.
    \end{equation*}
\end{prop}
\begin{proof}
    To obtain the surface equations, take the fixed point multiplier invariants and consider the ideal generated by
    \begin{align*}
        &(a^2b + ab^2 - 2ab + 3a + 3b - 6)-(ab - 1)\sigma_1,\\
        &(a^3b^3 - 2a^3b^2 - 2a^2b^3 + 4a^3b + 7a^2b^2 + 4ab^3 - 8a^2b - 8ab^2 + 7ab - 6a - 6b + 9)-(ab - 1)^2\sigma_2,\\
        &(-2a^3b^3 + 5a^3b^2 + 5a^2b^3 - 4a^3b - 12a^2b^2 - 4ab^3 + 4a^3 + 14a^2b + 14ab^2 + 4b^3 - 12a^2 - 18ab\\
        &- 12b^2 + 9a + 9b)-(ab - 1)^2\sigma_3.
    \end{align*}
    We saturate by the ideal $(ab-1)$ to avoid the parameters where the family is degenerate and compute the elimination ideal to eliminate the variables $a$ and $b$. This results in the single equation in $(\sigma_1,\sigma_2,\sigma_3)$ in the statement.

    The parameterization was computed in Magma and is easily checked by substituting $(\sigma_1,\sigma_2,\sigma_3) = (\phi_1/\phi_4, \phi_2/\phi_4, \phi_3/\phi_4)$ into the surface equation.

    The singular locus is also computed in Magma and its irreducible components computed in Sage are given by
    \begin{align*}
        S_1: & \begin{cases}
                \sigma_3 = 4\\
                \sigma_2 = 6\\
                \sigma_1 = 4
                \end{cases}\\
        S_2:& \begin{cases}
            \sigma_1 - 2\sigma_2 - \sigma_3 + 12=0\\
            4\sigma_2^2 + 4\sigma_2\sigma_3 + \sigma_3^2 - 60\sigma_2 - 28\sigma_3 + 216=0
            \end{cases}\\
        S_3:& \begin{cases}
          \sigma_2^2 - 3\sigma_1\sigma_3 - 24\sigma_1 + 12\sigma_2 + 36=0\\
          6\sigma_1^2 + \sigma_1\sigma_2 + 15\sigma_1 - 18\sigma_2 - 9\sigma_3 - 36=0.
          \end{cases}
    \end{align*}

    To analyze the components, we proceed similarly through elimination of variables. Starting with the ideal defined by the equations of the component along with the defining equations of the fixed point multiplier invariants we eliminate $(\sigma_1,\sigma_2,\sigma_3)$ to get an ideal in $a$ and $b$. Then we saturate by $(ab-1)$ to avoid the degenerate elements.
    The component $S_1$ corresponds to the degenerate case $a=b=1$.

    The component $S_2$ results in the components
    \begin{equation*}
        (a+b+6)(b-a)^2 = 0.
    \end{equation*}
    When $a=b$, this is the component of $\A_3(D_2)$ from Proposition \ref{prop_D2_1}.

    The component $S_3$ results in the components
    \begin{equation*}
        (ab + 2b - 3)^2(ab + 2a - 3)^2=0.
    \end{equation*}
    These correspond to $a = \frac{3}{b+2}$ and $b = \frac{3}{a+2}$. Since $f_{a,b}$ and $f_{b,a}$ are conjugate by $\alpha(z) = \frac{1}{z}$, this is a single component in the moduli space.
\end{proof}

\begin{code}
The code used is
\begin{lstlisting}[language=Python]
R.<a,b,s1,s2,s3>=PolynomialRing(QQ)
A=AffineSpace(R)
f1=(a^2*b + a*b^2 - 2*a*b + 3*a + 3*b - 6)-(a*b - 1)*s1
f2=(a^3*b^3 - 2*a^3*b^2 - 2*a^2*b^3 + 4*a^3*b + 7*a^2*b^2 + 4*a*b^3 - 8*a^2*b - 8*a*b^2 + 7*a*b - 6*a - 6*b + 9)-(a^2*b^2 - 2*a*b + 1)*s2
f3=(-2*a^3*b^3 + 5*a^3*b^2 + 5*a^2*b^3 - 4*a^3*b - 12*a^2*b^2 - 4*a*b^3 + 4*a^3 + 14*a^2*b + 14*a*b^2 + 4*b^3 - 12*a^2 - 18*a*b - 12*b^2 + 9*a + 9*b)-(a^2*b^2 - 2*a*b + 1)*s3
#f4=(a^3*b^3 - 4*a^3*b^2 - 4*a^2*b^3 + 4*a^3*b + 14*a^2*b^2 + 4*a*b^3 - 12*a^2*b - 12*a*b^2 + 9*a*b)-(a^2*b^2 - 2*a*b + 1)*s4
I=R.ideal(f1,f2,f3)
Ja=R.ideal(a*b-1)
I=I.saturation(Ja)[0]
I=I.elimination_ideal([a,b]) #since the bad point is a single point, this doesn't remove the degeneracy!
I.gens()
print(I.gens()[0])
print(I.gens()[0].homogenize('h'))


#to magma for the surface analysis
P3<s1,s2,s3,h>:=ProjectiveSpace(Rationals(),3);
S:= Surface(P3, 36*s1^5 - 12*s1^4*s2 + s1^3*s2^2 - 2*s1^2*s2^3 - 12*s1^4*s3 + 8*s1^3*s2*s3 - s1^2*s2^2*s3 + 4*s1^3*s3^2 - 60*s1^4*h - 194*s1^3*s2*h + 64*s1^2*s2^2*h - 4*s1*s2^3*h + 8*s2^4*h + 4*s1^3*s3*h + 60*s1^2*s2*s3*h - 36*s1*s2^2*s3*h + 4*s2^3*s3*h - 18*s1*s2*s3^2*h + s1^3*h^2 + 318*s1^2*s2*h^2 + 180*s1*s2^2*h^2 - 56*s2^3*h^2 + 261*s1^2*s3*h^2 - 108*s1*s2*s3*h^2 - 36*s2^2*s3*h^2 - 81*s1*s3^2*h^2 + 54*s2*s3^2*h^2 + 27*s3^3*h^2 + 684*s1^2*h^3 - 576*s1*s2*h^3 - 144*s2^2*h^3 - 648*s1*s3*h^3 + 108*s3^2*h^3 - 2160*s1*h^4 + 864*s2*h^4 + 432*s3*h^4 + 1728*h^5);
IsRational(S);
IsReduced(S);
IsIrreducible(S);
IsSingular(S);
GeometricGenusOfDesingularization(S);
P2<x,y,z>:=ProjectiveSpace(Rationals(),2);
ParametrizeProjectiveHypersurface(S,P2);
D:=SingularSubscheme(S);



h=1
r.<s1,s2,s3>=QQ[]
g1=-30*s1^4 - 97*s1^3*s2 + 32*s1^2*s2^2 - 2*s1*s2^3 + 4*s2^4 + 2*s1^3*s3 + 30*s1^2*s2*s3 - 18*s1*s2^2*s3 + 2*s2^3*s3 - 9*s1*s2*s3^2 + s1^3*h + 318*s1^2*s2*h + 180*s1*s2^2*h - 56*s2^3*h + 261*s1^2*s3*h - 108*s1*s2*s3*h - 36*s2^2*s3*h - 81*s1*s3^2*h + 54*s2*s3^2*h + 27*s3^3*h + 1026*s1^2*h^2 - 864*s1*s2*h^2 - 216*s2^2*h^2 - 972*s1*s3*h^2 + 162*s3^2*h^2 - 4320*s1*h^3 + 1728*s2*h^3 + 864*s3*h^3 + 4320*h^4
g2=-12*s1^4 + 8*s1^3*s2 - s1^2*s2^2 + 8*s1^3*s3 + 4*s1^3*h + 60*s1^2*s2*h - 36*s1*s2^2*h + 4*s2^3*h - 36*s1*s2*s3*h + 261*s1^2*h^2 - 108*s1*s2*h^2 -36*s2^2*h^2 - 162*s1*s3*h^2 + 108*s2*s3*h^2 + 81*s3^2*h^2 - 648*s1*h^3 + 216*s3*h^3 + 432*h^4
g3=-6*s1^4 + s1^3*s2 - 3*s1^2*s2^2 + 4*s1^3*s3 - s1^2*s2*s3 - 97*s1^3*h +64*s1^2*s2*h - 6*s1*s2^2*h + 16*s2^3*h + 30*s1^2*s3*h - 36*s1*s2*s3*h + 6*s2^2*s3*h - 9*s1*s3^2*h + 159*s1^2*h^2 + 180*s1*s2*h^2 - 84*s2^2*h^2 - 54*s1*s3*h^2 - 36*s2*s3*h^2 + 27*s3^2*h^2 - 288*s1*h^3 - 144*s2*h^3 + 432*h^4
g4=180*s1^4 - 48*s1^3*s2 + 3*s1^2*s2^2 - 4*s1*s2^3 - 48*s1^3*s3 + 24*s1^2*s2*s3 - 2*s1*s2^2*s3 + 12*s1^2*s3^2 - 240*s1^3*h - 582*s1^2*s2*h + 128*s1*s2^2*h - 4*s2^3*h + 12*s1^2*s3*h + 120*s1*s2*s3*h - 36*s2^2*s3*h - 18*s2*s3^2*h + 3*s1^2*h^2 + 636*s1*s2*h^2 + 180*s2^2*h^2 + 522*s1*s3*h^2 - 108*s2*s3*h^2 - 81*s3^2*h^2 + 1368*s1*h^3 - 576*s2*h^3 - 648*s3*h^3 - 2160*h^4
g5=36*s1^5 - 12*s1^4*s2 + s1^3*s2^2 - 2*s1^2*s2^3 - 12*s1^4*s3 + 8*s1^3*s2*s3 - s1^2*s2^2*s3 + 4*s1^3*s3^2 - 60*s1^4*h - 194*s1^3*s2*h + 64*s1^2*s2^2*h - 4*s1*s2^3*h + 8*s2^4*h + 4*s1^3*s3*h + 60*s1^2*s2*s3*h - 36*s1*s2^2*s3*h + 4*s2^3*s3*h - 18*s1*s2*s3^2*h + s1^3*h^2 + 318*s1^2*s2*h^2 + 180*s1*s2^2*h^2 - 56*s2^3*h^2 + 261*s1^2*s3*h^2 - 108*s1*s2*s3*h^2 - 36*s2^2*s3*h^2 - 81*s1*s3^2*h^2 + 54*s2*s3^2*h^2 + 27*s3^3*h^2 + 684*s1^2*h^3 - 576*s1*s2*h^3 - 144*s2^2*h^3 - 648*s1*s3*h^3 + 108*s3^2*h^3 - 2160*s1*h^4 + 864*s2*h^4 + 432*s3*h^4 + 1728*h^5

Ig=r.ideal([g1,g2,g3,g4,g5])

A=AffineSpace(r)
D=A.subscheme(Ig)
D.irreducible_components()


#component 1
R.<a,b>=PolynomialRing(QQbar)
A=AffineSpace(R)
s1=4
s2=6
s3=4
f1=(a^2*b + a*b^2 - 2*a*b + 3*a + 3*b - 6)-(a*b - 1)*s1
f2=(a^3*b^3 - 2*a^3*b^2 - 2*a^2*b^3 + 4*a^3*b + 7*a^2*b^2 + 4*a*b^3 - 8*a^2*b - 8*a*b^2 + 7*a*b - 6*a - 6*b + 9)-(a^2*b^2 - 2*a*b + 1)*s2
f3=(-2*a^3*b^3 + 5*a^3*b^2 + 5*a^2*b^3 - 4*a^3*b - 12*a^2*b^2 - 4*a*b^3 + 4*a^3 + 14*a^2*b + 14*a*b^2 + 4*b^3 - 12*a^2 - 18*a*b - 12*b^2 + 9*a + 9*b)-(a^2*b^2 - 2*a*b + 1)*s3
X=A.subscheme([f1,f2,f3])
X.rational_points()

#component 2
R.<a,b,s1,s2,s3>=QQ[]
A=AffineSpace(R)
f1=(a^2*b + a*b^2 - 2*a*b + 3*a + 3*b - 6)-(a*b - 1)*s1
f2=(a^3*b^3 - 2*a^3*b^2 - 2*a^2*b^3 + 4*a^3*b + 7*a^2*b^2 + 4*a*b^3 - 8*a^2*b - 8*a*b^2 + 7*a*b - 6*a - 6*b + 9)-(a^2*b^2 - 2*a*b + 1)*s2
f3=(-2*a^3*b^3 + 5*a^3*b^2 + 5*a^2*b^3 - 4*a^3*b - 12*a^2*b^2 - 4*a*b^3 + 4*a^3 + 14*a^2*b + 14*a*b^2 + 4*b^3 - 12*a^2 - 18*a*b - 12*b^2 + 9*a + 9*b)-(a^2*b^2 - 2*a*b + 1)*s3
h1=s2^2 - 3*s1*s3 - 24*s1 + 12*s2 + 36
h2=6*s1^2 + s1*s2 + 15*s1 - 18*s2 - 9*s3 - 36
h1=s1 - 2*s2 - s3 + 12
h2=4*s2^2 + 4*s2*s3 + s3^2 - 60*s2 - 28*s3 + 216
I=R.ideal([h1,h2] +[f1,f2,f3])
J=I.elimination_ideal([s1,s2,s3])
J2=J.saturation(R.ideal(a*b-1))[0]
J2.gens()[0].factor()

#componenet 3
R.<a,b,s1,s2,s3>=QQ[]
A=AffineSpace(R)
f1=(a^2*b + a*b^2 - 2*a*b + 3*a + 3*b - 6)-(a*b - 1)*s1
f2=(a^3*b^3 - 2*a^3*b^2 - 2*a^2*b^3 + 4*a^3*b + 7*a^2*b^2 + 4*a*b^3 - 8*a^2*b - 8*a*b^2 + 7*a*b - 6*a - 6*b + 9)-(a^2*b^2 - 2*a*b + 1)*s2
f3=(-2*a^3*b^3 + 5*a^3*b^2 + 5*a^2*b^3 - 4*a^3*b - 12*a^2*b^2 - 4*a*b^3 + 4*a^3 + 14*a^2*b + 14*a*b^2 + 4*b^3 - 12*a^2 - 18*a*b - 12*b^2 + 9*a + 9*b)-(a^2*b^2 - 2*a*b + 1)*s3
h1=s2^2 - 3*s1*s3 - 24*s1 + 12*s2 + 36
h2=6*s1^2 + s1*s2 + 15*s1 - 18*s2 - 9*s3 - 36
I=R.ideal([h1,h2] +[f1,f2,f3])
J=I.elimination_ideal([s1,s2,s3])
J2=J.saturation(R.ideal(a*b-1))[0]
J2.gens()[0].factor()


#check parameterization
R.<a,b,s1,s2,s3,x,y,z>=PolynomialRing(QQ)
A=AffineSpace(R)
f1=(a^2*b + a*b^2 - 2*a*b + 3*a + 3*b - 6)-(a*b - 1)*s1
f2=(a^3*b^3 - 2*a^3*b^2 - 2*a^2*b^3 + 4*a^3*b + 7*a^2*b^2 + 4*a*b^3 - 8*a^2*b - 8*a*b^2 + 7*a*b - 6*a - 6*b + 9)-(a^2*b^2 - 2*a*b + 1)*s2
f3=(-2*a^3*b^3 + 5*a^3*b^2 + 5*a^2*b^3 - 4*a^3*b - 12*a^2*b^2 - 4*a*b^3 + 4*a^3 + 14*a^2*b + 14*a*b^2 + 4*b^3 - 12*a^2 - 18*a*b - 12*b^2 + 9*a + 9*b)-(a^2*b^2 - 2*a*b + 1)*s3
#f4=(a^3*b^3 - 4*a^3*b^2 - 4*a^2*b^3 + 4*a^3*b + 14*a^2*b^2 + 4*a*b^3 - 12*a^2*b - 12*a*b^2 + 9*a*b)-(a^2*b^2 - 2*a*b + 1)*s4
I=R.ideal(f1,f2,f3)
print(I.elimination_ideal([a,b]))
Ja=R.ideal(a*b-1)
I=I.saturation(Ja)[0]
I=I.elimination_ideal([a,b])
I.gens()
print(I.gens()[0])
print(I.gens()[0].homogenize('h'))
t1 = 1/12*x^5*y + 6*x^5*z + 2*x^4*y*z + 24*x^4*z^2 + 18*x^3*y*z^2 + 72*x^2*y*z^3 + 288*x^2*z^4 + 108*x*y*z^4 - 864*x*z^5
t2=13/24*x^5*y + 9*x^5*z + 37/4*x^4*y*z - 27*x^4*z^2 + 51*x^3*y*z^2 + 216*x^3*z^3 + 90*x^2*y*z^3 - 504*x^2*z^4 + 54*x*y*z^4 - 432*x*z^5 + 324*y*z^5 + 1296*z^6
t3=x^5*y + 23/2*x^4*y*z + 36*x^3*y*z^2 + 36*x^2*y*z^3 - 648*y*z^5
t4=x^5*z + 23/2*x^4*z^2 + 36*x^3*z^3 + 36*x^2*z^4 - 648*z^6
I.gen(0)(a,b,t1/t4,t2/t4,t3/t4,x,y,z)

\end{lstlisting}
\end{code}

Now we move to the next family with a $C_2$ symmetry, $g_{a,b}(z)=\frac{az^2+1}{z^3+bz}$.
\begin{lemma}
    For the family $g_{a,b}(z) = \frac{az^2+1}{z^3+bz}$, the image of the map $\tau_1: [g_{a,b}] \mapsto (\sigma_1,\sigma_2,\sigma_3)$ is a curve given by
    \begin{align*}
        0&=\sigma_3 - 2\sigma_2 - \sigma_1 + 12\\
        0&=4\sigma_2^2 + 4\sigma_2\sigma_1 + \sigma_1^2 - 60\sigma_2 - 28\sigma_1 + 216.
    \end{align*}
\end{lemma}
\begin{proof}
    The fixed point multiplier invariants give the equations
    \begin{align*}
        0&=(4a^2 - 4ab + 4b^2 + 12) -\sigma_1(ab - 1)\\
        0&=(4a^4 - 12a^3b + 22a^2b^2 - 12ab^3 + 4b^4 + 28a^2 - 52ab + 28b^2 + 54)  - \sigma_2(a^2b^2 - 2ab + 1)\\
        0&=(-8a^4 + 28a^3b - 36a^2b^2 + 28ab^3 - 8b^4 - 60a^2 + 96ab - 60b^2 - 108)- \sigma_3(a^2b^2 - 2ab + 1).
    \end{align*}
    Saturating with respect to $(ab-1)$ to avoid degeneracy and eliminating $a$ and $b$ give the stated equations.
\end{proof}
\begin{code}
\begin{lstlisting}
R.<a,b,s3,s2,s1>=QQ[]
P.<x,y>=ProjectiveSpace(R,1)
f=DynamicalSystem([a*x^2*y+y^3, x^3 + b*x*y^2])
SI=f.sigma_invariants(1)
I=R.ideal([SI[i].numerator() - SI[i].denominator()*R.gen(2+i) for i in range(len(SI)-1)])
I.elimination_ideal([a])
Ja=R.ideal(a*b-1)
I=I.saturation(Ja)[0]
EI=I.elimination_ideal([a,b])
EI
\end{lstlisting}
\end{code}

To study the geometry of this family, we need to use the multiplier invariants of the periodic points of period $2$.

\begin{lemma}\label{lem_a3_c2_2_geo}
    Define $g_{a,b}(z) = \frac{az^2+1}{z^3+bz}$. The map
        \begin{align*}
            \varphi:\bbA^2\setminus \{ab=1\} &\to \M_3\\
            (a,b) &\mapsto [g_{a,b}]
        \end{align*}
        is four-to-one. The map
        \begin{align*}
            \tau_1: \varphi(\bbA^2\setminus \{ab=1\}) &\to \bbA^5\\
            [g_{a,b}] &\mapsto (\sigma_1^{(1)}, \sigma_2^{(1)}, \sigma_3^{(1)}, \sigma_1^{(2)}, \sigma_2^{(2)})
        \end{align*}
        is injective.
\end{lemma}
\begin{proof}
    To compute the degree of $\varphi$, we consider the ideal generated by
    \begin{equation*}
        \left(\sigma_1(a,b), \sigma_2(a,b), \sigma_3(a,b), \sigma_1^{(2)}(a,b), \sigma_2^{(2)}(a,b)\right) = \left(\sigma_1(c,d), \sigma_2(c,d), \sigma_3(c,d), \sigma_1^{(2)}(c,d), \sigma_2^{(2)}(c,d)\right)
    \end{equation*}
    as an ideal in $K[a,b]$ where $K$ is the function field $\overline{\Q}(c,d)$. This forms a zero dimensional variety and Singular (via Sage) computes the degree of its projective closure as $4$. For (almost) every choice of $(c,d)$, we have the four pairs $(a,b) \in \{(c,d), (d,c), (-c,-d), (-d,-c)\}$. The functions $g_{a,b}(z)$ and $g_{b,a}(z)$ are conjugate via $\alpha_1(z) = \frac{1}{z}$, and $g_{a,b}(z)$ and $g_{-a,-b}(z)$ are conjugate via $\alpha_2(-z) = iz$. So the four points with the same $\bbA^5$ coordinates all are from the same conjugacy class.
\end{proof}
\begin{code}
\begin{lstlisting}
SS.<c,d>=PolynomialRing(QQ)
S.<a,b>=PolynomialRing(FractionField(SS))
g1=(4*a^2 - 4*a*b + 4*b^2 + 12)*(c*d -1) - (4*c^2 - 4*c*d + 4*d^2 + 12)*(a*b -1)
g2=(4*a^4 - 12*a^3*b + 22*a^2*b^2 - 12*a*b^3 + 4*b^4 + 28*a^2 - 52*a*b + 28*b^2 + 54)*(c^2*d^2 - 2*c*d + 1) - (4*c^4 - 12*c^3*d + 22*c^2*d^2 - 12*c*d^3 + 4*d^4 + 28*c^2 - 52*c*d + 28*d^2 + 54)*(a^2*b^2 - 2*a*b + 1)
g3=(-8*a^4 + 28*a^3*b - 36*a^2*b^2 + 28*a*b^3 - 8*b^4 - 60*a^2 + 96*a*b - 60*b^2 - 108)*(c^2*d^2 - 2*c*d + 1) - (-8*c^4 + 28*c^3*d - 36*c^2*d^2 + 28*c*d^3 - 8*d^4 - 60*c^2 + 96*c*d - 60*d^2 - 108)*(a^2*b^2 - 2*a*b + 1)
f1=(2*a^3*b^3 + 16*a^4 + 4*a^2*b^2 + 16*b^4 + 18*a*b + 72)*(c^2*d^2 - 2*c*d + 1) - (2*c^3*d^3 + 16*c^4 + 4*c^2*d^2 + 16*d^4 + 18*c*d + 72)*(a^2*b^2 - 2*a*b + 1)
f2=(a^6*b^6 + 32*a^7*b^3 + 12*a^5*b^5 + 32*a^3*b^7 + 96*a^8 + 16*a^6*b^2 + 290*a^4*b^4 + 16*a^2*b^6 + 96*b^8 + 512*a^5*b + 204*a^3*b^3 + 512*a*b^5 + 80*a^4 - 151*a^2*b^2 + 80*b^4 + 2268)*(c^4*d^4 - 4*c^3*d^3 + 6*c^2*d^2 - 4*c*d + 1)-(c^6*d^6 + 32*c^7*d^3 + 12*c^5*d^5 + 32*c^3*d^7 + 96*c^8 + 16*c^6*d^2 + 290*c^4*d^4 + 16*c^2*d^6 + 96*d^8 + 512*c^5*d + 204*c^3*d^3 + 512*c*d^5 + 80*c^4 - 151*c^2*d^2 + 80*d^4 + 2268)*(a^4*b^4 - 4*a^3*b^3 + 6*a^2*b^2 - 4*a*b + 1)
I=S.ideal([g1,g2,g3,f1,f2])
A=AffineSpace(S)
X=A.subscheme(I)
Z=X.projective_closure()
Z.degree()
\end{lstlisting}
\end{code}

\begin{prop}\label{prop_a3_c2_2_geo}
     The component of $\A_3(C_2)$ given by $g_{a,b}(z) = \frac{az^2+1}{z^3+bz}$ is an irreducible surface defined by
    {\tiny\begin{align*}
        0&=\sigma_1 - 2\sigma_2 - \sigma_3 + 12\\
        0&=4\sigma_2^2 + 4\sigma_2\sigma_3 + \sigma_3^2 - 60\sigma_2 - 28\sigma_3 + 216\\
        0&=1632\sigma_2\sigma_3^5 + 792\sigma_3^6 + 3648\sigma_2\sigma_3^4\sigma_1^{(2)} + 1760\sigma_3^5\sigma_1^{(2)} + 1248\sigma_2\sigma_3^3(\sigma_1^{(2)})^2 + 596\sigma_3^4(\sigma_1^{(2)})^2 - 1152\sigma_2\sigma_3^2(\sigma_1^{(2)})^3 - 544\sigma_3^3(\sigma_1^{(2)})^3\\
        &- 216\sigma_2\sigma_3(\sigma_1^{(2)})^4 - 98\sigma_3^2(\sigma_1^{(2)})^4 + 96\sigma_2(\sigma_1^{(2)})^5 + 40\sigma_3(\sigma_1^{(2)})^5 + (\sigma_1^{(2)})^6 - 414816\sigma_2\sigma_3^4 - 195424\sigma_3^5 - 602880\sigma_2\sigma_3^3\sigma_1^{(2)}\\
        &- 284512\sigma_3^4\sigma_1^{(2)} - 111648\sigma_2\sigma_3^2(\sigma_1^{(2)})^2 - 53760\sigma_3^3(\sigma_1^{(2)})^2 + 54720\sigma_2\sigma_3(\sigma_1^{(2)})^3 + 28048\sigma_3^2(\sigma_1^{(2)})^3 + 1176\sigma_2(\sigma_1^{(2)})^4 + 1480\sigma_3(\sigma_1^{(2)})^4\\
        &- 624(\sigma_1^{(2)})^5 + 5248\sigma_2\sigma_3^3\sigma_2^{(2)} + 2512\sigma_3^4\sigma_2^{(2)} + 6784\sigma_2\sigma_3^2\sigma_1^{(2)}\sigma_2^{(2)} + 3200\sigma_3^3\sigma_1^{(2)}\sigma_2^{(2)} - 128\sigma_2\sigma_3(\sigma_1^{(2)})^2\sigma_2^{(2)} - 64\sigma_3^2(\sigma_1^{(2)})^2\sigma_2^{(2)}\\
        &- 768\sigma_2(\sigma_1^{(2)})^3\sigma_2^{(2)} - 320\sigma_3(\sigma_1^{(2)})^3\sigma_2^{(2)} - 12(\sigma_1^{(2)})^4\sigma_2^{(2)} + 18436608\sigma_2\sigma_3^3 + 9945984\sigma_3^4 + 16035200\sigma_2\sigma_3^2\sigma_1^{(2)} + 9804928\sigma_3^3\sigma_1^{(2)}\\
        &+ 1465472\sigma_2\sigma_3(\sigma_1^{(2)})^2 + 1204512\sigma_3^2(\sigma_1^{(2)})^2 - 313536\sigma_2(\sigma_1^{(2)})^3 - 434624\sigma_3(\sigma_1^{(2)})^3 - 6720(\sigma_1^{(2)})^4 - 515968\sigma_2\sigma_3^2\sigma_2^{(2)}\\
        &- 247040\sigma_3^3\sigma_2^{(2)} - 306944\sigma_2\sigma_3\sigma_1^{(2)}\sigma_2^{(2)} - 158400\sigma_3^2\sigma_1^{(2)}\sigma_2^{(2)} + 9984\sigma_2(\sigma_1^{(2)})^2\sigma_2^{(2)}+ 4736\sigma_3(\sigma_1^{(2)})^2\sigma_2^{(2)} + 4992(\sigma_1^{(2)})^3\sigma_2^{(2)}\\
        &+ 3968\sigma_2\sigma_3(\sigma_2^{(2)})^2 + 1824\sigma_3^2(\sigma_2^{(2)})^2 + 1536\sigma_2\sigma_1^{(2)}(\sigma_2^{(2)})^2 + 640\sigma_3\sigma_1^{(2)}(\sigma_2^{(2)})^2 + 48(\sigma_1^{(2)})^2(\sigma_2^{(2)})^2 - 244431360\sigma_2\sigma_3^2\\
        &- 201129984\sigma_3^3 - 120701952\sigma_2\sigma_3\sigma_1^{(2)}
        - 138880512\sigma_3^2\sigma_1^{(2)} - 5380992\sigma_2(\sigma_1^{(2)})^2 - 10551168\sigma_3(\sigma_1^{(2)})^2 + 1890944(\sigma_1^{(2)})^3\\
        &+ 7087104\sigma_2\sigma_3\sigma_2^{(2)} + 5656320\sigma_3^2\sigma_2^{(2)} + 1651968\sigma_2\sigma_1^{(2)}\sigma_2^{(2)} + 2407680\sigma_3\sigma_1^{(2)}\sigma_2^{(2)} - 59520(\sigma_1^{(2)})^2\sigma_2^{(2)} - 58752\sigma_2(\sigma_2^{(2)})^2\\
        &- 42624\sigma_3(\sigma_2^{(2)})^2 - 9984\sigma_1^{(2)}(\sigma_2^{(2)})^2 - 64(\sigma_2^{(2)})^3 + 1187592192\sigma_2\sigma_3 + 1880381952\sigma_3^2 + 265006080\sigma_2\sigma_1^{(2)} + 812934144\sigma_3\sigma_1^{(2)}\\
        &+ 32237568(\sigma_1^{(2)})^2 - 20653056\sigma_2\sigma_2^{(2)} - 49351680\sigma_3\sigma_2^{(2)} - 9967104\sigma_1^{(2)}\sigma_2^{(2)} + 345600(\sigma_2^{(2)})^2 - 1903564800\sigma_2 - 7759079424\sigma_3\\
        &- 1591031808\sigma_1^{(2)} + 123669504\sigma_2^{(2)} + 11418402816
    \end{align*}
    \begin{align*}
        0&=1224\sigma_3^7 + 3792\sigma_3^6\sigma_1^{(2)} - 46800\sigma_2\sigma_3^4(\sigma_1^{(2)})^2 - 20564\sigma_3^5(\sigma_1^{(2)})^2 - 79680\sigma_2\sigma_3^3(\sigma_1^{(2)})^3 - 40856\sigma_3^4(\sigma_1^{(2)})^3 + 1680\sigma_2\sigma_3^2(\sigma_1^{(2)})^4\\
        &- 374\sigma_3^3(\sigma_1^{(2)})^4 + 27600\sigma_2\sigma_3(\sigma_1^{(2)})^5 + 13988\sigma_3^2(\sigma_1^{(2)})^5 - 5580\sigma_2(\sigma_1^{(2)})^6 - 2665\sigma_3(\sigma_1^{(2)})^6 - 22(\sigma_1^{(2)})^7 + 670152\sigma_3^6\\
        &+ 26777472\sigma_2\sigma_3^4\sigma_1^{(2)} + 14457920\sigma_3^5\sigma_1^{(2)} + 27403200\sigma_2\sigma_3^3(\sigma_1^{(2)})^2 + 13581308\sigma_3^4(\sigma_1^{(2)})^2 - 17123904\sigma_2\sigma_3^2(\sigma_1^{(2)})^3 - 8575280\sigma_3^3(\sigma_1^{(2)})^3\\
        &- 7057920\sigma_2\sigma_3(\sigma_1^{(2)})^4 - 3318902\sigma_3^2(\sigma_1^{(2)})^4 + 3338448\sigma_2(\sigma_1^{(2)})^5 + 1331704\sigma_3(\sigma_1^{(2)})^5 + 78251(\sigma_1^{(2)})^6 + 120768\sigma_2\sigma_3^4\sigma_2^{(2)}\\
        &+ 66096\sigma_3^5\sigma_2^{(2)} + 204416\sigma_2\sigma_3^3\sigma_1^{(2)}\sigma_2^{(2)} + 114464\sigma_3^4\sigma_1^{(2)}\sigma_2^{(2)} - 142720\sigma_2\sigma_3^2(\sigma_1^{(2)})^2\sigma_2^{(2)} - 67136\sigma_3^3(\sigma_1^{(2)})^2\sigma_2^{(2)} - 179968\sigma_2\sigma_3(\sigma_1^{(2)})^3\sigma_2^{(2)}\\
        &- 93248\sigma_3^2(\sigma_1^{(2)})^3\sigma_2^{(2)} + 58512\sigma_2(\sigma_1^{(2)})^4\sigma_2^{(2)} + 28460\sigma_3(\sigma_1^{(2)})^4\sigma_2^{(2)} + 264(\sigma_1^{(2)})^5\sigma_2^{(2)} - 8650305696\sigma_2\sigma_3^4 - 4002683936\sigma_3^5\\
        &- 14959662720\sigma_2\sigma_3^3\sigma_1^{(2)} - 6856974368\sigma_3^4\sigma_1^{(2)} - 3459567584\sigma_2\sigma_3^2(\sigma_1^{(2)})^2 - 1619872288\sigma_3^3(\sigma_1^{(2)})^2 + 1631518720\sigma_2\sigma_3(\sigma_1^{(2)})^3\\
        &+ 743089904\sigma_3^2(\sigma_1^{(2)})^3 + 44906280\sigma_2(\sigma_1^{(2)})^4 + 48051640\sigma_3(\sigma_1^{(2)})^4 - 22136400(\sigma_1^{(2)})^5 + 53976320\sigma_2\sigma_3^3\sigma_2^{(2)} + 28303472\sigma_3^4\sigma_2^{(2)}\\
        &+ 128916096\sigma_2\sigma_3^2\sigma_1^{(2)}\sigma_2^{(2)} + 63843648\sigma_3^3\sigma_1^{(2)}\sigma_2^{(2)} + 5531008\sigma_2\sigma_3(\sigma_1^{(2)})^2\sigma_2^{(2)} + 2665024\sigma_3^2(\sigma_1^{(2)})^2\sigma_2^{(2)} - 27134208\sigma_2(\sigma_1^{(2)})^3\sigma_2^{(2)}\\
        &- 11018304\sigma_3(\sigma_1^{(2)})^3\sigma_2^{(2)} - 884100(\sigma_1^{(2)})^4\sigma_2^{(2)} + 352512\sigma_2\sigma_3^2(\sigma_2^{(2)})^2 + 184416\sigma_3^3(\sigma_2^{(2)})^2 + 278272\sigma_2\sigma_3\sigma_1^{(2)}(\sigma_2^{(2)})^2\\
        &+ 149184\sigma_3^2\sigma_1^{(2)}(\sigma_2^{(2)})^2
        - 200256\sigma_2(\sigma_1^{(2)})^2(\sigma_2^{(2)})^2 - 99760\sigma_3(\sigma_1^{(2)})^2(\sigma_2^{(2)})^2 - 1056(\sigma_1^{(2)})^3(\sigma_2^{(2)})^2 + 516653715456\sigma_2\sigma_3^3\\
        &+ 263589240192\sigma_3^4 + 484361825920\sigma_2\sigma_3^2\sigma_1^{(2)} + 278148291968\sigma_3^3\sigma_1^{(2)} + 48513480448\sigma_2\sigma_3(\sigma_1^{(2)})^2 + 38458271328\sigma_3^2(\sigma_1^{(2)})^2\\
        &- 10359552576\sigma_2(\sigma_1^{(2)})^3 - 13327529920\sigma_3(\sigma_1^{(2)})^3 - 254963456(\sigma_1^{(2)})^4 - 14139669632\sigma_2\sigma_3^2\sigma_2^{(2)} - 6383331328\sigma_3^3\sigma_2^{(2)}\\
        &- 9563209216\sigma_2\sigma_3\sigma_1^{(2)}\sigma_2^{(2)} - 4488508992\sigma_3^2\sigma_1^{(2)}\sigma_2^{(2)} + 294902016\sigma_2(\sigma_1^{(2)})^2\sigma_2^{(2)} + 92512768\sigma_3(\sigma_1^{(2)})^2\sigma_2^{(2)} + 179583360(\sigma_1^{(2)})^3\sigma_2^{(2)}\\
        &+ 98552320\sigma_2\sigma_3(\sigma_2^{(2)})^2 + 46508640\sigma_3^2(\sigma_2^{(2)})^2 + 55121664\sigma_2\sigma_1^{(2)}(\sigma_2^{(2)})^2 + 22765952\sigma_3\sigma_1^{(2)}(\sigma_2^{(2)})^2 + 3316752(\sigma_1^{(2)})^2(\sigma_2^{(2)})^2\\
        &+ 221952\sigma_2(\sigma_2^{(2)})^3 + 114240\sigma_3(\sigma_2^{(2)})^3 + 1408\sigma_1^{(2)}(\sigma_2^{(2)})^3 - 7578572244480\sigma_2\sigma_3^2 - 5917873162752\sigma_3^3 - 3908513912832\sigma_2\sigma_3\sigma_1^{(2)}\\
        &- 4290578293248\sigma_3^2\sigma_1^{(2)} - 184908562560\sigma_2(\sigma_1^{(2)})^2 - 351150788736\sigma_3(\sigma_1^{(2)})^2 + 62583639424(\sigma_1^{(2)})^3 + 225002594304\sigma_2\sigma_3\sigma_2^{(2)}\\
        &+ 166476388608\sigma_3^2\sigma_2^{(2)} + 55674839808\sigma_2\sigma_1^{(2)}\sigma_2^{(2)} + 76642562304\sigma_3\sigma_1^{(2)}\sigma_2^{(2)} - 1751409024(\sigma_1^{(2)})^2\sigma_2^{(2)} - 1933117056\sigma_2(\sigma_2^{(2)})^2\\
        &- 1197765504\sigma_3(\sigma_2^{(2)})^2 - 364151040\sigma_1^{(2)}(\sigma_2^{(2)})^2 - 4129472(\sigma_2^{(2)})^3 + 38553253060608\sigma_2\sigma_3 + 58922313076224\sigma_3^2\\
        &+ 8872947753984\sigma_2\sigma_1^{(2)} + 26426257744896\sigma_3\sigma_1^{(2)} + 1107358166016(\sigma_1^{(2)})^2 - 687110999040\sigma_2\sigma_2^{(2)} - 1576478785536\sigma_3\sigma_2^{(2)}\\
        &- 336460829184\sigma_1^{(2)}\sigma_2^{(2)} + 11296544256(\sigma_2^{(2)})^2 - 63399291660288\sigma_2 - 252407226636288\sigma_3 - 53281131595776\sigma_1^{(2)}\\
        &+ 4111788303360\sigma_2^{(2)} + 380265217671168.
    \end{align*}}
    This surface is singular with singular locus given by the conjugacy classes described by
    \begin{equation*}
        \{(a,b) : a^2 - ab + b^2 + 3 = 0\} \cup \{(a,b) : a=-b\}.
    \end{equation*}
\end{prop}
\begin{proof}
    We compute the fixed point multiplier invariants and the first two multiplier invariants for the points of period 2:
    \begin{align*}
        \sigma_1 &= \frac{4a^2 - 4ab + 4b^2 + 12}{ab -1}\\
        \sigma_2 &= \frac{4a^4 - 12a^3b + 22a^2b^2 - 12ab^3 + 4b^4 + 28a^2 - 52ab + 28b^2 + 54}{a^2b^2 - 2ab + 1}\\
        \sigma_3 &= \frac{-8a^4 + 28a^3b - 36a^2b^2 + 28ab^3 - 8b^4 - 60a^2 + 96ab - 60b^2 - 108}{a^2b^2 - 2ab + 1}\\
        \sigma_1^{(2)} &= \frac{2a^3b^3 + 16a^4 + 4a^2b^2 + 16b^4 + 18ab + 72}{a^2b^2 - 2ab + 1}\\
        \sigma_2^{(2)} &= \frac{1}{(ab-1)^4}(a^6b^6 + 32a^7b^3 + 12a^5b^5 + 32a^3b^7 + 96a^8 + 16a^6b^2 + 290a^4b^4 + 16a^2b^6 + 96b^8\\
        &+ 512a^5b + 204a^3b^3 + 512ab^5 + 80a^4 - 151a^2b^2 + 80b^4 + 2268).
    \end{align*}
    We clear denominators to create the associated ideal. We saturate by the ideal $(ab-1)$ to avoid degeneracy and eliminate the variables $a$ and $b$. This results in the surface defined by the four equations in the statement.

    In Magma the irreducible components of the singular locus are calculated. For each component, we eliminate variables to have equations in $(\sigma_1,\sigma_2,\sigma_3)$ to get
    \begin{align*}
        S_1&: \begin{cases}
            0=\sigma_2^2 - 24\sigma_2 - 4\sigma_3 + 108\\
            0=\sigma_2\sigma_3 + 36\sigma_2 + 6\sigma_3 - 216\\
            0=\sigma_3^2 - 108\sigma_2 - 36\sigma_3 + 648\\
            0=\sigma_1 - 2\sigma_2 - \sigma_3 + 12
        \end{cases}\\
        S_2&: \begin{cases}
            0=\sigma_2^2 - 3048\sigma_2 - 1372\sigma_3 + 13500\\
            0=\sigma_2\sigma_3 + 6408\sigma_2 + 2886\sigma_3 - 28512\\
            0=\sigma_3^2 - 13500\sigma_2 - 6084\sigma_3 + 60264\\
            0=\sigma_1 - 2\sigma_2 - \sigma_3 + 12
        \end{cases}\\
        S_3&: \begin{cases}
            0=\sigma_3^3 + 549\sigma_3^2 - 3037500\sigma_2 - 1333908\sigma_3 + 14819112\\
            0=\sigma_2^2 - 7/30\sigma_3^2 + 102\sigma_2 + 238/5\sigma_3 - 2808/5\\
            0=\sigma_2\sigma_3 + 29/60\sigma_3^2 - 117\sigma_2 - 273/5\sigma_3 + 3078/5\\
            0=\sigma_1 - 2\sigma_2 - \sigma_3 + 12.
        \end{cases}
    \end{align*}
    To determine the pairs $(a,b)$ for each of these components, we add in the equations defining the invariants in terms of $a$ and $b$ and eliminate $\sigma_1,\sigma_2,\sigma_3$. This results in the defining equations
    \begin{align*}
        S_1&: (a^2 - ab + b^2 + 3)^3=0 \\
        S_2&: (a + b)^6=0\\
        S_3&: (a + b)^8=0.
    \end{align*}
    Note that the component(s) with $a=-b$ is the component of $\A_3(D_2)$ from Proposition \ref{prop_D2_2}.
\end{proof}

\begin{code}
\begin{lstlisting}
R.<a,b,s3,s2,s1>=QQ[]
P.<x,y>=ProjectiveSpace(R,1)
f=DynamicalSystem([a*x^2*y+y^3, x^3 + b*x*y^2])
f.sigma_invariants(1)
f.sigma_invariants(2) #18 minutes


S.<a,b,s1,s2,s3,s5,s6>=PolynomialRing(QQ)
g1=4*a^2 - 4*a*b + 4*b^2 + 12 - s1*(a*b -1)
g2=4*a^4 - 12*a^3*b + 22*a^2*b^2 - 12*a*b^3 + 4*b^4 + 28*a^2 - 52*a*b + 28*b^2 + 54 - s2*(a^2*b^2 - 2*a*b + 1)
g3=-8*a^4 + 28*a^3*b - 36*a^2*b^2 + 28*a*b^3 - 8*b^4 - 60*a^2 + 96*a*b - 60*b^2 - 108 - s3*(a^2*b^2 - 2*a*b + 1)
f1=(2*a^3*b^3 + 16*a^4 + 4*a^2*b^2 + 16*b^4 + 18*a*b + 72) - s5*(a^2*b^2 - 2*a*b + 1)
f2=(a^6*b^6 + 32*a^7*b^3 + 12*a^5*b^5 + 32*a^3*b^7 + 96*a^8 + 16*a^6*b^2 + 290*a^4*b^4 + 16*a^2*b^6 + 96*b^8 + 512*a^5*b + 204*a^3*b^3 + 512*a*b^5 + 80*a^4 - 151*a^2*b^2 + 80*b^4 + 2268)-s6*(a^4*b^4 - 4*a^3*b^3 + 6*a^2*b^2 - 4*a*b + 1)
I=S.ideal([g1,g2,g3,f1,f2])
I=I.saturation(S.ideal(a*b-1))[0]
EI=I.elimination_ideal([a,b])

A.<s1,s2,s3,s5,s6>=AffineSpace(QQ,5)
phi=S.hom([0,0,s1,s2,s3,s5,s6],A.coordinate_ring())
C=A.subscheme([phi(t) for t in EI.gens()])
C.dimension()

for t in EI.gens():
    print(phi(t).homogenize('h'))

# magma:
P3<s1,s2,s3,s5,s6,h>:=ProjectiveSpace(Rationals(),5);
g1:=s1 - 2*s2 - s3 + 12*h;
g2:=4*s2^2 + 4*s2*s3 + s3^2 - 60*s2*h - 28*s3*h + 216*h^2;
g3:=1632*s2*s3^5 + 792*s3^6 + 3648*s2*s3^4*s5 + 1760*s3^5*s5 + 1248*s2*s3^3*s5^2 + 596*s3^4*s5^2 - 1152*s2*s3^2*s5^3 - 544*s3^3*s5^3 - 216*s2*s3*s5^4 - 98*s3^2*s5^4 + 96*s2*s5^5 + 40*s3*s5^5 + s5^6 - 414816*s2*s3^4*h - 195424*s3^5*h - 602880*s2*s3^3*s5*h - 284512*s3^4*s5*h - 111648*s2*s3^2*s5^2*h - 53760*s3^3*s5^2*h + 54720*s2*s3*s5^3*h + 28048*s3^2*s5^3*h + 1176*s2*s5^4*h + 1480*s3*s5^4*h - 624*s5^5*h + 5248*s2*s3^3*s6*h + 2512*s3^4*s6*h + 6784*s2*s3^2*s5*s6*h + 3200*s3^3*s5*s6*h - 128*s2*s3*s5^2*s6*h - 64*s3^2*s5^2*s6*h - 768*s2*s5^3*s6*h - 320*s3*s5^3*s6*h - 12*s5^4*s6*h + 18436608*s2*s3^3*h^2 + 9945984*s3^4*h^2 + 16035200*s2*s3^2*s5*h^2 + 9804928*s3^3*s5*h^2 + 1465472*s2*s3*s5^2*h^2 + 1204512*s3^2*s5^2*h^2 - 313536*s2*s5^3*h^2 - 434624*s3*s5^3*h^2 - 6720*s5^4*h^2 - 515968*s2*s3^2*s6*h^2 - 247040*s3^3*s6*h^2 - 306944*s2*s3*s5*s6*h^2 - 158400*s3^2*s5*s6*h^2 + 9984*s2*s5^2*s6*h^2 + 4736*s3*s5^2*s6*h^2 + 4992*s5^3*s6*h^2 + 3968*s2*s3*s6^2*h^2 + 1824*s3^2*s6^2*h^2 + 1536*s2*s5*s6^2*h^2 + 640*s3*s5*s6^2*h^2 + 48*s5^2*s6^2*h^2 - 244431360*s2*s3^2*h^3 - 201129984*s3^3*h^3 - 120701952*s2*s3*s5*h^3 - 138880512*s3^2*s5*h^3 - 5380992*s2*s5^2*h^3 - 10551168*s3*s5^2*h^3 + 1890944*s5^3*h^3 + 7087104*s2*s3*s6*h^3 + 5656320*s3^2*s6*h^3 + 1651968*s2*s5*s6*h^3 + 2407680*s3*s5*s6*h^3 - 59520*s5^2*s6*h^3 - 58752*s2*s6^2*h^3 - 42624*s3*s6^2*h^3 - 9984*s5*s6^2*h^3 - 64*s6^3*h^3 + 1187592192*s2*s3*h^4 + 1880381952*s3^2*h^4 + 265006080*s2*s5*h^4 + 812934144*s3*s5*h^4 + 32237568*s5^2*h^4 - 20653056*s2*s6*h^4 - 49351680*s3*s6*h^4 - 9967104*s5*s6*h^4 + 345600*s6^2*h^4 - 1903564800*s2*h^5 - 7759079424*s3*h^5 - 1591031808*s5*h^5 + 123669504*s6*h^5 + 11418402816*h^6;
g4:=1224*s3^7 + 3792*s3^6*s5 - 46800*s2*s3^4*s5^2 - 20564*s3^5*s5^2 - 79680*s2*s3^3*s5^3 - 40856*s3^4*s5^3 + 1680*s2*s3^2*s5^4 - 374*s3^3*s5^4 + 27600*s2*s3*s5^5 + 13988*s3^2*s5^5 - 5580*s2*s5^6 - 2665*s3*s5^6 - 22*s5^7 + 670152*s3^6*h + 26777472*s2*s3^4*s5*h + 14457920*s3^5*s5*h + 27403200*s2*s3^3*s5^2*h + 13581308*s3^4*s5^2*h - 17123904*s2*s3^2*s5^3*h - 8575280*s3^3*s5^3*h - 7057920*s2*s3*s5^4*h - 3318902*s3^2*s5^4*h + 3338448*s2*s5^5*h + 1331704*s3*s5^5*h + 78251*s5^6*h + 120768*s2*s3^4*s6*h + 66096*s3^5*s6*h + 204416*s2*s3^3*s5*s6*h + 114464*s3^4*s5*s6*h - 142720*s2*s3^2*s5^2*s6*h - 67136*s3^3*s5^2*s6*h - 179968*s2*s3*s5^3*s6*h - 93248*s3^2*s5^3*s6*h + 58512*s2*s5^4*s6*h + 28460*s3*s5^4*s6*h + 264*s5^5*s6*h - 8650305696*s2*s3^4*h^2 - 4002683936*s3^5*h^2 - 14959662720*s2*s3^3*s5*h^2 - 6856974368*s3^4*s5*h^2 - 3459567584*s2*s3^2*s5^2*h^2 - 1619872288*s3^3*s5^2*h^2 + 1631518720*s2*s3*s5^3*h^2 + 743089904*s3^2*s5^3*h^2 + 44906280*s2*s5^4*h^2 + 48051640*s3*s5^4*h^2 - 22136400*s5^5*h^2 + 53976320*s2*s3^3*s6*h^2 + 28303472*s3^4*s6*h^2 + 128916096*s2*s3^2*s5*s6*h^2 + 63843648*s3^3*s5*s6*h^2 + 5531008*s2*s3*s5^2*s6*h^2 + 2665024*s3^2*s5^2*s6*h^2 - 27134208*s2*s5^3*s6*h^2 - 11018304*s3*s5^3*s6*h^2 - 884100*s5^4*s6*h^2 + 352512*s2*s3^2*s6^2*h^2 + 184416*s3^3*s6^2*h^2 + 278272*s2*s3*s5*s6^2*h^2 + 149184*s3^2*s5*s6^2*h^2 - 200256*s2*s5^2*s6^2*h^2 - 99760*s3*s5^2*s6^2*h^2 - 1056*s5^3*s6^2*h^2 + 516653715456*s2*s3^3*h^3 + 263589240192*s3^4*h^3 + 484361825920*s2*s3^2*s5*h^3 + 278148291968*s3^3*s5*h^3 + 48513480448*s2*s3*s5^2*h^3 + 38458271328*s3^2*s5^2*h^3 - 10359552576*s2*s5^3*h^3 - 13327529920*s3*s5^3*h^3 - 254963456*s5^4*h^3 - 14139669632*s2*s3^2*s6*h^3 - 6383331328*s3^3*s6*h^3 - 9563209216*s2*s3*s5*s6*h^3 - 4488508992*s3^2*s5*s6*h^3 + 294902016*s2*s5^2*s6*h^3 + 92512768*s3*s5^2*s6*h^3 + 179583360*s5^3*s6*h^3 + 98552320*s2*s3*s6^2*h^3 + 46508640*s3^2*s6^2*h^3 + 55121664*s2*s5*s6^2*h^3 + 22765952*s3*s5*s6^2*h^3 + 3316752*s5^2*s6^2*h^3 + 221952*s2*s6^3*h^3 + 114240*s3*s6^3*h^3 + 1408*s5*s6^3*h^3 - 7578572244480*s2*s3^2*h^4 - 5917873162752*s3^3*h^4 - 3908513912832*s2*s3*s5*h^4 - 4290578293248*s3^2*s5*h^4 - 184908562560*s2*s5^2*h^4 - 351150788736*s3*s5^2*h^4 + 62583639424*s5^3*h^4 + 225002594304*s2*s3*s6*h^4 + 166476388608*s3^2*s6*h^4 + 55674839808*s2*s5*s6*h^4 + 76642562304*s3*s5*s6*h^4 - 1751409024*s5^2*s6*h^4 - 1933117056*s2*s6^2*h^4 - 1197765504*s3*s6^2*h^4 - 364151040*s5*s6^2*h^4 - 4129472*s6^3*h^4 + 38553253060608*s2*s3*h^5 + 58922313076224*s3^2*h^5 + 8872947753984*s2*s5*h^5 + 26426257744896*s3*s5*h^5 + 1107358166016*s5^2*h^5 - 687110999040*s2*s6*h^5 - 1576478785536*s3*s6*h^5 - 336460829184*s5*s6*h^5 + 11296544256*s6^2*h^5 - 63399291660288*s2*h^6 - 252407226636288*s3*h^6 - 53281131595776*s5*h^6 + 4111788303360*s6*h^6 + 380265217671168*h^7;
S:= Surface(P3, [g1,g2,g3,g4]);
D:=SingularSubscheme(S);
IC:=IrreducibleComponents(D);
for i in [1,2,3,4,5] do
  print EliminationIdeal(Ideal(IC[i]),{s1,s2,s3,h});
end for;


IsRational(S); False
IsReduced(S); True
IsIrreducible(S); True
IsSingular(S); True



#irreducible components
R.<a,b,s1,s2,s3>=QQ[]
A=AffineSpace(R)
g1=-a*b*s1 + 4*a^2 - 4*a*b + 4*b^2 + s1 + 12
g2=-a^2*b^2*s2 + 4*a^4 - 12*a^3*b + 22*a^2*b^2 - 12*a*b^3 + 4*b^4 + 2*a*b*s2 + 28*a^2 - 52*a*b + 28*b^2 - s2 + 54
g3 =-8*a^4 + 28*a^3*b - 36*a^2*b^2 + 28*a*b^3 - 8*b^4 - 60*a^2 + 96*a*b - 60*b^2 - 108 - s3*(a^2*b^2 - 2*a*b + 1)
h=1
h1= s2^2 - 24*s2*h - 4*s3*h + 108*h^2
h2= s2*s3 + 36*s2*h + 6*s3*h - 216*h^2
h3= s3^2 - 108*s2*h - 36*s3*h + 648*h^2
h4= s1 - 2*s2 - s3 + 12*h
I=R.ideal([h1,h2,h3,h4] +[g1,g2,g3])
J=I.elimination_ideal([s1,s2,s3])
J2=J.saturation(R.ideal(a*b-1))[0]
for G in J2.gens():
    print(G.factor())

#irreducible components
R.<a,b,s1,s2,s3>=QQ[]
A=AffineSpace(R)
g1=-a*b*s1 + 4*a^2 - 4*a*b + 4*b^2 + s1 + 12
g2=-a^2*b^2*s2 + 4*a^4 - 12*a^3*b + 22*a^2*b^2 - 12*a*b^3 + 4*b^4 + 2*a*b*s2 + 28*a^2 - 52*a*b + 28*b^2 - s2 + 54
g3 =-8*a^4 + 28*a^3*b - 36*a^2*b^2 + 28*a*b^3 - 8*b^4 - 60*a^2 + 96*a*b - 60*b^2 - 108 - s3*(a^2*b^2 - 2*a*b + 1)
h=1
h1= s2^2 - 3048*s2*h - 1372*s3*h + 13500*h^2
h2=s2*s3 + 6408*s2*h + 2886*s3*h - 28512*h^2
h3=s3^2 - 13500*s2*h - 6084*s3*h + 60264*h^2
h4=s1 - 2*s2 - s3 + 12*h
I=R.ideal([h1,h2,h3,h4] +[g1,g2,g3])
J=I.elimination_ideal([s1,s2,s3])
J2=J.saturation(R.ideal(a*b-1))[0]
for G in J2.gens():
    print(G.factor())


#irreducible components
R.<a,b,s1,s2,s3>=QQ[]
A=AffineSpace(R)
g1=-a*b*s1 + 4*a^2 - 4*a*b + 4*b^2 + s1 + 12
g2=-a^2*b^2*s2 + 4*a^4 - 12*a^3*b + 22*a^2*b^2 - 12*a*b^3 + 4*b^4 + 2*a*b*s2 + 28*a^2 - 52*a*b + 28*b^2 - s2 + 54
g3 =-8*a^4 + 28*a^3*b - 36*a^2*b^2 + 28*a*b^3 - 8*b^4 - 60*a^2 + 96*a*b - 60*b^2 - 108 - s3*(a^2*b^2 - 2*a*b + 1)
h=1
h1=s3^3 + 549*s3^2*h - 3037500*s2*h^2 - 1333908*s3*h^2 + 14819112*h^3
h2=s2^2 - 7/30*s3^2 + 102*s2*h + 238/5*s3*h - 2808/5*h^2
h3=s2*s3 + 29/60*s3^2 - 117*s2*h - 273/5*s3*h + 3078/5*h^2
h4=s1 - 2*s2 - s3 + 12*h
I=R.ideal([h1,h2,h3,h4] +[g1,g2,g3])
J=I.elimination_ideal([s1,s2,s3])
J2=J.saturation(R.ideal(a*b-1))[0]
for G in J2.gens():
    print(G.factor())

\end{lstlisting}
\end{code}

\section{Automorphism loci in $\M_4$} \label{sect_M4}
    As for $\A_3$ we utilize Miasnikov, Stout, and Williams \cite{Miasnikov} to determine the possible components of $\A_4$ and their dimensions. Unlike in $\M_3$, in $\M_4$ there are no automorphism groups other than cyclic and dihedral ones.
    \begin{lemma} \label{lem_containment_M4}
        When $\deg(f)=4$, $\Aut(f)$ must either be $C_2$, $C_3$, $C_4$, $C_5$, $D_3$, or $D_5$. The dimensions of $\A_4(\Gamma)$ for these groups are given by
        \begin{enumerate}
            \item $\dim \A_4(C_2)=3,
            \dim \A_4(C_3)=2$, $\dim\A_4(C_4)=1$, and $\dim\A_4(C_5)=0$.
            \item $\dim \A_4(D_3)=1$, and $\dim \A_4(D_5)=0$.
        \end{enumerate}
        Moreover,
        \begin{enumerate}
        \setcounter{enumi}{3}
            \item $\A_4(C_4)\subset \A_4(C_2)$.
            \item $\A_4(D_3)\subset \A_4(C_3)$.
            \item $\A_4(D_5)\subset \A_4(C_5)$.
        \end{enumerate}
    \end{lemma}
    \begin{proof}
        Which groups occur and the dimensions are calculated in Section 2 of MSW \cite{Miasnikov}. The second half follows from the observation that if $G$ is a subgroup of $H$, then $\A_d(H)\subseteq \A_d(G)$.
    \end{proof}

    The analysis in $\M_4$ for functions whose automorphism group contains cyclic and dihedral groups is more or less the same as in $\M_3$ -- in particular, we made use of the fact that for a function $f$, if $\Aut(f) \supseteq C_n$ then the equivalence class of $f$ in $\M_4$ can be written as $f = z\cdot\psi(z^n)$, where $\psi$ is a rational function \cite[Proposition 7.3]{Silverman12}.

\subsection{$\A_4(C_5)$ and $\A_4(D_5)$}

    \begin{prop}\label{prop_A4_C5}
        The loci $\A_4(C_5)$ and $\A_4(D_5)$ both are the single conjugacy class given by $f(z) = \frac{1}{z^4}$.
    \end{prop}
    \begin{proof}
        The only form $\psi$ can take such that $z\cdot\psi(z^5)$ is degree four is when $\psi(z^5) = \frac{a}{bz^5}$, where $ab \neq 0$. We can divide through by $a$ to write $z\cdot\psi(z^5) = \frac{1}{cz^4}$; and since this is conjugate to $f(z) = \frac{1}{z^4}$, we know that $\A_4(C_5)$ is just a point in $\M_4$. Furthermore, $\Aut(f) = D_5$ so that $\A_4(C_5) = \A_4(D_5)$.
    \end{proof}

\subsection{$\A_4(C_4)$}
    \begin{prop}\label{prop_A4_C4}
        The locus $\A_4(C_4)$ is an irreducible curve in $\M_4$ given by the 1-parameter family $f_k(z) = \frac{z^4 + 1}{kz^3}$ for $k \neq 0$.
    \end{prop}
    \begin{proof}
        Functions of the form $f(z) = z\cdot\psi(z^4)$ have degree 4 only when
        \begin{equation*}
            f_1(z) = \frac{az^4+b}{cz^3} \quad \text{or} \quad
            f_2(z) = \frac{bz}{cz^4+d},
        \end{equation*}
        where all the coefficients must be nonzero (or else we have a drop in degree due to cancellation). Thus, in both cases we can divide through by the coefficient of $z^4$ to have two-parameter families. Then in the first case, we can conjugate $f_1(z) = \frac{z^4 + k_1}{k_2z^3}$ via the matrix $\alpha = \begin{bmatrix}\sqrt[4]{k_1}&0\\0&1\end{bmatrix}$ to get a one-parameter family $f_1'(z) = \frac{z^4 + 1}{kz^3}$, which is what we should expect since the dimension of this locus is 1 by Lemma \ref{lem_containment_M4}. In the second case, we can  conjugate $f_2(z) = \frac{k_1z}{z^4 + k_2}$ by the matrix $\beta = \begin{bmatrix}\sqrt[4]{k_2}&0\\0&1\end{bmatrix}$ to get a one-parameter family $f_2'(z) =  \frac{kz}{z^4 + 1}$. These two one-parameter families are conjugate via $\begin{bmatrix}0&1\\1&0\end{bmatrix}$, so they are the same family in the moduli space.
    \end{proof}

\subsection{$\A_4(C_3)$ and $\A_4(D_3)$}
    \begin{prop}\label{prop_A4_C3}
        The locus $\A_4(C_3)$ is given by the family $f_{k_1,k_2}(z)=\frac{z^4+k_1z}{k_2z^3+1}$.
    \end{prop}
    \begin{proof}
        By Lemma \ref{lem_containment_M4} the locus $\A_4(C_3)$ is dimension 2 and functions of the form $f(z) = z\cdot\psi(z^3)$ have degree 4 only when
        \begin{equation*}
            f(z) = \frac{az^4+bz}{cz^3+d},
        \end{equation*}
        where $a\neq 0$ and $d\neq 0$. Thus, dividing through by $a$, we get the $3$-parameter family
        \begin{equation*}
            f(z) = \frac{z^4+k_1z}{k_2z^3+k_3},
        \end{equation*}
        where $k_3\neq 0$. Conjugating via the matrix $\alpha=\begin{bmatrix}\sqrt[3]{k_3}&0\\0&1\end{bmatrix}$, we get \begin{equation*}
            f^\alpha(z)=\frac{(\sqrt[3]{k_3}z)^4+k_1\sqrt[3]{k_3}z}{k_2(\sqrt[3]{k_3}z)^3+k_3}\cdot\frac{1}{\sqrt[3]{k_3}}=\frac{k_3z^4+k_1z}{k_2k_3z^3+k_3}=\frac{z^4+k_1/k_3z}{k_2z^3+1}.
        \end{equation*}
        Renaming $k_1=k_1/k_3$ and $k_2=k_2$, we see that $f^{\alpha}(z)=\frac{z^4+k_1z}{k_2z^3+1}$. Thus, every degree $4$ rational map with a $C_3$ automorphism is conjugate to a map of the form \begin{equation*}
            f_{k_1,k_2}(z)=\frac{z^4+k_1z}{k_2z^3+1}.
        \end{equation*}
    \end{proof}

    Recall that $D_3$ is generated by $\alpha_1(z)=\zeta_3z$ and $\alpha_2(z)=1/z$, where $\zeta_3$ is a primitive third root of unity. Maps with automorphism group containing $\zeta_3z$ are described in Proposition \ref{prop_A4_C3}, so we can start with that family and see which members additionally have $\alpha_2$ as automorphism.
    \begin{prop}\label{prop_A4_D3}
        The locus $\A_4(D_3)$ is an irreducible curve in $\M_4$ given by the family $f_k(z) = \frac{z^4 + kz}{kz^3+1}$.
    \end{prop}
    \begin{proof}
        Let $f_{k_1,k_2}(z) =  \frac{z^4 + k_1z}{k_2z^3+1}$. We compute
        \begin{equation*}
            f_{k_1,k_2}^{\alpha_2}(z) =  \frac{z^4 + k_2z}{k_1z^3+1}.
        \end{equation*}
        So for $f_{k_1,k_2}(z) = f_{k_1,k_2}^{\alpha_2}(z)$ we must have $k_1=k_2$.
    \end{proof}

\subsection{$\A_4(C_2)$}
\begin{prop}\label{prop_A4_C2}
    The locus $\A_4(C_2)$ is given by the 3-parameter family $ f_{k_1,k_2,k_3}(z)=\frac{z^4+k_1z^2+1}{k_2z^3+k_3z}$.
\end{prop}
\begin{proof}
    By Lemma \ref{lem_containment_M4} we know the automorphism loci of $C_2$ in the moduli space has dimension $3$. Furthermore, we know that a degree $4$ map $f$ has a $C_2$ automorphism if and only if it is of the form
        \begin{equation*}
            f_1(z) = \frac{az^4+bz^2+c}{dz^3+ez} \quad \text{or} \quad
            f_2(z) = \frac{az^3+bz}{cz^4+dz^2+e}.
        \end{equation*}
    First observe that for the family $f_2(z)$, we can conjugate by the matrix $\beta=\begin{bmatrix}0&1\\1&0\end{bmatrix}$ to obtain
    \begin{equation*}
        f^{\beta}_2(z) = \frac{c+dz^2+ez^4}{az+bz^3}.
    \end{equation*}
    Thus, the families $f_1$ and $f_2$ are the same in moduli space, so we consider only $f_1(z)$. For $f_1(z)$ to be degree $4$, we need $a\neq 0$ and $c\neq 0$. Dividing through by $a$, we get     \begin{equation*}
        f_1(z)=\frac{z^4+k_1z^2+k_2}{k_3z^3+k_4z},
    \end{equation*}
    where $k_2\neq 0$. Conjugating via the matrix $\alpha=\begin{bmatrix}\sqrt[4]{k_2}&0\\0&1\end{bmatrix}$, we get
    \begin{equation*}
        f_1^{\alpha}(z)=\frac{(\sqrt[4]{k_2})^4z^4+k_1(\sqrt[4]{k_2})^2z^2+k_2}{k_3(\sqrt[4]{k_2})^3z^3+k_4\sqrt[4]{k_2}z}\cdot\frac{1}{\sqrt[4]{k_2}}
        =\frac{z^4+\frac{k_1}{\sqrt[4]{k_2}^2}z^2+1}{k_3z^3+\frac{k_4}{\sqrt[4]{k_2}^2}z}.
    \end{equation*}
    Renaming $k_1=\frac{k_1}{\sqrt[4]{k_2}^2}$, and $k_2=k_3$, and $k_3=\frac{k_4}{\sqrt[4]{k_2}^2}$, we see that every $f_1$ map is conjugate to a map of the form
    \begin{equation*}
        f_{k_1,k_2,k_3}(z)=\frac{z^4+k_1z^2+1}{k_2z^3+k_3z}.
    \end{equation*}
\end{proof}

\section{Geometry of certain automorphism loci in $\M_4$} \label{sect_geom_A4}
In this section, we examine the geometry of automorphism loci $\mathcal{A}_4(\Gamma)\subset \M_3$. Similar to the methods in Section \ref{sect_geom_A3}, we study the embedding via multiplier invariants
    \begin{align*}
        \tau_n:\M_4 &\to \bbA^k\\
         [f] &\mapsto (\sigma^{(n)}_1,\ldots, \sigma^{(n)}_{4^n+1})
    \end{align*}
    and the geometry of the image of this embedding.
\subsection{$\A_4(C_4)$} We start with the family $f_k(z)=\frac{z^4+1}{kz^3}$ for $k\neq 0$.
\begin{lemma} \label{geometry of C_4}
    The maps
    \begin{align*}
            \varphi:\bbA^1\setminus\{0\} &\to \M_4\\
            k &\mapsto [f_{k}]
        \end{align*}
        and
        \begin{align*}
            \tau_1: \varphi(\bbA^1\setminus\{0\}) &\to \bbA^4\\
            [f_{k}] &\mapsto (\sigma_1^{(1)}, \sigma_2^{(1)}, \sigma_3^{(1)}, \sigma_4^{(1)})
        \end{align*}
        are injective.
\end{lemma}
\begin{proof}
We show the composition $\tau_1 \circ \varphi$ is injective, and thus, both maps are injective.

Consider the ideal generated by
\begin{equation*}
    (\sigma_1(k_1),\sigma_2(k_1),\sigma_3(k_1),\sigma_4(k_1))=(\sigma_1(k_2),\sigma_2(k_2),\sigma_3(k_2),\sigma_4(k_2).
\end{equation*}
Computing its lexicographic Groebner basis, we get two generators:
\begin{align*}
    &k_1^2 + 8k_1 - k_2^2 - 8k_2,\\
    &k_1k_2^2 + 8k_1k_2 + 16k_1 - k_2^3 - 8k_2^2 - 16k_2.
\end{align*}
These factor as
\begin{align*}
    &(k_1 - k_2)(k_1 + k_2 + 8),\\
    &(k_1 - k_2)(k_2 + 4)^2.
\end{align*}
The first says that for $f_{k_1}$ and $f_{k_2}$ to have the same fixed point multiplier invariants, we must have $k_1=k_2$ or $k_2=-4$. However, if $k_2=-4$, then the second generator tells us that $k_1 =-4$ and we are still in the case $k_1=k_2$.
Thus, the map $\varphi$ and the map $\tau_1$ are one-to-one.
\begin{code}
\begin{lstlisting}
R.<k1,k2>=PolynomialRing(QQ, order='lex')
P.<x,y>=ProjectiveSpace(R,1)
f1=DynamicalSystem([x^4 + y^4, k1*x^3*y])
f2=DynamicalSystem([x^4 + y^4, k2*x^3*y])
s1=f1.sigma_invariants(1)
s2=f2.sigma_invariants(1)
I=R.ideal([s1[i].numerator()*s2[i].denominator()-s2[i].numerator()*s1[i].denominator() for i in range(4)])
I.groebner_basis()
\end{lstlisting}
\end{code}
\end{proof}

\begin{prop} \label{prop_A4_C4_geo}
    The curve in $\bbA^4$ given by the image of $\tau_1$ on the family $f_k(z)=\frac{z^4+1}{kz^3}$ of $\A_4(C_4)$ is given by the system of equations
    \begin{align*}
        0&=36\sigma_1^3 + 3\sigma_1^2\sigma_2 + 222\sigma_1^2 - 96\sigma_1\sigma_2 - 8\sigma_2^2 + 240\sigma_1 - 560\sigma_2 - 800\\
        0&=-12\sigma_1^2 - \sigma_1\sigma_2 - 14\sigma_1 + 32\sigma_2 + 6\sigma_3 + 40\\
        0&=144\sigma_1^2 - \sigma_2^2 + 288\sigma_1 - 448\sigma_2 + 36\sigma_4 - 640.
    \end{align*}
    It is a singular irreducible curve of genus $0$ and the singularity corresponds to the rational map $$f=\frac{z^4+1}{-4z^3}.$$
\end{prop}
\begin{proof}
    The fixed point multiplier invariants give the equations
    \begin{align*}
        0&=(k^2 - 12k + 16)-k\sigma_1\\
        0&=(-12k^3 + 70k^2 - 144k + 96)-k^2\sigma_2\\
        0&=(54k^4 - 252k^3 + 528k^2 - 576k + 256)-k^3\sigma_3\\
        0&=(-108k^5 + 513k^4 - 1008k^3 + 1120k^2 - 768k + 256)-k^4\sigma_4.
    \end{align*}
    Saturating by the ideal $(k)$ and looking at the generators gives the stated equations. Using Sage, we can determine that these relations define a singular, irreducible curve of genus $0$ in $\bbA^3$. The only point of singularity is $(-20,160,-640,1280)$ and it corresponds to the rational map in the statement.
\begin{code}
\begin{lstlisting}
R.<k,s1,s2,s3,s4>=PolynomialRing(QQ)
S=[s1,s2,s3,s4]
P.<x,y>=ProjectiveSpace(R,1)
f=DynamicalSystem([x^4 + y^4, k*x^3*y])
s1=f.sigma_invariants(1)
I=R.ideal([s1[i].numerator()-s1[i].denominator()*S[i] for i in range(4)])
J=R.ideal(k)
I=I.saturation(J)[0]
I.elimination_ideal(k)

R.<s1,s2,s3,s4>=QQ[]
f1=s2*s3 - 6*s1*s4 + 150*s1 - 100*s2 + 50*s3 - 200
f2=30*s1*s3 + 3*s3^2 - 8*s2*s4 - 600*s1 + 400*s2 - 240*s3 - 160*s4 + 800
f3=45*s2^2 + 12*s3^2 - 32*s2*s4 - 2400*s1 + 1600*s2 - 960*s3 - 340*s4 + 3200
f4=135*s1*s2 - 3*s3^2 + 8*s2*s4 - 750*s1 + 320*s2 - 570*s3 - 320*s4 + 1000
f5=1620*s1^2 + 3*s3^2 - 8*s2*s4 + 2640*s1 - 4640*s2 - 240*s3 + 320*s4 - 6400
f6=6*s3^3 + 3*s3^2*s4 - 96*s1*s4^2 - 8*s2*s4^2 - 1035*s3^2 + 1200*s1*s4 - 440*s2*s4 + 1200*s3*s4 + 320*s4^2 + 126000*s1 - 84000*s2 + 42000*s3 - 1600*s4 - 168000
I=R.ideal([f1,f2,f3,f4,f5,f6])
gfan=I.groebner_fan()
for G in gfan:
    if len(G) <= 3:
        print(G)
G1=[s1^3 + 1/12*s1^2*s2 + 37/6*s1^2 - 8/3*s1*s2 - 2/9*s2^2 + 20/3*s1 - 140/9*s2 - 200/9, -2*s1^2 - 1/6*s1*s2 - 7/3*s1 + 16/3*s2 + s3 + 20/3, 4*s1^2 - 1/36*s2^2 + 8*s1 - 112/9*s2 + s4 - 160/9]
G1 = [G1[0]*36, G1[1]*6, G1[2]*36]
A=AffineSpace(R)
X=A.curve(G1)
X.genus()
X.is_irreducible()
X.singular_points()


#map associated to singularity
R.<k>=PolynomialRing(QQ)
S=[-20,160,-640,1280]
P.<x,y>=ProjectiveSpace(R,1)
f=DynamicalSystem([x^4 + y^4, k*x^3*y])
s1=f.sigma_invariants(1)
I=R.ideal([s1[i].numerator()-s1[i].denominator()*S[i] for i in range(4)])
I
\end{lstlisting}
\end{code}
\end{proof}
\subsection{$\A_4(D_3)$} We now move on to the family $f_k(z)=\frac{z^4+kz}{kz^3+1}$ with $k\neq \pm 1$.
\begin{lemma} \label{lem_A4_D3_geo}
    The maps
    \begin{align*}
        \varphi:\bbA^1 &\to \M_4\\
        k &\mapsto [f_{k}]
    \end{align*}
    and
    \begin{align*}
        \tau_1: \varphi(\bbA^1) &\to \bbA^4\\
        [f_{k}] &\mapsto (\sigma_1^{(1)}, \sigma_2^{(1)}, \sigma_3^{(1)}, \sigma_4^{(1)})
    \end{align*}
    are injective.
\end{lemma}
\begin{proof}
We show the composition $\tau_1 \circ \varphi$ is injective, and thus, both maps are injective.

We want to compute the ideal generated by
\begin{equation*}
    (\sigma_1(k_1),\sigma_2(k_2),\sigma_3(k_1),\sigma_4(k_1))=(\sigma_1(k_2),\sigma_2(k_2),\sigma_3(k_2),\sigma_4(k_2)).
\end{equation*}
This gives the ideal generated by
\begin{align*}
    &k_1^2 + 8k_1 - k_2^2 - 8k_2,\\
    &k_1k_2^2 + 8k_1k_2 + 16k_1 - k_2^3 - 8k_2^2 - 16k_2.
\end{align*}
This ideal is the same as the ideal considered in the proof of Lemma \ref{geometry of C_4}, so the result follows at once.
\begin{code}
\begin{lstlisting}
R.<k1,k2>=PolynomialRing(QQ, order='lex')
P.<x,y>=ProjectiveSpace(R,1)
f1=DynamicalSystem([x^4 + k1*x*y^3, k1*x^3*y+y^4])
f2=DynamicalSystem([x^4 + k2*x*y^3, k2*x^3*y+y^4])
s1=f1.sigma_invariants(1)
s2=f2.sigma_invariants(1)
I=R.ideal([s1[i].numerator()*s2[i].denominator()-s2[i].numerator()*s1[i].denominator() for i in range(4)])
I.groebner_basis()
\end{lstlisting}
\end{code}
\end{proof}

\begin{prop} \label{prop_A4_D3_geo}
    The curve given by the image of $\tau_1$ on the family $f_k(z)=\frac{z^4+kz}{kz^3+1}$ of $\A_4(D_3)$ is defined by
    \begin{align*}
        0&=\sigma_1^4 - 14\sigma_1^3 - 7\sigma_1^2\sigma_2 + 67\sigma_1^2 + 44\sigma_1\sigma_2 + 12\sigma_2^2 - 360\sigma_1 - 160\sigma_2 + 1200\\
        0&=\sigma_1^3 + 2\sigma_1^2 - 4\sigma_1\sigma_2 - 21\sigma_1 - 2\sigma_2 + 9\sigma_3 + 60\\
        0&=-2\sigma_1^3 + 5\sigma_1^2 + 8\sigma_1\sigma_2 - 48\sigma_1 - 32\sigma_2 + 9\sigma_4 + 240.
    \end{align*}
    It is a singular irreducible curve of genus $0$, and the singularity $(-20,160,-640,1280)$ corresponds to the rational map $$f=\frac{z^4-4z}{-4z^3+1}.$$
\end{prop}
\begin{proof}
The fixed point multiplier invariants give the equations
\begin{align*}
        0&=(2k^2 - 4k + 12)-(k + 1)\sigma_1\\
        0&=(k^4 - 10k^3 + 25k^2 - 24k + 48)-(k^2 + 2k + 1)\sigma_2\\
        0&=(-6k^5 + 24k^4 - 62k^3 + 60k^2 + 64)-(k^3 + 3k^2 + 3k + 1)\sigma_3\\
        0&=(12k^5 - 52k^4 + 96k^3 - 144k^2 + 128k)-(k^3 + 3k^2 + 3k + 1)\sigma_4.
\end{align*}
Eliminating $k$ gives the stated equations involving only the fixed point multiplier invariants. Choosing an appropriate monomial ordering, we arrive at the stated relations. These relations define a singular, irreducible curve of genus $0$ in $\bbA^3$. The only point of singularity is $(-20,160,-640,1280)$. This choice of fixed point multiplier invariants corresponds to the rational map in the statement.

\begin{code}
\begin{lstlisting}
R.<k,s1,s2,s3,s4>=PolynomialRing(QQ)
S=[s1,s2,s3,s4]
P.<x,y>=ProjectiveSpace(R,1)
f=DynamicalSystem([x^4 + k*x*y^3, k*x^3*y+y^4])
s1=f.sigma_invariants(1)
I=R.ideal([s1[i].numerator()-s1[i].denominator()*S[i] for i in range(4)])
J=R.ideal(k-1)
I=I.saturation(J)[0]
J=I.elimination_ideal(k)

R.<s1,s2,s3,s4>=QQ[]
f1,f2,f3,f4,f5,f6=J.gens()
I=R.ideal([f1,f2,f3,f4,f5,f6])
gfan=I.groebner_fan()
for G in gfan:
    if len(G) <= 3:
        print(G)

G1=[s1^4 - 14*s1^3 - 7*s1^2*s2 + 67*s1^2 + 44*s1*s2 + 12*s2^2 - 360*s1 - 160*s2 + 1200, 1/9*s1^3 + 2/9*s1^2 - 4/9*s1*s2 - 7/3*s1 - 2/9*s2 + s3 + 20/3, -2/9*s1^3 + 5/9*s1^2 + 8/9*s1*s2 - 16/3*s1 - 32/9*s2 + s4 + 80/3]
G1 = [G1[0], G1[1]*9, G1[2]*9]
A=AffineSpace(R)
X=A.curve(G1)
print(X.genus())
print(X.is_irreducible())
print(X.singular_points())

#map associated to singularity
R.<k>=PolynomialRing(QQ)
S=[-20,160,-640,1280]
P.<x,y>=ProjectiveSpace(R,1)
f=DynamicalSystem([x^4 + y^4, k*x^3*y])
s1=f.sigma_invariants(1)
I=R.ideal([s1[i].numerator()-s1[i].denominator()*S[i] for i in range(4)])
I
\end{lstlisting}
\end{code}
\end{proof}

\subsection{$\A_4(C_3)$}
Now we move on to study the family $f_{k_1,k_2}(z)=\frac{z^4+k_1z}{k_2z^3+1}$, with $k_1k_2\neq 1$.
\begin{lemma} \label{lem_A4_C3}
    The map \begin{align*}
            \varphi:\bbA^2\setminus\{k_1k_2=1\} &\to \M_4\\
            (k_1,k_2) &\mapsto [f_{k_1,k_2}]
        \end{align*}
        is two-to-one. The map \begin{align*}
            \tau_1: \varphi(\bbA^2\setminus\{k_1k_2=1\}) &\to \bbA^3\\
            [f_{k}] &\mapsto (\sigma_1^{(1)}, \sigma_2^{(1)}, \sigma_3^{(1)})
        \end{align*} is injective.
\end{lemma}
\begin{proof}
To compute the degree of $\varphi$, we consider the ideal generated by
\begin{equation*}
    (\sigma_1(k_1,k_2),\sigma_2(k_1,k_2),\sigma_3(k_1,k_2))=(\sigma_1(t_1,t_2),\sigma_2(t_1,t_2),\sigma_3(t_1,t_2))
\end{equation*}
as an ideal in $K[k_1,k_2]$, where $K$ is the function field $\bar{\Q}(t_1,t_2)$. This forms a zero dimensional variety, and Singular (via Sage) computes the degree of its projective closure as $2$. For (almost) every choice of $(t_1,t_2)$ we have two pairs $(k_1,k_2)\in \{(t_1,t_2),(t_2,t_1)\}$ and one can easily verify that $f_{k_1,k_2}$ and $f_{k_2,k_1}$ are conjugate via $\alpha(z)=\frac{1}{z}$.

\begin{code}
\begin{lstlisting}
R.<k1,k2,t1,t2>=QQ[]
P.<x,y>=ProjectiveSpace(R,1)
f1=DynamicalSystem([x^4 + k1*x*y^3, k2*x^3*y+y^4])
f2=DynamicalSystem([x^4 + t1*x*y^3, t2*x^3*y+y^4])
s1=f1.sigma_invariants(1)
s2=f2.sigma_invariants(1)
A=AffineSpace(R)
I=R.ideal([s1[i].numerator()*s2[i].denominator()-s2[i].numerator()*s1[i].denominator() for i in range(3)])
J1=R.ideal([t1*t2-1])
J2=R.ideal([k1*k2-1])
I=I.saturation(J1)[0]
I=I.saturation(J2)[0]

S.<t1,t2>=QQ[]
A.<k1,k2>=AffineSpace(FractionField(S),2)
F=(2*k1^3 - 3*k1^2*k2 - 3*k1*k2^2 + 2*k2^3 - 6*k1^2*t1 - 3*k1*k2*t1 - 6*k2^2*t1 + 6*k1*t1^2 + 6*k2*t1^2 - 2*t1^3 - 6*k1^2*t2 - 3*k1*k2*t2 - 6*k2^2*t2 + 3*k1*t1*t2 + 3*k2*t1*t2 + 3*t1^2*t2 + 6*k1*t2^2 + 6*k2*t2^2 + 3*t1*t2^2 - 2*t2^3 - 54*k1*k2 + 54*t1*t2 + 72*k1 + 72*k2 - 72*t1 - 72*t2, 2*k1^2*t1*t2 + k1*k2*t1*t2 + 2*k2^2*t1*t2 - k1*t1^2*t2 - k2*t1^2*t2 - t1^3*t2 - k1*t1*t2^2 - k2*t1*t2^2 + t1^2*t2^2 - t1*t2^3 + 6*k1*t1*t2 + 6*k2*t1*t2 - 6*t1^2*t2 - 6*t1*t2^2 - 2*k1^2 - k1*k2 - 2*k2^2 - 8*k1*t1 - 8*k2*t1 + 10*t1^2 - 8*k1*t2 - 8*k2*t2 + 17*t1*t2 + 10*t2^2 + 12*k1 + 12*k2 - 12*t1 - 12*t2, 5*k1*k2^3 - 2*k2^4 + 6*k1^2*k2*t1 + 3*k1*k2^2*t1 + 6*k2^3*t1 - 8*k1*k2*t1^2 - 6*k2^2*t1^2 + 2*k2*t1^3 + 6*k1^2*k2*t2 + 3*k1*k2^2*t2 + 6*k2^3*t2 - 4*k1*k2*t1*t2 - 3*k2^2*t1*t2 + k1*t1^2*t2 - 2*k2*t1^2*t2 - t1^3*t2 - 8*k1*k2*t2^2 - 6*k2^2*t2^2 + k1*t1*t2^2 - 2*k2*t1*t2^2 + 2*t1^2*t2^2 + 2*k2*t2^3 - t1*t2^3 + 2*k1^2*k2 + 56*k1*k2^2 - 20*k1*k2*t1 - 20*k1*k2*t2 + 14*k1*t1*t2 - 40*k2*t1*t2 + 4*t1^2*t2 + 4*t1*t2^2 - 14*k1^2 - 107*k1*k2 - 86*k2^2 + 8*k1*t1 + 80*k2*t1 + 6*t1^2 + 8*k1*t2 + 80*k2*t2 + 19*t1*t2 + 6*t2^2 + 68*k1 + 68*k2 - 68*t1 - 68*t2, k1^2*k2^2 + 2*k1^2*k2*t1 + 2*k1*k2^2*t1 - 2*k1*k2*t1^2 + 2*k1^2*k2*t2 + 2*k1*k2^2*t2 - k1*k2*t1*t2 - k1*t1^2*t2 - k2*t1^2*t2 + t1^3*t2 - 2*k1*k2*t2^2 - k1*t1*t2^2 - k2*t1*t2^2 - 2*t1^2*t2^2 + t1*t2^3 + 10*k1^2*k2 + 10*k1*k2^2 + 8*k1*k2*t1 + 8*k1*k2*t2 - 14*k1*t1*t2 - 14*k2*t1*t2 - 4*t1^2*t2 - 4*t1*t2^2 - 6*k1^2 + k1*k2 - 6*k2^2 + 8*k1*t1 + 8*k2*t1 - 2*t1^2 + 8*k1*t2 + 8*k2*t2 - 17*t1*t2 - 2*t2^2 - 44*k1 - 44*k2 + 44*t1 + 44*t2, 3*k1*k2*t1^2*t2 - 3*k1*t1^3*t2 - 3*k2*t1^3*t2 + 3*t1^4*t2 + 3*k1*k2*t1*t2^2 + 2*k1*t1^2*t2^2 + 2*k2*t1^2*t2^2 - 2*t1^3*t2^2 - 3*k1*t1*t2^3 - 3*k2*t1*t2^3 - 2*t1^2*t2^3 + 3*t1*t2^4 + 42*k1*k2*t1*t2 - 12*k1*t1^2*t2 - 12*k2*t1^2*t2 + 12*t1^3*t2 - 12*k1*t1*t2^2 - 12*k2*t1*t2^2 - 18*t1^2*t2^2 + 12*t1*t2^3 - 6*k1^2*t1 + 21*k1*k2*t1 - 6*k2^2*t1 + 24*k1*t1^2 + 24*k2*t1^2 - 18*t1^3 - 6*k1^2*t2 + 21*k1*k2*t2 - 6*k2^2*t2 - 52*k1*t1*t2 - 52*k2*t1*t2 + 13*t1^2*t2 + 24*k1*t2^2 + 24*k2*t2^2 + 13*t1*t2^2 - 18*t2^3 + 12*k1^2 - 90*k1*k2 + 12*k2^2 - 12*k1*t1 - 12*k2*t1 - 12*k1*t2 - 12*k2*t2 + 114*t1*t2 + 56*k1 + 56*k2 - 56*t1 - 56*t2, 3*k1*k2^2*t1*t2 - 6*k2^3*t1*t2 + 3*k2^2*t1^2*t2 - 3*k1*t1^3*t2 + 3*t1^4*t2 + 3*k2^2*t1*t2^2 + 2*k1*t1^2*t2^2 - k2*t1^2*t2^2 - 2*t1^3*t2^2 - 3*k1*t1*t2^3 - 2*t1^2*t2^3 + 3*t1*t2^4 + 24*k1*k2*t1*t2 - 18*k2^2*t1*t2 - 12*k1*t1^2*t2 + 6*k2*t1^2*t2 + 12*t1^3*t2 - 12*k1*t1*t2^2 + 6*k2*t1*t2^2 - 18*t1^2*t2^2 + 12*t1*t2^3 - 3*k1*k2^2 + 6*k2^3 - 6*k1^2*t1 - 3*k1*k2*t1 + 18*k2^2*t1 + 24*k1*t1^2 - 6*k2*t1^2 - 18*t1^3 - 6*k1^2*t2 - 3*k1*k2*t2 + 18*k2^2*t2 - 4*k1*t1*t2 - 55*k2*t1*t2 + 19*t1^2*t2 + 24*k1*t2^2 - 6*k2*t2^2 + 19*t1*t2^2 - 18*t2^3 + 12*k1^2 - 18*k1*k2 - 24*k2^2 - 12*k1*t1 + 24*k2*t1 - 12*k1*t2 + 24*k2*t2 + 6*t1*t2 + 8*k1 + 8*k2 - 8*t1 - 8*t2, 30*k2^5 + 42*k2^4*t1 + 234*k1^2*k2*t1^2 + 102*k1*k2^2*t1^2 - 306*k2^3*t1^2 - 102*k1*k2*t1^3 + 366*k2^2*t1^3 - 132*k2*t1^4 + 42*k2^4*t2 - 285*k2^3*t1*t2 + 243*k2^2*t1^2*t2 + 234*k1^2*k2*t2^2 + 102*k1*k2^2*t2^2 - 306*k2^3*t2^2 + 243*k2^2*t1*t2^2 - 416*k1*t1^2*t2^2 - 119*k2*t1^2*t2^2 - 52*t1^3*t2^2 - 102*k1*k2*t2^3 + 366*k2^2*t2^3 - 52*t1^2*t2^3 - 132*k2*t2^4 + 444*k2^4 + 1896*k1^2*k2*t1 + 48*k1*k2^2*t1 - 1332*k2^3*t1 + 1356*k1*k2*t1^2 + 1332*k2^2*t1^2 - 444*k2*t1^3 + 1896*k1^2*k2*t2 + 48*k1*k2^2*t2 - 1332*k2^3*t2 + 180*k1*k2*t1*t2 + 552*k2^2*t1*t2 - 2112*k1*t1^2*t2 + 72*k2*t1^2*t2 - 696*t1^3*t2 + 1356*k1*k2*t2^2 + 1332*k2^2*t2^2 - 2112*k1*t1*t2^2 + 72*k2*t1*t2^2 - 1416*t1^2*t2^2 - 444*k2*t2^3 - 696*t1*t2^3 + 7209*k1^2*k2 - 2628*k1*k2^2 - 822*k2^3 - 204*k1^2*t1 + 10362*k1*k2*t1 + 3300*k2^2*t1 - 1056*k1*t1^2 - 3738*k2*t1^2 + 1260*t1^3 - 204*k1^2*t2 + 10362*k1*k2*t2 + 3300*k2^2*t2 - 6680*k1*t1*t2 - 3029*k2*t1*t2 - 6949*t1^2*t2 - 1056*k1*t2^2 - 3738*k2*t2^2 - 6949*t1*t2^2 + 1260*t2^3 - 3096*k1^2 + 22026*k1*k2 + 8940*k2^2 - 1560*k1*t1 - 13596*k2*t1 + 4656*t1^2 - 1560*k1*t2 - 13596*k2*t2 - 6870*t1*t2 + 4656*t2^2 - 42488*k1 - 42488*k2 + 42488*t1 + 42488*t2, 9*k1*t1^3*t2^2 + 9*k2*t1^3*t2^2 - 9*t1^4*t2^2 - 12*k1*k2*t1*t2^3 - 11*k1*t1^2*t2^3 - 11*k2*t1^2*t2^3 + 2*t1^3*t2^3 + 12*k1*t1*t2^4 + 12*k2*t1*t2^4 + 11*t1^2*t2^4 - 12*t1*t2^5 + 54*k1*t1^3*t2 + 54*k2*t1^3*t2 - 54*t1^4*t2 - 168*k1*k2*t1*t2^2 - 30*k1*t1^2*t2^2 - 30*k2*t1^2*t2^2 - 24*t1^3*t2^2 + 102*k1*t1*t2^3 + 102*k2*t1*t2^3 + 96*t1^2*t2^3 - 102*t1*t2^4 + 9*k1^2*k2*t1 + 9*k1*k2^2*t1 + 72*k1*k2*t1^2 + 9*k1^2*k2*t2 + 9*k1*k2^2*t2 - 732*k1*k2*t1*t2 + 84*k1*t1^2*t2 + 84*k2*t1^2*t2 - 165*t1^3*t2 + 24*k1^2*t2^2 - 12*k1*k2*t2^2 + 24*k2^2*t2^2 + 388*k1*t1*t2^2 + 388*k2*t1*t2^2 + 218*t1^2*t2^2 - 96*k1*t2^3 - 96*k2*t2^3 - 289*t1*t2^3 + 72*t2^4 - 18*k1^2*k2 - 18*k1*k2^2 + 108*k1^2*t1 - 252*k1*k2*t1 + 108*k2^2*t1 - 432*k1*t1^2 - 432*k2*t1^2 + 324*t1^3 + 60*k1^2*t2 + 108*k1*k2*t2 + 60*k2^2*t2 + 672*k1*t1*t2 + 672*k2*t1*t2 - 30*t1^2*t2 - 384*k1*t2^2 - 384*k2*t2^2 - 486*t1*t2^2 + 324*t2^3 - 192*k1^2 + 996*k1*k2 - 192*k2^2 + 192*k1*t1 + 192*k2*t1 - 32*k1*t2 - 32*k2*t2 - 1156*t1*t2 + 224*t2^2 - 528*k1 - 528*k2 + 528*t1 + 528*t2, 27*k2^3*t1^2*t2 - 27*k2^2*t1^3*t2 + 27*k2^3*t1*t2^2 - 18*k2^2*t1^2*t2^2 - 9*k2*t1^3*t2^2 + 12*k1*k2*t1*t2^3 - 27*k2^2*t1*t2^3 + 20*k1*t1^2*t2^3 + 11*k2*t1^2*t2^3 + 16*t1^3*t2^3 - 12*k1*t1*t2^4 - 12*k2*t1*t2^4 - 20*t1^2*t2^4 + 12*t1*t2^5 + 378*k2^3*t1*t2 - 162*k2^2*t1^2*t2 - 216*k2*t1^3*t2 + 168*k1*k2*t1*t2^2 - 162*k2^2*t1*t2^2 + 120*k1*t1^2*t2^2 + 66*k2*t1^2*t2^2 + 96*t1^3*t2^2 - 48*k1*t1*t2^3 - 264*k2*t1*t2^3 - 24*t1^2*t2^3 + 48*t1*t2^4 - 63*k1^2*k2*t1 + 72*k1*k2^2*t1 - 54*k2^3*t1 - 72*k1*k2*t1^2 + 54*k2*t1^3 - 63*k1^2*k2*t2 + 72*k1*k2^2*t2 - 54*k2^3*t2 + 732*k1*k2*t1*t2 + 684*k2^2*t1*t2 + 60*k1*t1^2*t2 - 651*k2*t1^2*t2 + 21*t1^3*t2 - 24*k1^2*t2^2 + 12*k1*k2*t2^2 - 24*k2^2*t2^2 - 244*k1*t1*t2^2 - 955*k2*t1*t2^2 + 178*t1^2*t2^2 + 96*k1*t2^3 + 150*k2*t2^3 + 145*t1*t2^3 - 72*t2^4 + 126*k1^2*k2 - 144*k1*k2^2 - 324*k2^3 + 630*k1*k2*t1 - 1404*k2^2*t1 + 1728*k2*t1^2 + 48*k1^2*t2 + 270*k1*k2*t2 - 1356*k2^2*t2 - 2040*k1*t1*t2 + 1362*k2*t1*t2 - 492*t1^2*t2 - 48*k1*t2^2 + 1680*k2*t2^2 - 36*t1*t2^2 - 24*k1^2 - 1752*k1*k2 + 2496*k2^2 + 600*k1*t1 - 1920*k2*t1 - 576*t1^2 + 824*k1*t2 - 1696*k2*t2 + 2848*t1*t2 - 800*t2^2 + 672*k1 + 672*k2 - 672*t1 - 672*t2, 54*k2^4*t1*t2 - 54*k2^2*t1^3*t2 + 9*k2^2*t1^2*t2^2 - 9*k2*t1^3*t2^2 + 21*k1*k2*t1*t2^3 - 54*k2^2*t1*t2^3 + 35*k1*t1^2*t2^3 + 26*k2*t1^2*t2^3 + 28*t1^3*t2^3 - 21*k1*t1*t2^4 - 21*k2*t1*t2^4 - 35*t1^2*t2^4 + 21*t1*t2^5 + 864*k2^3*t1*t2 - 486*k2^2*t1^2*t2 - 378*k2*t1^3*t2 + 294*k1*k2*t1*t2^2 - 486*k2^2*t1*t2^2 + 210*k1*t1^2*t2^2 + 318*k2*t1^2*t2^2 + 168*t1^3*t2^2 - 84*k1*t1*t2^3 - 462*k2*t1*t2^3 - 42*t1^2*t2^3 + 84*t1*t2^4 - 54*k2^4 - 9*k1^2*k2*t1 + 126*k1*k2^2*t1 - 270*k2^3*t1 - 126*k1*k2*t1^2 + 270*k2^2*t1^2 + 54*k2*t1^3 - 9*k1^2*k2*t2 + 126*k1*k2^2*t2 - 270*k2^3*t2 + 1281*k1*k2*t1*t2 + 1899*k2^2*t1*t2 + 105*k1*t1^2*t2 - 1605*k2*t1^2*t2 - 24*t1^3*t2 - 42*k1^2*t2^2 + 21*k1*k2*t2^2 + 228*k2^2*t2^2 - 427*k1*t1*t2^2 - 2137*k2*t1*t2^2 + 406*t1^2*t2^2 + 168*k1*t2^3 + 222*k2*t2^3 + 193*t1*t2^3 - 126*t2^4 + 18*k1^2*k2 - 252*k1*k2^2 - 324*k2^3 + 1710*k1*k2*t1 - 3024*k2^2*t1 + 3348*k2*t1^2 + 84*k1^2*t2 + 1080*k1*k2*t2 - 2940*k2^2*t2 - 3570*k1*t1*t2 + 3072*k2*t1*t2 - 1590*t1^2*t2 - 84*k1*t2^2 + 3264*k2*t2^2 - 792*t1*t2^2 - 42*k1^2 - 4281*k1*k2 + 4638*k2^2 + 240*k1*t1 - 4440*k2*t1 - 198*t1^2 + 632*k1*t2 - 4048*k2*t2 + 8089*t1*t2 - 590*t2^2 + 2796*k1 + 2796*k2 - 2796*t1 - 2796*t2, 45*k2^2*t1^3*t2^2 - 45*k2*t1^4*t2^2 - 120*k2^3*t1*t2^3 + 5*k2^2*t1^2*t2^3 + 70*k2*t1^3*t2^3 + 45*t1^4*t2^3 + 120*k2^2*t1*t2^4 - 40*k1*t1^2*t2^4 - 45*k2*t1^2*t2^4 - 75*t1^3*t2^4 + 40*t1^2*t2^5 + 270*k2^2*t1^3*t2 - 270*k2*t1^4*t2 - 1680*k2^3*t1*t2^2 + 690*k2^2*t1^2*t2^2 + 720*k2*t1^3*t2^2 + 270*t1^4*t2^2 - 240*k1*k2*t1*t2^3 + 990*k2^2*t1*t2^3 - 560*k1*t1^2*t2^3 - 560*k2*t1^2*t2^3 - 430*t1^3*t2^3 + 690*k2*t1*t2^4 + 110*t1^2*t2^4 + 18*k2^4*t1 - 189*k1^2*k2*t1^2 + 198*k1*k2^2*t1^2 - 54*k2^3*t1^2 - 198*k1*k2*t1^3 + 54*k2^2*t1^3 - 18*k2*t1^4 + 18*k2^4*t2 - 7320*k2^3*t1*t2 + 4107*k2^2*t1^2*t2 + 2790*k2*t1^3*t2 + 405*t1^4*t2 + 51*k1^2*k2*t2^2 - 162*k1*k2^2*t2^2 + 186*k2^3*t2^2 - 3360*k1*k2*t1*t2^2 + 1067*k2^2*t1*t2^2 - 2224*k1*t1^2*t2^2 - 246*k2*t1^2*t2^2 - 818*t1^3*t2^2 - 198*k1*k2*t2^3 + 54*k2^2*t2^3 + 320*k1*t1*t2^3 + 6510*k2*t1*t2^3 - 818*t1^2*t2^3 - 258*k2*t2^4 + 85*t1*t2^4 - 36*k2^4 + 900*k1^2*k2*t1 - 1440*k1*k2^2*t1 + 648*k2^3*t1 + 306*k1*k2*t1^2 - 2268*k2^2*t1^2 + 1656*k2*t1^3 + 420*k1^2*k2*t2 - 720*k1*k2^2*t2 + 2088*k2^3*t2 - 14244*k1*k2*t1*t2 - 11610*k2^2*t1*t2 - 2232*k1*t1^2*t2 + 11286*k2*t1^2*t2 - 630*t1^3*t2 - 3054*k1*k2*t2^2 + 3972*k2^2*t2^2 + 8008*k1*t1*t2^2 + 14086*k2*t1*t2^2 - 1762*t1^2*t2^2 - 6024*k2*t2^3 + 650*t1*t2^3 - 1182*k1^2*k2 + 2124*k1*k2^2 + 6252*k2^3 - 396*k1^2*t1 - 16542*k1*k2*t1 + 30060*k2^2*t1 + 3096*k1*t1^2 - 33612*k2*t1^2 - 2700*t1^3 - 396*k1^2*t2 - 5982*k1*k2*t2 + 18860*k2^2*t2 + 38744*k1*t1*t2 - 20526*k2*t1*t2 + 12988*t1^2*t2 + 536*k1*t2^2 - 24972*k2*t2^2 - 6212*t1*t2^2 - 140*t2^3 + 792*k1^2 + 43512*k1*k2 - 46416*k2^2 - 7992*k1*t1 + 39216*k2*t1 + 7200*t1^2 - 13112*k1*t2 + 34096*k2*t2 - 69616*t1*t2 + 12320*t2^2 - 24544*k1 - 24544*k2 + 24544*t1 + 24544*t2)

X=A.subscheme(F)
X.dimension()
X.projective_closure().degree()
\end{lstlisting}
\end{code}
\end{proof}

It is interesting to note that $\sigma_4^{(1)}$ is determined uniquely by $(\sigma_1^{(1)},\sigma_2^{(1)},\sigma_3^{(1)})$. This is not typical for elements of $\M_4$. We do have the standard linear relationship between $(\sigma_1^{(1)},\sigma_2^{(1)},\sigma_3^{(1)})$ and $\sigma_5^{(1)}$ for all maps in $\M_4$, but the dependence of $\sigma_4^{(1)}$ is special to this family.

\begin{prop} \label{prop_A4_C3_geo}
    The image of the family $f_{k_1,k_2}(z)=\frac{z^4+k_1z}{k_2z^3+1}$ of $\A_4(C_3)$ under $\tau_1$ is the surface given by
    \begin{align*}
        0&=972\sigma_1^5 - 324\sigma_1^4\sigma_2 - 135\sigma_1^3\sigma_2^2 - 54\sigma_1^2\sigma_2^3 + 324\sigma_1^4\sigma_3 + 162\sigma_1^3\sigma_2\sigma_3 - 9\sigma_1^2\sigma_2^2\sigma_3 + 27\sigma_1^3\sigma_3^2 + 648\sigma_1^4\\
        &- 5724\sigma_1^3\sigma_2 + 1404\sigma_1^2\sigma_2^2 + 480\sigma_1\sigma_2^3 + 192\sigma_2^4 + 540\sigma_1^3\sigma_3 - 1224\sigma_1^2\sigma_2\sigma_3 - 648\sigma_1\sigma_2^2\sigma_3 + 32\sigma_2^3\sigma_3\\
        &- 108\sigma_1\sigma_2\sigma_3^2 + 1188\sigma_1^3 - 3672\sigma_1^2\sigma_2 + 7776\sigma_1\sigma_2^2 - 896\sigma_2^3 + 5364\sigma_1^2\sigma_3 - 3024\sigma_1\sigma_2\sigma_3 - 384\sigma_2^2\sigma_3\\
        &+ 1404\sigma_1\sigma_3^2 + 648\sigma_2\sigma_3^2 + 108\sigma_3^3 + 2160\sigma_1^2 - 5760\sigma_1\sigma_2 + 3840\sigma_2^2 + 4320\sigma_1\sigma_3 - 10560\sigma_2\sigma_3\\
        &+ 2160\sigma_3^2 - 33600\sigma_1 + 28800\sigma_2 - 17600\sigma_3 + 32000.
    \end{align*}
This surface is reduced, irreducible, and singular with singular locus given by the conjugacy classes described by
\begin{align}
 (k_1,k_2) \in \{(0,4/3)\} &\cup \{(-4,-4)\} \cup \{(9/4,-4)\} \notag\\
 &\bigcup
 \left\{\left(\frac{3888t + 46656}{9t^3 + 216t^2 + 1944t}, \frac{t^3 + 18t^2}{9t^2 + 216t + 1944}\right) : t\neq 0\in \Q\right\} \label{curve_1}\\
 &\bigcup
 \left\{\left(\frac{9t - 36}{-3t + 24}, -\frac{8-t}{3}\right) : t \in \Q \setminus \{8\}\right\} \label{curve_2}\\
 &\bigcup
 \left\{\left(\frac{-16}{2t-6}, \frac{-t-3}{2}\right) : t\in \Q \setminus\{3\}\right\} \label{curve_3}
\end{align}
\end{prop}
\begin{proof}
As usual, we first compute the fixed point multiplier invariants using only the first three:
\begin{align*}
\sigma_1&=\frac{1}{k_1k_2-1} (k_1^2k_2 + k_1k_2^2 - 6k_1k_2 + 8k_1 + 8k_2 - 12)\\
 \sigma_2&=\frac{1}{(k_1k_2-1)^2}(k_1^3k_2^3 - 6k_1^3k_2^2 - 6k_1^2k_2^3 + 9k_1^3k_2 + 28k_1^2k_2^2 + 9k_1k_2^3 - 42k_1^2k_2 - 42k_1k_2^2 + 18k_1^2\\
 &\hspace*{1in} + 85k_1k_2 + 18k_2^2 - 60k_1 - 60k_2 + 48)\\
 \sigma_3&=\frac{1}{(k_1k_2-1)^3}(-6k_1^4k_2^4 + 21k_1^4k_2^3 + 21k_1^3k_2^4 - 36k_1^4k_2^2 - 80k_1^3k_2^3 - 36k_1^2k_2^4 + 27k_1^4k_2 + 135k_1^3k_2^2\\
 &\hspace*{1in} + 135k_1^2k_2^3 + 27k_1k_2^4 - 90k_1^3k_2 - 210k_1^2k_2^2 - 90k_1k_2^3 + 153k_1^2k_2 + 153k_1k_2^2\\
 &\hspace*{1in} - 36k_1^2 - 180k_1k_2 - 36k_2^2 + 96k_1 + 96k_2 - 64).
\end{align*}
We look at the ideal in $\Q[\sigma_1, \sigma_2, \sigma_3]$ generated by the defining equations of the invariants. We saturate by the ideal $(k_1k_2-1)$ to avoid degeneracy and eliminate the variables $k_1$ and $k_2$. This results in the surface defined by the equation in the statement.

Using Magma, we see that the surface is reduced, irreducible, and singular. Since this is a hypersurface, we compute the singular locus as the points on the surface that also vanish on the partial derivatives of the defining equation. This variety has irreducible components
\begin{align*}
   S_1:&\begin{cases}
    0=\sigma_3=\sigma_2\\
    0=3\sigma_1 - 4
    \end{cases}\\
S_2:&\begin{cases}
    0=15\sigma_1 + 6\sigma_2 + \sigma_3 - 20\\
    0=36\sigma_2^2 + 12\sigma_2\sigma_3 + \sigma_3^2 - 320\sigma_ + 80\sigma_3
\end{cases}\\
S_3:&\begin{cases}
    0=15\sigma_1\sigma_2 + 8\sigma_2^2 - 18\sigma_1\sigma_3 - 150\sigma_1 + 80\sigma_2 - 90\sigma_3 + 200\\
    0=90\sigma_1^2 - 4\sigma_2^2 + 9\sigma_1\sigma_3 + 180\sigma_1 - 280\sigma_2 - 400\\
    0=90\sigma_2^3 + 128\sigma_2^2\sigma_3 - 288\sigma_1\sigma_3^2 + 1800\sigma_2^2 - 4950\sigma_1\sigma_3 + 1280\sigma_2\sigma_3 - 1215\sigma_3^2\\
    \phantom{0= }- 45000\sigma_1 + 30000\sigma_2 - 14800\sigma_3 + 60000.
 \end{cases}
\end{align*}

The general procedure of analyzing the components is similar to the one used in Section $3$: we start with the ideal defined by the equations of the component along with the defining equations of the fixed point multiplier invariants, eliminate $(\sigma_1,\sigma_2,\sigma_3)$, and finally saturate by $(k_1k_2-1)$. The first component $S_1$ corresponds to the case $k_1=0$ and $k_2=4/3$.

The component $S_2$ results in a rational curve given by
\begin{align*}
    0&=(k_1 + k_2 + 8)(9k_1^3k_2 - 14k_1^2k_2^2 + 9k_1k_2^3 + 12k_1^2k_2 + 12k_1k_2^2 + 4k_1k_2 - 48k_1 - 48k_2 + 64)\\
    0&=(k_2 + 4)^2(9k_1^3k_2 - 14k_1^2k_2^2 + 9k_1k_2^3 + 12k_1^2k_2 + 12k_1k_2^2 + 4k_1k_2 - 48k_1 - 48k_2 + 64).
\end{align*}
The common factor gives the rational curve \eqref{curve_1} parameterized in the statement. If $k_2=-4$, then either $k_1=-4$ or $k_1 = 4/9$ give the other two singular points.

The component $S_3$ results in a variety with the two irreducible components
\begin{align*}
    0=k_1k_2 + 3k_2 - 4 \quad \text{and} \quad k_1k_2 + 3k_1 - 4=0.
\end{align*}
These are conjugate by swapping $(k_1,k_2) \to (k_2,k_1)$ via $z \mapsto \frac{1}{z}$. This gives the two curves \eqref{curve_2} and \eqref{curve_3} and finishes the parameterization given in the statement.
\begin{code}
\begin{lstlisting}
R.<k1,k2,s1,s2,s3>=PolynomialRing(QQ)
S=[s1,s2,s3]
P.<x,y>=ProjectiveSpace(R,1)
f=DynamicalSystem([x^4 + k1*x*y^3, k2*x^3*y+y^4])
s1=f.sigma_invariants(1)
I=R.ideal([s1[i].numerator()-s1[i].denominator()*S[i] for i in range(3)])
J=R.ideal(k1*k2-1)
I=I.saturation(J)[0]
J=I.elimination_ideal([k1,k2])
print(len(J.gens()))

#singular subscheme
R.<s1,s2,s3>=PolynomialRing(QQ,3)
A=AffineSpace(R)
F=972*s1^5 - 324*s1^4*s2 - 135*s1^3*s2^2 - 54*s1^2*s2^3 + 324*s1^4*s3 + 162*s1^3*s2*s3 - 9*s1^2*s2^2*s3 + 27*s1^3*s3^2 + 648*s1^4 - 5724*s1^3*s2 + 1404*s1^2*s2^2 + 480*s1*s2^3 + 192*s2^4 + 540*s1^3*s3 - 1224*s1^2*s2*s3 - 648*s1*s2^2*s3 + 32*s2^3*s3 - 108*s1*s2*s3^2 + 1188*s1^3 - 3672*s1^2*s2 + 7776*s1*s2^2 - 896*s2^3 + 5364*s1^2*s3 - 3024*s1*s2*s3 - 384*s2^2*s3 + 1404*s1*s3^2 + 648*s2*s3^2 + 108*s3^3 + 2160*s1^2 - 5760*s1*s2 + 3840*s2^2 + 4320*s1*s3 - 10560*s2*s3 + 2160*s3^2 - 33600*s1 + 28800*s2 - 17600*s3 + 32000
X=A.subscheme(F)
f1 = F.derivative(s1)
f2 = F.derivative(s2)
f3 = F.derivative(s3)
Y=A.subscheme([F,f1,f2,f3])
Y.dimension()


#singular components
R.<k1,k2,s1,s2,s3>=QQ[]
P.<x,y>=ProjectiveSpace(R,1)
f=DynamicalSystem([x^4 + k1*x*y^3, k2*x^3*y+y^4])
sig=f.sigma_invariants(1)
S=[s1,s2,s3]

#S1
f1=s3
f2=s2
f3=3*s1 - 4
I=R.ideal([sig[i].numerator()-sig[i].denominator()*S[i] for i in range(3)] + [f1,f2,f3])
J=R.ideal(k1*k2-1)
I=I.saturation(J)[0]
J=I.elimination_ideal([s1,s2,s3])
A.<k1,k2>=AffineSpace(QQ,2)
X=A.subscheme([9*k1*k2 - 12*k1 - 12*k2 + 16, 9*k1^2 + 9*k2^2 - 16, 3*k2^3 - 4*k2^2])
X.rational_points()

#S2
R.<k1,k2,s1,s2,s3>=QQ[]
P.<x,y>=ProjectiveSpace(R,1)
f=DynamicalSystem([x^4 + k1*x*y^3, k2*x^3*y+y^4])
sig=f.sigma_invariants(1)
S=[s1,s2,s3]

f1=15*s1 + 6*s2 + s3 - 20
f2=36*s2^2 + 12*s2*s3 + s3^2 - 320*s2 + 80*s3
I=R.ideal([sig[i].numerator()-sig[i].denominator()*S[i] for i in range(3)] + [f1,f2])
J=R.ideal(k1*k2-1)
I=I.saturation(J)[0]
J=I.elimination_ideal([s1,s2,s3])

A.<k1,k2>=AffineSpace(QQ,2);
C=A.curve(gcd([9*k1^4*k2 - 5*k1^3*k2^2 - 5*k1^2*k2^3 + 9*k1*k2^4 + 84*k1^3*k2 - 88*k1^2*k2^2 + 84*k1*k2^3 + 100*k1^2*k2 + 100*k1*k2^2 - 48*k1^2 - 64*k1*k2 - 48*k2^2 - 320*k1 - 320*k2 + 512, 9*k1^3*k2^3 - 14*k1^2*k2^4 + 9*k1*k2^5 + 72*k1^3*k2^2 - 100*k1^2*k2^3 + 84*k1*k2^4 + 144*k1^3*k2 - 128*k1^2*k2^2 + 244*k1*k2^3 + 192*k1^2*k2 + 176*k1*k2^2 - 48*k2^3 - 320*k1*k2 - 320*k2^2 - 768*k1 - 256*k2 + 1024]))
C.rational_parameterization()

#S3
R.<k1,k2,s1,s2,s3>=QQ[]
P.<x,y>=ProjectiveSpace(R,1)
f=DynamicalSystem([x^4 + k1*x*y^3, k2*x^3*y+y^4])
sig=f.sigma_invariants(1)
S=[s1,s2,s3]

f1=15*s1*s2 + 8*s2^2 - 18*s1*s3 - 150*s1 + 80*s2 - 90*s3 + 200
f2=90*s1^2 - 4*s2^2 + 9*s1*s3 + 180*s1 - 280*s2 - 400
f3=90*s2^3 + 128*s2^2*s3 - 288*s1*s3^2 + 1800*s2^2 - 4950*s1*s3 + 1280*s2*s3 - 1215*s3^2 - 45000*s1 + 30000*s2 - 14800*s3 + 60000
I=R.ideal([sig[i].numerator()-sig[i].denominator()*S[i] for i in range(3)] + [f1,f2,f3])
J=R.ideal(k1*k2-1)
I=I.saturation(J)[0]
J=I.elimination_ideal([s1,s2,s3])
print(J)
A.<k1,k2>=AffineSpace(QQ,2)
X=A.subscheme([k1^4*k2^4 + 6*k1^4*k2^3 + 6*k1^3*k2^4 + 9*k1^4*k2^2 + 20*k1^3*k2^3 + 9*k1^2*k2^4 - 18*k1^3*k2^2 - 18*k1^2*k2^3 - 72*k1^3*k2 - 111*k1^2*k2^2 - 72*k1*k2^3 + 72*k1^2*k2 + 72*k1*k2^2 + 144*k1^2 + 320*k1*k2 + 144*k2^2 - 384*k1 - 384*k2 + 256])
X.irreducible_components()

Y=A.curve(k1*k2 + 3*k2 - 4)
Y.rational_parameterization()

Y=A.curve(k1*k2 + 3*k1 - 4)
Y.rational_parameterization()
\end{lstlisting}
\end{code}
\end{proof}

\subsection{$\A_4(C_2)$} We now move on to the family $f_{k_1,k_2,k_3}(z)=\frac{z^4+k_1z^2+1}{k_2z^3+k_3z}$.

This family has infinitely many conjugacy classes that have the same fixed point multiplier invariants. Computing the multiplier invariants for points of period two was computationally infeasible, but computing the multiplier invariants for the points of formal period two was possible. We recall the defintion of dynatomic polynomials and formal periodic points. For a rational map $f(z) = \frac{F(z)}{G(z)}$ define $\Phi_1(f) = F(z) - zG(z)$ and $\Phi_n(f) = \Phi_1(f^n)$ for $n > 1$. The roots of $\Phi_n(f)$ are the points with period $n$. We define the \emph{$n$th dynatomic polynomial} as $\Phi^{\ast}_n = \prod_{d \mid n} \Phi_d(f)^{\mu(n/d)}$ where $\mu$ is the M\"obius function. Its roots are the points of \emph{formal period $n$} and contain among them the points of minimal period $n$. Similarly, for strictly preperiodic points, we define the \emph{generalized $(m,n)$-dynatomic polynomial} as $\Phi^{\ast}_{m,n}(f) = \frac{\Phi^{\ast}_n(f^m)}{\Phi^{\ast}_{n-1}(f^m)}$. Its roots are the points of \emph{formal period $(m,n)$} and contain among them the points of minimal period $(m,n)$. See Silverman \cite[\S4.1]{Silverman10} for dynatomic polynomials and Hutz \cite{Hutz12} for generalized dynatomic polynomials. We can also construct invariants from the multipliers of the formal periodic points of period $n$ and notate them as $\sigma_i^{(n)\ast}$.

Utilizing just the $\sigma_{1}^{(2)\ast}$ did result in a finite-to-one map, but there were a number of spurious values appearing. Additionally including $\sigma_2^{(2)\ast}$ resulted in the correct mapping. As the computation is quite time and memory consuming, we record the two higher multiplier invariants here.
\begin{align*}
    \sigma_1^{(2)\ast} &= \frac{-1}{(k_2^2 + k_3^2-k_1k_2k_3)^2} \Big((8k_3^2 - 12)k_2^4 + (-8k_1k_3^3 + 40k_3^2 - 24k_1k_3 + (40k_1^2 - 96))k_2^3\\
    &+ (-8k_1k_3^3 + (-4k_1^2 + 120)k_3^2 + (-8k_1^3 - 64k_1)k_3 + (-8k_1^4 + 80k_1^2 - 192))k_2^2 + (-24k_3^4\\
    &+ (-8k_1^3 + 24k_1)k_3^3 + (8k_1^2 + 160)k_3^2 + (-16k_1^3 + 64k_1)k_3)k_2 + ((8k_1^2 - 60)k_3^4 - 32k_1k_3^3\\
    &+ (16k_1^2 - 64)k_3^2)\Big)\\
    \sigma_2^{(2)\ast} &=\frac{1}{(k_2^2 + k_3^2-k_1k_2k_3)^4} \Big((24k_3^4 - 136k_3^2 - 126)k_2^8 + (-48k_1k_3^5 + 320k_3^4 - 56k_1k_3^3\\
    &+ (320k_1^2 - 600)k_3^2 - 440k_1k_3 + (168k_1^2 + 288))k_2^7 + ((24k_1^2 - 16)k_3^6 - 384k_1k_3^5 + (200k_1^2\\
    &+ 1216)k_3^4 + (-384k_1^3 - 408k_1)k_3^3 + (-64k_1^4 + 2452k_1^2 - 2120)k_3^2 + (-664k_1^3 - 960k_1)k_3\\
    &+ (360k_1^4 - 2704k_1^2 + 5568))k_2^6 + (16k_1k_3^7 + (64k_1^2 - 32)k_3^6 + (-72k_1^3 - 1120k_1)k_3^5\\
    &+ (64k_1^4 + 664k_1^2 + 2712)k_3^4 + (64k_1^5 - 1168k_1^3 - 2680k_1)k_3^3 + (-232k_1^4 + 8312k_1^2 - 9120)k_3^2 \\
    &+ (-16k_1^5 - 528k_1^3 + 832k_1)k_3 + (-160k_1^6 + 2656k_1^4 - 12288k_1^2 + 16896))k_2^5 + (-24k_3^8\\
    &+ 160k_1k_3^7 + (64k_1^4 + 32k_1^2 - 296)k_3^6 + (-360k_1^3 - 112k_1)k_3^5 + (242k_1^4 - 1120k_1^2 + 6652)k_3^4 \\
    &+ (-232k_1^5 - 624k_1^3 - 7392k_1)k_3^3 + (40k_1^6 - 976k_1^4 + 10512k_1^2 - 21312)k_3^2 + (32k_1^7 - 1216k_1^5 \\
    &+ 4992k_1^3 - 2560k_1)k_3 + (16k_1^8 - 448k_1^6 + 3872k_1^4 - 13568k_1^2 + 16896))k_2^4 + (-96k_3^8 + (-64k_1^3 \\
    &+ 280k_1)k_3^7 + (64k_1^4 + 312k_1^2 - 2632)k_3^6 + (72k_1^5 - 920k_1^3 + 6008k_1)k_3^5 + (64k_1^6 + 952k_1^4\\
    &- 4200k_1^2 + 13152)k_3^4 + (64k_1^7 - 544k_1^5 - 608k_1^3 + 1920k_1)k_3^3 + (192k_1^6 + 832k_1^4\\
    &- 3072k_1^2 - 13312)k_3^2 + (192k_1^7 - 2368k_1^5 + 9728k_1^3 - 13312k_1)k_3)k_2^3 + (144k_3^8\\
    &+ (128k_1^3 - 600k_1)k_3^7 + (24k_1^6 - 280k_1^4 + 532k_1^2 - 7720)k_3^6 + (-128k_1^5 - 952k_1^3 + 6016k_1)k_3^5\\
    &+ (16k_1^6 + 152k_1^4 - 240k_1^2 + 5184)k_3^4 + (-192k_1^5 - 4864k_1^3 + 22528k_1)k_3^3 + (-160k_1^6 + 2112k_1^4\\
    &- 9216k_1^2 + 13312)k_3^2)k_2^2 + ((-192k_1^2 + 1224)k_3^8 + (-48k_1^5 + 680k_1^3 - 712k_1)k_3^7 + (320k_1^4\\
    &- 824k_1^2 - 6112)k_3^6 + (-160k_1^5 + 1776k_1^3 - 6080k_1)k_3^5 + (736k_1^4 + 512k_1^2 - 13824)k_3^4 + (-64k_1^5\\
    &+ 512k_1^3 - 1024k_1)k_3^3)k_2 + ((24k_1^4 - 472k_1^2 + 1554)k_3^8 + (-256k_1^3 + 1696k_1)k_3^7\\
    &+ (80k_1^4 - 1040k_1^2 + 3392)k_3^6 + (-640k_1^3 + 2560k_1)k_3^5 + (32k_1^4 - 256k_1^2 + 512)k_3^4)\Big).
\end{align*}

The first two fixed point multiplier invariants are given by
\begin{align*}
    \sigma_1^{(1)} &= \frac{1}{k_2^2 + k_3^2-k_1k_2k_3} \Big(-k_1k_2^2k_3 - 4k_1^2k_2 + k_2^3 + 4k_1k_2k_3 - 3k_2k_3^2 - 12k_2^2 + 4k_3^2 + 16k_2\Big)\\
    \sigma_2^{(1)} &=\frac{1}{(k_2^2 + k_3^2-k_1k_2k_3)^2} \Big(4k_1^3k_2^3k_3 - 4k_1^2k_2^3k_3^2 + 4k_1k_2^3k_3^3 + 4k_1^4k_2^2 - 4k_1^2k_2^4 - 12k_1^3k_2^2k_3\\
    &+ 16k_1k_2^4k_3 + 18k_1^2k_2^2k_3^2 - 4k_2^4k_3^2 - 12k_1k_2^2k_3^3 + 28k_1^2k_2^3 - 12k_2^5 + 8k_1^3k_2k_3 - 68k_1k_2^3k_3\\
    &- 20k_1^2k_2k_3^2 + 20k_2^3k_3^2 + 4k_1k_2k_3^3 - 40k_1^2k_2^2 + 70k_2^4 + 96k_1k_2^2k_3 - 8k_1^2k_3^2 - 28k_2^2k_3^2\\
    &+ 16k_1k_3^3 - 2k_3^4 - 144k_2^3 - 32k_1k_2k_3 - 16k_2k_3^2 + 96k_2^2 + 32k_3^2\Big).
\end{align*}

\begin{prop} \label{lem_A4_C2}
 The map \begin{align*}
        \varphi:\bbA^3\setminus\{k_2^2 + k_3^2=k_1k_2k_3\} &\to \M_4\\
        (k_1,k_2,k_3) &\mapsto [f_{k_1,k_2,k_3}]
    \end{align*}
    is generically two-to-one. The map \begin{align*}
        \tau: \varphi(\bbA^2\setminus\{k_2^2 + k_3^2=k_1k_2k_3\}) &\to \bbA^4\\
        [f_{k}] &\mapsto (\sigma_1^{(1)}, \sigma_2^{(1)},\sigma_1^{(2)\ast},\sigma_2^{(2)\ast})
    \end{align*} is injective.
\end{prop}
\begin{proof}
    Trying to compute the degree generically as in the previous families was not feasible with our hardware resources, so we instead specialize to a particular choice of invariants and show that the degree is invariant under perturbation.

    First note that the domain of $\varphi$ is irreducible and the composition $\tau \circ \varphi$ is a morphism, so the image is also irreducible. From Milne \cite[Section 10]{Milne} the dimension of fibers in this situation can only go up in a closed set, and the number of points in a specific fiber is at most the degree of the morphism and is equal to the degree for nonsingular fibers. We choose a nonsingular fiber where the fiber dimension is zero and has two points in it showing that the degree is two.

    Choosing
    \begin{equation*}
        (\sigma_1^{(1)}, \sigma_2^{(1)},\sigma_1^{(2)\ast},\sigma_2^{(2)\ast})
        = \left(\frac{-22}{7}, -\frac{82}{49}, \frac{2164}{49}, \frac{769442}{2401} \right)
    \end{equation*}
    generates a system of equations in $(k_1,k_2,k_3)$. We take the associated ideal and saturate with respect to $(-k_1k_2k_3 + k_2^2 + k_3^2)$ to remove any degenerate solutions from the system. The resulting ideal is given by
    \begin{equation*}
        I = (k_2 - 2, 3k_1 - k_3, k_3^2 - 9).
    \end{equation*}
    This has the two solution $(1,2,3)$ and $(-1,2,-3)$, which are conjugate via $z \mapsto iz$. Call $X$ the variety associated to $I$. The projective closure of $X$ is dimension $0$ and degree $2$, so there should be two points when counted with multiplicity. The two given solutions are both solutions of multiplicity $1$, so they are all the solutions to the system.

    In general, we always get the two conjugate solutions $(k_1,k_2,k_3)$ and $(-k_1,k_2,-k_3)$, so $\varphi$ is two-to-one. Since these two solutions are conjugate $\tau$ is injective.
\end{proof}

\begin{code}
\begin{lstlisting}
R.<k1,k2,k3,t> = QQ[]
P.<z> = R[]
A = AffineSpace(P)
S=A.coordinate_ring()
f = DynamicalSystem_affine((z^4+k1*z^2+1)/(k2*z^3+k3*z))
F=f.homogenize(1)
f2=f.nth_iterate_map(2)
phi2=S(f.dynatomic_polynomial(2))
df = f2[0].derivative(z)
dF = df.denominator()*t - df.numerator()


default(parisize, 42000000000)
F=(k2^4*t - k2^2)*z^30 + (2*k2^5*k3*t + 6*k1*k2^4*t + k1*k2^4 + 2*k2^3*k3*t - 3*k2^3*k3 - 5*k1*k2^2 - 3*k2*k3)*z^28 + (k2^6*k3^2*t + 8*k1*k2^5*k3*t - k1*k2^5*k3 + 15*k1^2*k2^4*t + 8*k2^4*k3^2*t + 3*k1^2*k2^4 + 3*k2^6 + 12*k1*k2^3*k3*t - 4*k1*k2^3*k3 + 6*k2^4*t + k2^2*k3^2*t - 9*k1^2*k2^2 - 15*k2^2*k3^2 - 19*k1*k2*k3 - 3*k2^2)*z^26 + (2*k1*k2^6*k3^2*t + 12*k1^2*k2^5*k3*t + 6*k2^5*k3^3*t - k1^2*k2^5*k3 + k2^7*k3 + 20*k1^3*k2^4*t + 32*k1*k2^4*k3^2*t + 2*k1^3*k2^4 + 3*k1*k2^6 - 7*k1*k2^4*k3^2 + 30*k1^2*k2^3*k3*t + 8*k2^5*k3*t + 12*k2^3*k3^3*t + 13*k1^2*k2^3*k3 + 21*k2^5*k3 + 30*k1*k2^4*t + 6*k1*k2^2*k3^2*t - 5*k1^3*k2^2 + k1*k2^4 - 50*k1*k2^2*k3^2 + 12*k2^3*k3*t - 51*k1^2*k2*k3 - 3*k2^3*k3 - 21*k2*k3^3 - 6*k1*k2^2 - 17*k2*k3)*z^24 + (k1^2*k2^6*k3^2*t + 8*k1^3*k2^5*k3*t + 12*k1*k2^5*k3^3*t + k1^3*k2^5*k3 - k1*k2^7*k3 + 15*k1^4*k2^4*t + 48*k1^2*k2^4*k3^2*t + 2*k2^6*k3^2*t + 15*k2^4*k3^4*t - 2*k1^4*k2^4 - 3*k1^2*k2^6 - 11*k1^2*k2^4*k3^2 + 9*k2^6*k3^2 + 40*k1^3*k2^3*k3*t + 24*k1*k2^5*k3*t + 48*k1*k2^3*k3^3*t + 32*k1^3*k2^3*k3 + 34*k1*k2^5*k3 - 18*k1*k2^3*k3^3 + 60*k1^2*k2^4*t + 15*k1^2*k2^2*k3^2*t + 32*k2^4*k3^2*t + 8*k2^2*k3^4*t + 5*k1^4*k2^2 - 4*k1^2*k2^4 - 3*k2^6 - 49*k1^2*k2^2*k3^2 + 54*k2^4*k3^2 + 60*k1*k2^3*k3*t - 75*k1^3*k2*k3 + 25*k1*k2^3*k3 - 84*k1*k2*k3^3 + 15*k2^4*t + 6*k2^2*k3^2*t + 15*k1^2*k2^2 - 39*k2^2*k3^2 - 9*k3^4 - 90*k1*k2*k3 + 3*k2^2)*z^22 + (2*k1^4*k2^5*k3*t + 6*k1^2*k2^5*k3^3*t + k1^4*k2^5*k3 + 6*k1^5*k2^4*t + 32*k1^3*k2^4*k3^2*t + 2*k1*k2^6*k3^2*t + 30*k1*k2^4*k3^4*t - 3*k1^5*k2^4 - 3*k1^3*k2^6 - k1^3*k2^4*k3^2 - 5*k1*k2^6*k3^2 + 30*k1^4*k2^3*k3*t + 24*k1^2*k2^5*k3*t + 72*k1^2*k2^3*k3^3*t + 12*k2^5*k3^3*t + 20*k2^3*k3^5*t + 23*k1^4*k2^3*k3 + 9*k1^2*k2^5*k3 - 3*k2^7*k3 - 34*k1^2*k2^3*k3^3 + 33*k2^5*k3^3 + 60*k1^3*k2^4*t + 20*k1^3*k2^2*k3^2*t + 96*k1*k2^4*k3^2*t + 32*k1*k2^2*k3^4*t + 9*k1^5*k2^2 - 18*k1^3*k2^4 - 18*k1*k2^6 + 4*k1^3*k2^2*k3^2 + 101*k1*k2^4*k3^2 - 22*k1*k2^2*k3^4 + 120*k1^2*k2^3*k3*t + 12*k2^5*k3*t + 48*k2^3*k3^3*t + 2*k2*k3^5*t - 65*k1^4*k2*k3 + 82*k1^2*k2^3*k3 + 3*k2^5*k3 - 125*k1^2*k2*k3^3 + 66*k2^3*k3^3 + 60*k1*k2^4*t + 30*k1*k2^2*k3^2*t + 60*k1^3*k2^2 - 6*k1*k2^4 - 61*k1*k2^2*k3^2 - 39*k1*k3^4 + 30*k2^3*k3*t - 195*k1^2*k2*k3 + 18*k2^3*k3 - 69*k2*k3^3 + 45*k1*k2^2 - 39*k2*k3)*z^20\
+ (k1^6*k2^4*t + 8*k1^4*k2^4*k3^2*t + 15*k1^2*k2^4*k3^4*t - k1^6*k2^4 + 3*k1^4*k2^4*k3^2 + 12*k1^5*k2^3*k3*t + 8*k1^3*k2^5*k3*t + 48*k1^3*k2^3*k3^3*t + 12*k1*k2^5*k3^3*t + 40*k1*k2^3*k3^5*t + 4*k1^5*k2^3*k3 - 4*k1^3*k2^5*k3 - 14*k1^3*k2^3*k3^3 - 9*k1*k2^5*k3^3 + 30*k1^4*k2^4*t + 15*k1^4*k2^2*k3^2*t + 96*k1^2*k2^4*k3^2*t + k2^6*k3^2*t + 48*k1^2*k2^2*k3^4*t + 30*k2^4*k3^4*t + 15*k2^2*k3^6*t + 5*k1^6*k2^2 - 20*k1^4*k2^4 - 15*k1^2*k2^6 + 31*k1^4*k2^2*k3^2 + 60*k1^2*k2^4*k3^2 - 19*k2^6*k3^2 - 46*k1^2*k2^2*k3^4 + 65*k2^4*k3^4 + 120*k1^3*k2^3*k3*t + 24*k1*k2^5*k3*t + 144*k1*k2^3*k3^3*t + 8*k1*k2*k3^5*t - 33*k1^5*k2*k3 + 66*k1^3*k2^3*k3 - 49*k1*k2^5*k3 - 80*k1^3*k2*k3^3 + 124*k1*k2^3*k3^3 - 13*k1*k2*k3^5 + 90*k1^2*k2^4*t + 60*k1^2*k2^2*k3^2*t + 48*k2^4*k3^2*t + 32*k2^2*k3^4*t + 75*k1^4*k2^2 - 30*k1^2*k2^4 - 15*k2^6 + 50*k1^2*k2^2*k3^2 + 42*k2^4*k3^2 - 66*k1^2*k3^4 + 39*k2^2*k3^4 + 120*k1*k2^3*k3*t - 220*k1^3*k2*k3 + 92*k1*k2^3*k3 - 193*k1*k2*k3^3 + 20*k2^4*t + 15*k2^2*k3^2*t + 150*k1^2*k2^2 - 6*k2^2*k3^2 - 33*k3^4 - 165*k1*k2*k3 + 25*k2^2)*z^18 + (2*k1^6*k2^3*k3*t + 12*k1^4*k2^3*k3^3*t + 20*k1^2*k2^3*k3^5*t - k1^6*k2^3*k3 + 2*k1^4*k2^3*k3^3 + 6*k1^5*k2^4*t + 6*k1^5*k2^2*k3^2*t + 32*k1^3*k2^4*k3^2*t + 32*k1^3*k2^2*k3^4*t + 30*k1*k2^4*k3^4*t + 30*k1*k2^2*k3^6*t + k1^7*k2^2 - 7*k1^5*k2^4 + 14*k1^5*k2^2*k3^2 + 13*k1^3*k2^4*k3^2 - 26*k1^3*k2^2*k3^4 - 5*k1*k2^4*k3^4 + 60*k1^4*k2^3*k3*t + 12*k1^2*k2^5*k3*t + 144*k1^2*k2^3*k3^3*t + 6*k2^5*k3^3*t + 12*k1^2*k2*k3^5*t + 40*k2^3*k3^5*t + 6*k2*k3^7*t - 9*k1^6*k2*k3 + k1^4*k2^3*k3 - 50*k1^2*k2^5*k3 - 15*k1^4*k2*k3^3 + 90*k1^2*k2^3*k3^3 - 51*k2^5*k3^3 - 29*k1^2*k2*k3^5 + 75*k2^3*k3^5 + 60*k1^3*k2^4*t + 60*k1^3*k2^2*k3^2*t + 96*k1*k2^4*k3^2*t + 96*k1*k2^2*k3^4*t + 42*k1^5*k2^2 - 42*k1^3*k2^4 - 21*k1*k2^6 + 126*k1^3*k2^2*k3^2 - 17*k1*k2^4*k3^2 - 54*k1^3*k3^4 + 61*k1*k2^2*k3^4 - 3*k1*k3^6 + 180*k1^2*k2^3*k3*t + 8*k2^5*k3*t + 72*k2^3*k3^3*t + 8*k2*k3^5*t - 135*k1^4*k2*k3 + 96*k1^2*k2^3*k3 - 57*k2^5*k3 - 162*k1^2*k2*k3^3 + 78*k2^3*k3^3 + 9*k2*k3^5 + 60*k1*k2^4*t + 60*k1*k2^2*k3^2*t + 210*k1^3*k2^2 - 14*k1*k2^4 + 112*k1*k2^2*k3^2 - 108*k1*k3^4 + 40*k2^3*k3*t - 270*k1^2*k2*k3 + 42*k2^3*k3 - 66*k2*k3^3 + 140*k1*k2^2 - 45*k2*k3)*z^16\
+ (k1^6*k2^2*k3^2*t + 8*k1^4*k2^2*k3^4*t + 15*k1^2*k2^2*k3^6*t + k1^6*k2^2*k3^2 - 2*k1^4*k2^2*k3^4 + 12*k1^5*k2^3*k3*t + 48*k1^3*k2^3*k3^3*t + 8*k1^3*k2*k3^5*t + 40*k1*k2^3*k3^5*t + 12*k1*k2*k3^7*t - k1^7*k2*k3 - 11*k1^5*k2^3*k3 + 4*k1^5*k2*k3^3 + 32*k1^3*k2^3*k3^3 - 19*k1^3*k2*k3^5 + 5*k1*k2^3*k3^5 + 15*k1^4*k2^4*t + 30*k1^4*k2^2*k3^2*t + 48*k1^2*k2^4*k3^2*t + 96*k1^2*k2^2*k3^4*t + 15*k2^4*k3^4*t + 30*k2^2*k3^6*t + k3^8*t + 9*k1^6*k2^2 - 18*k1^4*k2^4 + 53*k1^4*k2^2*k3^2 - 49*k1^2*k2^4*k3^2 - 21*k1^4*k3^4 + 45*k1^2*k2^2*k3^4 - 75*k2^4*k3^4 - 7*k1^2*k3^6 + 51*k2^2*k3^6 + 120*k1^3*k2^3*k3*t + 8*k1*k2^5*k3*t + 144*k1*k2^3*k3^3*t + 24*k1*k2*k3^5*t - 42*k1^5*k2*k3 - 12*k1^3*k2^3*k3 - 84*k1*k2^5*k3 - 18*k1^3*k2*k3^3 + 62*k1*k2^3*k3^3 + 2*k1*k2*k3^5 + 60*k1^2*k2^4*t + 90*k1^2*k2^2*k3^2*t + 32*k2^4*k3^2*t + 48*k2^2*k3^4*t + 135*k1^4*k2^2 - 36*k1^2*k2^4 - 9*k2^6 + 228*k1^2*k2^2*k3^2 - 78*k2^4*k3^2 - 126*k1^2*k3^4 + 57*k2^2*k3^4 + 120*k1*k2^3*k3*t - 210*k1^3*k2*k3 + 86*k1*k2^3*k3 - 76*k1*k2*k3^3 + 15*k2^4*t + 20*k2^2*k3^2*t + 270*k1^2*k2^2 + 66*k2^2*k3^2 - 42*k3^4 - 140*k1*k2*k3 + 45*k2^2)*z^14 + (2*k1^4*k2*k3^5*t + 6*k1^2*k2*k3^7*t + k1^6*k2*k3^3 - 3*k1^4*k2*k3^5 + 6*k1^5*k2^2*k3^2*t + 32*k1^3*k2^2*k3^4*t + 30*k1*k2^2*k3^6*t + 2*k1*k3^8*t - k1^5*k2^2*k3^2 - 3*k1^5*k3^4 + 23*k1^3*k2^2*k3^4 - 5*k1^3*k3^6 + 9*k1*k2^2*k3^6 + 30*k1^4*k2^3*k3*t + 72*k1^2*k2^3*k3^3*t + 24*k1^2*k2*k3^5*t + 20*k2^3*k3^5*t + 12*k2*k3^7*t - 5*k1^6*k2*k3 - 34*k1^4*k2^3*k3 + 23*k1^4*k2*k3^3 + 4*k1^2*k2^3*k3^3 - 3*k1^2*k2*k3^5 - 65*k2^3*k3^5 + 19*k2*k3^7 + 20*k1^3*k2^4*t + 60*k1^3*k2^2*k3^2*t + 32*k1*k2^4*k3^2*t + 96*k1*k2^2*k3^4*t + 33*k1^5*k2^2 - 22*k1^3*k2^4 + 96*k1^3*k2^2*k3^2 - 125*k1*k2^4*k3^2 - 60*k1^3*k3^4 + 68*k1*k2^2*k3^4 - 5*k1*k3^6 + 120*k1^2*k2^3*k3*t + 2*k2^5*k3*t + 48*k2^3*k3^3*t + 12*k2*k3^5*t - 75*k1^4*k2*k3 - 2*k1^2*k2^3*k3 - 39*k2^5*k3 + 48*k1^2*k2*k3^3 - 42*k2^3*k3^3 + 15*k2*k3^5 + 30*k1*k2^4*t + 60*k1*k2^2*k3^2*t + 220*k1^3*k2^2 - 11*k1*k2^4 + 202*k1*k2^2*k3^2 - 90*k1*k3^4 + 30*k2^3*k3*t - 150*k1^2*k2*k3 + 33*k2^3*k3 + 6*k2*k3^3 + 165*k1*k2^2 - 25*k2*k3)*z^12 + (k1^2*k3^8*t - k1^4*k3^6 + 8*k1^3*k2*k3^5*t + 12*k1*k2*k3^7*t + 3*k1^5*k2*k3^3 + 4*k1^3*k2*k3^5 + 5*k1*k2*k3^7 + 15*k1^4*k2^2*k3^2*t + 48*k1^2*k2^2*k3^4*t + 15*k2^2*k3^6*t + 2*k3^8*t - 14*k1^4*k2^2*k3^2 - 9*k1^4*k3^4 + 31*k1^2*k2^2*k3^4 - 6*k1^2*k3^6 - 33*k2^2*k3^6 + 3*k3^8 + 40*k1^3*k2^3*k3*t + 48*k1*k2^3*k3^3*t + 24*k1*k2*k3^5*t - 9*k1^5*k2*k3 - 46*k1^3*k2^3*k3 + 60*k1^3*k2*k3^3 - 80*k1*k2^3*k3^3 + 19*k1*k2*k3^5 + 15*k1^2*k2^4*t + 60*k1^2*k2^2*k3^2*t + 8*k2^4*k3^2*t + 32*k2^2*k3^4*t + 65*k1^4*k2^2 - 13*k1^2*k2^4 + 110*k1^2*k2^2*k3^2 - 66*k2^4*k3^2 - 54*k1^2*k3^4 - 3*k2^2*k3^4 + 60*k1*k2^3*k3*t - 60*k1^3*k2*k3 + 16*k1*k2^3*k3 + 90*k1*k2*k3^3 + 6*k2^4*t + 15*k2^2*k3^2*t + 195*k1^2*k2^2 + 69*k2^2*k3^2 - 18*k3^4 - 45*k1*k2*k3 + 39*k2^2)*z^10\
+ (2*k1*k3^8*t - k1^3*k3^6 + k1*k3^8 + 12*k1^2*k2*k3^5*t + 6*k2*k3^7*t + 2*k1^4*k2*k3^3 + 14*k1^2*k2*k3^5 - 9*k2*k3^7 + 20*k1^3*k2^2*k3^2*t + 32*k1*k2^2*k3^4*t - 26*k1^3*k2^2*k3^2 - 6*k1^3*k3^4 - 15*k1*k2^2*k3^4 - k1*k3^6 + 30*k1^2*k2^3*k3*t + 12*k2^3*k3^3*t + 8*k2*k3^5*t - 5*k1^4*k2*k3 - 29*k1^2*k2^3*k3 + 82*k1^2*k2*k3^3 - 54*k2^3*k3^3 + 3*k2*k3^5 + 6*k1*k2^4*t + 30*k1*k2^2*k3^2*t + 75*k1^3*k2^2 - 3*k1*k2^4 + 74*k1*k2^2*k3^2 - 12*k1*k3^4 + 12*k2^3*k3*t - 15*k1^2*k2*k3 + 9*k2^3*k3 + 39*k2*k3^3 + 90*k1*k2^2 - 3*k2*k3)*z^8 + (k3^8*t + k1^2*k3^6 - k3^8 + 8*k1*k2*k3^5*t - 2*k1^3*k2*k3^3 + 4*k1*k2*k3^5 + 15*k1^2*k2^2*k3^2*t + 8*k2^2*k3^4*t - 19*k1^2*k2^2*k3^2 + 6*k1^2*k3^4 - 21*k2^2*k3^4 + 12*k1*k2^3*k3*t + 5*k1^3*k2*k3 - 7*k1*k2^3*k3 + 56*k1*k2*k3^3 + k2^4*t + 6*k2^2*k3^2*t + 51*k1^2*k2^2 + 21*k2^2*k3^2 + 3*k3^4 + 6*k1*k2*k3 + 17*k2^2)*z^6 + (k1*k3^6 + 2*k2*k3^5*t - 3*k1^2*k2*k3^3 - 3*k2*k3^5 + 6*k1*k2^2*k3^2*t - 5*k1*k2^2*k3^2 + 9*k1*k3^4 + 2*k2^3*k3*t + 9*k1^2*k2*k3 + 15*k2*k3^3 + 19*k1*k2^2 + 3*k2*k3)*z^4 + (-k1*k2*k3^3 + k2^2*k3^2*t + 3*k3^4 + 5*k1*k2*k3 + 3*k2^2)*z^2 + k2*k3
G=(k2 + 1)*z^12 + (k2^2*k3 + 2*k1*k2 + k2*k3 + 3*k1 + k3)*z^10 + (k1*k2^2*k3 + k1^2*k2 - k2^3 + 2*k1*k2*k3 + 2*k2*k3^2 + 3*k1^2 - k2^2 + 2*k1*k3 + k3^2 + 2*k2 + 3)*z^8 + (k1^2*k2*k3 + 2*k1*k2*k3^2 + k1^3 - k1*k2^2 + k1^2*k3 - 2*k2^2*k3 + 2*k1*k3^2 + k3^3 + 2*k1*k2 + 6*k1 + 2*k3)*z^6 + (k1^2*k3^2 + k1*k3^3 - k2*k3^2 + 3*k1^2 - k2^2 + 2*k1*k3 + k3^2 + k2 + 3)*z^4 + (k1*k3^2 - k2*k3 + 3*k1 + k3)*z^2 + 1
r=polresultant(F,G,z)
s21=polcoeff(r,11,t)/polcoeff(r,12,t)

86+ hours, 21+Gb

sig21=-((8*k3^2 - 12)*k2^4 + (-8*k1*k3^3 + 40*k3^2 - 24*k1*k3 + (40*k1^2 - 96))*k2^3 + (-8*k1*k3^3 + (-4*k1^2 + 120)*k3^2 + (-8*k1^3 - 64*k1)*k3 + (-8*k1^4 + 80*k1^2 - 192))*k2^2 + (-24*k3^4 + (-8*k1^3 + 24*k1)*k3^3 + (8*k1^2 + 160)*k3^2 + (-16*k1^3 + 64*k1)*k3)*k2 + ((8*k1^2 - 60)*k3^4 - 32*k1*k3^3 + (16*k1^2 - 64)*k3^2))/(k2^4 - 2*k1*k3*k2^3 + (k1^2 + 2)*k3^2*k2^2 - 2*k1*k3^3*k2 + k3^4)

sig22 = ((24*k3^4 - 136*k3^2 - 126)*k2^8 + (-48*k1*k3^5 + 320*k3^4 - 56*k1*k3^3 + (320*k1^2 - 600)*k3^2 - 440*k1*k3 + (168*k1^2 + 288))*k2^7 + ((24*k1^2 - 16)*k3^6 - 384*k1*k3^5 + (200*k1^2 + 1216)*k3^4+ (-384*k1^3 - 408*k1)*k3^3 + (-64*k1^4 + 2452*k1^2 - 2120)*k3^2 + (-664*k1^3 - 960*k1)*k3 + (360*k1^4 - 2704*k1^2 + 5568))*k2^6 + (16*k1*k3^7 + (64*k1^2 - 32)*k3^6 + (-72*k1^3 - 1120*k1)*k3^5 + (64*k1^4 + 664*k1^2 + 2712)*k3^4 + (64*k1^5 - 1168*k1^3 - 2680*k1)*k3^3 + (-232*k1^4 + 8312*k1^2 - 9120)*k3^2 + (-16*k1^5 - 528*k1^3 + 832*k1)*k3 + (-160*k1^6 + 2656*k1^4 - 12288*k1^2 + 16896))*k2^5 + (-24*k3^8 + 160*k1*k3^7 + (64*k1^4 + 32*k1^2 - 296)*k3^6 + (-360*k1^3 - 112*k1)*k3^5 + (242*k1^4 - 1120*k1^2 + 6652)*k3^4 + (-232*k1^5 - 624*k1^3 - 7392*k1)*k3^3 + (40*k1^6 - 976*k1^4 + 10512*k1^2 - 21312)*k3^2 + (32*k1^7 - 1216*k1^5 + 4992*k1^3 - 2560*k1)*k3 + (16*k1^8 - 448*k1^6 + 3872*k1^4 - 13568*k1^2 + 16896))*k2^4 + (-96*k3^8 + (-64*k1^3 + 280*k1)*k3^7 + (64*k1^4 + 312*k1^2 - 2632)*k3^6 + (72*k1^5 - 920*k1^3 + 6008*k1)*k3^5 + (64*k1^6 + 952*k1^4 - 4200*k1^2 + 13152)*k3^4 + (64*k1^7 - 544*k1^5 - 608*k1^3 + 1920*k1)*k3^3 + (192*k1^6 + 832*k1^4 - 3072*k1^2 - 13312)*k3^2 + (192*k1^7 - 2368*k1^5 + 9728*k1^3 - 13312*k1)*k3)*k2^3 + (144*k3^8 + (128*k1^3 - 600*k1)*k3^7 + (24*k1^6 - 280*k1^4 + 532*k1^2 - 7720)*k3^6 + (-128*k1^5 - 952*k1^3 + 6016*k1)*k3^5 + (16*k1^6 + 152*k1^4 - 240*k1^2 + 5184)*k3^4 + (-192*k1^5 - 4864*k1^3 + 22528*k1)*k3^3 + (-160*k1^6 + 2112*k1^4 - 9216*k1^2 + 13312)*k3^2)*k2^2 + ((-192*k1^2 + 1224)*k3^8 + (-48*k1^5 + 680*k1^3 - 712*k1)*k3^7 + (320*k1^4 - 824*k1^2 - 6112)*k3^6 + (-160*k1^5 + 1776*k1^3 - 6080*k1)*k3^5 + (736*k1^4 + 512*k1^2 - 13824)*k3^4 + (-64*k1^5 + 512*k1^3 - 1024*k1)*k3^3)*k2 + ((24*k1^4 - 472*k1^2 + 1554)*k3^8 + (-256*k1^3 + 1696*k1)*k3^7 + (80*k1^4 - 1040*k1^2 + 3392)*k3^6 + (-640*k1^3 + 2560*k1)*k3^5 + (32*k1^4 - 256*k1^2 + 512)*k3^4))/(k2^8 - 4*k1*k3*k2^7 + (6*k1^2 + 4)*k3^2*k2^6 + (-4*k1^3 - 12*k1)*k3^3*k2^5 + (k1^4 + 12*k1^2 + 6)*k3^4*k2^4 + (-4*k1^3 - 12*k1)*k3^5*k2^3 + (6*k1^2 + 4)*k3^6*k2^2 - 4*k1*k3^7*k2 + k3^8)
\end{lstlisting}
\end{code}

\begin{code}
\begin{lstlisting}
#checking that these are the right 4 sigmas and 2-to-1
K=QQ
P.<x,y>=ProjectiveSpace(K,1)
a=1
b=2
c=3

f1=DynamicalSystem([x^4 + a*x^2*y^2+y^4,b*x^3*y+c*x*y^3])
R.<k1,k2,k3>=K[]
S=f1.sigma_invariants(1)[0:-1] + f1.sigma_invariants(2,formal=True)[0:2]
P.<x,y>=ProjectiveSpace(R,1)
f1=DynamicalSystem([x^4 + k1*x^2*y^2+y^4,k2*x^3*y+k3*x*y^3])
s1=f1.sigma_invariants(1)[:-1]
sig21 = -((8*k3^2 - 12)*k2^4 + (-8*k1*k3^3 + 40*k3^2 - 24*k1*k3 + (40*k1^2 - 96))*k2^3 + (-8*k1*k3^3 + (-4*k1^2 + 120)*k3^2 + (-8*k1^3 - 64*k1)*k3 + (-8*k1^4 + 80*k1^2 - 192))*k2^2 + (-24*k3^4 + (-8*k1^3 + 24*k1)*k3^3 + (8*k1^2 + 160)*k3^2 + (-16*k1^3 + 64*k1)*k3)*k2 + ((8*k1^2 - 60)*k3^4 - 32*k1*k3^3 + (16*k1^2 - 64)*k3^2))/(k2^4 - 2*k1*k3*k2^3 + (k1^2 + 2)*k3^2*k2^2 - 2*k1*k3^3*k2 + k3^4)
sig22 = ((24*k3^4 - 136*k3^2 - 126)*k2^8 + (-48*k1*k3^5 + 320*k3^4 - 56*k1*k3^3 + (320*k1^2 - 600)*k3^2 - 440*k1*k3 + (168*k1^2 + 288))*k2^7 + ((24*k1^2 - 16)*k3^6 - 384*k1*k3^5 + (200*k1^2 + 1216)*k3^4+ (-384*k1^3 - 408*k1)*k3^3 + (-64*k1^4 + 2452*k1^2 - 2120)*k3^2 + (-664*k1^3 - 960*k1)*k3 + (360*k1^4 - 2704*k1^2 + 5568))*k2^6 + (16*k1*k3^7 + (64*k1^2 - 32)*k3^6 + (-72*k1^3 - 1120*k1)*k3^5 + (64*k1^4 + 664*k1^2 + 2712)*k3^4 + (64*k1^5 - 1168*k1^3 - 2680*k1)*k3^3 + (-232*k1^4 + 8312*k1^2 - 9120)*k3^2 + (-16*k1^5 - 528*k1^3 + 832*k1)*k3 + (-160*k1^6 + 2656*k1^4 - 12288*k1^2 + 16896))*k2^5 + (-24*k3^8 + 160*k1*k3^7 + (64*k1^4 + 32*k1^2 - 296)*k3^6 + (-360*k1^3 - 112*k1)*k3^5 + (242*k1^4 - 1120*k1^2 + 6652)*k3^4 + (-232*k1^5 - 624*k1^3 - 7392*k1)*k3^3 + (40*k1^6 - 976*k1^4 + 10512*k1^2 - 21312)*k3^2 + (32*k1^7 - 1216*k1^5 + 4992*k1^3 - 2560*k1)*k3 + (16*k1^8 - 448*k1^6 + 3872*k1^4 - 13568*k1^2 + 16896))*k2^4 + (-96*k3^8 + (-64*k1^3 + 280*k1)*k3^7 + (64*k1^4 + 312*k1^2 - 2632)*k3^6 + (72*k1^5 - 920*k1^3 + 6008*k1)*k3^5 + (64*k1^6 + 952*k1^4 - 4200*k1^2 + 13152)*k3^4 + (64*k1^7 - 544*k1^5 - 608*k1^3 + 1920*k1)*k3^3 + (192*k1^6 + 832*k1^4 - 3072*k1^2 - 13312)*k3^2 + (192*k1^7 - 2368*k1^5 + 9728*k1^3 - 13312*k1)*k3)*k2^3 + (144*k3^8 + (128*k1^3 - 600*k1)*k3^7 + (24*k1^6 - 280*k1^4 + 532*k1^2 - 7720)*k3^6 + (-128*k1^5 - 952*k1^3 + 6016*k1)*k3^5 + (16*k1^6 + 152*k1^4 - 240*k1^2 + 5184)*k3^4 + (-192*k1^5 - 4864*k1^3 + 22528*k1)*k3^3 + (-160*k1^6 + 2112*k1^4 - 9216*k1^2 + 13312)*k3^2)*k2^2 + ((-192*k1^2 + 1224)*k3^8 + (-48*k1^5 + 680*k1^3 - 712*k1)*k3^7 + (320*k1^4 - 824*k1^2 - 6112)*k3^6 + (-160*k1^5 + 1776*k1^3 - 6080*k1)*k3^5 + (736*k1^4 + 512*k1^2 - 13824)*k3^4 + (-64*k1^5 + 512*k1^3 - 1024*k1)*k3^3)*k2 + ((24*k1^4 - 472*k1^2 + 1554)*k3^8 + (-256*k1^3 + 1696*k1)*k3^7 + (80*k1^4 - 1040*k1^2 + 3392)*k3^6 + (-640*k1^3 + 2560*k1)*k3^5 + (32*k1^4 - 256*k1^2 + 512)*k3^4))/(k2^8 - 4*k1*k3*k2^7 + (6*k1^2 + 4)*k3^2*k2^6 + (-4*k1^3 - 12*k1)*k3^3*k2^5 + (k1^4 + 12*k1^2 + 6)*k3^4*k2^4 + (-4*k1^3 - 12*k1)*k3^5*k2^3 + (6*k1^2 + 4)*k3^6*k2^2 - 4*k1*k3^7*k2 + k3^8)
s1 += [sig21,sig22]
A=AffineSpace(R)
L= [s1[i].numerator()-S[i]*(s1[i].denominator()) for i in [0,1,4,5]]
I=R.ideal(L)
J1=R.ideal([-k1*k2*k3 + k2^2 + k3^2])
I1=I.saturation(J1)[0]
I1

A=AffineSpace(R)
X=A.subscheme(I1)
Q1,Q2=X.rational_points()
Q1.multiplicity(), Q2.multiplicity()
X.projective_closure().degree()
\end{lstlisting}
\end{code}

Computing the equation of the dimension three hypersurface for this family via elimination was beyond the capabilities of our hardware.

\section{Rational Preperiodic Structures in $\M_3$} \label{sect_M3_preperiodic}
Having completed the description of the loci $\A_3$ and $\A_4$ and some of their geometric properties, we now turn to arithmetic dynamical properties of the families that make up the automorphism loci. Specifically, we look at the possible structures of $\Q$-rational preperiodic points. The main motivation is the far-reaching and open conjecture of Morton and Silverman that the number of rational preperiodic points should be bounded independently of the particular map chosen.
\begin{conj}[Morton-Silverman \cite{Silverman7}]
    Let $f:\bbP^N \to \bbP^N$ be a morphism of degree $d \geq 2$ defined over a number field $K$ of degree $D$. Then the number of $K$-rational preperiodic points for $f$ is bounded by a constant $C$ depending only on $N$, $d$, and $D$.
\end{conj}
The best known results typically make some kind of restriction of $f$, such as good reduction at certain primes. We are more interested in results such as Poonen \cite{Poonen} or Manes \cite{Manes2} that restrict to special families. There are a number of results in this area, but we focus primarily on these two as they are closest in type to our results for the families in $\A_3$ and $\A_4$. Poonen completely classified all possible $\Q$-rational preperiodic graph structures for the family $f_c(z) = z^2+c$ assuming that there are no rational periodic points with minimal period at least $4$. This has been generalized to quadratic fields in \cite{DFK}. Manes did the same for a family of quadratic rational maps with $C_2$ automorphisms assuming there were no $\Q$-rational periodic points with minimal period at least $5$. In particular, her family was the automorphism locus $\A_2$. We proceed along the same lines as follows:
\begin{enumerate}
    \item Provide computational evidence of an upper bound on the minimal period of a $\Q$-rational preperiodic point.
    \item Analyze all possible rational preperiodic structures assuming that bound on the minimal period.
\end{enumerate}
In the case of the families with dimension in moduli space at most $1$, we are able to complete the classification of rational preperiodic graph structures, with the exception of $\A_4(D_3)$, where the classification has possibly finitely many exceptional parameters. These results make heavy use of techniques for finding rational points on curves. For the families of dimension 2 and 3, the difficulty in finding all rational points on surfaces and dimension 3 varieties is the impediment to completing those classifications. In lieu of a full classification, we examine the existence of periodic points and take a computational census of the possible rational preperiodic graph structures. As techniques in these areas improve, it would be good to return to this topic and complete those classifications.

We end this introduction with a helpful lemma for points of curves and some helpful references.
\begin{lemma}
\label{bounds on the non-singular rational points}
    Let $C$ be a projective curve defined over $\Q$. Suppose that there is a birational map defined over $\Q$ between $C$ and a smooth projective hyperelliptic curve $X$ over $\Q$. Then the nonsingular rational points on $C$ is bounded above by $|X(\Q)|$.
\end{lemma}
\begin{proof}
    Since both $C$ and $X$ are curves, and $C$ is birationally equivalent to $X$, we know that there is a $\Q$-rational isomorphism between $X$ and $\overline{C}$, the smooth projective model of $C$ \cite[Theorem 3 in Section 7.5]{Fulton}. In particular, $|\overline{C}(\Q)|=|X(\Q)|$. Furthermore, we have a surjective map from $\Bar{C}$ to $C$, so the number of non-singular rational points on $C$ is bounded above by $|\overline{C}(\Q)|=|X(\Q)|$. The only other rational points on $C$ are the singular points that could come from non-rational points on $\overline{X}$.
\end{proof}
For references on curve quotienting and blow-ups, see Lorenzini \cite{Lorenzini} and Liu \cite{Liu}.

\subsection{The Single Conjugacy Classes}
The cases $\A_3(C_4)$, $\A_3(D_4)$, and $\A_3(\fA_4)$ consist of single conjugacy classes. Using the algorithm from Hutz \cite{Hutz12}, we compute each such structure.

\begin{theorem}\label{prop_rational_single}
    We have the following rational preperiodic structures.
    \begin{itemize}
        \item For $\A_3(C_4) = \A_3(D_4)$, we represent the conjugacy class as $f(z) = \frac{1}{z^3}$. This function has rational periodic structure given by
        \begin{equation*}
            \xygraph{
!{<0cm,0cm>;<1cm,0cm>:<0cm,1cm>::}
!{(0,0) }*+{\underset{0}{\bullet}}="a"
!{(1.5,0) }*+{\underset{\infty}{\bullet}}="b"
"a":@/^1pc/"b"
"b":@/^1pc/"a"
}
\quad
    \xygraph{
!{<0cm,0cm>;<1cm,0cm>:<0cm,1cm>::}
!{(1,0) }*+{\underset{1}{\bullet}}="a"
"a":@(rd,ru)"a"
}
\quad
\xygraph{
!{<0cm,0cm>;<1cm,0cm>:<0cm,1cm>::}
!{(1,0) }*+{\underset{-1}{\bullet}}="a"
"a":@(rd,ru)"a"
}
        \end{equation*}
        \item For $\A_3(\fA_4)$, we represent the conjugacy class as $f(z) = \frac{z^3-3}{-3z^2}$. This function has rational periodic structure given by
                \begin{equation*}
            \xygraph{
!{<0cm,0cm>;<1cm,0cm>:<0cm,2cm>::}
!{(0,0) }*+{\underset{0}{\bullet}}="a"
!{(1,0) }*+{\underset{\infty}{\bullet}}="b"
"a":"b"
"b":@(rd,ru)"b"
}
        \end{equation*}
    \end{itemize}
\end{theorem}
\begin{proof}
    Direct computation via the algorithm of Hutz \cite{Hutz12} as implemented in Sage.
\end{proof}

\subsection{$\A_3(C_3)$}
We saw in Proposition \ref{prop_A3C3} that the family $f_a(z) = \frac{z^3 + a}{az^2}$ with $a \neq 0$ gives $\A_3(C_3)$.

We first examine an upper bound on the minimal period of a $\Q$-rational periodic point. We make use of Lemma \ref{lem_irreducible_curves}.
\begin{lemma}\label{lem_irreducible_curves}
If $F\in \Q[x,y]$ is irreducible over $\Q$ but is reducible over some extension field $K$ of $\Q$, and all the components are defined over $K$, then every rational point on the curve defined by $F=0$ is singular.
\end{lemma}
\begin{proof}
Suppose $F$ factors into $F=g_1g_2\dots g_n$ over $K$. Consider the set $S=\{g_1^{\sigma}:\sigma\in \Gal(K/\Q)\}$. Observe first that $g_1^{\sigma}$ has to be another component of $F$, since $F^{\sigma}=g_1^{\sigma}g_2^{\sigma}\dots g_n^{\sigma}=F$. Then the polynomial $g=\prod_{h\in S}h$ must be invariant under the Galois action and is, therefore, defined over $\Q$. Since $g\mid F$ and $F$ is irreducible, we know that $g=F$ (up to scaling). In particular, for every $g_i$ and $g_j$, we can find a $\sigma\in \Gal(K/\Q)$ such that $g_i^\sigma=g_j$. If $P$ is a rational solution to $F=0$, it is a rational solution to some $g_i=0$. But $g_j(P)=g_i^\sigma(\sigma(P))=\sigma(g_i(P))=0$, so $P$ is in fact a root of all the components of $F$ and, thus, must be in the intersection. It then follows that $P$ must be a singular point on the curve defined by $F=0$.
\end{proof}

\begin{prop}\label{prop_A3_C3_periodic}
    Let $f_a(z) = \frac{z^3+a}{az^2}$.
    \begin{enumerate}
        \item The point $\infty$ is fixed for all choices of parameter $a$. There is a second $\Q$-rational fixed point for the parameters $a =\frac{1}{1-t^3}$ for $t \in \Q \setminus \{0, 1\}$. These are the only occurring $\Q$-rational fixed points.

        \item There are no rational parameters $a$ where $f_a(z)$ has a $\Q$-rational periodic point with minimal period $2$ or $3$.
    \end{enumerate}
\end{prop}
\begin{proof}
    Clearly $\infty$ is fixed, so to look for additional fixed points, we examine the first dynatomic polynomial
    \begin{equation*}
       \Phi^{\ast}_1(f_a) = (1-a)z^3 + a,
    \end{equation*}
    and the associated dynatomic curve $\Phi^{\ast}_1(f_a) = 0$. We want values of $a$ where this curve admits rational points. Linearity in $a$ allows us to quickly solve when $z = \frac{1}{t} \in \Q$,
    \begin{equation*}
        a = \frac{z^3}{z^3-1} = \frac{1}{1-t^3}.
    \end{equation*}
    Each finite fixed point $z \in \Q$ determines a unique $a$, so we can never have more than two rational fixed points.

    To determine rational $2$-cycles, we compute the second dynatomic polynomial
    \begin{equation*}
        \Phi^{\ast}_{2}(f_a)=(a + 1)z^6 + (-a^3 + a^2 + 2a)z^3 + a^2.
    \end{equation*}
    The curve defined by $\Phi^{\ast}_2(f_a) = 0$, called the second dynatomic curve, has genus $3$; so by Faltings' theorem, it only has finitely many rational points. Computing with Magma we see that its projective closure has an order two automorphism group and the quotient curve $C$ in $\bbP^4$ has defining equations
    \begin{align*}
        &-16x_0x_1 + 65x_1x_2 - 49x_3^2=0\\                                                                               &-588800x_0^2 - 108160x_0x_2 + 12675x_2^2 + 129654x_1x_3=0
    \end{align*}
    with projection map $\phi:\overline{V(\Phi^{\ast}_2(f_a))} \to C$.
    The curve $C$ is genus one and has rational point $(0:1:0:0)$. We compute a Weierstrass model as
    \begin{equation*}
        E: y^2 - \frac{134217728}{28588707}y = x^3 - \frac{40462027902156800}{7355827511386641}.
    \end{equation*}
    The curve $E$ is rank $0$ with torsion subgroup isomorphic $\Z/3\Z$. In particular, $C$ has at most $3$ rational points. We find them through a point search as
    \begin{equation*}
        \left\{\left(\frac{13}{160}:0:1:0\right), \left(-\frac{195}{736}:0:1:0\right), \left(0:1:0:0\right)\right\}.
    \end{equation*}
    Every rational point of the second dynatomic curve must project to one of these three points under $\phi$, so we compute all possible inverse images using Sage.
    \begin{align*}
        \phi^{-1}\left(\frac{13}{160}:0:1:0\right) &= \{(0:0:1),(0:1:0)\}\\
        \phi^{-1}\left(-\frac{195}{736}:0:1:0\right) &= \{(0:0:1),(0:1:0)\}\\
        \phi^{-1}\left(0:1:0:0\right) &= \{(0:0:1),(0:1:0),(1:0:0)\}.
    \end{align*}
    In particular the $\Q$-rational points on the second dynatomic curve are
    \begin{equation*}
        \{(0:0:1),(0:1:0),(1:0:0)\}.
    \end{equation*}
    only one of these is affine and it corresponds to $a=0$ which is degenerate. Hence, there is no $\Q$-rational value of $a$ where $f_a$ has a 2-cycle with $\Q$-rational points.

    Now we look at the rational $3$-cycles. We want to determine if there are any rational points on the third dynatomic curve defined by the vanishing of
    \begin{align*}
        \Phi^{\ast}_{3,f_a}(a,z)=&(a^2 + a + 1)z^{24} + (a^6 + a^5 + a^4 + 6a^3 + 7a^2 + 8a)z^{21} \\&+ (a^{10} + a^9 + a^8 + 2a^7 + 3a^6 + 4a^5 + 15a^4 + 21a^3 + 28a^2)z^{18} \\&+ (a^{11} + 2a^{10} + 3a^9 + 2a^7 + 5a^6 + 20a^5 + 35a^4 + 56a^3)z^{15} \\&+ (a^{12} + 2a^{11} + 4a^{10} - 2a^9 - 2a^8 + 15a^6 + 35a^5 + 70a^4)z^{12} \\&+ (a^{12} + 3a^{11} - a^{10} - 3a^9 - 5a^8 + 6a^7 + 21a^6 + 56a^5)z^9 \\&+ (a^{12} - a^{10} - 4a^9 + a^8 + 7a^7 + 28a^6)z^6 + (-a^{10} + a^8 + 8a^7)z^3 + a^8.
    \end{align*}
    This curve is reducible over $\Q$, and it has the following components:
    \begin{align*}
        X_1:&a^2z^6 + az^6 + z^6 + a^2z^3 + 2az^3 + a^2=0\\
    X_2:&a^8z^{12} + a^4z^{15} + a^9z^9 + z^{18} + 2a^5z^{12} + a^{10}z^6 + 6az^{15}\\ &+ 15a^2z^{12} - 2a^7z^6 + 20a^3z^9 - a^8z^3 + 15a^4z^6 + 6a^5z^3 + a^6=0.
    \end{align*}
    Both components are irreducible over $\Q$ but reducible over some extension field of $\Q$. For curves satisfying this condition, we have Lemma \ref{lem_irreducible_curves}. Thus, any $\Q$-rational points on the curve $X_1$ must be in the intersection of the two components. In particular, any $\Q$-rational points must be singular. The only singular point on $X_1$ is $(0,0)$. We can similarly look at the other component $X_2$. It turns out that $X_2$ is reducible over the extension field $\Q(\sqrt{-3})$ and both components are defined over this extension. Thus, we can use the same argument and compute the singular points for $X_2$. The only singular point is $(0,0)$. Thus, the only $\Q$-rational point on the dynatomic curve is $(0,0)$, but we cannot have $a=0$ since $a=0$ gives us a zero in the denominator for the rational map $f_a(z)$. Therefore, no member of the family has rational $3$-cycles.
\end{proof}
\begin{code}
The code we ran are attached below:
\begin{lstlisting}[language=Python]
#Second dynatomic polynomial
R.<a> = QQ[]
P.<z> = R[]
P = FractionField(P)
A = AffineSpace(P,1)
g = DynamicalSystem_affine((z^3+a)/(a*z^2))
g.dynatomic_polynomial(2)


#Second dynatomic curve: Magma
K := Rationals();
P<z,a,h> := ProjectiveSpace(K, 2);
Phi2 := Curve(P, (a + h)*z^6 + (-a^3 + a^2*h + 2*a*h^2)*z^3*h + a^2*h^5);
G := AutomorphismGroup(Phi2);
C,prj := CurveQuotient(G);
P:=C![0,1,0,0];
E,psi:=EllipticCurve(C,P);
Rank(E);
TorsionSubgroup(E);

RationalPoints(Phi2:Bound:=1000);
RationalPoints(C:Bound:=1000);

#In sage, compute the preimages of the rational points under the map prj
R.<z,a,h>=QQ[]
f1=9/64*z^6*a^2 - 9/16*z^6*a*h - 9/64*z^6*h^2 + 27/64*z^3*a^2*h^3 - 9/16*z^3*a*h^4 - 27/64*z^3*h^5 - 27/64*a*h^7
f2=z^8
f3=45/26*z^6*a^2 - 9/65*z^6*a*h - 45/26*z^6*h^2 - 207/130*z^3*a^2*h^3 - 9/65*z^3*a*h^4 + 207/130*z^3*h^5 + 207/130*a*h^7
f4=-3/2*z^7*a - 3/2*z^7*h + 3/2*z^4*a^2*h^2 - 3/2*z^4*a*h^3 - 3/2*z^4*h^4 - 3/2*z*a*h^6

L=[[13/160,0,1,0],[-195/736,0,1,0],[0,1,0,0]]
for Q in L:
    I=R.ideal([f1*Q[1]-f2*Q[0], f1*Q[2] - f3*Q[0], f1*Q[3] - f4*Q[0], f2*Q[2]-Q[1]*f3, f2*Q[3] - Q[1]*f3, f3*Q[3]-f3*Q[2]])
    P=ProjectiveSpace(R)
    X=P.subscheme(I.gens())
    print(X.rational_points())

#Third dynatomic polynomial
R.<a> = QQ[]
P.<z> = R[]
P = FractionField(P)
A = AffineSpace(P,1)
g = DynamicalSystem_affine((z^3+a)/(a*z^2))
G = g.homogenize(1)
g.dynatomic_polynomial(3)

#Irreducible components
A.<a,z>=AffineSpace(QQ,2)
R=A.coordinate_ring()
I=R.ideal((a^2 + a + 1)*z^24 + (a^6 + a^5 + a^4 + 6*a^3 + 7*a^2 + 8*a)*z^21 + (a^10 + a^9 + a^8 + 2*a^7 + 3*a^6 + 4*a^5 + 15*a^4 + 21*a^3 + 28*a^2)*z^18 + (a^11 + 2*a^10 + 3*a^9 + 2*a^7 + 5*a^6 + 20*a^5 + 35*a^4 + 56*a^3)*z^15 + (a^12 + 2*a^11 + 4*a^10 - 2*a^9 - 2*a^8 + 15*a^6 + 35*a^5 + 70*a^4)*z^12 + (a^12 + 3*a^11 - a^10 - 3*a^9 - 5*a^8 + 6*a^7 + 21*a^6 + 56*a^5)*z^9 + (a^12 - a^10 - 4*a^9 + a^8 + 7*a^7 + 28*a^6)*z^6 + (-a^10 + a^8 + 8*a^7)*z^3 + a^8)
S=A.subscheme(I)
S.irreducible_components()

#Field of geometric reducibility and singular points
A<z,a>:=AffineSpace(Rationals(),2);
X1:=Curve(A,a^2*z^6 + a*z^6 + z^6 + a^2*z^3 + 2*a*z^3 + a^2);
K:=FieldOfGeometricIrreducibility(X1);
K;
IrreducibleComponents(ChangeRing(X1,K));
SingularPoints(X1);
X2:=Curve(A,a^8*z^12 + a^4*z^15 + a^9*z^9 + z^18 + 2*a^5*z^12 + a^10*z^6 + 6*a*z^15 + 15*a^2*z^12 - 2*a^7*z^6 + 20*a^3*z^9 - a^8*z^3 + 15*a^4*z^6 + 6*a^5*z^3 + a^6);
K:=FieldOfGeometricIrreducibility(X2);
K;
IrreducibleComponents(ChangeRing(X2,K));
SingularPoints(X2);

#examining the components
X1,X2 = S.irreducible_components()
F=X1.defining_polynomials()[0]
for b in QQ.range_by_height(3):
    if b != 0:
        K2.<v>=NumberField(F.specialization({z:b}))
        A.<z> = AffineSpace(K2,1)
        g = DynamicalSystem_affine((z^3+v)/(v*z^2)).homogenize(1)
        print(len(g.periodic_points(3)))

F=X2.defining_polynomials()[0]
for b in QQ.range_by_height(3):
    if b != 0:
        K2.<v>=NumberField(F.specialization({z:b}))
        A.<z> = AffineSpace(K2,1)
        g = DynamicalSystem_affine((z^3+v)/(v*z^2)).homogenize(1)
        print(len(g.periodic_points(3)))
\end{lstlisting}
\end{code}

\begin{remark}
    The two components $X_1$ and $X_2$ appear to be the values of $a$ where there is one $3$-cycle with points defined over the field of definition and where all nine periodic points of period three are defined over the field of definition, respectively.
\end{remark}

A search for rational preperiodic structures with the parameter up to height $10,000$ using the algorithm from \cite{Hutz12} as implemented in Sage yields no parameters where $f_a$ has a $\Q$-rational periodic point with minimal period at least $4$. So we make the following conjecture.
\begin{conj}\label{conj_rational_A3_C3}
    There is no $a \in \Q$ so that $f_a(z) =  \frac{z^3 + a}{az^2}$ has a $\Q$-rational periodic point with minimal period at least $4$.
\end{conj}
\begin{code}
\begin{lstlisting}
#code for periodic search
#one parameter families
def make_dict(f, P=[2,3,5,7,11]):
    D = dict()
    c = f.base_ring().gens()[0]
    for p in P:
        Dp = dict()
        for a in GF(p):
            try:
                fa = f.specialization({c:a}).change_ring(GF(p))
            except ValueError:
                continue
            if fa.is_morphism():
                Dp.update({a:tuple(fa.possible_periods())})
        D.update({p:Dp})
    return D

@fork
def period_search(f, start,end, period_bound, period_dict,prime_bound):
    c = f.base_ring().gens()[0]
    for a in QQ.range_by_height(start,end):
        bad = False
        try:
            fa = f.specialization({c:a})
        except ValueError:
            continue
        if fa.is_morphism():
            bad_primes = fa.primes_of_bad_reduction()
            pos_per=set()
            first_prime = True
            for p in primes(2,prime_bound):
                if p not in bad_primes:
                    ap = a%p
                    period_dict[p][ap]
                    if first_prime:
                        pos_per = set(period_dict[p][ap])
                        first_prime=False
                    else:
                        pos_per = pos_per.intersection(period_dict[p][ap])
                    if max(list(pos_per)) < period_bound:
                        break

                    #if period not in period_dict[p][ap]:
                    #    bad = True
                    #    break
            #if not bad:
            if max(list(pos_per)) >= period_bound:
                print("potential:",a)

%%time
R.<c> = QQ[]
prime_bound = 100 #29s for 100
P.<x,y> = ProjectiveSpace(R,1)
f = DynamicalSystem_projective([x^2 + c*y^2, y^2])
f = DynamicalSystem_projective([x^3 + c*y^3, c*x^2*y])
period_dict = make_dict(f, P=[p for p in primes(prime_bound+1)])
#1000 2h50m ~6Gb
for i in range(0,100):
    print(i)
    period_search(f,i*100,(i+1)*100,3,period_dict,prime_bound)
\end{lstlisting}
\end{code}
Assuming this conjecture, we are able to classify all possible $\Q$-rational preperiodic structures.

One of the curves appearing in the proof of Theorem \ref{thm_M3_C3_preperiodic} is a non-hyperelliptic genus 3 with trivial automorphism group. The standard implementations do not yield a sharp point estimate, so we treat computing its rational points in Lemma \ref{lem_M3_C3_genus_3_points}.\footnote{Thanks to Michael Stoll for detailed help with the computation and Andrew Sutherland for access to preliminary data on the analytic rank for Lemma \ref{lem_M3_C3_genus_3_points}.}
\begin{lemma}\label{lem_M3_C3_genus_3_points}
    Assuming the (weak) Birch-Swinnerton-Dyer conjecture, the curve $C \subseteq \bbP^2$ defined by
    \begin{equation*}
        C:x^2y^2 + xy^3 - x^3z - x^2yz - xy^2z + yz^3=0
    \end{equation*}
    has exactly the following six points as $\Q$-rational points.
    \begin{equation*}
        \{(1 : 1 : 1), (1 : 1 : -2), (0 : 1 : 0), (0 : 0 : 1), (1 : 0 : 0), (-1 : 1 : 0)\}.
    \end{equation*}
\end{lemma}
\begin{proof}
    We first show that the differences of the six known points form a rank 1 subgroup of the Mordell-Weil group of the Jacobian of $C$, denoted $J$. This is adapted directly from Michael Stoll's Magma code for computing $\Q$-rational 6-cycles \cite{Stoll08}. We know that prime-to-$p$ torsion in $J(\Q)$ injects into $J(\F_p)$ for primes of good reduction. Magma computes
    \begin{equation*}
        \#J(\F_5) = 3 \cdot 79 \qquad \#J(\F_7) = 7 \cdot 83.
    \end{equation*}
    We conclude $J(\Q)$ has trivial torsion. To show the rank assertion, we use the homomorphism
    \begin{equation*}
        \Phi_S: \bigoplus_{i=0}^5 \Z P_i \to \Pic_C \to \prod_{p \in S}\Pic_{C/\F_p},
    \end{equation*}
    where $S$ is the set of primes of good reduction. We take $S= \{2,5,7,11,13\}$ and compute that the kernel of $\Phi_S$ is a subgroup of rank $5$ in $\Z^6 = \oplus\Z P_i$. We apply LLL and find that there are (at least) four relations among the points; this gives an upper bound of the rank as $1$. However, looking at the image of $\Phi_S$, we see that the degree $0$ subgroup of $\Z^6$ surjects onto $\Z/5\Z$; and since there is no torsion, the rank must be at least $1$. Hence, the rank is exactly $1$.

    This is just the rank of the subgroup supported on the known points, so is not a conclusive rank calculation. However, Andrew Sutherland in private communication has calculated the analytic rank as $1$ assuming that the $L$-function lies in the (polynomial) Selberg class, which is implied by the Hasse-Weil conjecture. Assuming BSD (and Hasse-Weil), this is a conclusive rank calculation.

    The curve $C$ is genus $3$ and we assume that the rank of $J(\Q)$ is $1$, so we can apply methods of Chabauty to show that these six points are the only $\Q$-rational points on $C$ \cite{BPS}.\footnote{Thanks to Michael Stoll for sharing the details of this computation.} Recall that there is a pairing
    \begin{align*}
        \Omega^1_J(\Q_p) \times J(\Q_P) &\to \Q_p\\
        (\omega, Q) \mapsto \int_0^Q \omega
    \end{align*}
    that induces a perfect $\Q_p$-bilinear pairing
    \begin{equation*}
        \Omega^1_J(\Q_p) \times J(\Q_p)^1 \otimes_{\Z_p} \Q_p \to \Q_p,
    \end{equation*}
    where $J(\Q_p)^1$ denotes the kernel of reduction. If $G \subset J(\Q_p)$ is a subgroup of rank less than $\dim(J) = 3$, then there is a nonzero differential $\omega$ that kills $G$ under this pairing. We apply this with $p=2$ and $G$ the subgroup generated by the known rational points. Fix a basis of regular differentials
    \begin{equation*}
        w_0 = \frac{x(zdx - xdz)}{F_y} \quad
        w_1 = \frac{y(zdx - xdz)}{F_y} \quad
        w_2 = \frac{z(zdx - xdz)}{F_y},
    \end{equation*}
    where $F(x,y,z)$ is the defining polynomial of $C$ and $F_y$ the partial derivative with respect to $y$; see, for example, \cite[Corollary to Theorem 1 p. 634]{BK}. For a given point $P$, find a uniformizing parameter of $C$ at $P$ that is also a uniformizer at $P$ modulo $2$. We find a basis of differentials in terms of the uniformizer that annihilate the known rational points. Looking at the degree of vanishing, we can determine whether one or two rational points lie above each of the points modulo $2$.

    The points $(-1:1:0)$ and $(1:1:-2)$ are in the same residue class. We calculate the basis of annihilating differentials as
    \begin{align*}
        &1 + t + t^2 + t^4 + t^5 + t^8 + t^{10} + t^{11} + t^{13} + t^{15} + t^{16} + O(t^{20})\\
        &1 + t + t^3 + t^4 + t^5 + t^8 + t^9 + t^{11} + t^{12} + t^{16} + t^{17} + O(t^{20})
    \end{align*}
    Since there is a non-vanishing constant term, there are at most two rational points in this residue class modulo $2$ \cite[Proposition 6.3]{Stoll06}.

    For the remain four points, we get the following four bases of annihilating differentials.
    \begin{align*}
        &\begin{cases}
        1 + t + t^3 + t^4 + t^5 + t^6 + t^{13} + t^{14} + t^{17} + t^{19} + O(t^{20})\\
        t + t^3 + t^6 + t^7 + t^8 + t^9 + t^{10} + t^{11} + t^{12} + t^{13} + t^{18} + O(t^{20})
        \end{cases}\\
    &\begin{cases}
    t^3 + t^5 + t^6 + t^9 + t^{10} + t^{12} + t^{15} + t^{18} +  O(t^{20})\\
    1 + t^7 + t^8 + t^{11} + t^{12} + t^{13} + t^{16} + t^{19} + O(t^{20})
    \end{cases}\\
    &\begin{cases}
    t + t^3 + t^5 + t^7 + t^8 + t^9 + t^{11} + t^{13} + t^{18} + t^{19} + O(t^{20})\\
    1 + t^2 + t^3 + t^4 + t^6 + t^7 + t^{10} + t^{11} + t^{12} + t^{15} + t^{17} + t^{19} +  O(t^{20})
    \end{cases}\\
    &\begin{cases}
    1 + t^2 + t^3 + t^6 + t^7 + t^9 + t^{10} + t^{12} + t^{15} + t^{16} + t^{17} + t^{18} + O(t^{20})\\
    t + t^2 + t^5 + t^7 + t^8 + t^{13} + t^{14} + t^{15} + t^{18} + t^{19} + O(t^{20})
    \end{cases}
    \end{align*}
    For each point there is an element such that the constant term is nonzero and the linear term is zero. By a standard Newton polygon argument; see \cite[IV.4]{Koblitz}, this implies that the corresponding logarithm has at most one zero on the residue disk of the point, so there is at most one rational point in the disk.

    Consequently, there are at most six rational points on the curve.
\end{proof}

\begin{code}
\begin{lstlisting}
# Magma
//We get the differentials from the adjoint Curve of degree d-3=1
P<x,y,z>:=ProjectiveSpace(Rationals(),2);
C:=Curve(P,[x^2*y^2 + x*y^3 - x^3*z - x^2*y*z - x*y^2*z + y*z^3]);
Adjoints(C,1);

// We first compute the class groups of the reduction mod 2,5,7,11,13.
Pr3<x,y,z> := ProjectiveSpace(Rationals(), 2);
C:=-y^2*x^2 - y*x^3 + y^3*z + y^2*x*z + y*x^2*z - x*z^3;
ptsC:=Points(Curve(Pr3,C):Bound:=1000);
primes := [2, 5, 7, 11, 13];
Cred := [BaseChange(Curve(Pr3,C), Bang(Rationals(), GF(p))) : p in primes];
Clgps := [<G, m> where G, m := ClassGroup(Cr) : Cr in Cred];
// Check that there is no rational torsion
Clnums := [&*[i : i in Invariants(pair[1]) | i ne 0] : pair in Clgps];
 [Factorization(Clnums[2]), Factorization(Clnums[3])];
// Set up the homomorphism Phi_S
// Get integral coordinates of P_0, ..., P_5
ptseqs := [Eltseq(ptsC[i+1]) : i in [0..5]];
ptseqs[2] := [2*a : a in ptseqs[2]]; ptseqs;
ptseqs := [ChangeUniverse(s, Integers()) : s in ptseqs];
prodG, incls, projs := DirectProduct([pair[1] : pair in Clgps]);
Z6 := FreeAbelianGroup(6);
homs := [hom<Z6 -> Clgps[i,1] | [Divisor(Place(Cred[i]!pt)) @@
Clgps[i,2] : pt in ptseqs]>  : i in [1..#Clgps]];
PhiS := hom<Z6 -> prodG | [&+[incls[i](homs[i](g)) : i in [1..#incls]] : g in OrderedGenerators(Z6)]>;
Invariants(Image(PhiS));
// Since we know there is no torsion, this implies that the rank of
// our subgroup of J(Q) is at least 1
Ker := Kernel(PhiS);
// Find small relations
L := Lattice(Matrix([Eltseq(Z6!k) : k in OrderedGenerators(Ker)]));
BL := Basis(LLL(L)); BL;
// We see 4 small elements in the kernel.
// Check that they correspond to principal divisors.
places := [Place(ptsC[i+1]) : i in [0..5]];
divisors := [&+[BL[i,j]*places[j] : j in [1..#places]] : i in [1..4]];
assert forall{d : d in divisors | IsPrincipal(d)};

//The following code is by Michael Stoll
// Give a conditional proof that the plane quartic curve
//  C : x^2 y^2 + x y^3 - x^3 z - x^2 y z - x y^2 z + y z^3 = 0
// has exactly six rational points.
// The result assumes that the Mordell-Weil rank of C's Jaobian is 1.

// Set up the projective plane over Q
Pr2<x,y,z> := ProjectiveSpace(Rationals(), 2);

// and define the curve.
pol := x^2*y^2 + x*y^3 - x^3*z - x^2*y*z - x*y^2*z + y*z^3;
C := Curve(Pr2, pol);

// These are the known rational points:
pts := [[1,0,0], [0,1,0], [0,0,1], [1,1,1], [-1,1,0], [1,1,-2]];
ptsC := ChangeUniverse(pts, C);
// their differences all lie in a rank-1 subgroup.

// We want to prove that these are all the rational points
// assuming that the Mordell-Weil rank is 1.
// (Then J(Q) = Z; there is no torsion.)

// We work 2-adically.
// The points (-1:1:0) and (1:1:-2) are in the same
// residue disk mod 2, so their difference is in
// the kernel of reduction.

// We fix a basis of the regular differentials:
// (omx, omy, omz) = (x, y, z)*((z dx - x dz)/F_y(x,y,z))
// where F = pol is the defining polynomial of C
// and F_y denotes the partial derivative w.r.t. y.

// Set up the power series ring over Q.
Pws<t> := LaurentSeriesAlgebra(Rationals());

function ombas(trip)
  xs, ys, zs := Explode(trip);
  // xs, ys, zs are coordinates as power series in a uniformizer t
  // returns [omxs, omys, omzs] s.t. omx = omxs(t) dt etc.
  om0 := (zs*Derivative(xs) - xs*Derivative(zs))
           / Evaluate(Derivative(pol, 2), [xs, ys, zs]);
  return [xs*om0, ys*om0, zs*om0];
end function;

// For a given point P, find a uniformizer
// (we can take a function like  x/y - x(P)/y(P) )
// and return (xs, ys, zs) expressing the coordinates in terms of it.
function uniformizer(P)
  // pick a coordinate that is = 1
  // (possible for our rational points)
  i := 1; while P[i] ne 1 do i +:= 1; end while;
  // find non-vanishing partial derivative at P
  j := 1; while j eq i or Evaluate(Derivative(pol, j), Eltseq(P)) eq 0 do j +:= 1; end while;
  // we set i-th coordinate = 1, k-th coordinate - P[k] = uniformizer
  // where {i,j,k} = {1,2,3}
  k := 6 - i - j;
  result := [Pws | 1, 1, 1];
  result[k] := P[k] + t;
  // solve for the remaining coordinate
  PPws := PolynomialRing(Pws);
  subs := ChangeUniverse(result, PPws);
  subs[j] := PPws.1;
  rts := Roots(Evaluate(pol, subs));
  // There may be additional roots;
  // we need the one that is regular and takes the correct value at 0.
  rts := [r[1] : r in rts | Valuation(r[1]) ge 0 and Coefficient(r[1], 0) eq P[j]];
  assert #rts eq 1; // there should be exactly one such root
  result[j] := rts[1];
  return result;
end function;

// Set up Q_2.
Q2 := pAdicField(2);

// ptsC[6] is the point with t = -2 in the residue class of ptsC[5]
triples := [uniformizer(P) : P in pts[1..5]];
assert pts[6] eq [Evaluate(ser, Q2!-2) : ser in triples[5]];

// Integrate the differentials from ptsC[5] to ptsC[6].
omspt5 := ombas(triples[5]);
logspt5 := [Pws| Integral(om) : om in omspt5]; // formal integrals
// the precision (O(t^20) for the series, O(2^20) for 2-adics) is sufficient.
values := [Evaluate(log, Q2!-2) : log in logspt5];

// Find basis of differentials that annihilate the MW group.
kermat := KernelMatrix(Matrix([[v] : v in values]));

// Find the image in differentials over F_2.
Z2 := Integers(Q2);
kermat := ChangeRing(kermat, Z2);
F2, red := ResidueClassField(Z2);
kermatmod2 := ChangeRing(kermat, F2);
// mod 2, the differentials correspond to the lines x = 0 and y+z = 0.
PwsF2<T> := LaurentSeriesAlgebra(F2);
omspt5mod2 := ChangeUniverse(omspt5, PwsF2);
annpt5 := [&+[kermatmod2[i,j]*omspt5mod2[j] : j in [1..3]] : i in [1..2]];

// Check that one of the annihilating differentials has non-vanishing
// constant term. This implies that there are at most two rational points
// in the same residue class mod 2 as ptsC[5].
// Since we know two (ptsC[5] and ptsC[6]), there are then exactly two.
assert exists{om : om in annpt5 | Coefficient(om, 0) ne 0};

// Now show that in all other residue classes, there is at most
// one rational point.

// First we check that the known rational points surject onto C(F_2).
CF2 := BaseChange(C, Bang(Rationals(), F2));
assert Set(Points(CF2)) eq {CF2| pt : pt in pts};

// Now for each point ptsC[i] with 1 <= i <= 4,
// consider the annihilating differentials mod 2 in terms of
// a local uniformizer
// and show that there is always one such that the constant
// term is nonzero and the linear term is zero.
// This implies that the corresponding logarithm has at most
// one zero on the residue disk of the point, so there is at
// most one rational point in the disk (and we know one, of course).
for i := 1 to 4 do
  // get basis of differentials in terms of local uniformizer
  oms := ombas(triples[i]);
  // consider them mod 2
  oms := ChangeUniverse(oms, PwsF2);
  // take appropriate linear combinations to get annihilating differentials
  ann := [&+[kermatmod2[i,j]*oms[j] : j in [1..3]] : i in [1..2]];
  Append(~ann, &+ann); // the third nonzero differential is the sum
  assert exists{om : om in ann | Coefficient(om, 0) ne 0 and Coefficient(om, 1) eq 0};
  print ann;
end for;

// This finishes the proof.

// It remains to verify that the Mordell-Weil rank is really 1.

\end{lstlisting}
\end{code}
\begin{theorem}\label{thm_M3_C3_preperiodic}
    Let $f \in \A_3(C_3)$. Assuming the (weak) Birch-Swinnerton-Dyer conjecture, if $f_a(z) = \frac{z^3+a}{az^2}$ does not have a $\Q$-rational periodic point of period at least $4$, then the $\Q$-rational preperiodic structure for $a \in \Q$ is one of the following;
    \begin{align*}
        G_1 &:=
            \xygraph{
        !{<0cm,0cm>;<1cm,0cm>:<0cm,1cm>::}
        !{(0,0) }*+{\underset{0}{\bullet}}="a"
        !{(1,0) }*+{\underset{\infty}{\bullet}}="b"
        !{(3,0) }*+{\bullet}="c"
        !{(2,0) }*+{\bullet}="d"
        !{(2,1) }*+{\bullet}="e"
        "a":"b"
        "b":@(rd,ru)"b"
        "c":@(rd,ru)"c"
        "d":"c"
        "e":"c"
        }, \quad \text{\parbox[t]{9cm}{$a = \frac{1}{1-t^3}, \; t = \frac{1}{xy - xz} \text{ for } (x:y:z) \in \{n(1:0:1) : n \geq 2 \text{ for addition of points on } y^2-y = x^3-1 \}$}}\\
        G_2 &:= \xygraph{
        !{<0cm,0cm>;<1cm,0cm>:<0cm,1cm>::}
        !{(0,0) }*+{\underset{0}{\bullet}}="a"
        !{(1,0) }*+{\underset{\infty}{\bullet}}="b"
        !{(2,0) }*+{\bullet}="c"
        "a":"b"
        "b":@(rd,ru)"b"
        "c":@(rd,ru)"c"
        }, \quad a = \frac{1}{1-t^3},\; t \in \Q\setminus\{0,1\} \text{ and not $G_1$}\\
        G_3 &:=
            \xygraph{
        !{<0cm,0cm>;<1cm,0cm>:<0cm,1cm>::}
        !{(0,0) }*+{\bullet}="a"
        !{(1,0) }*+{\underset{0}{\bullet}}="b"
        !{(2,0) }*+{\underset{\infty}{\bullet}}="c"
        "a":"b"
        "b":"c"
        "c":@(rd,ru)"c"
        }, \quad a = t^3, \; t \in \Q \setminus\{0\}\\
        G_4 &:=
            \xygraph{
        !{<0cm,0cm>;<1cm,0cm>:<0cm,1cm>::}
        !{(0,0) }*+{\underset{0}{\bullet}}="b"
        !{(1,0) }*+{\underset{\infty}{\bullet}}="c"
        "b":"c"
        "c":@(rd,ru)"c"
        }, \quad \text{all other parameters $a$.}
    \end{align*}
\end{theorem}
\begin{proof}
    Proposition \ref{prop_A3_C3_periodic} describes the periodic points, so we need only consider the strictly preperiodic points.

    The point at infinity denoted $\infty$ is fixed for $f_a$ for every $a$. Its only non-periodic preimage is $0$. So we always have at least the preperiodic structure
    \begin{equation*}
        \xygraph{
    !{<0cm,0cm>;<1cm,0cm>:<0cm,1cm>::}
    !{(0,0) }*+{\underset{0}{\bullet}}="a"
    !{(1,0) }*+{\underset{\infty}{\bullet}}="b"
    "a":"b"
    "b":@(rd,ru)"b"
    }
    \end{equation*}

We can also look for preimages of $0$. Solutions to $z^3 + a = 0$ for rational $a$ must look like $a = l^3$ for some $l \in \Q$. Of the three possible preimages of $0$, only one can ever be rational. This describes parameter values whose preperiodic structures contain graphs $G_2$ and $G_3$. 

To have both a second fixed point and a second non-periodic preimage of $\infty$ (i.e., the union of $G_2$ and $G_3$), we need $a = \frac{1}{1-t^3}$ and $a = l^3$ for $t,l \in \Q$. Setting these two equal, the problem comes down to finding rational points on the (affine) curve
\begin{equation}\label{eq3}
    l^3(1-t^3)-1=0.
\end{equation}
This curve has genus 1 and a rational point at $(t,l) = (0,1)$, which allows us to find an isomorphism to an elliptic curve using Magma.
\begin{code}
\begin{lstlisting}[language=Python]
K := Rationals();
A<t,l> := AffineSpace(K,2);
C := ProjectiveClosure(Curve(A, l^3*(1-t^3)-1));
p := C![0,1,1];
E, psi := EllipticCurve(C,p);
Rank(E);
TorsionSubgroup(E);
\end{lstlisting}
\end{code}
We get the Weierstrass model $y^2 -9y = x^3 -27$, which has rank 0 and whose torsion subgroup is isomorphic to $\Z/3\Z$. The rational points of the projective closure are $(t:l:r) = (0:1:0), (0:1:1), (1:0:0)$. The two points $(0:1:0)$ and $(1:0:0)$ are singular. When we blow them up they both correspond to one rational point. Since we only have three rational points on the elliptic curve, there are only three on our original curve.
\begin{code}
\begin{lstlisting}[language=Python]
// We search for the rational points using
RationalPoints(C:Bound:=1000);

// We evaluate the singular points with
A<l,t>:=AffineSpace(Rationals(),2);
Y:=Curve(A,l^3*(1-t^3)-1);
Y:=ProjectiveClosure(Y);
Points(Y:Bound:=1000);
sing:=SingularPoints(Y);
P1:=sing[1];
P2:=sing[2];
P1;
Q1:=Places(P1);
Q1;
Degree(Q1[1]);
Degree(Q1[2]);
P2;
Q2:=Places(P2);
Q2;
Degree(Q2[1]);
Degree(Q2[2]);
\end{lstlisting}
\end{code}
Only one of these three rational points is not at infinity. This point corresponds to $a = 1$ and gives $\Phi^{\ast}_1(f)(z) = 1$ which is never 0. Thus, we can never have both a second fixed point and a preimage of 0. Note that this eliminates the union of $G_1$ and $G_3$ as well.

We can also look at when the preimage of 0 itself has a preimage. We already know $a = t^3$ for some $t \in \Q$, so we need $\frac{z^3 + t^3}{t^3z^2} = -t$. This determines the curve
$$z^3 + z^2t^4 + t^3 = 0,$$
which is a singular genus 3 non-hyperelliptic curve. It has a $C_3$ automorphism given by $(z,t) \mapsto (\omega z, \omega t)$, where $\omega$ is a cube root of unity. Quotienting out by this action in Magma gives us all of $\bbP^1$, so a different analysis is necessary. We can also perform the change of coordinates $(z,t) \mapsto (u,v) = (\frac{z}{t}, z^2t)$ to get $v^3 + u + 1 = 0$ - this is exactly $\bbP^1$ as expected.

The first birational transformation we perform is sending the coordinates $(z,t)$ to $(x,y) = (\frac{t}{z}, \frac{t^2}{z})$, which gives us $y^3 = -x^5-x^2$. The simple change $y \mapsto -y$ allows us to write the nicer version
\begin{equation*}
    y^3 = x^5+x^2.
\end{equation*}
We see that this is now a superelliptic curve, but in a singular model so still not desirable to work with. Another birational transformation $(x,y) \mapsto (u,v) = (\frac{1}{x}, \frac{y}{x^2})$ gives
$$v^3 = u^4 + u,$$
a nonsingular Picard curve.
Magma computes that this curve has rank 0, and since this number is strictly less than 3 (its genus), we can use the Chabauty-Coleman \cite{Chabauty,Coleman} method to find all its rational points. We can check that $p=2$ is a prime of good reduction (define the curve in Magma over $\F_2$ and check if it is singular). Code written by Jan Tuitman and Jennifer Balakrishnan \cite{BT-coleman} verifies that, in fact, there can be at most three rational points on the curve.
\begin{code}
\begin{lstlisting}[language=Python]
load "coleman.m";
f:=x^4+x;
p:=5;
N:=30;
data:=coleman_data(y^3-f,p,N);
Qpoints:=Q_points(data,1000);
L,v:=effective_chabauty(data:bound:=1000,e:=50);
\end{lstlisting}
\end{code}
Thus, there are only three points on our original curve, which Magma finds to be $(0 : 1 : 0), (0 : 0 : 1), (1 : 0 : 0)$. The only affine point is $(0,0)$, so  $z = 0, t= 0$ is the only solution. However, these values send the expression $\frac{z^3 + t^3}{t^3z^2}$ to infinity, not 0, so it is not a valid solution. Thus, the rational preimage of 0 cannot have a rational preimage.

We next see if the second fixed point can have non-periodic rational preimages. To simplify the equations instead of the parameterization $a = \frac{1}{1-t^3}$, $t \neq 0,1$, for the existence of a second fixed point, we replace $t$ with $\frac{1}{t}$ to have $a= \frac{t^3}{t^3-1}$, $t \neq 0,1$ with (affine) fixed point $t$.
We know that
\begin{equation*}
    f_a(z) = t,
\end{equation*}
which produces equation
\begin{equation}\label{eq5}
    \frac{1}{t^3-1} (z^3(t^3-1) + -t^4z^2 + t^3)=0.
\end{equation}
The variety defined by equation \eqref{eq5} has two irreducible components: the case where the preimage is the fixed point itself, $z=t$, and the genus 1 curve defined by
\begin{equation*}
    C:z^2t^3 - z^2 - zt - t^2=0.
\end{equation*}
Using the point $(0:1:0)$ as the point at infinity, we get the model
\begin{equation}
    y^2 - y = x^3 - 1.
\end{equation}
It is rank 1 with generator $(1:0:1)$ with trivial torsion. We have a mapping
\begin{align*}
    \psi&:E \to C\\
    \psi(x,y,z) &= (xz,xy - xz, yz - z^2).
\end{align*}
So every rational point on the elliptic curve corresponds to a rational point on the curve $C$. The image of $(1:0:1)$ gives $t=1$, which is the degenerate case, but the other rational points correspond to parameters $a$ where the additional fixed point has a non-periodic rational preimage. Furthermore, since the curve $C$ defining the pair $(z,t)$ is degree $2$ in $z$, if there is one rational preimage of the fixed point, then there is typically two rational preimages of the fixed point. The exceptions are obtained by taking the discrminant of the defining equation of $C$ in $\Q[t][z]$. This discriminant is
\begin{equation*}
    -27\cdot t^2\left(t^8 + \frac{14}{27}t^4 + \frac{1}{9}\right).
\end{equation*}
So the only rational $t$ value where there could be a single preimage is $t=0$, which is degenerate. So we have structure $G_1$.

\begin{code}
\begin{lstlisting}
R.<z,t>=QQ[]
A=AffineSpace(R)
I=R.ideal([z^3*(t^3 - 1) + (-t^4*z^2 + t^3)])
X=A.subscheme(I)
X.irreducible_components()

K := Rationals();
P<z,t,h> := ProjectiveSpace(K,2);
f := z^2*t^3 - z^2*h^3 - z*t*h^3 - t^2*h^3;
D := Curve(P, [f]);
E,psi:=EllipticCurve(D,D![0,1,0]);
IsInvertible(psi);

#generating parameters
r.<x,y,z>=QQ[]
PP=ProjectiveSpace(r)
E = EllipticCurve([0, 0, -1, 0, -1])
P=E(1,0)
f(x,y,z) = (x*z,x*y - x*z,y*z - z^2)
for n in range(1,6):
    print(n*P)
    X,Y,Z=list(n*P)
    print(PP(list(f(X,Y,Z))))

#discriminant
R.<t>=QQ[]
S.<z>=R[]
f=z^3*t^3-z^2-z*t-t^2
f.discriminant().factor()
\end{lstlisting}
\end{code}

Now we check if structure $G_1$ can be extended by the finite fixed point having a second rational preimage. We continue to utilize the parameterization $a = \frac{t^3}{t^3-1}$, so we are looking for a rational point on a component of the curve defined by $f^2_a(z) = t$
\begin{align*}
    ((t^{12} &- 4t^9 + 6t^6 - 4t^3 + 1)z^9 + (4t^{12} - 9t^9 + 9t^6 - 3t^3)z^6 + (3t^{12} - 6t^9 + 3t^6)z^3 + t^{12} - t^9)\\
    &=t((t^{12} - 2t^9 + t^6)z^8 + (2t^{12} - 2t^9)z^5 + t^{12}z^2).
\end{align*}
After saturation by the ideal $(t,z)$, the degenerate case, this curve has three irreducible components. Two of these components correspond to fixed points and preperiodic points with period (1,1), respectively,
\begin{align*}
    z&=t\\
    t^3z^2 -z^2 -tz - t^2&=0.
\end{align*}
The third component represents the preperiodic points with period $(2,1)$ and is an irreducible genus 13 curve defined by
\begin{equation*}
    z^6t^9 - 3z^6t^6 - z^5t^7 - z^4t^8 + 2z^3t^9 + 3z^6t^3 + z^5t^4 - 4z^3t^6 - z^2t^7 + t^9 - z^6 + 2z^3t^3 - t^6=0.
\end{equation*}
Utilizing Magma, we quotient by the automorphism $(z,t) \mapsto (\zeta_3 z, \zeta_3 t)$, where $\zeta_3$ is a primitive third root of unity. This results in the genus 3 curve defined by
\begin{equation*}
    -x^3y + y^4 + 5x^3 - 2x^2y + xy^2 - 19y^3 + 11x^2 - 9xy + 136y^2 + 23x - 433y + 517=0.
\end{equation*}
Simplifying the (projective closure) equation with a change of variables defined by the $\SL_2(\Z)$ element $m=\begin{pmatrix} 1& -1&  1\\ 0 & 5& -4\\ 0 & 1 &-1\end{pmatrix}$, we get the reduced equation
\begin{equation*}
    x^2y^2 + xy^3 - x^3 - x^2y - xy^2 + y=0.
\end{equation*}
Lemma \ref{lem_M3_C3_genus_3_points} calculates the $\Q$-rational points of the projective closure of this curve as
\begin{equation*}
    \{(1 : 1 : 1), (-1/2 : -1/2 : 1), (0 : 1 : 0), (0 : 0 : 1), (1 : 0 : 0), (-1 : 1 : 0)\}.
\end{equation*}
These give the six points on the nonreduced curve
\begin{equation*}
    \{(1: 1: 0), (1: -13/2: -3/2), (1: -4: -1), (-1: 5: 1), (1: 0: 0), (-2: 5: 1)\}.
\end{equation*}
On the original curve, we find the four points
\begin{equation*}
    \{ (0 : 1 : 0), (0 : 0 : 1), (0 : 1 : 1), (1 : 0 : 0) \}.
\end{equation*}
Computing the inverse image of the quotient map of the six points on the nonreduced curve, we find only the four rational points on the original curve already known. So these are the only four rational points on the original curve. These are all either points at infinity or degenerate cases, so there are no non-periodic second rational preimages of the finite fixed point.

\begin{code}
\begin{lstlisting}
#Check that we get a rank 2 subgroup from the 6 known rational points and
#and that the jacobian as no torsion

// adapted from Stoll rational 6 cycle calculatioins:
// The next task is to prove Lemma 4.
// We first compute the class groups of the reduction mod 3,5,7,11,13.
Pr3<x,y,z> := ProjectiveSpace(Rationals(), 2);
Xpr3:=-y^2*x^2 - y*x^3 + y^3*z + y^2*x*z + y*x^2*z - x*z^3;
ptsXpr3:=Points(Curve(Pr3,Xpr3):Bound:=1000);

// These are points no. 8,7,3,5,0,4,9,2,6,1, in this order.
perm := [1, 2, 3, 4, 5, 6]; // ptsXpr11[perm[i+1]] = P_i

primes := [5, 7, 11, 13];
Xred := [BaseChange(Curve(Pr3,Xpr3), Bang(Rationals(), GF(p))) : p in primes];
Clgps := [<G, m> where G, m := ClassGroup(Xr) : Xr in Xred];
// Check that there is no rational torsion
Clnums := [&*[i : i in Invariants(pair[1]) | i ne 0] : pair in Clgps];
[Factorization(Clnums[1]), Factorization(Clnums[2])];
// Set up the homomorphism Phi_S
// Get integral coordinates of P_0, ..., P_5
ptseqs := [Eltseq(ptsXpr3[perm[i+1]]) : i in [0..5]];
ptseqs[2] := [2*a : a in ptseqs[2]]; ptseqs;
ptseqs := [ChangeUniverse(s, Integers()) : s in ptseqs];
prodG, incls, projs := DirectProduct([pair[1] : pair in Clgps]);
Z10 := FreeAbelianGroup(6);
homs := [hom<Z10 -> Clgps[i,1] | [Divisor(Place(Xred[i]!pt)) @@ Clgps[i,2] : pt in ptseqs]>  : i in [1..#Clgps]];
PhiS := hom<Z10 -> prodG | [&+[incls[i](homs[i](g)) : i in [1..#incls]] : g in OrderedGenerators(Z10)]>;
Invariants(Image(PhiS));
// Since we know there is no torsion, this implies that the rank of
// our subgroup of J(Q) is at least 3 (or the rank of the Picard group
// of C at least 4).
Ker := Kernel(PhiS);
// Find small relations
L := Lattice(Matrix([Eltseq(Z10!k) : k in OrderedGenerators(Ker)]));
BL := Basis(LLL(L)); BL;
// We see 4 small elements in the kernel.
// Check that they correspond to principal divisors.
places := [Place(ptsXpr3[i+1]) : i in [0..5]];
divisors := [&+[BL[i,j]*places[j] : j in [1..#places]] : i in [1..4]];
assert forall{d : d in divisors | IsPrincipal(d)};



load "coleman.m";
Q:= -x^5 - x^3*y + x^2*y^2 - x*y^2 + y^3 + x*y;
data:=coleman_data(Q,7,15);
Qpoints:=Q_points(data,1000);
L,v:=effective_chabauty(data:Qpoints:=Qpoints,e:=30);


load "coleman.m";
Q:=-y^2*x^2 - y*x^3 + y^3 + y^2*x + y*x^2 - x;
data:=coleman_data(Q,2,10);
Qpoints:=Q_points(data,1000);
L,v:=effective_chabauty(data:Qpoints:=Qpoints,e:=30);


A2<x,y> := AffineSpace(Rationals(),2);
f := -x^5 - x^3*y + x^2*y^2 - x*y^2 + y^3 + x*y;
C := Curve(A2,[f]);
D:=ProjectiveClosure(C);
Points(D:Bound:=10000);
\end{lstlisting}
\end{code}

\begin{code}
\begin{lstlisting}
#(2,1) points
R.<t>=QQ[]
A.<z>=AffineSpace(R,1)
a=t^3/(t^3-1)
f=DynamicalSystem_affine((z^3+a)/(a*z^2), domain=A)
F2=f.nth_iterate_map(2)
f2= F2[0].numerator() - t*F2[0].denominator()

from sage.rings.polynomial.flatten import FlatteningMorphism
phi=FlatteningMorphism(f2.parent())
R = phi.codomain()

A=AffineSpace(R)
I=R.ideal([phi(f2)])
X=A.subscheme(I)
J=I.saturation(R.ideal(R.gens()))[0]
X=A.subscheme(J)
X.irreducible_components()

K := Rationals();
P<z,t,h> := ProjectiveSpace(K,2);
f := z^6*t^9 - 3*z^6*t^6*h^3 - z^5*t^7*h^3 - z^4*t^8*h^3 + 2*z^3*t^9*h^3 + 3*z^6*t^3*h^6 + z^5*t^4*h^6 - 4*z^3*t^6*h^6 - z^2*t^7*h^6 + t^9*h^6 - z^6*h^9 + 2*z^3*t^3*h^9 - t^6*h^9;
D := Curve(P, [f]);
Genus(D);
RationalPoints(D:Bound:=1000);

K:=CyclotomicField(3);
P2<z,t,h> := ProjectiveSpace(K,2);
f := z^6*t^9 - 3*z^6*t^6*h^3 - z^5*t^7*h^3 - z^4*t^8*h^3 + 2*z^3*t^9*h^3 + 3*z^6*t^3*h^6 + z^5*t^4*h^6 - 4*z^3*t^6*h^6 - z^2*t^7*h^6 + t^9*h^6 - z^6*h^9 + 2*z^3*t^3*h^9 - t^6*h^9;
C := Curve(P2,[f]);
phi := iso<C->C|[K.1*z,K.1*t,h],[K.1^2*z,K.1^2*t,h]>;
// we will take the quotient by phi
G := AutomorphismGroup(C,[phi]);
CG,prj := CurveQuotient(G);

P2<z,t,h> := ProjectiveSpace(K,2);
f := z^6*t^9 - 3*z^6*t^6*h^3 - z^5*t^7*h^3 - z^4*t^8*h^3 + 2*z^3*t^9*h^3 + 3*z^6*t^3*h^6 + z^5*t^4*h^6 - 4*z^3*t^6*h^6 - z^2*t^7*h^6 + t^9*h^6 - z^6*h^9 + 2*z^3*t^3*h^9 - t^6*h^9;
C := Curve(P2,[f]);
phi := iso<C->C|[K.1*z,K.1*t,h],[K.1^2*z,K.1^2*t,h]>;
// we will take the quotient by phi
G := AutomorphismGroup(C,[phi]);
CG,prj := CurveQuotient(G);

P2<x,y,z> := ProjectiveSpace(Rationals(),2);
f := -x^3*y + y^4 + 5*x^3*z - 2*x^2*y*z + x*y^2*z - 19*y^3*z + 11*x^2*z^2 - 9*x*y*z^2 + 136*y^2*z^2 + 23*x*z^3 - 433*y*z^3 + 517*z^4;
C := Curve(P2,[f]);
D,m:=ReducePlaneCurve(C);


m=matrix(QQ,3,3,[1,-1,1,0,5,-4,0,1,-1])
L=[[1,1,1],[-1/2,-1/2,1],[0,0,1],[0,1,0],[1,0,0],[-1,1,0]]
for p in L:
    print((m*vector(p)))

K := Rationals();
P<z,t,h> := ProjectiveSpace(K,2);
f := z^6*t^9 - 3*z^6*t^6*h^3 - z^5*t^7*h^3 - z^4*t^8*h^3 + 2*z^3*t^9*h^3 + 3*z^6*t^3*h^6 + z^5*t^4*h^6 - 4*z^3*t^6*h^6 - z^2*t^7*h^6 + t^9*h^6 - z^6*h^9 + 2*z^3*t^3*h^9 - t^6*h^9;
D := Curve(P, [f]);
Genus(D);
RationalPoints(D:Bound:=1000);
E:=SingularPoints(D);
for i := 1 to #E do
    pl := Places(E[i]);
    print pl;
    for j := 1 to #pl do;
        print Degree(pl[j]);
    end for;
end for;
//E[1] blows up to 3 nonrational
//E[2] blows up to one rational (and one nonrational)
//E[3] blows up to 2 rational (and 2 nonrational)


##Finding the inverse images
K:=CyclotomicField(3);
P2<z,t,h> := ProjectiveSpace(K,2);
f := z^6*t^9 - 3*z^6*t^6*h^3 - z^5*t^7*h^3 - z^4*t^8*h^3 + 2*z^3*t^9*h^3 + 3*z^6*t^3*h^6 + z^5*t^4*h^6 - 4*z^3*t^6*h^6 - z^2*t^7*h^6 + t^9*h^6 - z^6*h^9 + 2*z^3*t^3*h^9 - t^6*h^9;
C := Curve(P2,[f]);
phi := iso<C->C|[K.1*z,K.1*t,h],[K.1^2*z,K.1^2*t,h]>;
// we will take the quotient by phi
G := AutomorphismGroup(C,[phi]);
CG,prj := CurveQuotient(G);
prj;

R.<z,t,h>=QQ[]
P=ProjectiveSpace(R)
f1=t^20*h^3 - 8*t^17*h^6 + 43*t^14*h^9 - 54*t^11*h^12 + 28*t^8*h^15
f2=-z^5*t^18 + 5*z^4*t^19 + z^3*t^20 - 15*z^5*t^15*h^3 - 48*z^4*t^16*h^3 + 2*z^3*t^17*h^3 + 5*z*t^19*h^3 - 4*t^20*h^3 + 58*z^5*t^12*h^6 + 211*z^4*t^13*h^6 - 17*z^3*t^14*h^6 - 13*z^2*t^15*h^6 - 36*z*t^16*h^6 + 38*t^17*h^6 - 96*z^5*t^9*h^9 - 399*z^4*t^10*h^9 + 92*z^3*t^11*h^9 + 17*z^2*t^12*h^9 + 168*z*t^13*h^9 - 196*t^14*h^9 + 105*z^5*t^6*h^12 + 402*z^4*t^7*h^12 - 126*z^3*t^8*h^12 - 47*z^2*t^9*h^12 - 220*z*t^10*h^12 + 252*t^11*h^12 - 73*z^5*t^3*h^15 - 207*z^4*t^4*h^15 + 116*z^3*t^5*h^15 + 53*z^2*t^6*h^15 + 159*z*t^7*h^15 - 164*t^8*h^15 + 22*z^5*h^18 + 36*z^4*t*h^18 - 34*z^3*t^2*h^18 - 22*z^2*t^3*h^18 - 36*z*t^4*h^18 + 34*t^5*h^18
f3=z^4*t^19 - 4*z^5*t^15*h^3 - 9*z^4*t^16*h^3 + 2*z^3*t^17*h^3 + z*t^19*h^3 - t^20*h^3 + 14*z^5*t^12*h^6 + 40*z^4*t^13*h^6 - 11*z^3*t^14*h^6 - 2*z^2*t^15*h^6 - 7*z*t^16*h^6 + 9*t^17*h^6 - 24*z^5*t^9*h^9 - 76*z^4*t^10*h^9 + 33*z^3*t^11*h^9 + 34*z*t^13*h^9 - 46*t^14*h^9 + 28*z^5*t^6*h^12 + 77*z^4*t^7*h^12 - 39*z^3*t^8*h^12 - 8*z^2*t^9*h^12 - 43*z*t^10*h^12 + 59*t^11*h^12 - 20*z^5*t^3*h^15 - 39*z^4*t^4*h^15 + 31*z^3*t^5*h^15 + 12*z^2*t^6*h^15 + 31*z*t^7*h^15 - 39*t^8*h^15 + 6*z^5*h^18 + 6*z^4*t*h^18 - 8*z^3*t^2*h^18 - 6*z^2*t^3*h^18 - 6*z*t^4*h^18 + 8*t^5*h^18
F=(f1,f2,f3)
for Q in [[1,1,0],[1,-13/2,-3/2],[1,-4,-1],[-1,5,1],[1,0,0],[-2,5,1]]:
    I=R.ideal([F[0]*Q[1]-F[1]*Q[0], F[0]*Q[2]-F[2]*Q[0], F[1]*Q[2]-F[2]*Q[1]])
    X=P.subscheme(I)
    print(X.rational_points())

\end{lstlisting}
\end{code}

There are no possible ways to extend the structures $G_1$, $G_2$, and $G_3$ rationally, so these are the only possible structures of $\Q$-rational preperiodic points for $f_a(z)$ for $a \in \Q$.
\end{proof}

\subsection{$\A_3(D_2)$ First Component}
There are two families with a $D_2$ symmetry. We first look at the family $f_a(z)=\frac{az^2+1}{z^3+az}$. Note that $a = \pm1$ are degenerate cases.
\begin{prop}\label{prop_A3_D2_1}
    The following describes $\Q$-rational periodic points for $f_a(z) =\frac{az^2+1}{z^3+az}$ for $a \in \Q$
    \begin{enumerate}
        \item For every $a$, $f_a(z)$ has the rational periodic points
        \begin{equation*}
           \xygraph{
            !{<0cm,0cm>;<1cm,0cm>:<0cm,1cm>::}
            !{(0,0) }*+{\underset{0}{\bullet}}="a"
            !{(1.5,0) }*+{\underset{\infty}{\bullet}}="b"
            !{(2.5,0) }*+{\underset{1}{\bullet}}="c"
            !{(3.5,0) }*+{\underset{-1}{\bullet}}="d"
            "a":@/^1pc/"b"
            "b":@/^1pc/"a"
            "c":@(rd,ru)"c"
            "d":@(rd,ru)"d"
            }
        \end{equation*}
        \item The points $\pm 1$ are the only $\Q$-rational fixed points for all $a \in \Q \setminus \{\pm1\}$.
        \item If there are any additional rational periodic points with period $2$, then all six points of period $2$ are $\Q$-rational.
        \item There is no $a \in \Q \setminus \{\pm 1\}$ so that $f_a(z)$ has a $\Q$-rational $3$-cycle.
    \end{enumerate}
\end{prop}
\begin{proof}
The first dynatomic polynomial is $\Phi^{\ast}_{1}(f_a)=-z^4+1$. In particular, every member of the family has two rational fixed points $1$ and $-1$ and two complex fixed points $\pm i$.
\begin{code}
\begin{lstlisting}[language=Python]
R.<a> = QQ[]
P.<z> = R[]
P = FractionField(P)
A = AffineSpace(P,1)
g = DynamicalSystem_affine((a*z^2+1)/(z^3+a*z))
g.dynatomic_polynomial(1)
\end{lstlisting}
\end{code}
The second dynatomic polynomial is given by
\begin{equation*}
    \Phi^{\ast}_{2}(f_a)=(a^2 - 1)z^5 + (2a^3 - 2a)z^3 + (a^2 - 1)z.
\end{equation*}
The associated curve is reducible over $\Q$ and the irreducible components are
\begin{align}
    &z=0 \notag\\
    &a-1=0 \notag\\
    &a+1=0 \notag\\
    &z^4 + 2az^2 + 1=0. \label{6.10 eq4}
\end{align}
Note that $a=\pm 1$ are degenerate cases. The points $0$ and $\infty$ form a 2-cycle. From the last component \eqref{6.10 eq4}, for every $z$, we can find $a=\frac{-z^4-1}{2z^2}$ such that $z$ is periodic with period $2$. Furthermore, we know if $z$ is a rational point of period $2$, then $1/z$, $-z$, and $-1/z$ are all points of period $2$ since the maps $z \mapsto \pm 1/z$ and $z \mapsto -z$ are automorphism of $f_a(z)$. Since $f_a$ is a degree $3$ map, we know that there are at most three rational $2$-cycles and six points of period $2$. Thus, for every member of the family, either we have only $1$ rational $2$-cycle ($0$ and $\infty$) or all $2$-cycles are rational.
\begin{code}
\begin{lstlisting}[language=Python]
R.<a> = QQ[]
P.<z> = R[]
P = FractionField(P)
A = AffineSpace(P,1)
g = DynamicalSystem_affine((a*z^2+1)/(z^3+a*z))
g.dynatomic_polynomial(2)

A.<a,z>=AffineSpace(QQ,2)
R=A.coordinate_ring()
I=R.ideal([(a^2 - 1)*z^5 + (2*a^3 - 2*a)*z^3 + (a^2 - 1)*z])
A.subscheme(I).irreducible_components()
\end{lstlisting}
\end{code}

Now we look at rational $3$-cycles. The third dynatomic polynomial is calculated as
\begin{align*}
    \Phi^{\ast}_{3}(f_a)=&z^{24} + (a^5 + 2a^3 + 9a)z^{22} + (a^8 + 9a^6 + 14a^4 + 41a^2 + 1)z^{20} \\
    &+ (7a^9 + 28a^7 + 55a^5 + 118a^3 + 12a)z^{18} + (15a^{10} + 54a^8 + 118a^6 + 248a^4 + 59a^2 + 1)z^{16} \\
    &+ (12a^{11} + 73a^9 + 140a^7 + 376a^5 + 180a^3 + 11a)z^{14} \\
    &+ (3a^{12} + 52a^{10} + 115a^8 + 334a^6 + 361a^4 + 58a^2 + 1)z^{12} \\
    &+ (12a^{11} + 73a^9 + 140a^7 + 376a^5 + 180a^3 + 11a)z^{10}\\
    &+ (15a^{10} + 54a^8 + 118a^6 + 248a^4 + 59a^2 + 1)z^8 + (7a^9 + 28a^7 + 55a^5 + 118a^3 + 12a)z^6\\
    &+ (a^8 + 9a^6 + 14a^4 + 41a^2 + 1)z^4 + (a^5 + 2a^3 + 9a)z^2 + 1.
\end{align*}
The associated dynatomic curve has genus $27$ and is difficult to work with directly. First observe that if $(a,z)$ is a point on $\Phi^{\ast}_3(f_a) = 0$, then $(a,f_a(z))$ and $(a,f_a^2(z))$ are also on the curve. Thus, we can quotient this curve by a $C_3$ symmetry by setting $t=z+f_a(z)+f_a^2(z)$. We need to find the curve $X$ obtained through this quotient. We know that finding this curve $X$ is equivalent to finding the minimal polynomial of $t=z+f_a(z)+f_a^2(z)\in K$, where $K$ is the extension field of $\Q(a)$ defined by the third dynatomic polynomial.

The curve $X$ defined by this minimal polynomial is given by
\begin{equation*}
    X:z^8 + (a^5 + 4a^3 - a)z^6 + (a^8 + 2a^6 - 4a^4 + 8a^2 - 1)z^4 + (3a^7 +4a^5 - 23a^3 + 20a)z^2 + 9a^4 - 24a^2 + 16=0.
\end{equation*}
Observe that we can quotient out by another $C_2$ action by identifying $x=z^2$. This gives us the curve
\begin{equation*}
    Y:x^4 + (a^5 + 4a^3 - a)x^3 + (a^8 + 2a^6 - 4a^4 + 8a^2 - 1)x^2 + (3a^7 +4a^5 - 23a^3 + 20a)x + 9a^4 - 24a^2 + 16=0.
\end{equation*}
This curve has genus $3$ and is hyperelliptic with simplified model
\begin{equation*}
    Y':y^2 = x^8 - 6x^6 + 19x^4 - 30x^2 + 9.
\end{equation*}
Observe, that this curve covers a genus 1 curve by setting $x=x^2$,
\begin{equation*}
    E:y^2 = x^4 - 6x^3 + 19x^2 - 30x + 9.
\end{equation*}
This curve $E$ has smooth model as an elliptic curve
\begin{equation*}
    E': y^2 - \frac{10}{3}xy - \frac{32}{9}y = x^3 + \frac{4}{3}x^2.
\end{equation*}
The curve $E'$ has rank $0$ and the torsion subgroup has order $6$. Therefore, there are at most six rational points on $E$. A search using height bound $1000$ on the projective closure of $E$ finds five rational points:
\begin{equation*}
    (3 : 3 : 1), (3 : -3 : 1), (0 : 1 : 0), (0 : -3 : 1), (0 : 3 : 1).
\end{equation*}
This list is complete since $(0:1:0)$ is a singular point and blows up to $2$ rational points. Observe that affine rational points on $E$ can be lifted to affine rational points on $Y'$ only if the $x$-coordinate is a square. Thus, the only affine rational points on $Y'$ are $(0,-3)$ and $(0,3)$. We also know that since $Y'$ is of the form $y^2=f(x)$, where $f(x)$ has even degree with leading coefficient a square, $Y'$ has two points at infinity, both of which are rational. Thus, on $Y'$, there are only four rational points. Then $Y$ has at most four non-singular rational points. Searching using a height bound $1000$ on the projective closure of $Y$, we find four rational points:
\begin{equation*}
    (0:1:0), (1:0:0), (-1:1:1), (1:-1: 1),
\end{equation*}
where the first coordinate is $a$, the second $x$, and the third a homogenizing variable. All these points are singular points. When we blow them up, $(0:1:0)$ blows up to be two rational points, $(1:0:0)$ blows up to be two rational points, $(-1:1:1)$ becomes a single point that is not rational, and $(1:-1:1)$ also becomes a point that is not rational. Thus, we have found all four rational points in the smooth projective model of $Y$. We only lift the affine rational point $(-1,1)$ since $(1,-1)$ has $x=-1$ and we replaced $x^2$ with $x$ to get from $X$ to $Y$. This point has $a=1$, which is degenerate. Thus, there is no choice of $a\in \Q$ for which $f_a(z)$ has a $\Q$-rational $3$-cycle.
\begin{code}
\begin{lstlisting}[language=Python]
R.<a> = QQ[]
P.<z> = R[]
P = FractionField(P)
A = AffineSpace(P,1)
g = DynamicalSystem_affine((a*z^2+1)/(z^3+a*z))
g.dynatomic_polynomial(3)

#Quotienting by identifying (a,z) with (a,f(z)) and (a,f^2(z)) (run in magma)
Fc<a> := FunctionField(Rationals());
R<z> := PolynomialRing(Fc);
Phi3 :=z^24 + (a^5 + 2*a^3 + 9*a)*z^22 + (a^8 + 9*a^6 + 14*a^4 + 41*a^2 + 1)*z^20 + (7*a^9 + 28*a^7 + 55*a^5 + 118*a^3 + 12*a)*z^18 + (15*a^10 + 54*a^8 + 118*a^6 + 248*a^4 + 59*a^2 + 1)*z^16 + (12*a^11 + 73*a^9 + 140*a^7 + 376*a^5 + 180*a^3 + 11*a)*z^14 + (3*a^12 + 52*a^10 + 115*a^8 + 334*a^6 + 361*a^4 + 58*a^2 + 1)*z^12 + (12*a^11 + 73*a^9 + 140*a^7 + 376*a^5 + 180*a^3 + 11*a)*z^10 + (15*a^10 + 54*a^8 + 118*a^6 + 248*a^4 + 59*a^2 + 1)*z^8 + (7*a^9 + 28*a^7 + 55*a^5 + 118*a^3 + 12*a)*z^6 + (a^8 + 9*a^6 + 14*a^4 + 41*a^2 + 1)*z^4 + (a^5 + 2*a^3 + 9*a)*z^2 + 1;
Fcz<z> := ext<Fc | Phi3>;
t := z+((a*z^2+1)/(z^3+a*z))+((a^3*z^7 + a^4*z^5 + z^9 + 3*a*z^7 + 5*a^2*z^5 + 3*a^3*z^3 + a*z^3 + a^2*z)/(a^2*z^8 + 3*a^3*z^6 + a^4*z^4 + a*z^6 + 5*a^2*z^4 + a^3*z^2 + 3*a*z^2 + 1));
MinimalPolynomial(t);

#Getting the simplified model for the hyperelliptic curve (run in magma)
A<a,x>:=AffineSpace(Rationals(),2);
Y:=Curve(A,x^4 + (a^5 + 4*a^3 - a)*x^3 + (a^8 + 2*a^6 - 4*a^4 + 8*a^2 - 1)*x^2 + (3*a^7 +4*a^5 - 23*a^3 + 20*a)*x + 9*a^4 - 24*a^2 + 16);
b,X:=IsHyperelliptic(Y);
X;
YY:=SimplifiedModel(X);
YY;

#Rational Points on elliptic curve (run in magma)
A<x,y>:=AffineSpace(Rationals(),2);
E:=ProjectiveClosure(Curve(A,y^2 - (x^4 - 6*x^3 + 19*x^2 - 30*x + 9)));
P:=E![3,3,1];
EE:=EllipticCurve(E,P);
RankBounds(EE);
TorsionSubgroup(EE);
Points(E:Bound:=1000);
SingularPoints(E);
Degree(Places(SingularPoints(E)[1])[1]);
Degree(Places(SingularPoints(E)[1])[2]);

#Rational Points on hyperelliptic curve (run in magma)
A<a,x>:=AffineSpace(Rationals(),2);
Y:=Curve(A,x^4 + (a^5 + 4*a^3 - a)*x^3 + (a^8 + 2*a^6 - 4*a^4 + 8*a^2 - 1)*x^2 + (3*a^7 +4*a^5 - 23*a^3 + 20*a)*x + 9*a^4 - 24*a^2 + 16);
Y:=ProjectiveClosure(Y);
Points(Y:Bound:=1000);
sing:=SingularPoints(Y);
P1:=sing[1];
P2:=sing[2];
P3:=sing[3];
P4:=sing[4];
Q1:=Places(P1);
Q1;
Degree(Q1[1]);
Q2:=Places(P2);
Q2;
Degree(Q2[1]);
Q3:=Places(P3);
Q3;
Degree(Q3[1]);
Degree(Q3[2]);
Q4:=Places(P4);
Q4;
Degree(Q4[1]);
Degree(Q4[2]);
\end{lstlisting}
\end{code}
\end{proof}

A search for rational preperiodic structures with the parameter up to height $10,000$ using the algorithm from \cite{Hutz12} as implemented in Sage yields no parameters when $f_a(z)$ has a $\Q$-rational periodic point with minimal period at least $4$.

\begin{conj} \label{conj_A3_D2_1}
    There are no $a \in \Q$ such that $f_a(z) = \frac{az^2+1}{z^3+az}$ has a $\Q$-rational periodic point of minimal period at least $4$.
\end{conj}

\begin{theorem}\label{thm_A3_D2_1_preperiodic}
    Assuming Conjecture \ref{conj_A3_D2_1}, the possible rational preperiodic structures for $f_a(z) = \frac{az^2+1}{z^3+az}$ for $a \in \Q$ are the following.
    \begin{align*}
G_1:&\xygraph{
!{<0cm,0cm>;<1cm,0cm>:<0cm,1cm>::}
!{(0,0) }*+{\underset{1}{\bullet}}="a"
!{(1,0) }*+{\underset{-1}{\bullet}}="b"
!{(2,0) }*+{\underset{0}{\bullet}}="c"
!{(3.5,0) }*+{\underset{\infty}{\bullet}}="d"
"a":@(rd,ru)"a"
"b":@(rd,ru)"b"
"c":@/^1pc/"d"
"d":@/^1pc/"c"
}, \qquad a \text{ not in one of the families $G_2$, $G_3$, or $G_4$} \\
G_2:&\xygraph{
!{<0cm,0cm>;<1cm,0cm>:<0cm,1cm>::}
!{(1,0) }*+{\underset{1}{\bullet}}="a"
!{(0,1) }*+{\bullet}="a1"
!{(0,0) }*+{\bullet}="a2"
!{(3,0) }*+{\underset{-1}{\bullet}}="b"
!{(2,1) }*+{\bullet}="b1"
!{(2,0) }*+{\bullet}="b2"
!{(4,0) }*+{\underset{0}{\bullet}}="c"
!{(5.5,0) }*+{\underset{\infty}{\bullet}}="d"
"a1":"a"
"a2":"a"
"a":@(rd,ru)"a"
"b1":"b"
"b2":"b"
"b":@(rd,ru)"b"
"c":@/^1pc/"d"
"d":@/^1pc/"c"
}, \quad a = \frac{t^2+t+1}{t}, \quad t \in \Q \setminus \{\pm 1, 0\}\\
G_3:&\xygraph{
!{<0cm,0cm>;<1cm,0cm>:<0cm,1cm>::}
!{(0,0) }*+{\underset{1}{\bullet}}="a"
!{(1,0) }*+{\underset{-1}{\bullet}}="b"
!{(2,0) }*+{\underset{0}{\bullet}}="c"
!{(3.5,0) }*+{\underset{\infty}{\bullet}}="d"
!{(4,0) }*+{\bullet}="e"
!{(5.5,0) }*+{\bullet}="f"
!{(6,0) }*+{\bullet}="g"
!{(7.5,0) }*+{\bullet}="h"
"a":@(rd,ru)"a"
"b":@(rd,ru)"b"
"c":@/^1pc/"d"
"d":@/^1pc/"c"
"e":@/^1pc/"f"
"f":@/^1pc/"e"
"g":@/^1pc/"h"
"h":@/^1pc/"g"
}, \quad a=\frac{-t^4-1}{2t^2}, \quad t \in \Q \setminus \{\pm 1, 0\}\\
G_4:&\xygraph{
!{<0cm,0cm>;<1cm,0cm>:<0cm,1cm>::}
!{(0,0) }*+{\underset{1}{\bullet}}="a"
!{(1,0) }*+{\underset{-1}{\bullet}}="b"
!{(3,0) }*+{\underset{0}{\bullet}}="c"
!{(4.5,0) }*+{\underset{\infty}{\bullet}}="d"
!{(2,0.5) }*+{\bullet}="c1"
!{(2,-0.5) }*+{\bullet}="c2"
!{(5.5,-0.5) }*+{\bullet}="d1"
!{(5.5,0.5) }*+{\bullet}="d2"
"a":@(rd,ru)"a"
"b":@(rd,ru)"b"
"c":@/^1pc/"d"
"d":@/^1pc/"c"
"c1":"c"
"c2":"c"
"d1":"d"
"d2":"d"
},\quad a=-t^2, \quad t \in \Q \setminus \{\pm 1,0\}.
\end{align*}
\end{theorem}
\begin{proof}
    Proposition \ref{prop_A3_D2_1} combined with Conjecture \ref{conj_A3_D2_1} classifies the $\Q$-rational periodic structure. For preperiodic structures, we start by looking for non-periodic rational preimages of the fixed points. Recall from Proposition \ref{prop_A3_D2_1} that for every $a\in \Q$, $f_a(z)$ has exactly two $\Q$-rational fixed points $\pm 1$. Observe that if $f_a(z)=1$, then $f_a(-z)=-1$, so it suffices to consider the preimages of $1$. To have $f_a(z)=1$, we must have $(az^2+1)-(z^3+az)=0$. This defines a reducible curve over $\Q$, and the irreducible components are
    \begin{equation*}
        z-1=0 \quad \text{ and } \quad az - z^2 - z - 1=0.
    \end{equation*}
    We only need to consider the second component. This component is linear in $a$ so we solve as
    \begin{equation*}
        a = \frac{t^2+t+1}{t} \quad t \not\in \{0, \pm 1\}.
    \end{equation*}
    We exclude $t\in \{0, -1\}$ for degenerate $a$ values and $t=1$ is the fixed point as its own preimage.
    Thus, for every $t\in \Q\setminus \{\pm 1,0\}$, we can find an $a$ such that $1$ has a non-periodic $\Q$-rational preimage under $f_a(z)$. Since $z \mapsto \frac{1}{z}$ is an automorphism of $f_a(z)$, if $f_a(z)=1$, then $f_a(1/z)=1$. Furthermore, since this irreducible component is quadratic, it can have at most two rational points. Thus, the two rational points must be $z$ and $1/z$.
\begin{code}
\begin{lstlisting}[language=Python]
R.<a,z>=QQ[]
A=AffineSpace(R)
I=R.ideal((a*z^2+1)-(z^3+a*z))
S=A.subscheme(I)
S.irreducible_components()
\end{lstlisting}
\end{code}

    Now we look for a second non-periodic rational preimage of $1$, i.e., a $z\in \Q$ with $f_a^2(z)=1$. We need to find (rational) solutions to
    \begin{align*}
        &(a^3z^7 + a^4z^5 + z^9 + 3az^7 + 5a^2z^5 + 3a^3z^3 + az^3 + a^2z\\
        &-(a^2z^8 + 3a^3z^6 + a^4z^4 + az^6 + 5a^2z^4 + a^3z^2 + 3az^2 + 1)=0.
    \end{align*}
    This defines a reducible curve over $\Q$ with irreducible components
    \begin{align*}
        &z-1=0 \quad \text{(fixed points),}\\
        &az - z^2 - z - 1=0 \quad \text{(points satisfying $f_a(z)=1$), and}\\
        &a^2z^5 + a^3z^3 - a^2z^4 - az^5 - z^6 - a^2z^3 - 2az^4 - a^2z^2 + az^3 + a^2z - 2az^2 - z^3 - az - 1=0.
    \end{align*}
    We need only consider the third component. It defines a genus $3$ curve, but we can quotient by a $C_2$ symmetry by identifying $(a,z)$ with $(a,1/z)$. If we set $t=z+1/z$, we get the curve after this quotienting as
    \begin{equation*}
        X:t^3 + (-a^2 + a)t^2 + (a^2 + 2a - 3)t - a^3 + 3a^2 - 3a + 1=0.
    \end{equation*}
    This is a genus 1 curve and is birationally equivalent to the elliptic curve
    \begin{equation*}
        X':y^2 - 2xy + 2y = x^3 - 3x^2 + 2x
    \end{equation*}
    via the map on projective closures
    \begin{align*}
        x=&2at^4 + 2a^2t^2h - 8at^3h - 8a^2th^2 +4at^2h^2 - 8t^3h^2 + 16ath^3 + 26t^2h^3 -8th^4\\
        y=&2at^4 + 4t^4h - 2a^3h^2 + 2a^2th^2 -26at^2h^2 \\
        &- 14t^3h^2 - 26a^2h^3 + 36ath^3 -6t^2h^3 + 58ah^4 + 90th^4 - 30h^5\\
        z=&-t^5 + 2t^4h.
    \end{align*}
    We compute that $X'$ has rank $0$ and its torsion subgroup has order $6$. Thus, there are at most six non-singular rational points on $X$. Searching using height bound $1000$ on the projective closure produces
    \begin{equation*}
        (a,t,h) \in \{(-1: -2: 1), (3: 2: 1), (0: 1: 0), (1: 0: 0), (1: 0: 1), (-1: 2:1)\},
    \end{equation*}
    where $h$ is a homogenizing variable.
    The curve $X$ has singular points $(1:0:1)$ and $(-1:2:1)$. We blow up the singular points and in the blow up, $(1:0:1)$ becomes two points, both of which are rational points and $(-1:2:1)$ becomes one closed point which is not rational. Thus, the above list is complete since we have found six rational points. Now we need to lift these rational points back to the original curve. Recall that $t=z+1/z$, so we can solve for $z$. Only $t=\pm 2$ produces rational points, but $t=\pm 2$ lifts back to $z=\pm 1$, which are fixed points. Thus, there is no $a \in \mathbb{Q}$ such that $f_a^2(z)=1$ and $f_a(z) \neq 1$ for some $z\in \Q$.
\begin{code}
\begin{lstlisting}[language=Python]
R.<a,z>=PolynomialRing(QQ)
f=((a*z^2+1)/(z^3+a*z))
g=((a*f^2+1)/(f^3+a*f))
g

R.<a,z>=QQ[]
A=AffineSpace(R)
I=R.ideal((a^3*z^7 + a^4*z^5 + z^9 + 3*a*z^7 + 5*a^2*z^5 + 3*a^3*z^3 + a*z^3 + a^2*z)-(a^2*z^8 + 3*a^3*z^6 + a^4*z^4 + a*z^6 + 5*a^2*z^4 + a^3*z^2 + 3*a*z^2 + 1))
S=A.subscheme(I)
S.irreducible_components()

R.<a,z>=QQ[]
A=AffineSpace(R)
I=R.ideal(a^2*z^5 + a^3*z^3 - a^2*z^4 - a*z^5 - z^6 - a^2*z^3 - 2*a*z^4 - a^2*z^2 + a*z^3 + a^2*z - 2*a*z^2 - z^3 - a*z - 1)
C=A.curve(I.gens())
C.genus()

#Quotienting by C2 (run in magma)
Fc<a>:=FunctionField(Rationals());
R<z>:=PolynomialRing(Fc);
Phi3:=a^2*z^5 + a^3*z^3 - a^2*z^4 - a*z^5 - z^6 - a^2*z^3 - 2*a*z^4 - a^2*z^2 + a*z^3 + a^2*z - 2*a*z^2 - z^3 - a*z - 1;
Fcz<z>:=ext<Fc | Phi3>;
t:=z+1/z;
MinimalPolynomial(t);

#Rational Points (run in magma)
A<a,z>:=AffineSpace(Rationals(),2);
C:=ProjectiveClosure(Curve(A,z^3 + (-a^2 + a)*z^2 + (a^2 + 2*a - 3)*z - a^3 + 3*a^2 - 3*a + 1));
Points(C:Bound:=1000);
sing:=SingularPoints(C);
sing;
P1:=sing[1];
P1;
P2:=sing[2];
P2;
Q1:=Places(P1);
Degree(Q1[1]);
Degree(Q1[2]);
Q2:=Places(P2);
Degree(Q2[1]);
P:=C![1,0,1];
E,psi:=EllipticCurve(C,P);
E;
Rank(E);
TorsionSubgroup(E);
RationalPoints(E:Bound:=1000);
\end{lstlisting}
\end{code}

    Now we look at non-periodic preimages of the $2$-cycle comprised of $0$ and $\infty$. First consider if there is an $a$ such that $f_a(z)=0$. It suffices to look at $f_a(z)=0$ since if $f_a(z)=\infty$, we have $f_a(1/z)=0$ due to the automorphism $z \mapsto 1/z$. Solving $az^2+1=0$, we get $z=\pm\sqrt{\frac{-1}{a}}$. Thus, as long as $a=-k^2$ for some $k\in \Q$, we have $f_a(z)=0$ for two $\Q$-rational $z$ values.

    Now we consider non-periodic second preimages: $f^2_a(z)=0$ for some $z$ such that $f_a(z)\neq 0$. We need to find rational solutions to the equation
    \begin{equation*}
        a^3z^7 + a^4z^5 + z^9 + 3az^7 + 5a^2z^5 + 3a^3z^3 + az^3 + a^2z=0.
    \end{equation*}
    This equation defines a reducible curve over $\Q$. The irreducible components are
    \begin{align}
        &z=0 \notag\\
        &z^2 + a=0 \notag\\
        &a^3z^4 + z^6 + 2az^4 + 3a^2z^2 + a=0. \label{eq_2}
    \end{align}
    The first two components correspond to $f_a(z)$ is the point at infinity, so we only need to study the third component. It has genus $3$, but observe that it covers a genus $1$ curve:
    \begin{equation*}
        X:a^3x^2 + x^3 + 2ax^2 + 3a^2x + a=0.
    \end{equation*}
    This curve is birationally equivalent to the elliptic curve
    \begin{equation*}
        E:y^2 + 8xy + 2y = x^3 - 12x^2 - 4x
    \end{equation*}
    via the map (of projective closures)
    \begin{align*}
        x&=2a^2x^4 - 4a^2x^3h + 4ax^4h - 6a^2x^2h^2 + 8ax^3h^2 + 8x^4h^2 - 12ax^2h^3 + 4x^3h^3 - 16axh^4\\
        &- 10x^2h^4 - 8xh^5 - 2h^6\\
         y&=-4a^2x^4 - 8ax^4h + 36a^2x^2h^2 - 32ax^3h^2 - 12x^4h^2 + 8ax^2h^3 - 32x^3h^3 + 96axh^4 + 36x^2h^4\\
         &+ 32xh^5 + 8h^6\\
         z&=-x^5h - 3x^4h^2 - 2x^3h^3 + 2x^2h^4 + 3xh^5 + h^6,
    \end{align*}
    where $h$ and $z$ are the homogenizing variables of $X$ and $E$, respectively.
\begin{code}
\begin{lstlisting}
r.<a,x,t>=QQ[]
X=2*a^2*x^4 - 4*a^2*x^3*t + 4*a*x^4*t - 6*a^2*x^2*t^2 + 8*a*x^3*t^2 + 8*x^4*t^2 - 12*a*x^2*t^3 + 4*x^3*t^3 - 16*a*x*t^4 - 10*x^2*t^4 - 8*x*t^5 - 2*t^6
Y=-4*a^2*x^4 - 8*a*x^4*t + 36*a^2*x^2*t^2 - 32*a*x^3*t^2 - 12*x^4*t^2 + 8*a*x^2*t^3 - 32*x^3*t^3 + 96*a*x*t^4 + 36*x^2*t^4 + 32*x*t^5 + 8*t^6
Z=-x^5*t - 3*x^4*t^2 - 2*x^3*t^3 + 2*x^2*t^4 + 3*x*t^5 + t^6
F=a^3*x^2 + x^3*t^2 + 2*a*x^2*t^2 + 3*a^2*x*t^2 + a*t^4
G=Y^2*Z + 8*X*Y*Z + 2*Y*Z^2 - (X^3 - 12*X^2*Z - 4*X*Z^2)
G.quo_rem(F)[1]
\end{lstlisting}
\end{code}
    The curve $E$ has rank $0$ and its torsion subgroup has order $6$. Searching with a height bound $1000$ on the projective closure, we find five rational points on $X$:
    \begin{equation*}
        (a,x,h) \in \{(0: 1: 0), (0: 0: 1), (1: 0: 0), (-1: 1: 1), (1: -1: 1)\}.
    \end{equation*}
    Among these five points, four of them are singular points:
    \begin{equation*}
        (1: -1: 1), (-1: 1: 1), (0: 1: 0), (1: 0: 0).
    \end{equation*}
    The first point blows up to be two rational points, the second blows up to be two rational points, and the third blows up to be one rational point. The last one blows up to a closed point, but that point is not rational. Thus, we have found all six points on the projective smooth model of $X$, and the list must be complete. Recall that we cannot have $a=1$ or $-1$, and we only look at affine rational points. Thus, the only rational point on the projective closure of $X$ that we care about is $(0:0:1)$. It lifts to $(0:0:1)$ on the third irreducible component in equation \eqref{eq_2}, but $z=0$ is a periodic point. Therefore, we can conclude that there is no $a \in \mathbb{Q}$ such that $f_a^2(z)=0$ and $f_a(z)\neq 0$.
\begin{code}
\begin{lstlisting}[language=Python]
R.<a,z>=QQ[]
A=AffineSpace(R)
I=R.ideal((a^3*z^7 + a^4*z^5 + z^9 + 3*a*z^7 + 5*a^2*z^5 + 3*a^3*z^3 + a*z^3 + a^2*z))
S=A.subscheme(I)
S.irreducible_components()

#Rational points (run in magma)
A<a,x>:=AffineSpace(Rationals(),2);
C:=ProjectiveClosure(Curve(A,a^3*x^2 + x^3 + 2*a*x^2 + 3*a^2*x + a));
Points(C:Bound:=1000);
sing := SingularPoints(C);
sing;
p1 := sing[1];
p2 := sing[2];
p3 := sing[3];
p4 := sing[4];
Q1:=Places(p1);
Q1;
Degree(Q1[1]);
Degree(Q1[2]);
Q2:=Places(p2);
Q2;
Degree(Q2[1]);
Degree(Q2[2]);
Q3:=Places(p3);
Q3;
Degree(Q3[1]);
Q4:=Places(p4);
Q4;
Degree(Q4[1]);
P:=C![1,-1,1];
E,psi:=EllipticCurve(C,P);
E;
Rank(E);
TorsionSubgroup(E);
RationalPoints(E:Bound:=1000);
\end{lstlisting}
\end{code}

    Now we look for non-periodic preimages of one of the other possible rational 2-cycles. In particular, if we can find an $a$ such that $f_a(z)$ enters into a 2-cycle that is not the $0-\infty$ cycle. Recall from component \eqref{6.10 eq4} of the second dynatomic curve in the proof of Proposition 6.10 that we need $a=\frac{-t^4 -1}{2t^2}$ for $t \in \Q \setminus \{\pm 1, 0\}$ to have extra 2-cycles. Furthermore, we can compute the $(1,2)$ generalized dynatomic polynomial characterizing points with formal period (1,2), (see Hutz \cite[Section 3]{Hutz12}) as
    \begin{equation*}
        \Phi^{\ast}_{(1,2)}(f_a)=(x^2 + a)(ax^2 + 1)(x^8 + (2a^3 + 2a)x^6 + (a^4 + 6a^2 - 1)x^4 + (2a^3 + 2a)x^2 + 1).
    \end{equation*}
    When we set $a=\frac{-t^4 -1}{2t^2}$, the irreducible components of $\Phi^{\ast}_{(1,2)}(f_a)$ become
    \begin{align*}
        &xt^4 - 2x^2t + 2t^3 - x=0\\
        &xt^4 + 2x^2t - 2t^3 - x=0\\
        &2x^2t^3 - xt^4 + x - 2t=0\\
        &2x^2t^3 + xt^4 - x - 2t=0\\
        &2x^2t^2 - t^4 - 1=0\\
        &x^2t^4 + x^2 - 2t^2=0.
    \end{align*}
    The first four are quadratic in $x$ and all have discriminant $t^8 + 14t^4 + 1$. So rational points occur when there are rational points on the curve
    \begin{equation*}
        C:k^2 = t^8 + 14t^4 + 1.
    \end{equation*}
    This covers the curve
    \begin{equation*}
        E:y^2 = x^4 + 14x^2 + 1,
    \end{equation*}
    via the substitution $y=k, x=t^2$.
    The curve $E$ is birational via
    \begin{align*}
        (x,y,z) \mapsto (2x^2z + 2yz^2 + 14z^3,4x^3 + 4xyz + 28xz^2, z^3)
    \end{align*}
    to the rank 0 elliptic curve $y^2 = x^3 - 28x^2 + 192x$ with torsion subgroup isomorphic to $\Z/2\Z \times \Z/4\Z$.
    A point search gives the eight points
    \begin{equation*}
        (0 : 1 : 0), (0 : 0 : 1), (8 : 16 : 1), (8 : -16 : 1), (12 : 0 : 1), (16 : 0 : 1), (24 : 48 : 1), (24 : -48 : 1).
    \end{equation*}
    The curve $E$ has the seven rational points
    \begin{equation*}
        (-1 : 4 : 1), (0 : -1 : 1), (0 : 1 : 0), (1 : -4 : 1), (-1 : -4 : 1), (1 : 4 : 1), (0 : 1 : 1),
    \end{equation*}
    where $(0:1:0)$ is singular and blows-up to two rational points. On the original curve $C$, we find the same seven rational points
    \begin{equation*}
        (-1 : 4 : 1), (0 : -1 : 1), (0 : 1 : 0), (1 : -4 : 1), (-1 : -4 : 1), (1 : 4 : 1), (0 : 1 : 1).
    \end{equation*}
    These all have $t \in \{\pm 1,0\}$ so correspond to degenerate cases.
    The last two components are both birational to the same rank 0 elliptic curve, $y^2 + 2xy = x^3 - 4x^2 + 2x$, with torsion subgroup isomorphic to $\Z/2\Z \times \Z/2\Z$, so they have at most four rational points. In both cases, we have the rational points $(x,t) = (\pm 1, \pm 1)$. Since $t=\pm 1$ are degenerate cases, these components do not have any rational points. Thus, we cannot have $f_a(z)$ entering a 2-cycle that is not the $0-\infty$ cycle.
\begin{code}
\begin{lstlisting}[language=Python]
R.<a>=QQ[]
P.<x,y>=ProjectiveSpace(FractionField(R),1)
f=DynamicalSystem_projective([a*x^2*y + y^3, x^3 + a*x*y^2])
f.dynatomic_polynomial(2).factor(proof=False)
A=(-x^4 -1)/(2*x^2)
(x^4 + 2*A*x^2 + 1)
print(f.dynatomic_polynomial([1,2]).factor(proof=False))
R.<x,t>=QQ[]
A2=AffineSpace(R)
X1=A2.curve(1024*t^8*x^8 + (-256*t^14 - 1792*t^10 - 1792*t^6 - 256*t^2)*x^6 + (64*t^16 + 1792*t^12 + 2432*t^8 + 1792*t^4 + 64)*x^4 + (-256*t^14 - 1792*t^10 - 1792*t^6 - 256*t^2)*x^2 + 1024*t^8)
X2=A2.curve(2*t^2*x^2 - t^4 - 1)
X3=A2.curve((-t^4 - 1)*x^2 + 2*t^2)
for Z in [X1,X2,X3]:
    for Y in Z.irreducible_components():
        print(Y.defining_polynomials())
        F=Y.defining_polynomials()[0]
        print(F.discriminant(x))

P<t,k,h>:=ProjectiveSpace(Rationals(),2);
C:=Curve(P,k^2*h^6 - (t^8+14*t^4*h^4+h^8));
Points(C:Bound:=1000);
SingularPoints(C);
P:=C![0,1,0];
Q:=Places(P);
Q;
Degree(Q[1]);
Degree(Q[2]);

P<x,y,z>:=ProjectiveSpace(Rationals(),2);
C:=Curve(P,y^2*z^2 - (x^4+14*x^2*z^2+z^4));
Points(C:Bound:=1000);
P:=C![0,1,0];
E,psi:=EllipticCurve(C,P);
E;
Rank(E);
TorsionSubgroup(E);
Points(E:Bound:=1000);

#########
X2.projective_closure().rational_points(bound=100)

A<x,t>:=AffineSpace(Rationals(),2);
C:=ProjectiveClosure(Curve(A,(2*t^2*x^2 - t^4 - 1)));
Points(C:Bound:=1000);
P:=C![1,1,1];
E,psi:=EllipticCurve(C,P);
E;
Rank(E);
TorsionSubgroup(E);
RationalPoints(E:Bound:=1000);

############
X3.projective_closure().rational_points(bound=100)

A<x,t>:=AffineSpace(Rationals(),2);
C:=ProjectiveClosure(Curve(A,(-t^4 - 1)*x^2 + 2*t^2));
Points(C:Bound:=1000);
P:=C![1,1,1];
E,psi:=EllipticCurve(C,P);
E;
Rank(E);
TorsionSubgroup(E);
RationalPoints(E:Bound:=1000);
\end{lstlisting}
\end{code}

    Now that we have identified each of the separate possible components of the rational preperiodic structure, we need to consider which of these components can occur simultaneously.
    First we ask if $0$ can have a non-periodic rational preimage at the same time there are the additional 2-cycles. Recall that we need $a=-k^2$ for some $k\in \Q$ to have $f_a(z)=0$ and we need $a=\frac{-z^4-1}{2z^2}$ to have extra 2-cycles. Thus, we are looking for rational points on the curve
    \begin{equation*}
        X:z^4+1-2k^2z^2=0.
    \end{equation*}
    This is a genus 1 curve that is birationally equivalent to the elliptic curve
    \begin{equation*}
        E:y^2+2xy=x^3-4x^2+2x
    \end{equation*}
    via the map (of projective closures with $h$ and $z$ the homogenizing variables of $X$ and $E$, respectively,)
    \begin{align*}
        x&=2kz^2 + 2z^3 - 2kzh - 4z^2h + 4zh^2 - 2h^3\\
        y&=-4kz^2 - 4z^3 + 4z^2h - 4zh^2\\
        z&=z^3 - 3z^2h + 3zh^2 - h^3.
    \end{align*}
    This elliptic curve $E$ has rank $0$ with torsion subgroup of order $4$. Searching with a height bound $1000$ on the projective closure, we find five rational points on $X$:
    \begin{equation*}
        (k,z,h) \in \{(1: 1: 1), (1: 0: 0), (-1: 1: 1), (1: -1: 1), (-1: -1: 1)\},
    \end{equation*}
    where $h$ is a homogenizing variable.
    Only one of these points is singular, so we know this is a complete list of rational points. The affine rational points force $z=\pm 1$, which are fixed points. Thus, we cannot have $f_a(z)=0$ and extra 2-cycles at the same time.
\begin{code}
\begin{lstlisting}[language=Python]
#Rational points (run in magma)
A<k,z>:=AffineSpace(Rationals(),2);
C:=ProjectiveClosure(Curve(A,z^4+1-2*k^2*z^2));
Points(C:Bound:=1000);
SingularPoints(C);
P:=C![1,1,1];
E,psi:=EllipticCurve(C,P);
E;
Rank(E);
TorsionSubgroup(E);
RationalPoints(E:Bound:=1000);
SingularPoints(E);
\end{lstlisting}
\end{code}

    Now we check if both the fixed points and the $0-\infty$ 2-cycle can have non-periodic rational preimages at the same time. Recall that we need $a=-k^2$ for some $k\in \Q$ to have $f_a(z_1)=0$ and we need $a=\frac{t^2 + t + 1}{t}$ for some $t\neq 0,\pm 1 \in \Q$ to have $f_a(z_2)=1$. Thus, we are interested in the following curve
    \begin{equation*}
        X:t^2+t+1 + k^2t=0.
    \end{equation*}
    This is a nonsingular curve of genus $1$ that is birational to the elliptic curve
    \begin{equation*}
        E:y^2 = x^3 -4x^2 + 16x,
    \end{equation*}
    via the map
    \begin{align*}
        x&=4kt\\\
        y&=8t^2 + 8th + 8h^2\\
        z&=-kh,
    \end{align*}
    where $k$ and $z$ are the homogenizing variables of $X$ and $E$, respectively.
    This curve $E$ has rank $0$ and a torsion subgroup of order $4$. Searching using height bound $1000$ on the projective closure, we find four rational points on $X$:
    \begin{equation*}
        (k,t,h) \in \{(0 : 1 : 0), (1 : 0 : 0), (1 : -1 : 1), (-1 : -1 : 1)\}.
    \end{equation*}
    The affine points all have $t$ in the excluded set $\{0,\pm 1\}$. Thus, we can conclude that there is no $a$ such that $f_a(z_1)=0$ and $f_a(z_2)=1$.
\begin{code}
\begin{lstlisting}[language=Python]
#Rational Points (run in magma)
A<k,t>:=AffineSpace(Rationals(),2);
C:=ProjectiveClosure(Curve(A,t^2+t+1 + k^2*t));
Points(C:Bound:=1000);
sing := SingularPoints(C);
sing;
P:=C![0,1,0];
E,psi:=EllipticCurve(C,P);
E;
Rank(E);
TorsionSubgroup(E);
RationalPoints(E:Bound:=1000);
\end{lstlisting}
\end{code}

    Next we check if we can have additional 2-cycles as well as non-periodic preimages of the fixed points. Recall that we need $a=\frac{t^2+3}{-t^2+1}$
    to have $f(z)=1$ and we need $a=\frac{k^4+1}{2k^2}$ to have extra $2$-cycles. Thus, we need to find rational points on the curve    \begin{equation*}
        X:(t^2+3)(2k^2)=(k^4+1)(-t^2+1).
    \end{equation*}
    This is a genus 1 curve that is birationally equivalent to the elliptic curve
    \begin{equation*}
        E:y^2 = x^3 - 6x^2 - 4x + 24
    \end{equation*}
    via the map
    \begin{align*}
        x&=2tk^3 + 2tkh^2 - 2kh^3\\
        y&=4tk^2h + 12k^2h^2 + 4th^3 - 4h^4\\
        z&=-k^3h,
    \end{align*}
    where $h$ and $z$ are the homogenizing variables of $X$ and $E$, respectively.
    This curve $E$ has rank $0$ and the torsion subgroup has order $4$. Searching using height bound $1000$ on the projective closure, we find four rational points on the curve $X$:
    \begin{equation*}
        (t,k,h) \in \{(0:1:0),(-1: 0:1), (1: 0: 0), (1: 0: 1)\}.
    \end{equation*}
    Among these four points, two of them are singular: $(0: 1: 0), (1: 0: 0)$. The point $(0:1:0)$ blows up to be two rational points, and the point $(1:0:0)$ blows up to be a closed point that is not rational. Thus, this search produces all the rational points on $X$. Considering only the affine points, we are left with $(-1:0:1)$ and $(1:0:1)$, but $k=0$ corresponds to $a$ not being defined. Thus, there is no $a$ such that we can have extra 2-cycles and $f_a(z)=1$.
\begin{code}
\begin{lstlisting}[language=Python]
#Rational Points (run in magma)
A<t,k>:=AffineSpace(Rationals(),2);
C:=ProjectiveClosure(Curve(A,(t^2+3)*2*k^2-(k^4+1)*(-t^2+1)));
Points(C:Bound:=1000);
sing := SingularPoints(C);
sing;
P1:=sing[1];
P2:=sing[2];
Q1:=Places(P1);
Q1;
Degree(Q1[1]);
Degree(Q1[2]);
Q2:=Places(P2);
Q2;
Degree(Q2[1]);
P:=C![-1,0,1];
E,psi:=EllipticCurve(C,P);
E;
Rank(E);
TorsionSubgroup(E);
RationalPoints(E:Bound:=1000);
\end{lstlisting}
\end{code}

    We have exhausted all possibilities for rational preperiodic structures, leaving only those enumerated in the statement.
\end{proof}

\subsection{$\A_3(D_2)$ Second Component}
Now we look at the family with $D_2$ symmetry $g_a(z)=\frac{az^2-1}{z^3-az}$. We first consider the possible rational periodic points.
\begin{prop} \label{prop_A3_D2_2}
    Let $g_a(z) =\frac{az^2-1}{z^3-az}$ for $a \neq \pm 1$.
    \begin{enumerate}
        \item For every $t \in \Q\setminus \{0, \pm 1\}$, the value $a= \frac{t^4+1}{2t^2}$ satisfies $g_a(z)$ has the four fixed points $\left\{ \pm t, \pm \frac{1}{t} \right\}$. For no other values of $a \in \Q$ does $g_a(z)$ have a $\Q$-rational fixed point.
        \item For every $a \in \Q \setminus {\pm 1}$, the function $g_a(z)$ has exactly two 2-cycles with $\Q$-rational points: swapping $0$ and $\infty$ and swapping $1$ and $-1$.
        \item There is no $a \in \Q$ so that $g_a(z)$ has a 3-cycle with $\Q$-rational points.
    \end{enumerate}
\end{prop}
\begin{proof}
    The first dynatomic polynomial is
    \begin{equation*}
        \Phi^{\ast}_{1}(g_a) =-z^4 + 2az^2 - 1.
    \end{equation*}
    This polynomial is linear in $a$, so it has a zero when $a=\frac{z^4+1}{2z^2}$. Thus, for every $z\in \Q$, we can find an $a$ such that $z$ is a fixed point. Once $z$ is a fixed point, we know from the automorphism group that $-z$, $\frac{1}{z}$ and $-\frac{1}{z}$ are all fixed points. Furthermore, for every $a$, the first dynatomic polynomial has at most four zeros, so $z,-z,\frac{1}{z}$ and $-\frac{1}{z}$ are all the fixed points.

    Now we look for 2-cycles with $\Q$-rational points. The second dynatomic polynomial is given by
    \begin{equation*}
        \Phi^{\ast}_{2}(g_a)=(-a^2 + 1)z^5 + (a^2 - 1)z = (1-a^2)z(z-1)(z+1)(z^2+1).
    \end{equation*}
    Since $a = \pm 1$ are degenerate cases, we see that the roots of $\Phi^{\ast}_2(g_a)$ do not depend on $a$ and every member of the family $g_a$ has exactly two 2-cycles with $\Q$-rational points: swapping $0$ and $\infty$ and swapping $1$ and $-1$.

Now we look for 3-cycles with $\Q$-rational points. The third dynatomic polynomial is
\begin{align*}
    \Phi^{\ast}_{3}(g_a)=&-z^{24}+ (a^5 + 4a^3 + 7a)z^{22} + (-a^8 - 13a^6 - 24a^4 - 29a^2 + 1)z^{20}\\
    &+ (7a^9 + 48a^7 + 91a^5 + 76a^3 - 2a)z^{18} + (-11a^{10} - 102a^8 - 214a^6 - 168a^4 + a^2 - 1)z^{16}\\\
    &+ (6a^{11} + 103a^9 + 348a^7 + 304a^5 + 30a^3 + a)z^{14} \\&+ (-a^{12} - 44a^{10} - 289a^8 - 470a^6 - 111a^4 - 10a^2 + 1)z^{12}\\ &+ (6a^{11} + 103a^9 + 348a^7 + 304a^5 + 30a^3 + a)z^{10}\\
    &+ (-11a^{10} - 102a^8 - 214a^6 - 168a^4 + a^2 - 1)z^8\\ &+ (7a^9 + 48a^7 + 91a^5 + 76a^3 - 2a)z^6 + (-a^8 - 13a^6 - 24a^4 - 29a^2 + 1)z^4\\
    &+ (a^5 + 4a^3 + 7a)z^2 - 1.
\end{align*}
The dynatomic curve $\Phi^{\ast}_{3}(g_a)=0$ has genus $31$. We can quotient by a $C_3$ symmetry by identifying $(a,z)$ with $(a,g_a(z))$ and $(a,g_a^2(z))$. Setting $t=z+g_a(z)+g_a^2(z)$, this quotienting produces the following curve:
\begin{align*}
    X:t^8 &+ (-a^5 - 6a^3 - 13a)t^6 + (a^8 + 14a^6 + 46a^4 + 56a^2 + 1)t^4\\
    &+(-4a^9 - 33a^7 - 78a^5 - 69a^3 + 4a)t^2 + 4a^8 + 12a^6 + 25a^4 +24a^2 + 16=0.
\end{align*}
This curve has genus $11$. We should be able to quotient by another $C_3$ symmetry by identifying $(a,z)$ with $(a,\frac{1}{z})$. However, in this curve we have used $t=z+g_a(z)+g_a^2(z)$ so we need to identify $(a,t)$ with $(a,t')$, where $t'=\frac{1}{z}+\frac{1}{g_a(z)}+\frac{1}{g_a^2(z)}$.
We can find the minimal polynomial of $t+t'$, which defines the curve
\begin{equation*}
    Y':u^4 + (-a^5 - 2a^4 - 6a^3 - 10a^2 - 13a - 8)u^2 + 4a^6 + 12a^5 + 25a^4 + 36a^3 + 34a^2 + 24a + 9=0.
\end{equation*}
This curve has genus $5$, but we can identify $u^2=x$ and get
\begin{equation*}
    Y:x^2 + (-a^5 + 2a^4 - 6a^3 + 10a^2 - 13a + 8)x + 4a^6 - 12a^5 + 25a^4 -36a^3 + 34a^2 - 24a + 9=0.
\end{equation*}
The new curve $Y$ has genus $2$ so is hyperelliptic. It is birational to the curve
\begin{equation*}
    H:y^2 = 4x^6 - 12x^5 + 25x^4 - 30x^3 + 25x^2 -12x + 4
\end{equation*}
via tha map
\begin{align*}
    x&=1/2a^3 + 3/2a^2 + 3/2a + 1/2\\
    y&=1/2xa^4 + 2xa^3 + 3xa^2 + 2xa + 1/2x - 1/4a^9 - 3/2a^8 - 5a^7 - 25/2a^6 - 49/2a^5\\
    &- 73/2a^4 - 39a^3 - 55/2a^2 - 45/4a - 2\\
    z&=a^2 + 2a + 1,
\end{align*}
where $z$ is the homogenizing variable of $H$.
The curve $H$ has genus $2$ and its Jacobian has rank 0 with torsion subgroup isomorphic to $\Z/6\Z \times \Z/6\Z$. Using Chabauty's method for rank 0 Jacobians as implemented in Magma yields the six rational points on (the weighted projective closure of) $H$ as
\begin{equation*}
   (1 : -2 : 0), (1 : 2 : 0), (1 : 2 : 1), (1 : -2 : 1), (0 : 2 : 1), (0 : -2 : 1) .
\end{equation*}
A point search up to height bound 1000 yields the points on (the projective closure of) $Y$:
\begin{equation*}
    (z,a,h) \in \{(4 : 1 : 0), (0 : -1 : 1), (4 : 1 : 1), (1 : 0 : 0), (36 : 1 : 1)\}.
\end{equation*}
The point $(0:-1:1)$ is a singular point which blows up to two rational points, so these are all the rational points on $Y$. The affine points all have $a = \pm 1$, which is the degenerate case for this family. So there are no points that corresponds to a rational $3$-cycle.
\end{proof}
\begin{code}
\begin{lstlisting}[language=Python]
#First dynatomic polynomial
R.<a> = QQ[]
P.<z> = R[]
P = FractionField(P)
A = AffineSpace(P,1)
g = DynamicalSystem_affine((a*z^2-1)/(z^3-a*z))
g.dynatomic_polynomial(1)

#Second dynatomic polynomial and its irreducible components
R.<a> = QQ[]
P.<z> = R[]
P = FractionField(P)
A = AffineSpace(P,1)
g = DynamicalSystem_affine((a*z^2-1)/(z^3-a*z))
g.dynatomic_polynomial(2)

A.<a,z>=AffineSpace(QQ,2)
R=A.coordinate_ring()
I=R.ideal((-a^2 + 1)*z^5 + (a^2 - 1)*z)
S=A.subscheme(I.gens())
S.irreducible_components()

#Third dynatomic polynomial
R.<a> = QQ[]
P.<z> = R[]
P = FractionField(P)
A = AffineSpace(P,1)
g = DynamicalSystem_affine((a*z^2-1)/(z^3-a*z))
g.dynatomic_polynomial(3)

#Quotienting by C_3 (run in Magma)
Fc<a>:=FunctionField(Rationals());
R<z>:=PolynomialRing(Fc);
Phi3:=-z^24 + (a^5 + 4*a^3 + 7*a)*z^22 + (-a^8 - 13*a^6 - 24*a^4 - 29*a^2 + 1)*z^20 + (7*a^9 + 48*a^7 + 91*a^5 + 76*a^3 - 2*a)*z^18 + (-11*a^10 - 102*a^8 - 214*a^6 - 168*a^4 + a^2 - 1)*z^16 + (6*a^11 + 103*a^9 + 348*a^7 + 304*a^5 + 30*a^3 + a)*z^14 + (-a^12 - 44*a^10 - 289*a^8 - 470*a^6 - 111*a^4 - 10*a^2 + 1)*z^12 + (6*a^11 + 103*a^9 + 348*a^7 + 304*a^5 + 30*a^3 + a)*z^10 + (-11*a^10 - 102*a^8 - 214*a^6 - 168*a^4 + a^2 - 1)*z^8 + (7*a^9 + 48*a^7 + 91*a^5 + 76*a^3 - 2*a)*z^6 + (-a^8 - 13*a^6 - 24*a^4 - 29*a^2 + 1)*z^4 + (a^5 + 4*a^3 + 7*a)*z^2 - 1;
Fcz<z>:=ext<Fc | Phi3>;
t:=z+((a*z^2-1)/(z^3-a*z))+((a^3*z^7 - a^4*z^5 - z^9 + 3*a*z^7 - 5*a^2*z^5 + 3*a^3*z^3 + a*z^3 - a^2*z)/(-a^2*z^8 + 3*a^3*z^6 - a^4*z^4 + a*z^6 - 5*a^2*z^4 + a^3*z^2 + 3*a*z^2 - 1));
t2:=1/z + ((z^3-a*z)/(a*z^2-1)) + ((-a^2*z^8 + 3*a^3*z^6 - a^4*z^4 + a*z^6 - 5*a^2*z^4 + a^3*z^2 + 3*a*z^2 - 1)/(a^3*z^7 - a^4*z^5 - z^9 + 3*a*z^7 - 5*a^2*z^5 + 3*a^3*z^3 + a*z^3 - a^2*z));
MinimalPolynomial(t+t2);

#Genus of curve
A<z,a>:=AffineSpace(Rationals(),2);
Y:=Curve(A,z^4 + (-a^5 - 2*a^4 - 6*a^3 - 10*a^2 - 13*a - 8)*z^2 + 4*a^6 + 12*a^5 + 25*a^4 + 36*a^3 + 34*a^2 + 24*a + 9);
Genus(Y);

#Simplified Model (run in Magma)
A<z,a>:=AffineSpace(Rationals(),2);
Y:=Curve(A,z^2 + (-a^5 - 2*a^4 - 6*a^3 - 10*a^2 - 13*a - 8)*z + 4*a^6 + 12*a^5 + 25*a^4 + 36*a^3 + 34*a^2 + 24*a + 9);
b,YY,psi:=IsHyperelliptic(Y);
YYY,psi2:=SimplifiedModel(YY);
YYY;
Expand(psi*psi2);
#change of coordinate in Sage
(4*(1/2*x)^6 - 12*(1/2*x)^5 + 25*(1/2*x)^4 - 30*(1/2*x)^3 + 25*(1/2*x)^2 - 12*(1/2*x) + 4)*(16)

#Rational points on H
A<x,y>:=AffineSpace(Rationals(),2);
Y:=Curve(A,y^2-(4*x^6 + (-12)*x^5 + 25*x^4 + (-30)*x^3 + 25*x^2 + (-12)*x + 4));
b,YY:=IsHyperelliptic(Y);
YY,psi:=SimplifiedModel(YY);
psi;
YY;
J:=Jacobian(YY);
RankBounds(J);
TorsionSubgroup(J);
Chabauty0(J);

A<z,a>:=AffineSpace(Rationals(),2);
Y:=Curve(A,z^2 + (-a^5 - 2*a^4 - 6*a^3 - 10*a^2 - 13*a - 8)*z + 4*a^6 + 12*a^5 + 25*a^4 + 36*a^3 + 34*a^2 + 24*a + 9);
YP:=ProjectiveClosure(Y);
Points(YP:Bound:=1000);
SingularPoints(Y);
Multiplicity(Y![0,-1]);

\end{lstlisting}
\end{code}

A search for rational preperiodic structures with the parameter up to height $10,000$ using the algorithm from \cite{Hutz12} as implemented in Sage yields no parameters where $g_a$ has a $\Q$-rational periodic point with minimal period at least $4$.

\begin{conj} \label{conj_A3_D2_2}
    There are no $a \in \Q$ such that $g_a(z) = \frac{az^2-1}{z^3-az}$ has a $\Q$-rational periodic point of minimal period at least $4$.
\end{conj}
Assuming Conjecture \ref{conj_A3_D2_2}, we classify all rational preperiodic structures.
\begin{theorem} \label{thm_A3_D2_2_preperiodic}
    Assuming Conjecture \ref{conj_A3_D2_2}, the possible rational preperiodic structures for $g_a(z) = \frac{az^2-1}{z^3-az}$ for $a \in \Q$ are the following.
    \begin{align*}
G_1&=\xygraph{
!{<0cm,0cm>;<1cm,0cm>:<0cm,1cm>::}
!{(0,0) }*+{\underset{1}{\bullet}}="a"
!{(1.5,0) }*+{\underset{-1}{\bullet}}="b"
!{(2.5,0) }*+{\underset{0}{\bullet}}="c"
!{(4,0) }*+{\underset{\infty}{\bullet}}="d"
"a":@/^1pc/"b"
"b":@/^1pc/"a"
"c":@/^1pc/"d"
"d":@/^1pc/"c"
}, \quad \text{for $t$ not in the following cases}\\
G_2&=\xygraph{
!{<0cm,0cm>;<1cm,0cm>:<0cm,1cm>::}
!{(0,0) }*+{\underset{1}{\bullet}}="a"
!{(1.5,0) }*+{\underset{-1}{\bullet}}="b"
!{(2.5,0) }*+{\underset{0}{\bullet}}="c"
!{(4,0) }*+{\underset{\infty}{\bullet}}="d"
!{(5,0) }*+{\underset{t}{\bullet}}="e"
!{(6,0) }*+{\underset{-t}{\bullet}}="f"
!{(7,0) }*+{\underset{1/t}{\bullet}}="g"
!{(8,0) }*+{\underset{-1/t}{\bullet}}="h"
"a":@/^1pc/"b"
"b":@/^1pc/"a"
"c":@/^1pc/"d"
"d":@/^1pc/"c"
"e":@(rd,ru)"e"
"f":@(rd,ru)"f"
"g":@(rd,ru)"g"
"h":@(rd,ru)"h"
}, \quad a = \frac{t^4+1}{2t^2} \text{ for } t \in \Q\setminus \{\pm 1, 0\}\\
G_{3a}&=\xygraph{
!{<0cm,0cm>;<1cm,0cm>:<0cm,1cm>::}
!{(0,0) }*+{\underset{1}{\bullet}}="a"
!{(1.5,0) }*+{\underset{-1}{\bullet}}="b"
!{(3,0) }*+{\underset{0}{\bullet}}="c"
!{(4.5,0) }*+{\underset{\infty}{\bullet}}="d"
!{(2,-0.5) }*+{\bullet}="c1"
!{(2,0.5) }*+{\bullet}="c2"
!{(5.5,-0.5) }*+{\bullet}="d1"
!{(5.5,0.5) }*+{\bullet}="d2"
"a":@/^1pc/"b"
"b":@/^1pc/"a"
"c":@/^1pc/"d"
"d":@/^1pc/"c"
"c1":"c"
"c2":"c"
"d1":"d"
"d2":"d"
}, \quad a = t^2, \text{ for } t \in \Q \setminus \{0\}\\
G_{3b}&=\xygraph{
!{<0cm,0cm>;<1cm,0cm>:<0cm,1cm>::}
!{(1,0) }*+{\underset{1}{\bullet}}="a"
!{(2.5,0) }*+{\underset{-1}{\bullet}}="b"
!{(4.5,0) }*+{\underset{0}{\bullet}}="c"
!{(6,0) }*+{\underset{\infty}{\bullet}}="d"
!{(0,-0.5) }*+{\bullet}="c1"
!{(0,0.5) }*+{\bullet}="c2"
!{(3.5,-0.5) }*+{\bullet}="d1"
!{(3.5,0.5) }*+{\bullet}="d2"
"a":@/^1pc/"b"
"b":@/^1pc/"a"
"c":@/^1pc/"d"
"d":@/^1pc/"c"
"c1":"a"
"c2":"a"
"d1":"b"
"d2":"b"
}, \quad a=\frac{3t^2 + 3t + 3}{2t^2 + 5t + 2}, t \in \Q\setminus \{-2,-1/2\}.
\end{align*}
\end{theorem}
\begin{proof}
We start with non-periodic preimages of fixed points. Recall from Proposition \ref{prop_A3_D2_2} that we have (four) fixed points when $a=\frac{t^4+1}{2t^2}$ for some $t \in \Q\setminus \{\pm 1, 0\}$. We want to know if these points can have rational preimages. Consider the equation $g_a(x) = t$ and substitute $a=\frac{t^4+1}{2t^2}$. The resulting equation factors as
\begin{equation*}
    -(x - t)(2t^3x^2 + (t^4 - 1)x - 2t).
\end{equation*}
Since $t$ is the fixed point we need only consider the second factor component. Its vanishing defines a genus 3 curve but is quadratic in $x$, so we have rational solutions when the discriminant, with respect to $x$, is a square. This gives the curve
\begin{equation*}
    t^8 + 14t^4 + 1 = k^2.
\end{equation*}
Replacing $u = t^2$, this becomes
\begin{equation*}
    u^4 + 14u^2 + 1 = k^2.
\end{equation*}
Using the point $(1:0:0)$ as the point at infinity, this is an elliptic curve with model $y^2 = x^3 - 28x^2 + 192x$. This curve has rank 0 and torsion isomorphism to $\Z/2\Z \times \Z/4\Z$. So the original curve has at most eight rational points. A point search up to height 1000 yields the seven points (on the projective closure)
\begin{equation*}
    (x,t) \in \{(\pm 4 : \pm 1 : 1), (\pm 1 : 0 : 1), (1 : 0 : 0)\}.
\end{equation*}
The point $(1:0:0)$ is singular and blows up to two rational points, so we have found all eight $\Q$-rational points. The affine points all have $t$ values that are degenerate, so there are no $a \in \Q$ such that $g_a(z)$ has a rational fixed point with a non-periodic rational preimage.

Now we look for non-periodic rational preimages of the points of period $2$. Recall that every member of the family has exactly two $2$-cycles $\{0,\infty\}$ and $\{1,-1\}$. We want to know if we can have $g_a(z_1)=0$ or $g_a(z_2)=1$. It suffices to look for non-periodic rational preimages of $0$ and $1$ since preimages of $-1$ and $\infty$ are then obtained from the same automorphisms that produce four rational fixed points when there is one. We first look at $g_a(z_1)=0$. We need to find rational solutions to the equation $az^2-1=0.$ This equation defines a genus $0$ curve with rational parameterization
\begin{equation*}
    t\mapsto (t^2,1/t)=(a,z).
\end{equation*}
Thus, for every $t\in \Q\setminus \{0\}$, $g_a(1/t)=0$, where $a=t^2$.

Now we look at $g_a(z_2)=1$. We need to find rational solutions to the equation $(az^2-1)-(z^3-az)=0.$ This equation defines a reducible curve over $\Q$ and the irreducible components are
\begin{align*}
    &z+1=0\\
    &az - z^2 + z - 1=0.
\end{align*}
This first component corresponds to a periodic point, so we only look at the second component. It defines a genus $0$ curve with rational parameterization:
\begin{equation*}
    t\mapsto \left(\frac{3t^2 + 3t + 3}{2t^2 + 5t + 2}, \frac{t + 2}{2t + 1}\right).
\end{equation*}
Thus, for every $t\in \Q$, we can find $a=\frac{3t^2 + 3t + 3}{2t^2 + 5t + 2}$ such that $1$ has a non-periodic rational preimage under $g_a$.

Now we look for non-periodic rational second preimages of $0$; i.e., determine if we can have $g_a^2(z_1)=0$ and $g_a(z_1)\neq 0$ or $\infty$. We need to find rational solutions to the equation
\begin{equation*}
    a^3z^7 - a^4z^5 - z^9 + 3az^7 - 5a^2z^5 + 3a^3z^3 + az^3 - a^2z=0.
\end{equation*}
This equation defines a reducible curve and the irreducible components are
\begin{align*}
    &z=0\\
    &z^2 - a=0\\
    &a^3z^4 - z^6 + 2az^4 - 3a^2z^2 + a=0.
\end{align*}
The first component corresponds to periodic points. The second component corresponds to $z$ such that $g_a(z)=\infty$. Thus, we can just look at the third component. Via the replacement $z^2 \mapsto z$, it covers an elliptic curve:
\begin{equation*}
    X:a^3z^2 - z^3 + 2az^2 - 3a^2z + a=0.
\end{equation*}
The elliptic curve $X$ is birationally equivalent to
\begin{equation*}
    E:y^2 - 8xy - 2y = x^3 - 12x^2 - 4x,
\end{equation*}
via the map
\begin{align*}
    x&=2a^2z^4 + 4a^2z^3h + 4az^4h - 6a^2z^2h^2 -
    8az^3h^2 + 8z^4h^2 - 12az^2h^3\\
    &- 4z^3h^3 +
    16azh^4 - 10z^2h^4 + 8zh^5 - 2h^6\\
    y&=4a^2z^4 + 8az^4h - 36a^2z^2h^2 - 32az^3h^2 +
    12z^4h^2 - 8az^2h^3 - 32z^3h^3\\
    &+ 96azh^4 -
    36z^2h^4 + 32zh^5 - 8h^6\\
    z&=z^5h - 3z^4h^2 + 2z^3h^3 + 2z^2h^4 - 3zh^5 + h^6,
\end{align*}
where $h$ and $z$ are the homogenizing variables of $X$ and $E$, respectively.
This elliptic curve $E$ has rank $0$ and its torsion subgroup has order $6$. Searching using height bound $1000$, we find five rational points on the projective closure of $X$:
\begin{equation*}
    (a,z,h) \in \{1: 1: 1), (0: 1: 0), (0: 0: 1), (1: 0: 0), (-1: -1: 1)\},
\end{equation*}
where $h$ is a homogenizing variable.
Among these five points, $(-1: -1: 1)$, $(1: 1:1)$, $(0: 1: 0)$, and $(1: 0: 0)$  are singular points. The point $(-1:-1:1)$ blows up to be two rational points, point $(1: 1:1)$ blows up to be two rational points, point $(0:1:0)$ blows up to be one rational point, and point $(1:0:0)$ blows up to be a closed point that is not rational. Thus, our search has found all the rational points on $X$. The only affine rational point such that $a\neq \pm 1$ is $(0:0:1)$. But this corresponds to $z=0$, which is the periodic point. Thus, there is no $a \in \Q$ such that $g_a^2(z_1)=0$ and $g_a(z_1)\neq 0$ or $\infty$ for rational $z_1$.

Now we look for non-periodic rational second preimages of $1$; i.e., rational $z_2$ so that $g^2_a(z_2)=1$ and $g_a(z_2)\neq 1$ or $-1$. We need to find rational solutions to the equation
\begin{multline*}
    a^3z^7 - a^4z^5 - z^9 + 3az^7 - 5a^2z^5 + 3a^3z^3 + az^3 - a^2z\\
    -(-a^2z^8 + 3a^3z^6 - a^4z^4 + az^6 - 5a^2z^4 + a^3z^2 + 3az^2 - 1)=0.
\end{multline*}
This equation defines a reducible curve over $\Q$ and the irreducible components are
\begin{align*}
    &z-1=0\\
    &az + z^2 + z + 1=0\\
    &a^2z^5 - a^3z^3 - a^2z^4 + az^5 - z^6 - a^2z^3 + 2az^4 - a^2z^2 - az^3 + a^2z + 2az^2 - z^3 + az - 1=0.
\end{align*}
The first component corresponds to $z=1$ and the second component to $g_a(z)=-1$. Thus, we focus on the third component. It defines a genus $3$ curve. We can quotient by a $C_2$ action by identifying $(a,z)$ with $(a,1/z)$. This gives curve
\begin{equation*}
    X:t^3 + (-a^2 - a)t^2 + (a^2 - 2a - 3)t + a^3 + 3a^2 + 3a + 1=0,
\end{equation*}
where $t = z + \frac{1}{z}$.
This is a genus 1 curve birational to the elliptic curve
\begin{equation*}
    E:y^2 - 2xy + 2y = x^3 - 3x^2 + 2x,
\end{equation*}
via the map
\begin{align*}
    x &= 2at^4 - 2a^2t^2h - 8at^3h + 8a^2th^2 +
    4at^2h^2 + 8t^3h^2 + 16ath^3 - 26t^2h^3 +
    8th^4\\
    y&=2at^4 - 4t^4h - 2a^3h^2 - 2a^2th^2 -
    26at^2h^2 + 14t^3h^2 + 26a^2h^3 + 36ath^3 +
    6t^2h^3\\
    &+ 58ah^4 - 90th^4 + 30h^5\\
    z&=t^5 - 2t^4h,
\end{align*}
where $h$ and $z$ are the homogenizing variables of $X$ and $E$, respectively.
The curve $E$ has rank $0$ and its torsion subgroup has six elements. A search for rational points on the projective closure of $X$ using height bound $1000$ finds
\begin{equation*}
    (a,t,h) \in \{(1: -2: 1), (1: 2: 1), (0: 1: 0), (-1: 0: 1), (1: 0: 0), (-3:2:1)\},
\end{equation*}
where $h$ is the homogenizing variable.
Among these points, $(-1: 0: 1)$ and $(1: 2: 1)$ are singular points. The first one blows up to be two rational points, and the second one blows up to be a closed point that is not rational. Therefore, we have found all the rational points on $X$. Observe that the only affine rational point where $a\neq \pm 1$ is $(-3:2:1)$. Thus, we need to solve $t=z+1/z=2$. The only solution is $z=1$. Thus, there is no $a$ such that $g_a^2(z_2)=1$ and $g_a(z_2)\neq 1$ or $-1$ for rational $z_2$.

Next we check if both pairs of 2-cycles can have preperiodic tails at the same time; i.e., if we can have $g_a(z_1)=0$ and $g_a(z_2)=1$. Recall that we need $a=t_1^2$ to have $g_a(z_1)=0$ and $a=\frac{3t_2^2+3t_2+3}{2t_2^2+5t_2+2}$ to have $g_a(z_2)=1$. Thus, we need to find rational solutions to
\begin{equation*}
    C:(3t_2^2 + 3t_2 + 3)-t_1^2(2t_2^2 + 5t_2 + 2)=0.
\end{equation*}
This defines a genus $1$ curve that is birationally equivalent to the elliptic curve
\begin{equation*}
    E:y^2 = x^3 + 180x^2 + 116646x + 279936
\end{equation*}
via the map
\begin{align*}
    x &= 54t_2h\\
    y &= 108t_1t_2 + 216t_1h\\
    z &= -t_2h - 1/2h^2,
\end{align*}
where $h$ and $z$ are the homogenizing variables of $C$ and $E$, respectively.
The curve $E$ has rank $0$ and its torsion subgroup has order $4$. Searching for rational points on the projective closure of $X$ using height bound $1000$ produces
\begin{equation*}
    (t_1,t_2,h) \in \{(0: 1: 0), (1: 0: 0), (1: 1: 1), (-1: 1: 1)\},
\end{equation*}
where $h$ is the homogenizing variable.
Among these points, $(0: 1: 0)$ and $(1: 0: 0)$ are singular points. The first one blows up to be two rational points, where as the second one blows up to be only one closed point that is not rational. Thus, we have found all rational points on $X$. Checking whether these points produce valid members of the family, we see that $t_2=1$ so that $a=t_2^2=1$, which produces degeneracy. Thus, we cannot have $g_a(z_1)=0$ and $g_a(z_2)=1$ at the same time for rational $z_1$ and $z_2$.

Now we need to ask if we can have rational fixed points at the same time as preperiodic tails for a 2-cycle. Recall that we have rational fixed points if $a=\frac{t^4+1}{2t^2}$ and we have $g_a(z)=0$ if $a=u^2$ and we have $g_a(z)=1$ if $a=\frac{3v^2+3v+3}{2v^2+5v+2}$. We first study the curve
\begin{equation*}
    X:t^4+1-2u^2t^2=0.
\end{equation*}
This defines a genus 1 curve birational to the elliptic curve
\begin{equation*}
  E:y^2 + 2xy = x^3 - 4x^2 + 2x
\end{equation*}
via the map
\begin{align*}
    x&=2t^3 + 2t^2u - 4t^2h - 2tuh + 4th^2 - 2h^3\\
    y&=-4t^3 - 4t^2u + 4t^2h - 4th^2\\
    z&=t^3 - 3t^2h + 3th^2 - h^3
\end{align*}
where $h$ and $z$ are the homogenizing variables of $X$ and $E$, respectively.
The curve $E$ has rank $0$ and $4$ rational torsion points. Searching up to height $1000$ on the projective closure of $X$, we find
\begin{equation*}
    (t,u,h) \in \{(1 : 1 : 1), (0 : 1 : 0), (-1 : 1 : 1), (1 : -1 : 1), (-1 : -1 : 1)\}.
\end{equation*}
The point $(0:1:0)$ is a singular point and does not blow up to a rational point. Thus, we have found all rational points on $X$. Note that all the affine points have $t \in\{0, \pm 1\}$ so correspond to either periodic $t=0$ or to degenerate $t = \pm 1$. Thus, we cannot have a rational preperiodic tail for $0$ and rational fixed points at the same time. Now we study the curve
\begin{align*}
    X:&(t^4+1)(2v^2 + 5v + 2)-(3v^2 + 3v + 3)(2t^2)=0
\end{align*}
corresponding to a rational preperiodic tail for $1$ and rational fixed points at the same time.
This curve is genus 1 and is  birational to the elliptic curve
\begin{equation*}
    E: y^2 = x^3 + 9x^2 -54x -216
\end{equation*}
via the map
\begin{align*}
    x&=9t^2vh + 18t^2h^2 - 15vh^3 - 12h^4\\
    y&=27t^3v + 54t^3h - 81tvh^2\\
    z&=-vh^3 + h^4,
\end{align*}
where $h$ and $z$ are the homogenizing variables of $X$ and $E$, respectively.
The curve $E$ has rank $0$, and $4$ rational torsion points. Searching using height bound $1000$ on the projective closure of $X'$ produces:
\begin{equation*}
    (t,z,h) \in \{(1 : 1 : 1), (0 : 1 : 0), (0 : -2 : 1), (0 : -1/2 : 1), (1 : 0 : 0), (-1 : 1 : 1) \}.
\end{equation*}
The singular points are
\begin{equation*}
    \{(-1 : 1 : 1), (1 : 1 : 1), (0 : 1 : 0), (1 : 0 : 0)\}.
\end{equation*}
The points $\{(-1 : 1 : 1), (1 : 1 : 1), (0 : 1 : 0)\}$ blow up to closed points that are not rational. The point $(1:0:0)$ blows up to two rational points. Thus, we have found all the rational points on $X$. The only affine rational points have $t\in \{0, \pm 1\}$ which are exclude in this case. Therefore, we cannot have a rational preperiodic tail of $1$ and rational fixed points at the same time.
\end{proof}

\begin{code}
\begin{lstlisting}[language=Python]
#Preimages of fixed points
r.<t>=QQ[]
P.<x,y>=ProjectiveSpace(r,1)
a=(t^4+1)
b=(2*t^2)
f=DynamicalSystem([(a*x^2*y-b*y^3),(b*x^3-a*x*y^2)])
F=f.dehomogenize(1)
(F[0].numerator() - F[0].denominator()*t).factor()

A<x,t>:=AffineSpace(Rationals(),2);
C:=ProjectiveClosure(Curve(A,t^4 + 14*t^2 + 1 - x^2));
P:=C![1,0,0];
E:=EllipticCurve(C,P);
E;
RankBounds(E);
TorsionSubgroup(E);
Points(C:Bound:=1000);
Q1:=Places(C![1,0,0]);
Q1;
Degree(Q1[1]);
Degree(Q1[2]);

#Preimages of 0 or 1
R.<a,z>=QQ[]
A=AffineSpace(R)
I=R.ideal(a*z^2-1)
C=A.curve(I.gens())
C.genus()
C.rational_parameterization()

R.<a,z>=QQ[]
A=AffineSpace(R)
I=R.ideal((a*z^2-1)-(z^3-a*z))
S=A.subscheme(I)
S.irreducible_components()

#Preimage of preimages of 0
R.<a,z>=QQ[]
A=AffineSpace(R)
I=R.ideal((a^3*z^7 - a^4*z^5 - z^9 + 3*a*z^7 - 5*a^2*z^5 + 3*a^3*z^3 + a*z^3 - a^2*z))
S=A.subscheme(I)
S.irreducible_components()

#rational points (run in magma)
A<a,z>:=AffineSpace(Rationals(),2);
C:=ProjectiveClosure(Curve(A,a^3*z^2 - z^3 + 2*a*z^2 - 3*a^2*z + a));
Points(C:Bound:=1000);
P:=C![1,1,1];
E:=EllipticCurve(C,P);
E;
RankBounds(E);
TorsionSubgroup(E);
sing:=SingularPoints(C);
sing;
P1:=sing[1];
P2:=sing[2];
P3:=sing[3];
P4:=sing[4];
Q1:=Places(P1);
Q1;
Degree(Q1[1]);
Degree(Q1[2]);
Q2:=Places(P2);
Q2;
Degree(Q2[1]);
Degree(Q2[2]);
Q3:=Places(P3);
Q3;
Degree(Q3[1]);
Q4:=Places(P4);
Q4;
Degree(Q4[1]);

#Preimage of preimages of 1
R.<a,z>=QQ[]
A=AffineSpace(R)
I=R.ideal((a^3*z^7 - a^4*z^5 - z^9 + 3*a*z^7 - 5*a^2*z^5 + 3*a^3*z^3 + a*z^3 - a^2*z)-(-a^2*z^8 + 3*a^3*z^6 - a^4*z^4 + a*z^6 - 5*a^2*z^4 + a^3*z^2 + 3*a*z^2 - 1))
S=A.subscheme(I)
S.irreducible_components()
#Checking the second component
R.<a,z>=QQ[]
A=AffineSpace(R)
I=R.ideal((a*z^2-1)+(z^3-a*z))
S=A.subscheme(I)
S.irreducible_components()

#Rational Points (run in magma)
Fc<a>:=FunctionField(Rationals());
R<z>:=PolynomialRing(Fc);
Phi3:=a^2*z^5 - a^3*z^3 - a^2*z^4 + a*z^5 - z^6 - a^2*z^3 + 2*a*z^4 - a^2*z^2 - a*z^3 + a^2*z + 2*a*z^2 - z^3 + a*z - 1;
Fcz<z>:=ext<Fc | Phi3>;
t:=z+1/z;
MinimalPolynomial(t);

A<a,z>:=AffineSpace(Rationals(),2);
C:=ProjectiveClosure(Curve(A,z^3 + (-a^2 - a)*z^2 + (a^2 - 2*a - 3)*z + a^3 + 3*a^2 + 3*a + 1));
Points(C:Bound:=1000);
sing:=SingularPoints(C);
sing;
P1:=sing[1];
P2:=sing[2];
Q1:=Places(P1);
Q1;
Degree(Q1[1]);
Degree(Q1[2]);
Q2:=Places(P2);
Degree(Q2[1]);
P:=C![-1,0,1];
E:=EllipticCurve(C,P);
E;
RankBounds(E);
TorsionSubgroup(E);

#Preimages of 0 and 1 at the same time (run in magma)
A<t1,t2>:=AffineSpace(Rationals(),2);
C:=ProjectiveClosure(Curve(A,(3*t2^2 + 3*t2 + 3)-t1^2*(2*t2^2 + 5*t2 + 2)));
Points(C:Bound:=1000);
P:=C![1,0,0];
E:=EllipticCurve(C,P);
E;
RankBounds(E);
TorsionSubgroup(E);
sing:=SingularPoints(C);
sing;
P1:=sing[1];
P2:=sing[2];
Q1:=Places(P1);
Q1;
Degree(Q1[1]);
Q2:=Places(P2);
Q2;
Degree(Q2[1]);
Degree(Q2[2]);

#Preimages of 0 and extra fixed points
A<t,u>:=AffineSpace(Rationals(),2);
C:=ProjectiveClosure(Curve(A,t^4+1-2*u^2*t^2));
Points(C:Bound:=1000);
P:=C![1,1,1];
E:=EllipticCurve(C,P);
E;
RankBounds(E);
TorsionSubgroup(E);
sing:=SingularPoints(C);
sing;
P1:=sing[1];
Q1:=Places(P1);
Q1;
Degree(Q1[1]);

#Preimages of 1 and extra fixed points
A<t,v>:=AffineSpace(Rationals(),2);
C:=ProjectiveClosure(Curve(A,(t^4+1)*(2*v^2 + 5*v + 2)-(3*v^2 + 3*v + 3)*(2*t^2)));
Points(C:Bound:=1000);
P:=C![1,0,0];
E:=EllipticCurve(C,P);
E;
RankBounds(E);
TorsionSubgroup(E);
sing:=SingularPoints(C);
sing;
P1:=sing[1];
P2:=sing[2];
P3:=sing[3];
P4:=sing[4];
Q1:=Places(P1);
Q1;
Degree(Q1[1]);
Q2:=Places(P2);
Q2;
Degree(Q2[1]);
Q3:=Places(P3);
Q3;
Degree(Q3[1]);
Degree(Q3[2]);
Q4:=Places(P4);
Q4;
Degree(Q4[1]);
Degree(Q4[2]);
\end{lstlisting}
\end{code}

\subsection{$C_2$ First Component} \label{sect_A3_C2_2_preperiodic}
We consider the {\it two}-parameter family $f_{a,b}(z) = \frac{z^3 + az}{bz^2 + 1}$.

We start by examining the fixed points and 2-periodic points.
\begin{prop} \label{prop_A3_C2_2_periodic}
    For $f_{a,b}(z) = \frac{z^3 + az}{bz^2 + 1}$, we have the following periodic points.
    \begin{enumerate}
        \item The points $0$ and $\infty$ are always fixed. There are two additional $\Q$-rational fixed points for pairs $(a,b)$ parameterized by
            \begin{equation*}
                (a,t) = \left(a, 1+\frac{a-1}{t^2}\right) \quad t \neq 0,\; t^2 \neq -a.
            \end{equation*}
            The two additional $\Q$-two rational fixed points are given by $z = \pm t$.

        \item The pairs $(a,b)$ for which $f_{a,b}$ has a 2-cycle with $\Q$-rational points is parameterized by the union of two surfaces.
        \begin{align*}
            S_1&: (a,t) = \left(a, -1 + \frac{-1-a}{t^2}\right) \quad t \neq 0,\; t^2 \neq -a \\
            S_2&: (a,t) = \left(a, -\frac{1}{t^2} -a - t^2\right) \quad t \neq 0,\; t^2 \neq -a, \; t^2 \neq - \frac{1}{a}.
        \end{align*}
        For $S_1$ there is one 2-cycle with points $z = \pm t$. For $S_2$ there are two 2-cycles with points given by the pairs \begin{equation*}
            z \in \left\{t, -\frac{1}{t}\right\} \quad z \in \left\{ -t, \frac{1}{t}\right\}.
        \end{equation*}
        These two surfaces intersect on the line $a+b=-2$ and on this line there is one $2$-cycle with multiplicity $3$ and points $z = \pm 1$.
    \end{enumerate}
\end{prop}
\begin{proof}
    We first look at the fixed points. We compute the first dynatomic polynomial as
    \begin{equation*}
        \Phi^{\ast}_1(f_{a,b}) = -bz^3 + z^3 + az - z.
    \end{equation*}
    After removing the component $z=0$, we set $z=t$ and solve for $b$ in terms of $a$. This is a rational surface with parameterization given in the statement. To see that there are exactly two fixed points, we substitute the parameterizations of $a$ and $b$ into $\Phi^{\ast}_1(f_{a,b})$ to get
    \begin{equation*}
        \Phi^{\ast}_1(f_{a,b}) = z(a - 1)(t - z)(t + z).
    \end{equation*}
    From this we see the two new $\Q$-rational fixed points are $z = \pm t$.
\begin{code}
\begin{lstlisting}
r.<a,b>=QQ[]
P.<x,y>=ProjectiveSpace(r,1)
f=DynamicalSystem([x^3 + a*x*y^2, b*x^2*y + y^3])
phi=f.dehomogenize(1).dynatomic_polynomial(1)
s.<a,b,x>=QQ[]
phis=s(phi)
phis.factor()

R.<a,t>=QQ[]
P.<x,y>=ProjectiveSpace(R,1)
b=1+(a-1)/t^2
f=DynamicalSystem([x^3 + a*x*y^2, b*x^2*y + y^3])
f.resultant().factor()

f.dynatomic_polynomial(1)*t^2
R.<a,t,x,y>=QQ[]
((-a + 1)*x^3*y + (a*t^2 - t^2)*x*y^3).factor()
\end{lstlisting}
\end{code}

    Now we consider the second periodic points via the second dynatomic polynomial
    \begin{equation*}
        \Phi^{\ast}_2(f_{a,b}) = (bz^2 + z^2 + a + 1)(z^4 + az^2 + bz^2 + 1).
    \end{equation*}
    Each component gives the rational surface parameterized by the given parameterization obtained by setting $z=t$ and solving for $b$ in terms of $a$. The points of period two are obtained by factoring the dynatomic polynomial after substituting in the parameterization of $a$ and $b$. Their intersection contains two curves: $a+b=-2$ and $ab=1$, but the second curve is the degenerate pairs $(a,b)$.

\begin{code}
\begin{lstlisting}
r.<a,b>=QQ[]
P.<x,y>=ProjectiveSpace(r,1)
f=DynamicalSystem([x^3 + a*x*y^2, b*x^2*y + y^3])
phi=f.dehomogenize(1).dynatomic_polynomial(2)
s.<a,b,x>=QQ[]
phis=s(phi)
phis.factor()

R.<a,t>=QQ[]
P.<x,y>=ProjectiveSpace(R,1)
b=-1+(-1-a)/t^2
f=DynamicalSystem([x^3 + a*x*y^2, b*x^2*y + y^3])
f.resultant().factor()

R.<a,t>=QQ[]
P.<x,y>=ProjectiveSpace(R,1)
b=-1/t^2 - a - t^2
f=DynamicalSystem([x^3 + a*x*y^2, b*x^2*y + y^3])
f.resultant().factor()


# 2-cycle points
R.<a,t>=QQ[]
P.<x,y>=ProjectiveSpace(R,1)
b=-1/t^2 - a - t^2
f=DynamicalSystem([x^3 + a*x*y^2, b*x^2*y + y^3])
phi=f.dehomogenize(1).dynatomic_polynomial(2)
s.<a,t,x>=QQ[]
phis=(-t^8 - a*t^6 + t^6 - t^4)*x^6 + (t^10 + a*t^8 - t^8 + a*t^6 + 3*t^6 + a*t^4 - t^4 + t^2)*x^4 + (-a*t^8 - 2*t^8 - a*t^6 + t^6 - a*t^4 - 2*t^4)*x^2 + a*t^6 + t^6
phis.factor()
f(P(t))

#intersection
A.<a,b,x>=AffineSpace(QQ,3)
X=A.subscheme([(b*x^2 + x^2 + a + 1), (x^4 + a*x^2 + b*x^2 + 1)])
X.irreducible_components()

r.<a>=QQ[]
P.<x,y>=ProjectiveSpace(r,1)
f=DynamicalSystem([x^3 + a*x*y^2, (-2-a)*x^2*y + y^3])
phi=f.dehomogenize(1).dynatomic_polynomial(2)
s.<a,x>=QQ[]
phis=s(phi)
phis.factor()

\end{lstlisting}
\end{code}
\end{proof}

        Periodic points with higher periods and strictly preperiodic points are difficult to study for this family mainly because computational tools for rational points on surfaces is much less well developed than for curves. Consequently, we content ourselves with a census of $\Q$-rational preperiodic structures for parameters $a$ and $b$ in $\Q$ with small height. By no means do we think this census is exhaustive; rather, it gives a sense of the diversity of possibilities when there are only $C_2$ symmetries.

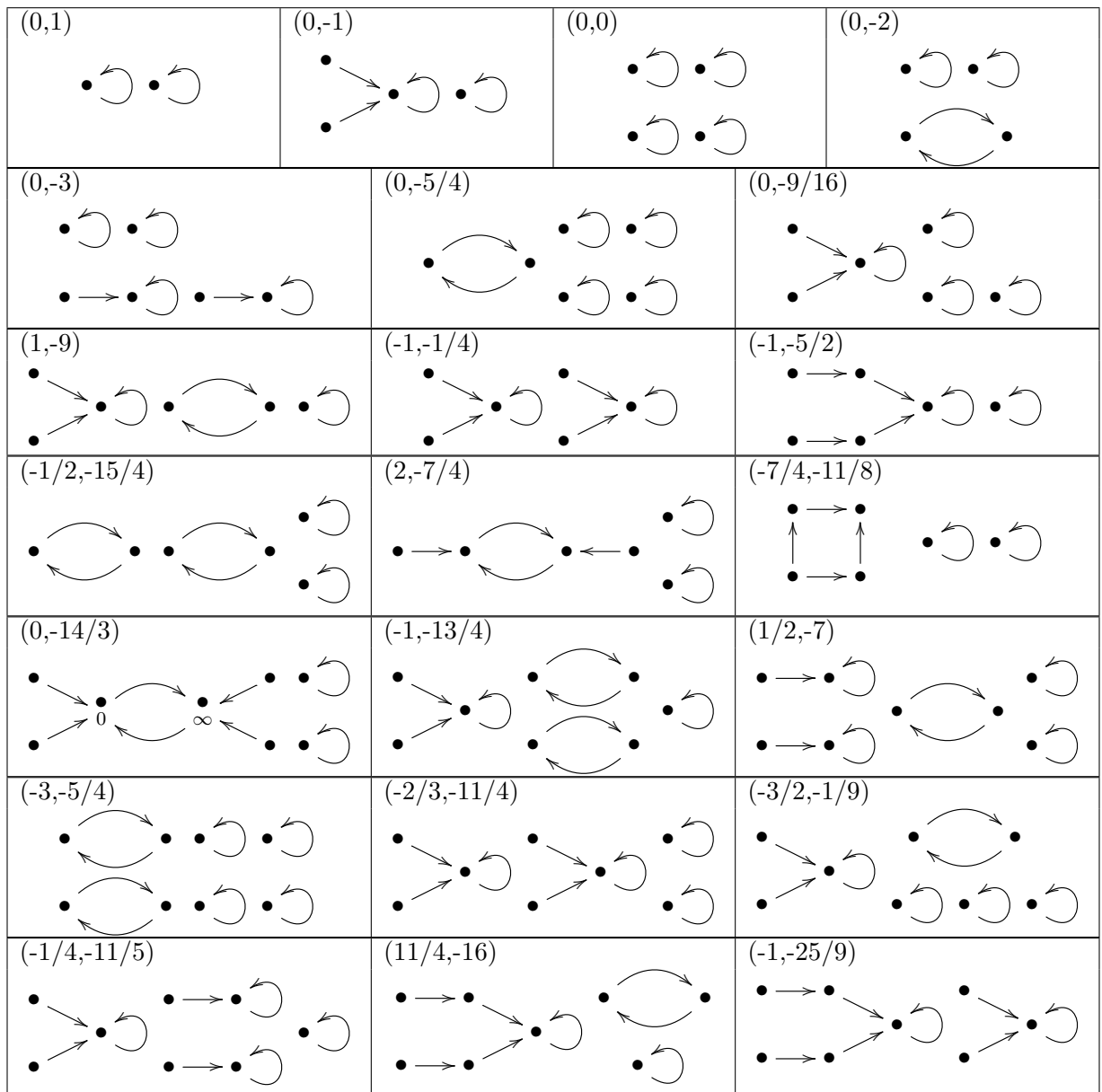
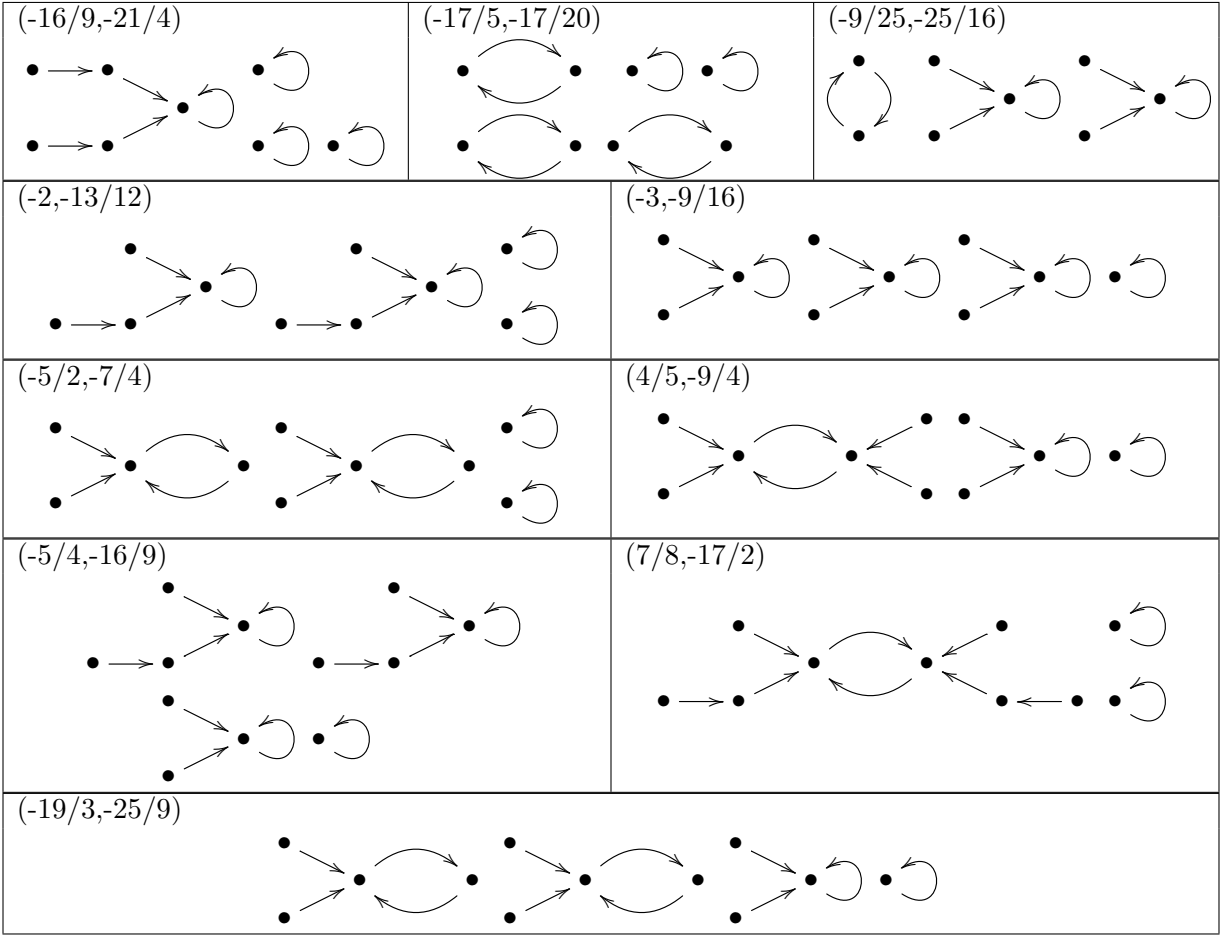
\begin{center}
\middlexcolumn
\offinterlineskip
\begin{tabularx}{0.98\textwidth}{%
  | *{4}{>{\centering\arraybackslash}X|}
}
\hline
\heading{(0,1)} &  \heading{(0,-1)} &  \heading{(0,0)} & \heading{(0,-2)}\\
\pre{\xygraph{
!{<0cm,0cm>;<1cm,0cm>:<0cm,1cm>::}
!{(0,0) }*+{\bullet}="a"
!{(1,0) }*+{\bullet}="b"
"a":@(rd,ru)"a"
"b":@(rd,ru)"b"
}}&
\pre{\xygraph{
!{<0cm,0cm>;<1cm,0cm>:<0cm,1cm>::}
!{(0,0) }*+{\bullet}="a"
!{(0,1) }*+{\bullet}="b"
!{(1,0.5) }*+{\bullet}="c"
!{(2,0.5) }*+{\bullet}="d"
"a":"c"
"b":"c"
"c":@(rd,ru)"c"
"d":@(rd,ru)"d"
}}&
\pre{\xygraph{
!{<0cm,0cm>;<1cm,0cm>:<0cm,1cm>::}
!{(0,0) }*+{\bullet}="a"
!{(1,0) }*+{\bullet}="b"
!{(0,1) }*+{\bullet}="c"
!{(1,1) }*+{\bullet}="d"
"a":@(rd,ru)"a"
"b":@(rd,ru)"b"
"c":@(rd,ru)"c"
"d":@(rd,ru)"d"
}}&
\pre{\xygraph{
!{<0cm,0cm>;<1cm,0cm>:<0cm,1cm>::}
!{(0,0) }*+{\bullet}="a"
!{(1.5,0) }*+{\bullet}="b"
!{(0,1) }*+{\bullet}="c"
!{(1,1) }*+{\bullet}="d"
"a":@/^1pc/"b"
"b":@/^1pc/"a"
"c":@(rd,ru)"c"
"d":@(rd,ru)"d"
}}
\\
\hline
\end{tabularx}

\begin{tabularx}{0.98\textwidth}{%
  | *{3}{>{\centering\arraybackslash}X|}
}
\hline
\heading{(0,-3)} &  \heading{(0,-5/4)} &  \heading{(0,-9/16)} \\
\pre{\xygraph{
!{<0cm,0cm>;<1cm,0cm>:<0cm,1cm>::}
!{(0,0) }*+{\bullet}="a"
!{(1,0) }*+{\bullet}="b"
!{(2,0) }*+{\bullet}="c"
!{(3,0) }*+{\bullet}="d"
!{(0,1) }*+{\bullet}="e"
!{(1,1) }*+{\bullet}="f"
"a":"b"
"b":@(rd,ru)"b"
"c":"d"
"d":@(rd,ru)"d"
"e":@(rd,ru)"e"
"f":@(rd,ru)"f"
}}&
\pre{\xygraph{
!{<0cm,0cm>;<1cm,0cm>:<0cm,1cm>::}
!{(0,0.5) }*+{\bullet}="a"
!{(1.5,0.5) }*+{\bullet}="b"
!{(2,0) }*+{\bullet}="c"
!{(3,0) }*+{\bullet}="d"
!{(2,1) }*+{\bullet}="e"
!{(3,1) }*+{\bullet}="f"
"a":@/^1pc/"b"
"b":@/^1pc/"a"
"c":@(rd,ru)"c"
"d":@(rd,ru)"d"
"e":@(rd,ru)"e"
"f":@(rd,ru)"f"
}}&
\pre{\xygraph{
!{<0cm,0cm>;<1cm,0cm>:<0cm,1cm>::}
!{(0,0) }*+{\bullet}="a"
!{(0,1) }*+{\bullet}="b"
!{(1,0.5) }*+{\bullet}="c"
!{(2,0) }*+{\bullet}="d"
!{(3,0) }*+{\bullet}="e"
!{(2,1) }*+{\bullet}="f"
"a":"c"
"b":"c"
"c":@(rd,ru)"c"
"d":@(rd,ru)"d"
"e":@(rd,ru)"e"
"f":@(rd,ru)"f"
}}\\
\hline
\end{tabularx}

\begin{tabularx}{0.98\textwidth}{%
  | *{3}{>{\centering\arraybackslash}X|}
}
\hline
\heading{(1,-9)} & \heading{(-1,-1/4)} & \heading{(-1,-5/2)}\\
\pre{\xygraph{
!{<0cm,0cm>;<1cm,0cm>:<0cm,1cm>::}
!{(0,0) }*+{\bullet}="a"
!{(0,1) }*+{\bullet}="b"
!{(1,0.5) }*+{\bullet}="c"
!{(2,0.5) }*+{\bullet}="d"
!{(3.5,0.5) }*+{\bullet}="e"
!{(4,0.5) }*+{\bullet}="f"
"a":"c"
"b":"c"
"c":@(rd,ru)"c"
"d":@/^1pc/"e"
"e":@/^1pc/"d"
"f":@(rd,ru)"f"
}}&
\pre{\xygraph{
!{<0cm,0cm>;<1cm,0cm>:<0cm,1cm>::}
!{(0,0) }*+{\bullet}="a"
!{(0,1) }*+{\bullet}="b"
!{(1,0.5) }*+{\bullet}="c"
!{(2,0) }*+{\bullet}="d"
!{(2,1) }*+{\bullet}="e"
!{(3,0.5) }*+{\bullet}="f"
"a":"c"
"b":"c"
"c":@(rd,ru)"c"
"d":"f"
"e":"f"
"f":@(rd,ru)"f"
}}&
\pre{\xygraph{
!{<0cm,0cm>;<1cm,0cm>:<0cm,1cm>::}
!{(0,0) }*+{\bullet}="a"
!{(0,1) }*+{\bullet}="b"
!{(1,0) }*+{\bullet}="c"
!{(1,1) }*+{\bullet}="d"
!{(2,0.5) }*+{\bullet}="e"
!{(3,0.5) }*+{\bullet}="f"
"a":"c"
"b":"d"
"c":"e"
"d":"e"
"e":@(rd,ru)"e"
"f":@(rd,ru)"f"
}}\\
\hline
\end{tabularx}

\begin{tabularx}{0.98\textwidth}{%
  | *{3}{>{\centering\arraybackslash}X|}
}
\hline
\heading{(-1/2,-15/4)} & \heading{(2,-7/4)} & \heading{(-7/4,-11/8)}\\
\pre{\xygraph{
!{<0cm,0cm>;<1cm,0cm>:<0cm,1cm>::}
!{(0,0.5) }*+{\bullet}="a"
!{(1.5,0.5) }*+{\bullet}="b"
!{(2,0.5)}*+{\bullet}="c"
!{(3.5,0.5) }*+{\bullet}="d"
!{(4,0) }*+{\bullet}="e"
!{(4,1) }*+{\bullet}="f"
"a":@/^1pc/"b"
"b":@/^1pc/"a"
"c":@/^1pc/"d"
"d":@/^1pc/"c"
"e":@(rd,ru)"e"
"f":@(rd,ru)"f"
}}&
\pre{\xygraph{
!{<0cm,0cm>;<1cm,0cm>:<0cm,1cm>::}
!{(0,0.5) }*+{\bullet}="a"
!{(1,0.5) }*+{\bullet}="b"
!{(2.5,0.5) }*+{\bullet}="c"
!{(3.5,0.5) }*+{\bullet}="d"
!{(4,0) }*+{\bullet}="e"
!{(4,1) }*+{\bullet}="f"
"a":"b"
"b":@/^1pc/"c"
"c":@/^1pc/"b"
"d":"c"
"e":@(rd,ru)"e"
"f":@(rd,ru)"f"
}}&
\pre{\xygraph{
!{<0cm,0cm>;<1cm,0cm>:<0cm,1cm>::}
!{(0,0) }*+{\bullet}="a"
!{(1,0) }*+{\bullet}="b"
!{(0,1) }*+{\bullet}="c"
!{(1,1) }*+{\bullet}="d"
!{(2,0.5) }*+{\bullet}="e"
!{(3,0.5) }*+{\bullet}="f"
"a":"b"
"b":"d"
"c":"d"
"a":"c"
"e":@(rd,ru)"e"
"f":@(rd,ru)"f"
}}\\
\hline
\end{tabularx}

\begin{tabularx}{0.98\textwidth}{%
  | *{3}{>{\centering\arraybackslash}X|}
}
\hline
\heading{(0,-14/3)} & \heading{(-1,-13/4)} & \heading{(1/2,-7)}\\
\pre{\xygraph{
!{<0cm,0cm>;<1cm,0cm>:<0cm,1cm>::}
!{(1,0.5) }*+{\underset{0}{\bullet}}="a"
!{(2.5,0.5) }*+{\underset{\infty}{\bullet}}="b"
!{(3.5,0) }*+{\bullet}="c"
!{(3.5,1) }*+{\bullet}="d"
!{(0,0) }*+{\bullet}="e"
!{(0,1) }*+{\bullet}="f"
!{(4,0) }*+{\bullet}="g"
!{(4,1) }*+{\bullet}="h"
"e":"a"
"f":"a"
"a":@/^1pc/"b"
"b":@/^1pc/"a"
"c":"b"
"d":"b"
"g":@(rd,ru)"g"
"h":@(rd,ru)"h"
}}&
\pre{\xygraph{
!{<0cm,0cm>;<1cm,0cm>:<0cm,1cm>::}
!{(0,0) }*+{\bullet}="a"
!{(0,1) }*+{\bullet}="b"
!{(1,0.5) }*+{\bullet}="c"
!{(2,0) }*+{\bullet}="d"
!{(3.5,0) }*+{\bullet}="e"
!{(2,1) }*+{\bullet}="f"
!{(3.5,1) }*+{\bullet}="g"
!{(4,0.5) }*+{\bullet}="h"
"a":"c"
"b":"c"
"c":@(rd,ru)"c"
"d":@/^1pc/"e"
"e":@/^1pc/"d"
"f":@/^1pc/"g"
"g":@/^1pc/"f"
"h":@(rd,ru)"h"
}}&
\pre{\xygraph{
!{<0cm,0cm>;<1cm,0cm>:<0cm,1cm>::}
!{(0,0) }*+{\bullet}="a"
!{(1,0) }*+{\bullet}="b"
!{(2,0.5) }*+{\bullet}="c"
!{(3.5,0.5) }*+{\bullet}="d"
!{(0,1) }*+{\bullet}="e"
!{(1,1) }*+{\bullet}="f"
!{(4,0) }*+{\bullet}="g"
!{(4,1) }*+{\bullet}="h"
"a":"b"
"b":@(rd,ru)"b"
"c":@/^1pc/"d"
"d":@/^1pc/"c"
"e":"f"
"f":@(rd,ru)"f"
"g":@(rd,ru)"g"
"h":@(rd,ru)"h"
}}\\
\hline
\end{tabularx}

\begin{tabularx}{0.98\textwidth}{%
  | *{3}{>{\centering\arraybackslash}X|}
}
\hline
\heading{(-3,-5/4)}& \heading{(-2/3,-11/4)}& \heading{(-3/2,-1/9)}\\
\pre{\xygraph{
!{<0cm,0cm>;<1cm,0cm>:<0cm,1cm>::}
!{(0,0) }*+{\bullet}="a"
!{(1.5,0) }*+{\bullet}="b"
!{(0,1) }*+{\bullet}="c"
!{(1.5,1) }*+{\bullet}="d"
!{(2,0) }*+{\bullet}="e"
!{(3,0) }*+{\bullet}="f"
!{(2,1) }*+{\bullet}="g"
!{(3,1) }*+{\bullet}="h"
"a":@/^1pc/"b"
"b":@/^1pc/"a"
"c":@/^1pc/"d"
"d":@/^1pc/"c"
"e":@(rd,ru)"e"
"f":@(rd,ru)"f"
"g":@(rd,ru)"g"
"h":@(rd,ru)"h"
}} &
\pre{\xygraph{
!{<0cm,0cm>;<1cm,0cm>:<0cm,1cm>::}
!{(0,0) }*+{\bullet}="a"
!{(1,0.5) }*+{\bullet}="b"
!{(0,1) }*+{\bullet}="c"
!{(2,0) }*+{\bullet}="d"
!{(3,0.5) }*+{\bullet}="e"
!{(2,1) }*+{\bullet}="f"
!{(4,0) }*+{\bullet}="g"
!{(4,1) }*+{\bullet}="h"
"a":"b"
"b":@(rd,ru)"b"
"c":"b"
"d":"e"
"e":@(rd,ru)"e"
"f":"e"
"g":@(rd,ru)"g"
"h":@(rd,ru)"h"
}}&
\pre{\xygraph{
!{<0cm,0cm>;<1cm,0cm>:<0cm,1cm>::}
!{(0,0) }*+{\bullet}="a"
!{(1,0.5) }*+{\bullet}="b"
!{(0,1) }*+{\bullet}="c"
!{(2.25,1) }*+{\bullet}="d"
!{(3.75,1) }*+{\bullet}="e"
!{(2,0) }*+{\bullet}="f"
!{(3,0) }*+{\bullet}="g"
!{(4,0) }*+{\bullet}="h"
"a":"b"
"b":@(rd,ru)"b"
"c":"b"
"d":@/^1pc/"e"
"e":@/^1pc/"d"
"f":@(rd,ru)"f"
"g":@(rd,ru)"g"
"h":@(rd,ru)"h"
}}\\
\hline
\end{tabularx}

\begin{tabularx}{0.98\textwidth}{%
  | *{3}{>{\centering\arraybackslash}X|}
}
\hline
\heading{(-1/4,-11/5)} & \heading{(11/4,-16)}& \heading{(-1,-25/9)}\\
\pre{\xygraph{
!{<0cm,0cm>;<1cm,0cm>:<0cm,1cm>::}
!{(0,0) }*+{\bullet}="a"
!{(1,0.5) }*+{\bullet}="b"
!{(0,1) }*+{\bullet}="c"
!{(2,0) }*+{\bullet}="d"
!{(3,0) }*+{\bullet}="e"
!{(2,1) }*+{\bullet}="f"
!{(3,1) }*+{\bullet}="g"
!{(4,0.5) }*+{\bullet}="h"
"a":"b"
"b":@(rd,ru)"b"
"c":"b"
"d":"e"
"e":@(rd,ru)"e"
"f":"g"
"g":@(rd,ru)"g"
"h":@(rd,ru)"h"
}}&
\pre{\xygraph{
!{<0cm,0cm>;<1cm,0cm>:<0cm,1cm>::}
!{(0,0) }*+{\bullet}="a"
!{(0,1) }*+{\bullet}="b"
!{(1,0) }*+{\bullet}="c"
!{(1,1) }*+{\bullet}="d"
!{(2,0.5) }*+{\bullet}="e"
!{(3,1) }*+{\bullet}="f"
!{(4.5,1) }*+{\bullet}="g"
!{(3.5,0) }*+{\bullet}="h"
"a":"c"
"b":"d"
"c":"e"
"d":"e"
"e":@(rd,ru)"e"
"f":@/^1pc/"g"
"g":@/^1pc/"f"
"h":@(rd,ru)"h"
}}&
\pre{\xygraph{
!{<0cm,0cm>;<1cm,0cm>:<0cm,1cm>::}
!{(0,0) }*+{\bullet}="a"
!{(0,1) }*+{\bullet}="b"
!{(1,0) }*+{\bullet}="c"
!{(1,1) }*+{\bullet}="d"
!{(2,0.5) }*+{\bullet}="e"
!{(3,0) }*+{\bullet}="g"
!{(3,1) }*+{\bullet}="h"
!{(4,0.5) }*+{\bullet}="f"
"a":"c"
"b":"d"
"c":"e"
"d":"e"
"e":@(rd,ru)"e"
"g":"f"
"h":"f"
"f":@(rd,ru)"f"
}}\\
\hline
\end{tabularx}

\begin{tabularx}{0.98\textwidth}{%
  | *{3}{>{\centering\arraybackslash}X|}
}
\hline
\heading{(-16/9,-21/4)} & \heading{(-17/5,-17/20)}& \heading{(-9/25,-25/16)}\\
\pre{\xygraph{
!{<0cm,0cm>;<1cm,0cm>:<0cm,1cm>::}
!{(2,0.5) }*+{\bullet}="a"
!{(1,0) }*+{\bullet}="c"
!{(1,1) }*+{\bullet}="d"
!{(0,0) }*+{\bullet}="e"
!{(0,1) }*+{\bullet}="g"
!{(3,0) }*+{\bullet}="h"
!{(4,0) }*+{\bullet}="f"
!{(3,1) }*+{\bullet}="i"
"a":@(rd,ru)"a"
"c":"a"
"d":"a"
"e":"c"
"g":"d"
"h":@(rd,ru)"h"
"f":@(rd,ru)"f"
"i":@(rd,ru)"i"
}}&
\pre{\xygraph{
!{<0cm,0cm>;<1cm,0cm>:<0cm,1cm>::}
!{(0,0) }*+{\bullet}="a"
!{(1.5,0) }*+{\bullet}="b"
!{(0,1) }*+{\bullet}="c"
!{(1.5,1) }*+{\bullet}="d"
!{(2,0) }*+{\bullet}="e"
!{(3.5,0) }*+{\bullet}="f"
!{(2.25,1) }*+{\bullet}="g"
!{(3.25,1) }*+{\bullet}="h"
"a":@/^1pc/"b"
"b":@/^1pc/"a"
"c":@/^1pc/"d"
"d":@/^1pc/"c"
"e":@/^1pc/"f"
"f":@/^1pc/"e"
"g":@(rd,ru)"g"
"h":@(rd,ru)"h"
}}&
\pre{\xygraph{
!{<0cm,0cm>;<1cm,0cm>:<0cm,1cm>::}
!{(0,0) }*+{\bullet}="a"
!{(0,1) }*+{\bullet}="b"
!{(2,0.5) }*+{\bullet}="e"
!{(1,0) }*+{\bullet}="f"
!{(1,1) }*+{\bullet}="g"
!{(4,0.5) }*+{\bullet}="h"
!{(3,1) }*+{\bullet}="j"
!{(3,0) }*+{\bullet}="k"
"a":@/^1pc/"b"
"b":@/^1pc/"a"
"e":@(rd,ru)"e"
"f":"e"
"g":"e"
"h":@(rd,ru)"h"
"j":"h"
"k":"h"
}}\\
\hline
\end{tabularx}

\begin{tabularx}{0.98\textwidth}{%
  | *{2}{>{\centering\arraybackslash}X|}
}
\hline
\heading{(-2,-13/12)} & \heading{(-3,-9/16)}\\
\pre{\xygraph{
!{<0cm,0cm>;<1cm,0cm>:<0cm,1cm>::}
!{(0,0) }*+{\bullet}="a"
!{(1,0) }*+{\bullet}="b"
!{(2,0.5) }*+{\bullet}="c"
!{(1,1) }*+{\bullet}="d"
!{(3,0) }*+{\bullet}="e"
!{(4,0) }*+{\bullet}="f"
!{(5,0.5) }*+{\bullet}="g"
!{(4,1) }*+{\bullet}="h"
!{(6,0) }*+{\bullet}="i"
!{(6,1) }*+{\bullet}="j"
"a":"b"
"b":"c"
"c":@(rd,ru)"c"
"d":"c"
"e":"f"
"f":"g"
"g":@(rd,ru)"g"
"h":"g"
"i":@(rd,ru)"i"
"j":@(rd,ru)"j"
}}&
\pre{\xygraph{
!{<0cm,0cm>;<1cm,0cm>:<0cm,1cm>::}
!{(0,0) }*+{\bullet}="a"
!{(1,0.5) }*+{\bullet}="b"
!{(0,1) }*+{\bullet}="c"
!{(2,0) }*+{\bullet}="d"
!{(3,0.5) }*+{\bullet}="e"
!{(2,1) }*+{\bullet}="f"
!{(4,0) }*+{\bullet}="g"
!{(5,0.5) }*+{\bullet}="h"
!{(4,1) }*+{\bullet}="i"
!{(6,0.5) }*+{\bullet}="j"
"a":"b"
"b":@(rd,ru)"b"
"c":"b"
"d":"e"
"e":@(rd,ru)"e"
"f":"e"
"g":"h"
"h":@(rd,ru)"h"
"i":"h"
"j":@(rd,ru)"j"
}}\\
\hline
\end{tabularx}

\begin{tabularx}{0.98\textwidth}{%
  | *{2}{>{\centering\arraybackslash}X|}
}
\hline
\heading{(-5/2,-7/4)} & \heading{(4/5,-9/4)}\\
\pre{\xygraph{
!{<0cm,0cm>;<1cm,0cm>:<0cm,1cm>::}
!{(0,0) }*+{\bullet}="a"
!{(0,1) }*+{\bullet}="b"
!{(1,0.5) }*+{\bullet}="c"
!{(2.5,0.5) }*+{\bullet}="d"
!{(3,0) }*+{\bullet}="e"
!{(3,1) }*+{\bullet}="f"
!{(4,0.5) }*+{\bullet}="g"
!{(5.5,0.5) }*+{\bullet}="h"
!{(6,0) }*+{\bullet}="i"
!{(6,1) }*+{\bullet}="j"
"a":"c"
"b":"c"
"c":@/^1pc/"d"
"d":@/^1pc/"c"
"e":"g"
"f":"g"
"g":@/^1pc/"h"
"h":@/^1pc/"g"
"i":@(rd,ru)"i"
"j":@(rd,ru)"j"
}}&
\pre{\xygraph{
!{<0cm,0cm>;<1cm,0cm>:<0cm,1cm>::}
!{(0,0) }*+{\bullet}="a"
!{(0,1) }*+{\bullet}="b"
!{(1,0.5) }*+{\bullet}="c"
!{(2.5,0.5) }*+{\bullet}="d"
!{(3.5,0) }*+{\bullet}="e"
!{(3.5,1) }*+{\bullet}="f"
!{(4,0) }*+{\bullet}="g"
!{(4,1) }*+{\bullet}="h"
!{(5,0.5) }*+{\bullet}="i"
!{(6,0.5) }*+{\bullet}="j"
"a":"c"
"b":"c"
"c":@/^1pc/"d"
"d":@/^1pc/"c"
"e":"d"
"f":"d"
"g":"i"
"h":"i"
"i":@(rd,ru)"i"
"j":@(rd,ru)"j"
}}\\
\hline
\end{tabularx}

\begin{tabularx}{0.98\textwidth}{%
  | *{2}{>{\centering\arraybackslash}X|}
}
\hline
\heading{(-5/4,-16/9)} & \heading{(7/8,-17/2)}\\
\pre{\xygraph{
!{<0cm,0cm>;<1cm,0cm>:<0cm,1cm>::}
!{(0,1.5) }*+{\bullet}="a"
!{(1,1.5) }*+{\bullet}="b"
!{(1,2.5) }*+{\bullet}="c"
!{(2,2) }*+{\bullet}="d"
!{(3,1.5) }*+{\bullet}="e"
!{(4,1.5) }*+{\bullet}="f"
!{(4,2.5) }*+{\bullet}="g"
!{(5,2) }*+{\bullet}="h"
!{(1,0) }*+{\bullet}="i"
!{(1,1) }*+{\bullet}="j"
!{(2,0.5) }*+{\bullet}="k"
!{(3,0.5) }*+{\bullet}="l"
"a":"b"
"b":"d"
"c":"d"
"d":@(rd,ru)"d"
"e":"f"
"f":"h"
"g":"h"
"h":@(rd,ru)"h"
"i":"k"
"j":"k"
"k":@(rd,ru)"k"
"l":@(rd,ru)"l"
}}&
\pre{\xygraph{
!{<0cm,0cm>;<1cm,0cm>:<0cm,1cm>::}
!{(2,0.5) }*+{\bullet}="a"
!{(3.5,0.5) }*+{\bullet}="b"
!{(1,0) }*+{\bullet}="c"
!{(1,1) }*+{\bullet}="d"
!{(0,0) }*+{\bullet}="e"
!{(4.5,0) }*+{\bullet}="g"
!{(4.5,1) }*+{\bullet}="h"
!{(5.5,0) }*+{\bullet}="f"
!{(6,0) }*+{\bullet}="i"
!{(6,1) }*+{\bullet}="j"
"a":@/^1pc/"b"
"b":@/^1pc/"a"
"c":"a"
"d":"a"
"e":"c"
"g":"b"
"h":"b"
"f":"g"
"i":@(rd,ru)"i"
"j":@(rd,ru)"j"
}}\\
\hline
\end{tabularx}

\begin{tabularx}{0.98\textwidth}{%
  | *{1}{>{\centering\arraybackslash}X|}
}
\hline
\heading{(-19/3,-25/9)}\\
\pre{\xygraph{
!{<0cm,0cm>;<1cm,0cm>:<0cm,1cm>::}
!{(1,0.5) }*+{\bullet}="a"
!{(2.5,0.5) }*+{\bullet}="b"
!{(0,0) }*+{\bullet}="c"
!{(0,1) }*+{\bullet}="d"
!{(4,0.5) }*+{\bullet}="e"
!{(5.5,0.5) }*+{\bullet}="f"
!{(3,0) }*+{\bullet}="g"
!{(3,1) }*+{\bullet}="h"
!{(7,0.5) }*+{\bullet}="i"
!{(6,1) }*+{\bullet}="j"
!{(6,0) }*+{\bullet}="k"
!{(8,0.5) }*+{\bullet}="l"
"a":@/^1pc/"b"
"b":@/^1pc/"a"
"c":"a"
"d":"a"
"e":@/^1pc/"f"
"f":@/^1pc/"e"
"g":"e"
"h":"e"
"i":@(rd,ru)"i"
"j":"i"
"k":"i"
"l":@(rd,ru)"l"
}}\\
\hline
\end{tabularx}
\end{center}

\begin{code}
\begin{lstlisting}[language=python]
gr=[[0,1],[0,-1],[0,0],[0,-2],[0,-3],[0,-5/4],[0,-9/16],[1,-9], [-1,-1/4],[-1,-5/2],[-1/2,-15/4],[2,-7/4],[-7/4,-11/8],[0,-14/3], [-1,-13/4],[1/2,-7],[-3,-5/4],[-2/3,-11/4],[-3/2,-1/9],[-1/4,-11/5], [11/4,-16],[-1,-25/9],[-16/9,-21/4],[-17/5,-17/20],[-9/25,-25/16], [-2,-13/12],[-3,-9/16],[-5/2,-7/4],[4/5,-9/4],[-5/4,-16/9],[7/8,-17/2], [-19/3,-25/9]]

set_verbose(None)
P.<x,y>=ProjectiveSpace(QQ,1)
L=[]
for A,B in gr:
    f=DynamicalSystem([x^3 + A*x*y^2, B*x^2*y + y^3])
    if f.is_morphism():
        G = f.rational_preperiodic_graph()
        found = False
        for a,b,g in L:
            if g.is_isomorphic(G):
                found=True
                break
        if not found:
            print(A,B,G, len(L)+1)
            L.append((A,B,G))
len(L)
\end{lstlisting}
\end{code}

\subsection{$C_2$ Second Component}\label{sect_A3_C2_1_preperiodic}
We move to the {\it two}-parameter family $g_{a,b}(z) = \frac{az^2 + 1}{z^3 + bz}$.

We start by examining the fixed points and 2-periodic points.
\begin{prop}\label{prop_A3_C2_1_periodic}
    For $g_{a,b}(z) = \frac{az^2 + 1}{z^3 + bz}$, we have the following periodic points.
    \begin{enumerate}
        \item The pairs $(a,b)$ for which $g_{a,b}$ has a $\Q$-rational fixed point are parameterized by
            \begin{align*}
                (a,b) = \left( a, -t^2 + a + \frac{1}{t^2}\right),
            \end{align*}
            for $t \neq 0$, $t^2 \neq a$, and $t^2 \neq -\frac{1}{a}$.
            There are exactly two $\Q$-rational fixed points given by
            \begin{equation*}
                z = \pm t
            \end{equation*}

        \item For every (non-degenerate) pair $(a,b)$, $g_{a,b}$ has the 2-cycle $0$--$\infty$. The pairs $(a,b)$ for which there are additional 2-cycles with $\Q$-rational points are parameterized by \begin{align*}
                (a,b) = \left( a, -t^2 - a - \frac{1}{t^2}\right),
            \end{align*}
            for $t \neq 0$, $t^2 \neq -a$, and $t^2 \neq -\frac{1}{a}$.
            There are exactly two $2$-cycles with $\Q$-rational points given by
            \begin{equation*}
                \{\pm t\} \quad \text{and} \quad \left\{\pm \frac{1}{t}\right\}.
            \end{equation*}
    \end{enumerate}
\end{prop}
\begin{proof}
    We first look at the fixed points. The fixed points are given by the equation
    \begin{equation*}
        az^2 + 1 = z^4 + bz^2.
    \end{equation*}
    Let $z=t$ be a fixed point, we can solve for $b$ in terms of $a$ and $t$ as
    \begin{equation*}
        b = -t^2 + a + \frac{1}{t^2}.
    \end{equation*}
    This gives the stated parameterization.
    The omitted $t$ values are those where the map is degenerate. With this parameterization, the first dynatomic polynomial factors as
    \begin{equation*}
        \Phi^{\ast}_1(g_{a,b}) = (z-t)(z+t)(z^2t^2 + 1)
    \end{equation*}
    so there are two $\Q$-rational fixed points $\pm t$.

\begin{code}
\begin{lstlisting}
r.<a,b>=QQ[]
P.<x,y>=ProjectiveSpace(r,1)
f=DynamicalSystem([a*x^2*y + y^3, x^3 + b*x*y^2])
f.dynatomic_polynomial(1)

r.<a,t>=QQ[]
P.<x,y>=ProjectiveSpace(r,1)
b = -t^2 + a + 1/t^2
f=DynamicalSystem([a*x^2*y + y^3, x^3 + b*x*y^2])
f.resultant().factor()

r.<z,t>=QQ[]
F=r(-t^2*z^4 + ((t^4 - 1))*z^2 + t^2)
F.factor()
\end{lstlisting}
\end{code}

    Now we consider the second periodic points via the second dynatomic polynomial
    \begin{equation*}
        \Phi^{\ast}_2(g_{a,b}) = z(ab - 1)(z^4 + (a+b)z^2 + 1).
    \end{equation*}
    We see the periodic $2$-cycle $0$--$\infty$ occurs for every choice of $(a,b)$. Taking the last factor and setting $z=t$ as a periodic point, we solve for $b$ in terms of $t$ to get
    \begin{equation*}
        b = -t^2 - a - \frac{1}{t^2}.
    \end{equation*}
    Then additional $2$-cycles with $\Q$-rationl points are given by the stated parameterization. The omitted $t$ values are those where the map is degenerate. With this parameterization the second dynatomic polynomial factors as
    \begin{equation*}
        z(t - z)(t + z)(-tz + 1)(tz + 1)(t^2 + a)(at^2 + 1).
    \end{equation*}
    So the additional $\Q$-rational points of period $2$ are $\pm t$ and $\pm \frac{1}{t}$.

\begin{code}
\begin{lstlisting}
r.<a,b>=QQ[]
P.<x,y>=ProjectiveSpace(r,1)
f=DynamicalSystem([a*x^2*y + y^3, x^3 + b*x*y^2])
f.dynatomic_polynomial(2)
r.<a,b,x,y>=QQ[]
F=(a*b - 1)*x^5*y + (a^2*b + a*b^2 - a - b)*x^3*y^3 + (a*b - 1)*x*y^5
F.factor()

r.<a,t>=QQ[]
P.<x,y>=ProjectiveSpace(r,1)
b=-t^2-a-1/t^2
f=DynamicalSystem([a*x^2*y + y^3, x^3 + b*x*y^2])
f.resultant().factor()

f.dynatomic_polynomial(2)
r.<a,t,x,y>=QQ[]
F=(t^2*(-a*t^4 - a^2*t^2 - t^2 - a))*x^5*y + ((a*t^8 + a^2*t^6 + t^6 + 2*a*t^4 + a^2*t^2 + t^2 + a))*x^3*y^3 + (t^2*(-a*t^4 - a^2*t^2 - t^2 - a))*x*y^5
F.factor()
\end{lstlisting}
\end{code}
\end{proof}

Similar to the first two-parameter family, periodic points with higher periods and (strictly) preperiodic points are difficult to study for this family. Again, we content ourselves with a census of $\Q$-rational preperiodic structures for parameters $a$ and $b$ with small height.

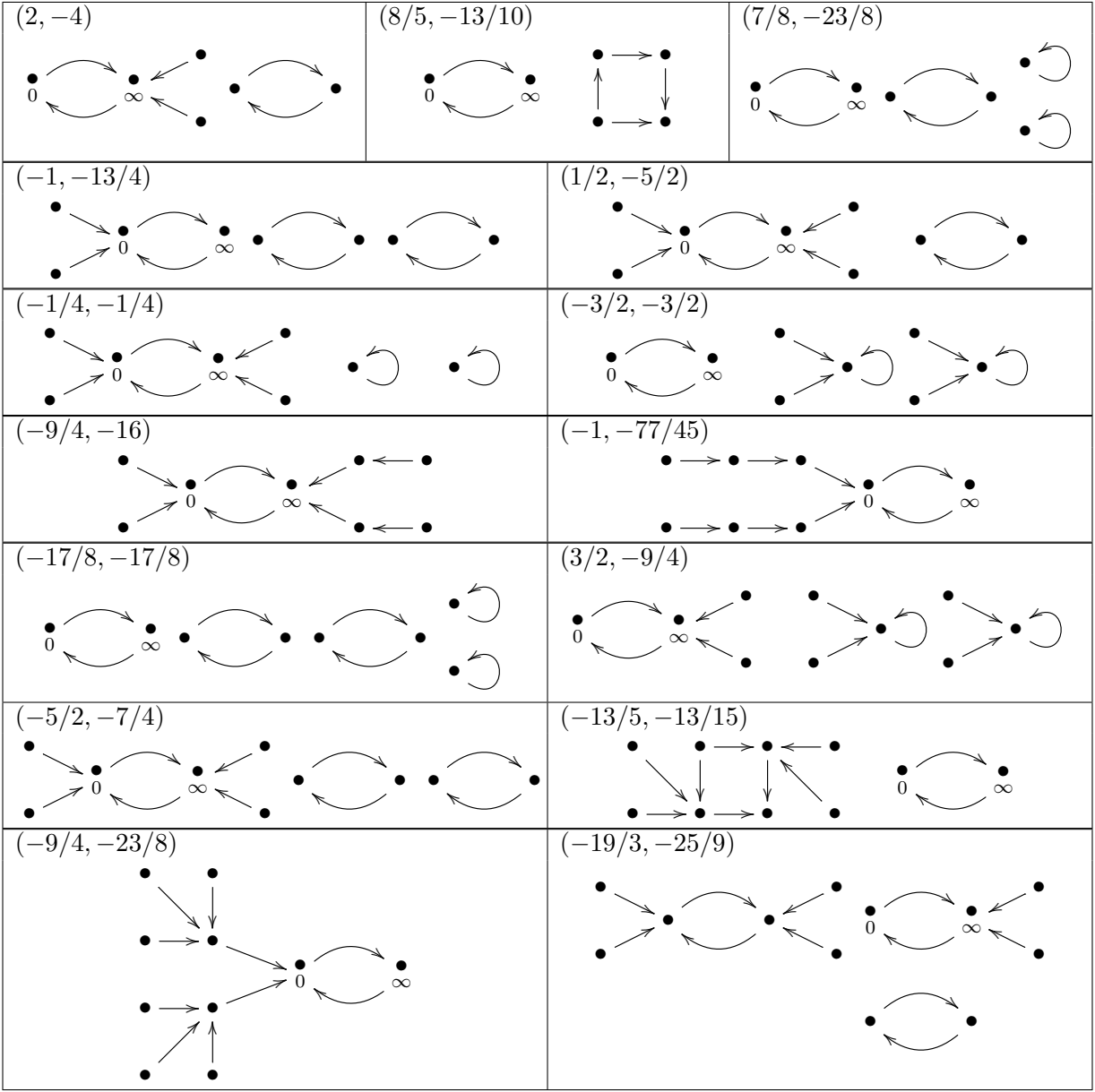
\begin{center}
Table: Preperiodic Graphs for Pairs $(a,b)$

\smallskip

\middlexcolumn
\offinterlineskip
\begin{tabularx}{0.98\textwidth}{%
  | *{3}{>{\centering\arraybackslash}X|}
}
\hline
\heading{$(0,0)$} &  \heading{$(0,-1)$} &  \heading{$(0,-2)$} \\
\pre{\xygraph{
!{<0cm,0cm>;<1cm,0cm>:<0cm,1cm>::}
!{(0,0) }*+{\underset{0}{\bullet}}="a"
!{(1.5,0) }*+{\underset{\infty}{\bullet}}="b"
!{(2.5,0) }*+{\bullet}="c"
!{(3.5,0) }*+{\bullet}="d"
"a":@/^1pc/"b"
"b":@/^1pc/"a"
"c":@(rd,ru)"c"
"d":@(rd,ru)"d"
}}&
\pre{\xygraph{
!{<0cm,0cm>;<1cm,0cm>:<0cm,1cm>::}
!{(0,0.5) }*+{\underset{0}{\bullet}}="a"
!{(1.5,0.5) }*+{\underset{\infty}{\bullet}}="b"
!{(2.5,1) }*+{\bullet}="c"
!{(2.5,0) }*+{\bullet}="d"
"a":@/^1pc/"b"
"b":@/^1pc/"a"
"c":"b"
"d":"b"
}}&
\pre{\xygraph{
!{<0cm,0cm>;<1cm,0cm>:<0cm,1cm>::}
!{(0,0) }*+{\underset{0}{\bullet}}="a"
!{(1.5,0) }*+{\underset{\infty}{\bullet}}="b"
!{(2.5,0) }*+{\bullet}="c"
!{(4.0,0) }*+{\bullet}="d"
"a":@/^1pc/"b"
"b":@/^1pc/"a"
"c":@/^1pc/"d"
"d":@/^1pc/"c"
}}\\\hline
\end{tabularx}

\begin{tabularx}{0.98\textwidth}{%
  | *{3}{>{\centering\arraybackslash}X|}
}
\hline
\heading{$(-1,-1/4)$} &  \heading{$(-1,-5/2)$} & \heading{$(-1,11/4)$}\\
\pre{\xygraph{
!{<0cm,0cm>;<1cm,0cm>:<0cm,1cm>::}
!{(1,0.5) }*+{\underset{0}{\bullet}}="a"
!{(2.5,0.5) }*+{\underset{\infty}{\bullet}}="b"
!{(3.5,0) }*+{\bullet}="c"
!{(3.5,1) }*+{\bullet}="d"
!{(0,0) }*+{\bullet}="e"
!{(0,1) }*+{\bullet}="f"
"e":"a"
"f":"a"
"a":@/^1pc/"b"
"b":@/^1pc/"a"
"c":"b"
"d":"b"
}}&
\pre{\xygraph{
!{<0cm,0cm>;<1cm,0cm>:<0cm,1cm>::}
!{(2,0.5) }*+{\underset{0}{\bullet}}="a"
!{(3.5,0.5) }*+{\underset{\infty}{\bullet}}="b"
!{(1,0) }*+{\bullet}="c"
!{(1,1) }*+{\bullet}="d"
!{(0,0) }*+{\bullet}="e"
!{(0,1) }*+{\bullet}="f"
"e":"c"
"f":"d"
"a":@/^1pc/"b"
"b":@/^1pc/"a"
"c":"a"
"d":"a"
}}
&
\pre{\xygraph{
!{<0cm,0cm>;<1cm,0cm>:<0cm,1cm>::}
!{(1,0.5) }*+{\underset{0}{\bullet}}="a"
!{(2.5,0.5) }*+{\underset{\infty}{\bullet}}="b"
!{(3,0.5) }*+{\bullet}="c"
!{(4,0.5) }*+{\bullet}="d"
!{(0,0) }*+{\bullet}="e"
!{(0,1) }*+{\bullet}="f"
"e":"a"
"f":"a"
"a":@/^1pc/"b"
"b":@/^1pc/"a"
"c":@(rd,ru)"c"
"d":@(rd,ru)"d"
}}
\\
\hline
\end{tabularx}

\begin{tabularx}{0.98\textwidth}{%
  | *{2}{>{\centering\arraybackslash}X|}
}
\hline
\heading{$(-1/2,-15/4)$} &  \heading{$(2, -7/4)$} \\
\pre{\xygraph{
!{<0cm,0cm>;<1cm,0cm>:<0cm,1cm>::}
!{(0,0) }*+{\underset{0}{\bullet}}="a"
!{(1.5,0) }*+{\underset{\infty}{\bullet}}="b"
!{(2,0) }*+{\bullet}="c"
!{(3.5,0) }*+{\bullet}="d"
!{(4,0) }*+{\bullet}="g"
!{(5.5,0) }*+{\bullet}="h"
"a":@/^1pc/"b"
"b":@/^1pc/"a"
"c":@/^1pc/"d"
"d":@/^1pc/"c"
"g":@/^1pc/"h"
"h":@/^1pc/"g"
}} &
\pre{\xygraph{
!{<0cm,0cm>;<1cm,0cm>:<0cm,1cm>::}
!{(0,0) }*+{\underset{0}{\bullet}}="a"
!{(1.5,0) }*+{\underset{\infty}{\bullet}}="b"
!{(2,0) }*+{\bullet}="c"
!{(3,0) }*+{\bullet}="d"
!{(4,0) }*+{\bullet}="e"
!{(5,0) }*+{\bullet}="f"
!{(0,-0.5)}*+{}="g"
"a":@/^1pc/"b"
"b":@/^1pc/"a"
"c":"d"
"d":@(rd,ru)"d"
"e":"f"
"f":@(rd,ru)"f"
}}\\
\hline
\end{tabularx}

\begin{tabularx}{0.98\textwidth}{%
  | *{3}{>{\centering\arraybackslash}X|}
}
\hline
\heading{$(2,-4)$} & \heading{$(8/5,-13/10)$} & \heading{$(7/8, -23/8)$}\\
\pre{\xygraph{
!{<0cm,0cm>;<1cm,0cm>:<0cm,1cm>::}
!{(0,0.5) }*+{\underset{0}{\bullet}}="a"
!{(1.5,0.5) }*+{\underset{\infty}{\bullet}}="b"
!{(2.5,0) }*+{\bullet}="c"
!{(2.5,1) }*+{\bullet}="d"
!{(3,0.5) }*+{\bullet}="e"
!{(4.5,0.5) }*+{\bullet}="f"
"a":@/^1pc/"b"
"b":@/^1pc/"a"
"c":"b"
"d":"b"
"e":@/^1pc/"f"
"f":@/^1pc/"e"
}}&
\pre{\xygraph{
!{<0cm,0cm>;<1cm,0cm>:<0cm,1cm>::}
!{(0,0.5) }*+{\underset{0}{\bullet}}="a"
!{(1.5,0.5) }*+{\underset{\infty}{\bullet}}="b"
!{(2.5,0) }*+{\bullet}="c"
!{(2.5,1) }*+{\bullet}="d"
!{(3.5,0) }*+{\bullet}="e"
!{(3.5,1) }*+{\bullet}="f"
"a":@/^1pc/"b"
"b":@/^1pc/"a"
"c":"d"
"d":"f"
"f":"e"
"c":"e"
}} &
\pre{\xygraph{
!{<0cm,0cm>;<1cm,0cm>:<0cm,1cm>::}
!{(0,0.5) }*+{\underset{0}{\bullet}}="a"
!{(1.5,0.5) }*+{\underset{\infty}{\bullet}}="b"
!{(2,0.5) }*+{\bullet}="c"
!{(3.5,0.5) }*+{\bullet}="d"
!{(4,0) }*+{\bullet}="e"
!{(4,1) }*+{\bullet}="f"
"a":@/^1pc/"b"
"b":@/^1pc/"a"
"c":@/^1pc/"d"
"d":@/^1pc/"c"
"e":@(rd,ru)"e"
"f":@(rd,ru)"f"
}}\\
\hline
\end{tabularx}

\begin{tabularx}{0.98\textwidth}{%
  | *{2}{>{\centering\arraybackslash}X|}
}
\hline
\heading{$(-1, -13/4)$} & \heading{$(1/2,-5/2)$}\\
\pre{\xygraph{
!{<0cm,0cm>;<1cm,0cm>:<0cm,1cm>::}
!{(1,0.5) }*+{\underset{0}{\bullet}}="a"
!{(2.5,0.5) }*+{\underset{\infty}{\bullet}}="b"
!{(0,0) }*+{\bullet}="e"
!{(0,1) }*+{\bullet}="f"
!{(3,0.5) }*+{\bullet}="c"
!{(4.5,0.5) }*+{\bullet}="d"
!{(5,0.5) }*+{\bullet}="g"
!{(6.5,0.5) }*+{\bullet}="h"
"e":"a"
"f":"a"
"a":@/^1pc/"b"
"b":@/^1pc/"a"
"c":@/^1pc/"d"
"d":@/^1pc/"c"
"g":@/^1pc/"h"
"h":@/^1pc/"g"
}} &
\pre{\xygraph{
!{<0cm,0cm>;<1cm,0cm>:<0cm,1cm>::}
!{(1,0.5) }*+{\underset{0}{\bullet}}="a"
!{(2.5,0.5) }*+{\underset{\infty}{\bullet}}="b"
!{(3.5,0) }*+{\bullet}="c"
!{(3.5,1) }*+{\bullet}="d"
!{(0,0) }*+{\bullet}="e"
!{(0,1) }*+{\bullet}="f"
!{(4.5,0.5) }*+{\bullet}="g"
!{(6,0.5) }*+{\bullet}="h"
"e":"a"
"f":"a"
"a":@/^1pc/"b"
"b":@/^1pc/"a"
"c":"b"
"d":"b"
"g":@/^1pc/"h"
"h":@/^1pc/"g"
}}\\
\hline
\end{tabularx}

\begin{tabularx}{0.98\textwidth}{%
  | *{2}{>{\centering\arraybackslash}X|}
}
\hline
\heading{$(-1/4,-1/4)$} & \heading{$(-3/2,-3/2)$}\\
\pre{\xygraph{
!{<0cm,0cm>;<1cm,0cm>:<0cm,1cm>::}
!{(1,0.5) }*+{\underset{0}{\bullet}}="a"
!{(2.5,0.5) }*+{\underset{\infty}{\bullet}}="b"
!{(3.5,0) }*+{\bullet}="c"
!{(3.5,1) }*+{\bullet}="d"
!{(0,0) }*+{\bullet}="e"
!{(0,1) }*+{\bullet}="f"
!{(4.5,0.5) }*+{\bullet}="g"
!{(6,0.5) }*+{\bullet}="h"
"e":"a"
"f":"a"
"a":@/^1pc/"b"
"b":@/^1pc/"a"
"c":"b"
"d":"b"
"g":@(rd,ru)"g"
"h":@(rd,ru)"h"
}} &
\pre{\xygraph{
!{<0cm,0cm>;<1cm,0cm>:<0cm,1cm>::}
!{(0,0.5) }*+{\underset{0}{\bullet}}="a"
!{(1.5,0.5) }*+{\underset{\infty}{\bullet}}="b"
!{(2.5,0) }*+{\bullet}="e"
!{(3.5,0.5) }*+{\bullet}="f"
!{(2.5,1) }*+{\bullet}="g"
!{(4.5,0) }*+{\bullet}="h"
!{(4.5,1) }*+{\bullet}="i"
!{(5.5,0.5) }*+{\bullet}="j"
"a":@/^1pc/"b"
"b":@/^1pc/"a"
"e":"f"
"g":"f"
"f":@(rd,ru)"f"
"h":"j"
"i":"j"
"j":@(rd,ru)"j"
}}\\
\hline
\end{tabularx}

\begin{tabularx}{0.98\textwidth}{%
  | *{2}{>{\centering\arraybackslash}X|}
}
\hline
\heading{$(-9/4, -16)$} & \heading{$(-1,-77/45)$} \\
\pre{\xygraph{
!{<0cm,0cm>;<1cm,0cm>:<0cm,1cm>::}
!{(1,0.5) }*+{\underset{0}{\bullet}}="a"
!{(2.5,0.5) }*+{\underset{\infty}{\bullet}}="b"
!{(3.5,0) }*+{\bullet}="c"
!{(3.5,1) }*+{\bullet}="d"
!{(0,0) }*+{\bullet}="e"
!{(0,1) }*+{\bullet}="f"
!{(4.5,0) }*+{\bullet}="g"
!{(4.5,1) }*+{\bullet}="h"
"e":"a"
"f":"a"
"a":@/^1pc/"b"
"b":@/^1pc/"a"
"c":"b"
"d":"b"
"g":"c"
"h":"d"
}} &
\pre{\xygraph{
!{<0cm,0cm>;<1cm,0cm>:<0cm,1cm>::}
!{(3,0.5) }*+{\underset{0}{\bullet}}="a"
!{(4.5,0.5) }*+{\underset{\infty}{\bullet}}="b"
!{(1,0) }*+{\bullet}="c"
!{(1,1) }*+{\bullet}="d"
!{(0,0) }*+{\bullet}="e"
!{(0,1) }*+{\bullet}="f"
!{(2,0) }*+{\bullet}="g"
!{(2,1) }*+{\bullet}="h"
"e":"c"
"f":"d"
"a":@/^1pc/"b"
"b":@/^1pc/"a"
"c":"g"
"d":"h"
"g":"a"
"h":"a"
}}\\
\hline
\end{tabularx}

\begin{tabularx}{0.98\textwidth}{%
  | *{2}{>{\centering\arraybackslash}X|}
}
\hline
\heading{$(-17/8,-17/8)$} & \heading{$(3/2,-9/4)$}\\
\pre{\xygraph{
!{<0cm,0cm>;<1cm,0cm>:<0cm,1cm>::}
!{(0,0.5) }*+{\underset{0}{\bullet}}="a"
!{(1.5,0.5) }*+{\underset{\infty}{\bullet}}="b"
!{(2,0.5) }*+{\bullet}="c"
!{(3.5,0.5) }*+{\bullet}="d"
!{(4,0.5) }*+{\bullet}="e"
!{(5.5,0.5) }*+{\bullet}="f"
!{(6,0) }*+{\bullet}="g"
!{(6,1) }*+{\bullet}="h"
"a":@/^1pc/"b"
"b":@/^1pc/"a"
"c":@/^1pc/"d"
"d":@/^1pc/"c"
"e":@/^1pc/"f"
"f":@/^1pc/"e"
"g":@(rd,ru)"g"
"h":@(rd,ru)"h"
}}&
\pre{\xygraph{
!{<0cm,0cm>;<1cm,0cm>:<0cm,1cm>::}
!{(0,0.5) }*+{\underset{0}{\bullet}}="a"
!{(1.5,0.5) }*+{\underset{\infty}{\bullet}}="b"
!{(2.5,0) }*+{\bullet}="c"
!{(2.5,1) }*+{\bullet}="d"
!{(3.5,0) }*+{\bullet}="e"
!{(4.5,0.5) }*+{\bullet}="f"
!{(3.5,1) }*+{\bullet}="g"
!{(5.5,0) }*+{\bullet}="h"
!{(5.5,1) }*+{\bullet}="i"
!{(6.5,0.5) }*+{\bullet}="j"
"a":@/^1pc/"b"
"b":@/^1pc/"a"
"c":"b"
"d":"b"
"e":"f"
"g":"f"
"f":@(rd,ru)"f"
"h":"j"
"i":"j"
"j":@(rd,ru)"j"
}}
\end{tabularx}

\begin{tabularx}{0.98\textwidth}{%
  | *{2}{>{\centering\arraybackslash}X|}
}
\hline
\heading{$(-5/2,-7/4)$} & \heading{$(-13/5, -13/15)$}\\
\pre{\xygraph{
!{<0cm,0cm>;<1cm,0cm>:<0cm,1cm>::}
!{(1,0.5) }*+{\underset{0}{\bullet}}="a"
!{(2.5,0.5) }*+{\underset{\infty}{\bullet}}="b"
!{(3.5,0) }*+{\bullet}="c"
!{(3.5,1) }*+{\bullet}="d"
!{(0,0) }*+{\bullet}="e"
!{(0,1) }*+{\bullet}="f"
!{(4.0,0.5) }*+{\bullet}="g"
!{(5.5,0.5) }*+{\bullet}="h"
!{(6,0.5) }*+{\bullet}="i"
!{(7.5,0.5) }*+{\bullet}="j"
"e":"a"
"f":"a"
"a":@/^1pc/"b"
"b":@/^1pc/"a"
"c":"b"
"d":"b"
"g":@/^1pc/"h"
"h":@/^1pc/"g"
"i":@/^1pc/"j"
"j":@/^1pc/"i"
}} &
\pre{\xygraph{
!{<0cm,0cm>;<1cm,0cm>:<0cm,1cm>::}
!{(4,0.5) }*+{\underset{0}{\bullet}}="a"
!{(5.5,0.5) }*+{\underset{\infty}{\bullet}}="b"
!{(1,1) }*+{\bullet}="c"
!{(2,1) }*+{\bullet}="d"
!{(1,0) }*+{\bullet}="e"
!{(2,0) }*+{\bullet}="f"
!{(0,0) }*+{\bullet}="g"
!{(0,1) }*+{\bullet}="h"
!{(3,0) }*+{\bullet}="i"
!{(3,1) }*+{\bullet}="j"
"a":@/^1pc/"b"
"b":@/^1pc/"a"
"c":"d"
"c":"e"
"d":"f"
"e":"f"
"g":"e"
"h":"e"
"i":"d"
"j":"d"
}}\\
\hline
\end{tabularx}

\begin{tabularx}{0.98\textwidth}{%
  | *{2}{>{\centering\arraybackslash}X|}
}
\hline
\heading{$(-9/4, -23/8)$} & \heading{$(-19/3, -25/9)$}\\
\pre{\xygraph{
!{<0cm,0cm>;<1cm,0cm>:<0cm,1cm>::}
!{(2.5,0.5) }*+{\underset{0}{\bullet}}="a"
!{(4,0.5) }*+{\underset{\infty}{\bullet}}="b"
!{(1.2,0) }*+{\bullet}="c"
!{(1.2,1) }*+{\bullet}="d"
!{(1.2,-1) }*+{\bullet}="e"
!{(0.2,-1) }*+{\bullet}="f"
!{(0.2,0) }*+{\bullet}="g"
!{(0.2,1) }*+{\bullet}="h"
!{(0.2,2) }*+{\bullet}="i"
!{(1.2,2) }*+{\bullet}="j"
"a":@/^1pc/"b"
"b":@/^1pc/"a"
"c":"a"
"d":"a"
"e":"c"
"f":"c"
"g":"c"
"h":"d"
"i":"d"
"j":"d"
}}&
\pre{\xygraph{
!{<0cm,0cm>;<1cm,0cm>:<0cm,1cm>::}
!{(4,0.5) }*+{\underset{0}{\bullet}}="a"
!{(5.5,0.5) }*+{\underset{\infty}{\bullet}}="b"
!{(1,0.5) }*+{\bullet}="c"
!{(2.5,0.5) }*+{\bullet}="d"
!{(0,0) }*+{\bullet}="e"
!{(0,1) }*+{\bullet}="f"
!{(3.5,0) }*+{\bullet}="g"
!{(3.5,1) }*+{\bullet}="h"
!{(6.5,0) }*+{\bullet}="i"
!{(6.5,1) }*+{\bullet}="j"
!{(4,-1) }*+{\bullet}="k"
!{(5.5,-1) }*+{\bullet}="l"
"a":@/^1pc/"b"
"b":@/^1pc/"a"
"i":"b"
"j":"b"
"c":@/^1pc/"d"
"d":@/^1pc/"c"
"e":"c"
"f":"c"
"g":"d"
"h":"d"
"k":@/^1pc/"l"
"l":@/^1pc/"k"
}}\\
\hline
\end{tabularx}

\end{center}

\begin{code}
\begin{lstlisting}[language=python]
gr=[[0,0],[0,-1],[0,-2],[-1,-1/4],[-1,-5/2],[-1,11/4],[-1/2,-15/4],[2,-7/4], [2,-4],[8/5,-13/10],[7/8,-23/8],[-1,-13/4],[1/2,-5/2],[-1/4,-1/4], [-3/2,-3/2],[-9/4,-16],[-1,-77/45],[-17/8,-17/8],[3/2,-9/4],[-5/2,-7/4], [-13/5,-13/5],[-9/4,-23/8],[-19/3,-25/9]]

set_verbose(None)
P.<x,y>=ProjectiveSpace(QQ,1)
L=[]
for A,B in gr:
    f=DynamicalSystem([A*x^2*y + y^3, x^3 +B*x*y^2])
    if f.is_morphism():
        G = f.rational_preperiodic_graph()
        found = False
        for a,b,g in L:
            if g.is_isomorphic(G):
                found=True
                break
        if not found:
            print(A,B,G, len(L)+1)
            L.append((A,B,G))
len(L)
\end{lstlisting}
\end{code}

\section{Rational Preperiodic Structures in $\M_4$} \label{sect_M4_preperiodic}
    In this section, we examine the $\Q$-rational preperiodic point structures of the families covering $\A_4$ given in Section \ref{sect_M4}.

We start with the dimension $0$ family.

\subsection{The zero dimension loci $\A_4(C_5)$ and $\A_4(D_5)$}
    Recall from Proposition \ref{prop_A4_C5} that $\A_4(D_5) = \A_4(C_5)$ are given by the single conjugacy class $f(z)  = \frac{1}{z^4}$.
    \begin{theorem}\label{M4_rational_single}
        For $\A_4(D_5) = \A_4(C_5)$ the single conjugacy class $f(z)  = \frac{1}{z^4}$ has $\Q$-rational preperiodic structure given by
        \begin{equation*}
            \xygraph{
    !{<0cm,0cm>;<1cm,0cm>:<0cm,1cm>::}
    !{(1,0) }*+{\underset{0}{\bullet}}="a"
    !{(2.5,0) }*+{\underset{\infty}{\bullet}}="b"
    !{(3.5,0) }*+{\underset{-1}{\bullet}}="c"
    !{(4.5,0) }*+{\underset{1}{\bullet}}="d"
    "a":@/^1pc/"b"
    "b":@/^1pc/"a"
    "c":"d"
    "d":@(rd,ru)"d"
    }
        \end{equation*}
    \end{theorem}
    \begin{proof}
        Direct computation with the algorithm of Hutz \cite{Hutz12} as implemented in Sage.
    \end{proof}

\subsection{The dimension one family $\A_4(C_4)$}
The automorphism locus of $\A_4(C_4)$ is covered by the family $f_k(z) = \frac{z^4 + 1}{kz^3}$. We start with the periodic points.
\begin{prop}\label{prop_A4_C4_periodic}
    We have the following periodic points for $f_k(z) = \frac{z^4+1}{kz^3}$.
    \begin{enumerate}
        \item The point $\infty$ is a fixed point for every $k$. When $k= 1 + t^4$ for $t \in \Q\setminus \{0\}$, we have exactly two additional $\Q$-rational fixed points, $\pm \frac{1}{t}$.

        \item There are two 2-cycles with $\Q$-rational points for each of the following two families of parameters
        \begin{align*}
            k &= t^2 + t^{-2} \quad \text{ for } t \in \Q \setminus \{0,\pm 1\}\\
            k &= -(t^2 + t^{-2}) \quad \text{ for } t \in \Q \setminus \{0, \pm 1\}
        \end{align*}
        and one $2$-cycle for the family
        \begin{equation*}
            k = -(1+t^4)\quad \text{ for } t \in \Q \setminus \{0\}.
        \end{equation*}
        Furthermore, no $\Q$-rational value of $a$ appears in more than one of the given parameterizations. In particular, there are at most two 2-cycles with $\Q$-rational points.

        \item The only $\Q$-rational fixed point that occurs in conjunction with a 2-cycle with $\Q$-rational points is the fixed point $\infty$, which occurs for every parameter value.
    \end{enumerate}
\end{prop}
\begin{proof}
    Starting with rational fixed points, we can immediately see that $\infty$ is always a fixed point. The first dynatomic polynomial is $\Phi^{\ast}_{1}(f_k) = (1-k)z^4 + 1$. Since this is linear in $k$, there are additional rational fixed points when
    \begin{equation*}
        k = \frac{z^4 + 1}{z^4}.
    \end{equation*}
    In particular, only two of the four fixed points $\pm z, \pm iz$ are $\Q$-rational. Substituting $z = \frac{1}{t}$, we get $k = 1+t^4$.

    Next we look for $2$-cycles with $\Q$-rational points. The second dynatomic polynomial is
    \begin{align*}
        \Phi^{\ast}_{2}(f_k)
        &= (z^4 - kz^2 + 1)(z^4 + kz^2 + 1)((1+k)z^4 + 1).
    \end{align*}
    All three components are linear in $k$, so we get one-parameter families each of which has two 2-cycles with $\Q$-rational points.
    \begin{enumerate}
        \item The first factor vanishes when $k = \frac{z^4+1}{z^2}$, which for $z = \pm t$ or $z = \pm \frac{1}{t}$ becomes $k = t^2 + t^{-2}$. In this case, there are two 2-cycles with $\Q$-rational points $\{t, \frac{1}{t}\}$ and $\{-t, -\frac{1}{t}\}$ for $t \in \Q \setminus \{0\}$ except when $t = \pm 1$ where the two 2-cycles collapse to fixed points.
        \item The second factor vanishes when $k = \frac{z^4+1}{-z^2}$, which for $z = \pm t$ or $z = \pm \frac{1}{t}$ becomes $k = -(t^2 + t^{-2})$. In this case there are two 2-cycles with $\Q$-rational points $\{t, - \frac{1}{t}\}$ and $\{-t, \frac{1}{t}\}$ for $t \in \Q \setminus \{0\}$ excpet when $t = \pm 1$ where the two 2-cycles collapse to a single 2-cycle.
        \item The third factor vanishes when $k = \frac{1 + z^4}{-z^4}$, which for $z = \pm\frac{1}{t}$ becomes $k = -(1+t^4)$. In this case there is one 2-cycle with $\Q$-rational points $\{\frac{1}{t}, -\frac{1}{t}\}$ for $t \in \Q\setminus \{0\}$.
    \end{enumerate}
    We must also check if any of these three sets of 2-cycles can occur at the same time. We must find points on the pairwise intersections of the three curves.
    \begin{enumerate}
        \item For $k = t_1^2 + t_1^{-2} = -(t_2^2 + t_2^{-2})$ there are two components:
        \begin{align*}
            t_1^2 + t_2^2 &= 0\\
            t_1^2t_2^2 &= -1
        \end{align*}
        neither of which has a $\Q$-rational point.

        \item For $k = t_1^2 + t_1^{-2}$ and $k = -(1+t_3^4)$, we get the genus $5$ curve
        \begin{equation*}
            t_1^2t_3^4 + t_1^4 + t_1^2 + 1 =0.
        \end{equation*}
        Substituting $u=t_1^2$ and $v=t_3^2$, this covers the elliptic curve
        \begin{equation*}
            X:uv^2 + u^2 + u + 1=0.
        \end{equation*}
        The curve $X$ is birational to the elliptic curve
        \begin{equation*}
            E:y^2 = x^3-x^2+x,
        \end{equation*}
        via the map
        \begin{equation*}
            (u,v,w) \mapsto (z,y,-x)
        \end{equation*}
        where $w$ and $z$ are the homogenizing variables of $X$ and $E$, respectively.
        The curve $E$ is rank $0$ with torsion subgroup isomorphic to $\Z/4\Z$. The four torsion points are
        \begin{equation*}
            \{(0 : 1 : 0), (1 : 0 : 0), (-1 : 1 : 1), (-1 : -1 : 1)\}.
        \end{equation*}
        Only the two points $\{(0 : 1 : 0), (1 : 0 : 0)\}$ correspond to rational points on the original curve, and these points do not correspond to valid parameters.

        \item For $k = -(t_1^2 + t_1^{-2})$ and $k = -(1+t_2^4)$, we get the genus $5$ curve
        \begin{equation*}
            t_1^2t_2^4 - t_1^4 + t_1^2 - 1 =0.
        \end{equation*}
        Substituting $u=t_1^2$ and $v=t_3^2$, this covers the elliptic curve
        \begin{equation*}
            X:v^2 - u^2 + u - 1=0.
        \end{equation*}
        The curve $X$ is birational to the elliptic curve
        \begin{equation*}
            E:y^2 = x^3-x^2+x,
        \end{equation*}
        via the map
        \begin{equation*}
            (u,v,w) \mapsto (z,-y,x)
        \end{equation*}
        where $w$ and $z$ are the homogenizing variables of $X$ and $E$, respectively.
        The curve $E$ is rank $0$ with torsion subgroup isomorphic to $\Z/4\Z$. The four torsion points are
        \begin{equation*}
            \{(0 : 1 : 0), (1 : 0 : 0), (-1 : 1 : 1), (-1 : -1 : 1)\}.
        \end{equation*}
        Only the two points $\{(0 : 1 : 0), (1 : 0 : 0)\}$ correspond to rational points on the original curve, and these points do not correspond to valid parameters.
    \end{enumerate}

    Now we investigate whether we can get rational fixed points other than the point at infinity if we have rational 2-cycles. The first case is when the third factor of $\Phi^{\ast}_2(f_k)=0$ has a solution, so we want to see if $k = -(1+t^4) = 1 + l^4$ is possible for some rational $t$ and $l$. This gives the curve
    \begin{equation*}
        2 + l^4 + t^4 = 0,
    \end{equation*}
    which clearly has no solutions over $\Q$. Doing the same thing for the second component of $\Phi^{\ast}_2(f_k)=0$ yields
    \begin{align*}
        t^4 + (1+l^4)t^2  + 1&= 0.
    \end{align*}
    The substitution $x = t^2, y = l^2$ reduces this curve to
    \begin{equation*}
        X:x^2 + (1+y^2)x  + 1 = 0,
    \end{equation*}
    which is a genus 1 curve with the rational point $(x,y) = (-1,-1)$. We find that this curve is birational to the elliptic curve
    \begin{equation} \label{eq_E}
        E:y^2 = x^3 - 16x^2 + 96x - 192
    \end{equation}
    via the map
    \begin{align*}
        x&= 8xy - 8y^2\\
        y&=-8xy - 8y^2 - 8xz + 8yz - 16z^2\\
        z&=xy - y^2 + xz - yz
    \end{align*}
    where $z$ is the homogenizing variable for both curves.
    The curve $E$ is a rank 0 elliptic curve with torsion subgroup isomorphic to $\Z/4\Z$. Thus, the curve $X$ has four rational (projective) points, which we find to be $\{(0:1:0), (1:0:0), (-1:1:1),(-1:-1:1)\}$. These points are all nonsingular. The only affine points are $(-1,\pm 1)$. Both of these points have at least one negative coordinate so cannot lift to any points in $\Q$ since we covered by the squaring map. Because there are no rational points on the original curve, there cannot be any rational values of $k$ for which $f_k$ has rational fixed points and rational 2-cycles in this case.

    The first factor of $\Phi^{\ast}_2(f_k)=0$ is similar. The curve we get this time is
    \begin{equation*}
        X':t^4 - (1+l^4)t^2 + 1 = 0,
    \end{equation*}
    which differs from the previous case by a minus sign. We can use the same cover by $x$ and $y$ to get that this curve is birational to the same elliptic curve $E$ in equation \eqref{eq_E}. Again there are only four rational points on the curve $X'$:
    \begin{equation*}
        \{(1 : 1 : 1), (0 : 1 : 0), (1 : 0 : 0), (1 : -1 : 1)\}.
    \end{equation*}
    This time we see that one point has two positive coordinates, so the original curve has rational (affine) solutions $\{(1,1), (1,-1), (-1,1), (-1,-1)\}$. They all correspond to the value $k = 2$ ($t= \pm 1$), which is the case where the 2-cycles collapse to fixed points.
\begin{code}
\begin{lstlisting}
#fixed points
R.<k>=QQ[]
P.<x,y>=ProjectiveSpace(R,1)
f=DynamicalSystem([x^4 + y^4, k*x^3*y])
f.dynatomic_polynomial(1)

#2 periodic
R.<k>=QQ[]
P.<x,y>=ProjectiveSpace(R,1)
f=DynamicalSystem([x^4 + y^4, k*x^3*y])
f.dehomogenize(1).dynatomic_polynomial(2).factor()

#multiple 2-cycles
#1,2
s.<u,v>=QQ[]
F=u^2 + 1/u^2 + v^2 + 1/v^2
F=F.numerator()
A=AffineSpace(s)
C=A.curve(F)
for D in C.irreducible_components():
    D=A.curve(D.defining_polynomials()[0])
    print(D.genus())
    print(D)

#1,3
s.<u,v>=QQ[]
F=u^2 + 1/u^2  + (1+v^4)
F=F.numerator()
A=AffineSpace(s)
C=A.curve(F)
C

K := Rationals();
A<u,v> := AffineSpace(K,2);
C := ProjectiveClosure(Curve(A, u*v^2 + u^2+u+1));
p := C![0,1,0];
E, psi := EllipticCurve(C,p);
Rank(E);
TorsionSubgroup(E);
Points(C:Bound:=1000);

#2,3
s.<u,v>=QQ[]
F=u^2 + 1/u^2  - (1+v^4)
F=F.numerator()
A=AffineSpace(s)
C=A.curve(F)
C

K := Rationals();
A<u,v> := AffineSpace(K,2);
C := ProjectiveClosure(Curve(A, u*v^2 - u^2+u-1));
p := C![0,1,0];
E, psi := EllipticCurve(C,p);
Rank(E);
TorsionSubgroup(E);
Points(C:Bound:=1000);


#fixed + 2-cycles case 2
K := Rationals();
A<x,y> := AffineSpace(K, 2);
C := ProjectiveClosure(Curve(A, x^2 + (1+y^2)*x + 1));
p := C![-1,-1,1];
E, psi := EllipticCurve(C, p);
Rank(E);
TorsionSubgroup(E);
Points(C:Bound:=1000);

#fixed + 2-cycles case 3
K := Rationals();
A<x,y> := AffineSpace(K, 2);
C := ProjectiveClosure(Curve(A, x^2 - (1+y^2)*x + 1));
p := C![1,1,1];
E, psi := EllipticCurve(C, p);
Rank(E);
TorsionSubgroup(E);
Points(C:Bound:=1000);
\end{lstlisting}
\end{code}
\end{proof}

A search for rational preperiodic structures with the parameter up to height $10,000$ using the algorithm from \cite{Hutz12} as implemented in Sage yields no parameters where $f_k$ has a $\Q$-rational periodic point with minimal period at least $3$.

\begin{conj}\label{conj_M4_C4_periodic}
    There is no $k \in \Q$ such that $f_k(z) = \frac{z^4 + 1}{kz^3}$ has a $\Q$-rational periodic point of minimal period at least $3$.
\end{conj}
Assuming Conjecture \ref{conj_M4_C4_periodic}, we classify the $\Q$-rational preperiodic structures.

\begin{theorem} \label{theorem_A4_C4_preperiodic}
 Assuming Conjecture \ref{conj_M4_C4_periodic}, the possible $\Q$-rational preperiodic structures for $f_k(z) = \frac{z^4+1}{kz^3}$ for $k \in \Q$ are the following.
 \begin{align*}
G_1:&\xygraph{
!{<0cm,0cm>;<1cm,0cm>:<0cm,1cm>::}
!{(0,0) }*+{\underset{0}{\bullet}}="a"
!{(1,0) }*+{\underset{\infty}{\bullet}}="b"
"a":"b"
"b":@(rd,ru)"b"
}, \quad \text{all $k$ not in families $G_2$, $G_3$, or $G_4$}\\
G_2:&\xygraph{
!{<0cm,0cm>;<1cm,0cm>:<0cm,1cm>::}
!{(0,0) }*+{\underset{0}{\bullet}}="a"
!{(1,0) }*+{\underset{\infty}{\bullet}}="b"
!{(2,0) }*+{\bullet}="c"
!{(3,0) }*+{\bullet}="d"
"a":"b"
"b":@(rd,ru)"b"
"c":@(rd,ru)"c"
"d":@(rd,ru)"d"
},  \quad k = 1+t^4 \quad \text{ for } t \in \Q\setminus\{0\}\\
G_3:&\xygraph{
!{<0cm,0cm>;<1cm,0cm>:<0cm,1cm>::}
!{(0,0) }*+{\underset{0}{\bullet}}="a"
!{(1,0) }*+{\underset{\infty}{\bullet}}="b"
!{(2,0) }*+{\bullet}="c"
!{(3.5,0) }*+{\bullet}="d"
"a":"b"
"b":@(rd,ru)"b"
"c":@/^1pc/"d"
"d":@/^1pc/"c"
}, \quad k = -(1+t^4)\quad \text{ for } t \in \Q \setminus \{0\}\\
G_4:&\xygraph{
!{<0cm,0cm>;<1cm,0cm>:<0cm,1cm>::}
!{(0,0) }*+{\underset{0}{\bullet}}="f"
!{(1,0) }*+{\underset{\infty}{\bullet}}="a"
!{(2,0) }*+{\bullet}="b"
!{(3.5,0) }*+{\bullet}="c"
!{(4,0) }*+{\bullet}="d"
!{(5.5,0) }*+{\bullet}="e"
"f":"a"
"a":@(rd,ru)"a"
"b":@/^1pc/"c"
"c":@/^1pc/"b"
"d":@/^1pc/"e"
"e":@/^1pc/"d"
}, \quad k = \pm (t^2 + t^{-2}) \quad \text{ for } t \in \Q \setminus \{0, \pm 1\}.
\end{align*}
\end{theorem}
\begin{proof}
    Starting with rational fixed points, we can immediately see that $\infty$ is always a fixed point and $0$ is its only non-periodic rational (first) preimage for all parameters $a$. Since the numerator of $f_k(z)$ is $z^4 + 1$, there are no rational preimages of 0 in $\Q$. However, we might have non-periodic rational preimages of the additional rational fixed points when they do appear. We examine each of the two rational fixed points $z = \pm \frac{1}{t}$ when $k = 1+t^4$ and $t \neq 0$.

    In the first case, we look at $z = \frac{1}{t}$. We know that $k = 1 + t^4$, so finding the preimages of this point amounts to solving the equation
    \begin{equation*}
        \frac{z^4 + 1}{(1+t^4)z^3} = \frac{1}{t},
    \end{equation*}
    which determines the curve
    \begin{equation*}
        (1-zt)(z^2t^3 - z^3 + zt^2 + t) = 0.
    \end{equation*}
    The first factor corresponds to the fixed point itself since $z = \frac{1}{t}$ is fixed. The second factor can be reduced to the form
    \begin{equation*}
        E:y^2 = x^3 + x^2 + x
    \end{equation*}
    using the transformation $x = zt, y = z^2$.
    The curve $E$ has rank 0 and torsion subgroup isomorphic to $\Z/2\Z$. Since one of the coordinates of the transformation was the square map, each point of $E$ has one or two preimages. We find the original curve has three rational projective points  $\{(0 : 1 : 0), (0 : 0 : 1), (1 : 0 : 0)\}$. The two points $(0:1:0)$ and $(1:0:0)$ are singular. The point $(0:1:0)$ does not blow-up to a rational point and $(1:0:0)$ blows-up to a single rational point. Thus, these points give us the two points which are the preimages of the torsion points of $E$. The only affine one of these is $(0,0)$, but it is not valid since $t=0$ corresponds to all fixed points being the point at infinity.

    Now we look at $z = -\frac{1}{t}$. In this case, we want to solve the equation
    \begin{equation*}
        \frac{z^4 + 1}{(1+t^4)z^3} = -\frac{1}{t},
    \end{equation*}
    which is the same as finding rational points on
    \begin{equation*}
        (1 + zt)(z^2t^3 + z^3 - zt^2 +t) = 0.
    \end{equation*}
    The same transformation as above can be used to get
    \begin{equation*}
        y^2 = -x^3 + x^2 - x,
    \end{equation*}
    which is also rank 0 with torsion subgroup isomorphic to $\Z/2\Z$. For the same reason, there are three projective rational points but only one corresponds to an affine point. In particular, we again see that $(0,0)$ is the only $\Q$-rational solution in $z$ and $t$, which is invalid. Thus, the extra fixed points never have non-periodic preimages.

    We now investigate the case of points in 2-cycles having $\Q$-rational preperiodic tails. The first case is when $k = -(1+t^4)$, in which case the points $z = \frac{1}{t}$ and $z = -\frac{1}{t}$ map to each other. Finding rational preimages of these points results in the same equations as solving for preimages of fixed points, which we know do not exist.

    The second case is $k = t^2 + t^{-2} = \frac{t^4 + 1}{t^2}$, which has 2-cycles $\{t,\frac{1}{t}\}$ and $\{-t,-\frac{1}{t}\}$. The cases of $f_k(z) = \pm t$ reduce to the curves
    \begin{equation*}
        (zt\pm 1)(z^2t^3 \mp z^3 \pm zt^2 + t)=0.
    \end{equation*}
    The first component is the other point in the 2-cycle, so we focus on the second component which is two genus 3 curves isomorphic under $z \mapsto -z$. These curves as the same as analyzed earlier in this proof and the only affine point has $t=0$, which is degenerate.

    Next we look at $f_k(z) = \frac{t^2(z^4+1)}{z^3(t^4+1)} = \pm \frac{1}{t}$, which give the curves
    \begin{equation*}
        (z\pm t)(z^3t^3 \pm z^2 - zt \pm t^2)=0.
    \end{equation*}
    The first component is the other point in the 2-cycle, so we focus on the second component.

    This curve is genus $3$ and we can quotient by the order $2$ automorphism
    \begin{equation*}
        (z,t,h) \mapsto (-z,-t,h)
    \end{equation*}
    to obtain a genus 1 curve. A point search of low height gives the two points
    \begin{equation*}
        \{(-17/25:0:1:0), (17/125:0:1:0)\}.
    \end{equation*}
    Using the first point, the curve is birational to the elliptic curve
    \begin{equation*}
        y^2 = x^3 + 83521/688747536x,
    \end{equation*}
    which is rank $0$ with torsion subgroup isomorphic to $\Z/2\Z$. So every rational point on the original curve must map to one of these two rational points on the quotient curve.
    Using the equations of the quotient map, we find the rational points in the inverse image of each point to get
    \begin{equation*}
        \{(0:0:1),(0:1:0),(1:0:0)\}.
    \end{equation*}
    The affine point has $t=0$, which is the degenerate case. Thus, there are no non-periodic rational preimages of the points in any 2-cycle for the given parameterization of $k$.

The final case is when $k = -(t^2 + t^{-2})$, but it reduces to finding rational points on the same curves as the previous case.

\begin{code}
\begin{lstlisting}
#preimages of 2-cycles
R.<t>=QQ[]
P.<x,y>=ProjectiveSpace(R,1)
a=(t^4+1)/(t^2)
f=DynamicalSystem([x^4 + y^4, a*x^3*y])
s.<x,t>=QQ[]
F=(t^2*(x^4+1) + t*(t^4+1)*x^3)
A=AffineSpace(s)
C=A.curve(F)
C.irreducible_components()

K := Rationals();
P<z,t,h> := ProjectiveSpace(K, 2);
C := Curve(P, z^2*t^3 + z^3*h^2 - z*t^2*h^2 + t*h^4);
phi := iso<C->C|[-z,-t,h],[-z,-t,h]>;
// we will take the quotient by phi
G := AutomorphismGroup(C,[phi]);
CG,prj := CurveQuotient(G);
Points(CG:Bound:=1000);
E:=EllipticCurve(CG,CG![0,0,1,0]);
Rank(E);
TorsionSubgroup(E);

R.<z,t,h>=QQ[]
P=ProjectiveSpace(R)
L = [[0, 0, 1, 0], [0,1,0,0]]
f0=t*h^4
f1=t^3*h^2
f2=-z*t^4 - z^2*t*h^2 - z*h^4
f3=-z*t^4 - z^2*t*h^2 + t^3*h^2
for Q in L:
    X=P.subscheme([z^2*t^3 + z^3*h^2 - z*t^2*h^2 + t*h^4,Q[0]*f1-f0*Q[1], Q[0]*f2-Q[2]*f0, Q[0]*f3-Q[3]*f0, Q[1]*f2-Q[2]*f1, Q[1]*f3-Q[3]*f1, Q[2]*f3-Q[3]*f2])
    print(X.rational_points())



#z=\pm 1/t
R.<t>=QQ[]
P.<x,y>=ProjectiveSpace(R,1)
a=(t^4+1)/(t^2)
f=DynamicalSystem([x^4 + y^4, a*x^3*y])
s.<x,t>=QQ[]
F=(t^3*(x^4+1) + (t^4+1)*x^3)
A=AffineSpace(s)
C=A.curve(F)
C.irreducible_components()

K := Rationals();
P<z,t,h> := ProjectiveSpace(K, 2);
C := Curve(P, z^3*t^3 + z^2*h^4 - z*t*h^4 + t^2*h^4);
G := AutomorphismGroup(C);
phi := Generators(G)[2];
G := AutomorphismGroup(C,[phi]);
CG,prj := CurveQuotient(G);
Points(CG:Bound:=1000);
E:=EllipticCurve(CG,CG![-17/25,0,1,0]);
Rank(E);
TorsionSubgroup(E);

R.<z,t,h>=QQ[]
P=ProjectiveSpace(R)
L = [[-17/25,0,1,0],[17/125,0,1,0]]
f0=-27/25*z^2*t^3*h - 9/25*z*t^4*h - 27/25*z*h^5 + 27/25*t*h^5
f1=t^3*h^3
f2=27/17*z^2*t^3*h - 45/17*z*t^4*h + 27/17*z*h^5 - 27/17*t*h^5
f3=3*z^2*t^4 + 3*z*t*h^4
for Q in L:
    X=P.subscheme([z^3*t^3 + z^2*h^4 - z*t*h^4 + t^2*h^4,Q[0]*f1-f0*Q[1], Q[0]*f2-Q[2]*f0, Q[0]*f3-Q[3]*f0, Q[1]*f2-Q[2]*f1, Q[1]*f3-Q[3]*f1, Q[2]*f3-Q[3]*f2])
    print(X.rational_points())
\end{lstlisting}
\end{code}
\end{proof}

\subsection{$\A_4(D_3)$}
The model for this family is $f_k(z) = \frac{z^4 + kz}{kz^3 + 1}$. We start by classifying periodic points.
\begin{prop} \label{prop_A4_D3_periodic}
    For the family $f_k(z) = \frac{z^4+kz}{kz^3+1}$, we have the following $\Q$-rational periodic points.
    \begin{enumerate}
        \item The points $0$, $1$, and $\infty$ are fixed points for every choice of $k \in \Q \setminus \{\pm 1\}$. There are no other rational fixed points.
        \item For $k = t^2 + t + 1 + t^{-1} + t^{-2}$ with $t \in \Q\setminus \{0, \pm 1\}$, $f_k(z)$ has a single $2$-cycle with $\Q$-rational points $\left\{ t, \frac{1}{t}\right\}$.
    \end{enumerate}
\end{prop}
\begin{proof}
    We start with rational fixed points and see that $\infty$ is always fixed. For additional $\Q$-rational fixed points, we consider the first dynatomic polynomial
    \begin{equation*}
        \Phi^{\ast}_1(f_k) = (1-k)z(z-1)(z^2 + z +1).
    \end{equation*}
    Note that $k = 1$ is degenerate since it would lead us to cancelling a factor of $z^3 + 1$ in the function, resulting in a function that is not degree 4. The roots of $z^2 + z + 1 = 0$ are cube roots of unity so are not rational; thus, the fixed points are $0$, $1$, and $\infty$ regardless of $k$.

    Looking for 2-cycles with $\Q$-rational points, we consider the second dynatomic polynomial
    \begin{multline}\label{eq6}
        \Phi^{\ast}_2(f_k) = (k+1)(z^4 + z^3 + (1-k)z^2 + z + 1)\\
        (z^8 - z^7 + kz^6 + (k+1)z^5 + (k^2 - 2k - 1)z^4 + (k+1)z^3 + kz^2 -z + 1.
    \end{multline}
    Looking at the curve $\Phi^{\ast}_2(f_k) = 0$, we do not consider the degenerate component $k=-1$. The degree 4 component produces a genus 0 curve since it is linear in $k$. Solving for $k$, we get
    \begin{equation*}
        k = t^2 + t + 1 + t^{-1} + t^{-2}
    \end{equation*}
    for the $2$-cycle consisting of $z \in \left\{t, \frac{1}{t} \right\}$. With this parameterization of $k$ values, we look for addition $2$-cycles with $\Q$-rational points by factoring the second dynatomic polynomial
    \begin{multline*}
        \Phi_2^{\ast}(f_k) = (t^2 + 1)(t^2 + t + 1)(t - z)(zt - 1)(z^2 + zt + t^2)(z^2t^2 + zt + 1)(z^2t + z(t^2 + t +1) + t)\\
        (z^4t^2 + z^3( -t^3 -t^2-t) + z^2(t^4 + 2t^3 + 2t^2 + 2t + 1) + z(-t^3 - t^2 -t) + t^2).
    \end{multline*}
    It remains to be seen that no choice of $t \in \Q$ produces an additional 2-cycle with $\Q$-rational points. We examine each component in turn.
    We apply the quadratic equation to the component
    \begin{equation*}
        z^2 + zt + t^2 = 0
    \end{equation*}
    to have
    \begin{equation*}
        z = \frac{-t \pm t\sqrt{- 3}}{2}
    \end{equation*}
    which is never rational. Similarly for the component
    \begin{equation*}
        z^2t^2 + zt + 1=0
    \end{equation*}
    we get
    \begin{equation*}
        z = \frac{-1 \pm \sqrt{-3}}{2t}
    \end{equation*}
    which is rational only for $t$ a multiple of $\sqrt{-3}$, in which case $k \not\in \Q$. For the third quadratic component we also apply the quadratic formular to have
    \begin{equation*}
         z= \frac{-(t^2+t+1) \pm \sqrt{t^4 + 2t^3 - t^2 + 2t + 1}}{2t^2}.
    \end{equation*}
    For $z$ to be rational we need the discriminant to be a square, so we examine the curve
    \begin{equation*}
        X:t^4 + 2t^3 - t^2 + 2t + 1 = \ell^2.
    \end{equation*}
    This is a genus 1 curve birational to an ellitpic curve with Weierstrass model
    \begin{equation*}
        E:y^2 + 2xy + 8y = x^3 + 4x^2
    \end{equation*}
    via the map
    \begin{align*}
        x&=2t^2h + 2th^2 + 2\ell h^2 - 2h^3\\
        y&=4t^3 + 4t^2h + 4t\ell h - 4th^2\\
        z&=h^3,
    \end{align*}
    where $h$ and $z$ are the homogenizing variables of $X$ and $E$, respectively.
    The curve $E$ has rank 0 and torsion subgroup isomorphic to $\Z/4\Z$. A rational point search on the projective closure of $X$ yields the three points
    \begin{equation*}
        \{(0:-1:1), (0:1:0), (0:1:1)\}.
    \end{equation*}
    The point $(0:1:0)$ is singular and blows-up to two rational points, so these are all of the rational points on $X$. These all have $t=0$ which is $k=1$, which is degenerate. The last component (degree four in $x$) is irreducible over $\Q$ but is reducible over $\Q(\omega)$ where $\omega$ is a cube root of unity. We can factor it as
    \begin{multline*}
        z^4t^2 + z^3( -t^3 -t^2-t) + z^2(t^4 + 2t^3 + 2t^2 + 2t + 1) + z(-t^3 - t^2 -t) + t^2 =\\
        (t^2z + wtz^2 + tz + (-w - 1)t + z)(t^2z + (-w - 1)tz^2 + tz + wt + z).
    \end{multline*}
    Since the curve factors over an extension of $\Q$, rational points on it must be singular (Lemma \ref{lem_irreducible_curves}). The singular points are $(t,z) \in \{(-1,-1), (0,0)\}$ but $t \in \{0,-1\} $ are degenerate cases. Therefore, there is only one 2-cycle with $\Q$-rational points.
    Note that $t=1$ has the $2$-cycle collapsing to the fixed point $z=1$.

    Now the degree $8$ component of $\Phi^{\ast}_2(f_k)$ in \eqref{eq6} is irreducible over $\Q$ but is reducible over $\Q(\omega)$ where $\omega$ is a cube root of unity. We can factor it as
    \begin{multline*}
        z^8 - z^7 + kz^6 + (k+1)z^5 + (k^2 - 2k - 1)z^4 + (k+1)z^3 + kz^2 -z + 1 =\\
        (kz^2 - \omega z^4 +(\omega +1)z^3 -z^2 -\omega z + \omega +1)(kz^2+(\omega+1)z^4 - \omega z^3 - z^2 + (\omega + 1)z - \omega).
    \end{multline*}
    Since the curve factors over an extension of $\Q$, rational points on it must be singular (Lemma \ref{lem_irreducible_curves}). The singular points are just $(k,x) \in \{(-1,1), (1,-1)\}$ but $k = \pm 1$ are degenerate cases, so there are no valid $\Q$-rational points on this component.

\begin{code}
\begin{lstlisting}
#fixed points
R.<k>=QQ[]
P.<x,y>=ProjectiveSpace(R,1)
f=DynamicalSystem([x^4 + k*x*y^3, k*x^3*y+ y^4])
f.dehomogenize(1).dynatomic_polynomial(1).factor()

#2 periodic
R.<k>=QQ[]
P.<x,y>=ProjectiveSpace(R,1)
f=DynamicalSystem([x^4 + k*x*y^3, k*x^3*y+ y^4])
f.dehomogenize(1).dynatomic_polynomial(2).factor()

#only one 2-cycle:
R.<t>=QQ[]
P.<x,y>=ProjectiveSpace(R,1)
k=t^2+t+1+1/t+1/t^2
f=DynamicalSystem([x^4 + k*x*y^3, k*x^3*y+ y^4])
f.dehomogenize(1).dynatomic_polynomial(2).factor()

R.<t>=QQ[]
P.<x,y>=ProjectiveSpace(R,1)
k=t^2+t+1+1/t+1/t^2
f=DynamicalSystem([x^4 + k*x*y^3, k*x^3*y+y^4])
f.dynatomic_polynomial(2)*t^8
r.<t,x,y>=QQ[]
F=(t^10 + t^9 + 2*t^8 + t^7 + t^6)*x^12 + (3*t^12 + 6*t^11 + 13*t^10 + 16*t^9 + 20*t^8 + 16*t^7 + 13*t^6 + 6*t^5 + 3*t^4)*x^9*y^3 + (-t^16 - 4*t^15 - 8*t^14 - 14*t^13 - 17*t^12 - 20*t^11 - 16*t^10 - 16*t^9 - 12*t^8 - 16*t^7 - 16*t^6 - 20*t^5 - 17*t^4 - 14*t^3 - 8*t^2 - 4*t - 1)*x^6*y^6 + (3*t^12 + 6*t^11 + 13*t^10 + 16*t^9 + 20*t^8 + 16*t^7 + 13*t^6 + 6*t^5 + 3*t^4)*x^3*y^9 + (t^10 + t^9 + 2*t^8 + t^7 + t^6)*y^12
F.factor()

A<t,l>:=AffineSpace(Rationals(),2);
C:=ProjectiveClosure(Curve(A,l^2 - ((t^2+t+1)^2 - 4*t*t)));
Genus(C);
Points(C:Bound:=1000);
E,psi:=EllipticCurve(C,C![0,1,0]);
TorsionSubgroup(E);
Rank(E);
sing:=SingularPoints(C);
P:=sing[1];
P;
Q:=Places(P);
Q;
Degree(Q[1]);
Degree(Q[2]);

K.<w>=CyclotomicField(3)
A.<t,z>=AffineSpace(K,2)
C=Curve((t^4*z^2 - t^3*z^3 + 2*t^3*z^2 - t^3*z + t^2*z^4 - t^2*z^3 + 2*t^2*z^2 - t^2*z + t^2 - t*z^3 + 2*t*z^2 - t*z + z^2))
C.irreducible_components()
C.singular_points()

A.<a,x>=AffineSpace(QQ,2)
X=A.curve([x^8 - x^7 + a*x^6 + (a + 1)*x^5 + (a^2 - 2*a - 1)*x^4 + (a + 1)*x^3 + a*x^2 - x + 1])
X.singular_points()


#3-periodic
R.<a>=QQ[]
P.<x,y>=ProjectiveSpace(R,1)
f=DynamicalSystem([x^4 + a*x*y^3, a*x^3*y+ y^4])
f.dehomogenize(1).dynatomic_polynomial(3).factor()

K.<w>=CyclotomicField(3)
A.<a,x>=AffineSpace(K,2)
X=A.curve([((a^2 + a + 1)*x^6 + (a^2 + 4*a + 1)*x^3 + a^2 + a + 1)])
X.irreducible_components()
X.singular_points()

A<a,x>:=AffineSpace(Rationals(),2);
F:=(x^54 + (a^5 + 2*a^4 + 2*a^2 + 13*a)*x^51 + (2*a^8 + 4*a^7 + 12*a^6 + 19*a^5 + 5*a^4 + 30*a^3 + 78*a^2 + 3*a)*x^48 + (a^12 + 20*a^9 + 33*a^8 + 72*a^7 + 102*a^6 + 60*a^5 + 189*a^4 + 296*a^3 + 42*a^2 + 1)*x^45 + (a^14 + 5*a^13 - 3*a^12 + 9*a^11 + 78*a^10 + 123*a^9 + 295*a^8 + 374*a^7 + 333*a^6 + 732*a^5 + 834*a^4 + 255*a^3 + 10*a^2 + 14*a)*x^42 + (3*a^15 + 10*a^14 - a^13 + 33*a^12 + 135*a^11 + 321*a^10 + 795*a^9 + 958*a^8 + 1217*a^7 + 1992*a^6 + 1929*a^5 + 969*a^4 + 117*a^3 + 85*a^2 + 5*a)*x^39 + (3*a^16 + 13*a^15 + 12*a^14 + 39*a^13 + 151*a^12 + 528*a^11 + 1251*a^10 + 1919*a^9 + 2997*a^8 + 3993*a^7 + 4010*a^6 + 2616*a^5 + 618*a^4 + 356*a^3 + 57*a^2 + 1)*x^36 + (a^17 + 11*a^16 + 18*a^15 + 38*a^14 + 130*a^13 + 417*a^12 + 1226*a^11 + 2698*a^10 + 4560*a^9 + 6475*a^8 + 7310*a^7 + 5220*a^6 + 2258*a^5 + 1129*a^4 + 294*a^3 + 28*a^2 + 11*a)*x^33 + (5*a^17 + 7*a^16 + 39*a^15 + 107*a^14 + 172*a^13 + 684*a^12 + 2011*a^11 + 4391*a^10 + 7968*a^9 + 9862*a^8 + 8681*a^7 + 5733*a^6 + 2633*a^5 + 1183*a^4 + 219*a^3 + 55*a^2 + 8*a)*x^30 + (a^18 + 21*a^16 + 46*a^15 + 84*a^14 + 318*a^13 + 710*a^12 + 2166*a^11 + 5688*a^10 + 9270*a^9 + 11094*a^8 + 9114*a^7 + 5600*a^6 + 3360*a^5 + 762*a^4 + 322*a^3 + 63*a^2 + 1)*x^27 + (5*a^17 + 7*a^16 + 39*a^15 + 107*a^14 + 172*a^13 + 684*a^12 + 2011*a^11 + 4391*a^10 + 7968*a^9 + 9862*a^8 + 8681*a^7 + 5733*a^6 + 2633*a^5 + 1183*a^4 + 219*a^3 + 55*a^2 + 8*a)*x^24 + (a^17 + 11*a^16 + 18*a^15 + 38*a^14 + 130*a^13 + 417*a^12 + 1226*a^11 + 2698*a^10 + 4560*a^9 + 6475*a^8 + 7310*a^7 + 5220*a^6 + 2258*a^5 + 1129*a^4 + 294*a^3 + 28*a^2 + 11*a)*x^21 + (3*a^16 + 13*a^15 + 12*a^14 + 39*a^13 + 151*a^12 + 528*a^11 + 1251*a^10 + 1919*a^9 + 2997*a^8 + 3993*a^7 + 4010*a^6 + 2616*a^5 + 618*a^4 + 356*a^3 + 57*a^2 + 1)*x^18 + (3*a^15 + 10*a^14 - a^13 + 33*a^12 + 135*a^11 + 321*a^10 + 795*a^9 + 958*a^8 + 1217*a^7 + 1992*a^6 + 1929*a^5 + 969*a^4 + 117*a^3 + 85*a^2 + 5*a)*x^15 + (a^14 + 5*a^13 - 3*a^12 + 9*a^11 + 78*a^10 + 123*a^9 + 295*a^8 + 374*a^7 + 333*a^6 + 732*a^5 + 834*a^4 + 255*a^3 + 10*a^2 + 14*a)*x^12 + (a^12 + 20*a^9 + 33*a^8 + 72*a^7 + 102*a^6 + 60*a^5 + 189*a^4 + 296*a^3 + 42*a^2 + 1)*x^9 + (2*a^8 + 4*a^7 + 12*a^6 + 19*a^5 + 5*a^4 + 30*a^3 + 78*a^2 + 3*a)*x^6 + (a^5 + 2*a^4 + 2*a^2 + 13*a)*x^3 + 1);
FF:=(x^18 + (a^5 + 2*a^4 + 2*a^2 + 13*a)*x^17 + (2*a^8 + 4*a^7 + 12*a^6 + 19*a^5 + 5*a^4 + 30*a^3 + 78*a^2 + 3*a)*x^16 + (a^12 + 20*a^9 + 33*a^8 + 72*a^7 + 102*a^6 + 60*a^5 + 189*a^4 + 296*a^3 + 42*a^2 + 1)*x^15 + (a^14 + 5*a^13 - 3*a^12 + 9*a^11 + 78*a^10 + 123*a^9 + 295*a^8 + 374*a^7 + 333*a^6 + 732*a^5 + 834*a^4 + 255*a^3 + 10*a^2 + 14*a)*x^14 + (3*a^15 + 10*a^14 - a^13 + 33*a^12 + 135*a^11 + 321*a^10 + 795*a^9 + 958*a^8 + 1217*a^7 + 1992*a^6 + 1929*a^5 + 969*a^4 + 117*a^3 + 85*a^2 + 5*a)*x^13 + (3*a^16 + 13*a^15 + 12*a^14 + 39*a^13 + 151*a^12 + 528*a^11 + 1251*a^10 + 1919*a^9 + 2997*a^8 + 3993*a^7 + 4010*a^6 + 2616*a^5 + 618*a^4 + 356*a^3 + 57*a^2 + 1)*x^12 + (a^17 + 11*a^16 + 18*a^15 + 38*a^14 + 130*a^13 + 417*a^12 + 1226*a^11 + 2698*a^10 + 4560*a^9 + 6475*a^8 + 7310*a^7 + 5220*a^6 + 2258*a^5 + 1129*a^4 + 294*a^3 + 28*a^2 + 11*a)*x^11 + (5*a^17 + 7*a^16 + 39*a^15 + 107*a^14 + 172*a^13 + 684*a^12 + 2011*a^11 + 4391*a^10 + 7968*a^9 + 9862*a^8 + 8681*a^7 + 5733*a^6 + 2633*a^5 + 1183*a^4 + 219*a^3 + 55*a^2 + 8*a)*x^10 + (a^18 + 21*a^16 + 46*a^15 + 84*a^14 + 318*a^13 + 710*a^12 + 2166*a^11 + 5688*a^10 + 9270*a^9 + 11094*a^8 + 9114*a^7 + 5600*a^6 + 3360*a^5 + 762*a^4 + 322*a^3 + 63*a^2 + 1)*x^9 + (5*a^17 + 7*a^16 + 39*a^15 + 107*a^14 + 172*a^13 + 684*a^12 + 2011*a^11 + 4391*a^10 + 7968*a^9 + 9862*a^8 + 8681*a^7 + 5733*a^6 + 2633*a^5 + 1183*a^4 + 219*a^3 + 55*a^2 + 8*a)*x^28 + (a^17 + 11*a^16 + 18*a^15 + 38*a^14 + 130*a^13 + 417*a^12 + 1226*a^11 + 2698*a^10 + 4560*a^9 + 6475*a^8 + 7310*a^7 + 5220*a^6 + 2258*a^5 + 1129*a^4 + 294*a^3 + 28*a^2 + 11*a)*x^7 + (3*a^16 + 13*a^15 + 12*a^14 + 39*a^13 + 151*a^12 + 528*a^11 + 1251*a^10 + 1919*a^9 + 2997*a^8 + 3993*a^7 + 4010*a^6 + 2616*a^5 + 618*a^4 + 356*a^3 + 57*a^2 + 1)*x^6 + (3*a^15 + 10*a^14 - a^13 + 33*a^12 + 135*a^11 + 321*a^10 + 795*a^9 + 958*a^8 + 1217*a^7 + 1992*a^6 + 1929*a^5 + 969*a^4 + 117*a^3 + 85*a^2 + 5*a)*x^5 + (a^14 + 5*a^13 - 3*a^12 + 9*a^11 + 78*a^10 + 123*a^9 + 295*a^8 + 374*a^7 + 333*a^6 + 732*a^5 + 834*a^4 + 255*a^3 + 10*a^2 + 14*a)*x^4 + (a^12 + 20*a^9 + 33*a^8 + 72*a^7 + 102*a^6 + 60*a^5 + 189*a^4 + 296*a^3 + 42*a^2 + 1)*x^3 + (2*a^8 + 4*a^7 + 12*a^6 + 19*a^5 + 5*a^4 + 30*a^3 + 78*a^2 + 3*a)*x^2 + (a^5 + 2*a^4 + 2*a^2 + 13*a)*x + 1);
X:=Curve(A,FF);
Genus(X);
\end{lstlisting}
\end{code}
\end{proof}

    Turning to periodic points of period $3$, we look at the vanishing of the third dynatomic polynomial $\Phi^{\ast}_3(f_k)$. It has a degree 6 and a degree 54 component. The degree 6 component is given by
    \begin{equation*}
        (k^2 + k + 1)z^6 + (k^2 + 4k + 1)z^3 + k^2 + k + 1=0.
    \end{equation*}
    This is irreducible of $\Q$, but reduced over $\Q(\omega)$, where $\omega$ is a cube root of unity. Since the curve factors over an extension of $\Q$, rational points on it must be singular (Lemma \ref{lem_irreducible_curves}). The singular points are $(k,x) \in \{(-1,1), (1,-1)\}$ but $k = \pm 1$ are degenerate cases, so there are no (valid) rational points on this component. The degree 54 component can be simplified by replacing $z^3$ with $z$ to have a degree $18$ equation in $z$. This gives a genus 23 curve which is still problematic computationally. The only points of small height on it corresponded to $k = \pm 1$, so are degenerate; however, we are not able to fully analyze this curve.

A search for rational preperiodic structures with the parameter up to height $10,000$ using the algorithm from \cite{Hutz12} as implemented in Sage yields no parameters $k \in \Q$ where $f_k$ has a $\Q$-rational periodic point with minimal period at least $3$.

\begin{conj}\label{conj_M4_D3_periodic}
    There is no $k \in \Q$ such that $f_k(z) = \frac{z^4 + kz}{kz^3 + 1}$ has a $\Q$-rational periodic point of minimal period at least $3$.
\end{conj}
Assuming Conjecture \ref{conj_M4_D3_periodic}, we classify the $\Q$-rational preperiodic structures.

\begin{theorem}\label{thm_A4_D3_preperiodic}
 Assuming Conjecture \ref{conj_M4_D3_periodic}, the possible $\Q$-rational preperiodic structures for $f_k(z) = \frac{z^4+kz}{kz^3+1}$ for $k \in \Q$, with possibly finitely many exceptional values of the parameter $k$, are the following.
 \begin{align*}
G_1:&\xygraph{
!{<0cm,0cm>;<1cm,0cm>:<0cm,1cm>::}
!{(0,0) }*+{\underset{-1}{\bullet}}="a"
!{(1,0) }*+{\underset{1}{\bullet}}="b"
!{(2,0) }*+{\underset{0}{\bullet}}="c"
!{(3,0) }*+{\underset{\infty}{\bullet}}="d"
"a":"b"
"b":@(rd,ru)"b"
"c":@(rd,ru)"c"
"d":@(rd,ru)"d"
}, \quad \text{all $k$ not in $G_2$, $G_3$, $G_4$, or $G_5$}\\
G_2:&\xygraph{
!{<0cm,0cm>;<1cm,0cm>:<0cm,1cm>::}
!{(0,0) }*+{\underset{-1}{\bullet}}="a1"
!{(0,1) }*+{\bullet}="a2"
!{(1,1) }*+{\bullet}="a3"
!{(1,0) }*+{\underset{1}{\bullet}}="b"
!{(2,0) }*+{\underset{0}{\bullet}}="c"
!{(3,0) }*+{\underset{\infty}{\bullet}}="d"
"a1":"b"
"a2":"b"
"a3":"b"
"b":@(rd,ru)"b"
"c":@(rd,ru)"c"
"d":@(rd,ru)"d"
}, \quad \text{$k = t + t^{-1}$ for $t \in \Q \setminus \{0, \pm 1\}$}\\
G_3:&\xygraph{
!{<0cm,0cm>;<1cm,0cm>:<0cm,1cm>::}
!{(0,0) }*+{\underset{-1}{\bullet}}="b1"
!{(1,0) }*+{\underset{1}{\bullet}}="b"
!{(2,0) }*+{\bullet}="c1"
!{(3,0) }*+{\underset{0}{\bullet}}="c"
!{(4,0) }*+{\bullet}="d1"
!{(5,0) }*+{\underset{\infty}{\bullet}}="d"
"b1":"b"
"b":@(rd,ru)"b"
"c1":"c"
"c":@(rd,ru)"c"
"d1":"d"
"d":@(rd,ru)"d"
}, \quad \text{$k=t^3$ for $t \in \Q \setminus \{0, \pm 1\}$}\\
G_4:&\xygraph{
!{<0cm,0cm>;<1cm,0cm>:<0cm,1cm>::}
!{(0,0) }*+{\underset{-1}{\bullet}}="a"
!{(1,0) }*+{\underset{1}{\bullet}}="b"
!{(2,0) }*+{\underset{0}{\bullet}}="c"
!{(3,0) }*+{\underset{\infty}{\bullet}}="d"
!{(4,0) }*+{\bullet}="e"
!{(5.5,0) }*+{\bullet}="f"
"a":"b"
"b":@(rd,ru)"b"
"c":@(rd,ru)"c"
"d":@(rd,ru)"d"
"e":@/^1pc/"f"
"f":@/^1pc/"e"
}, \quad \text{$k = t^2 + t + 1 + t^{-1} + t^{-2}$ for $t \in \Q\setminus \{0, \pm 1\}$}\\
G_5:&\xygraph{
!{<0cm,0cm>;<1cm,0cm>:<0cm,1cm>::}
!{(0,-0.5) }*+{\bullet}="a1"
!{(0,0.5) }*+{\bullet}="a2"
!{(1,0) }*+{\underset{-1}{\bullet}}="a"
!{(2,0) }*+{\underset{1}{\bullet}}="b"
!{(3,0) }*+{\underset{0}{\bullet}}="c"
!{(4,0) }*+{\underset{\infty}{\bullet}}="d"
"a1":"a"
"a2":"a"
"a":"b"
"b":@(rd,ru)"b"
"c":@(rd,ru)"c"
"d":@(rd,ru)"d"
}, \quad \text{for $k = \frac{-1-t^4}{t^3+t}$ for $t \in \Q \setminus{0, \pm 1}$}.
\end{align*}
\end{theorem}
\begin{proof}
    We start with non-periodic rational preimages of the fixed points. We begin with $0$, which amounts to solving $z^4 + kz = 0$, giving us the point $z = -t$ as the preimage of $0$ when $k = t^3$. The other non-periodic preimages of $0$ are not rational. Similarly, preimages of $\infty$ come from solutions to $az^3 + 1 = 0$, which are parameterized by $k = t^3$ as well, but the preimage is $z = -\frac{1}{t}$. The other non-periodic preimages of $\infty$ are not rational. Thus, the parameter $k = t^3$ gives two additional rational points in the preperiodic structure.

    Preimages of 1 come from solutions to $f_k(z) = 1$, which gives the curve
    \begin{equation*}
        z^4 - kz^3 + kz - 1 = (z-1)(z+1)(z^2 - kz + 1) = 0.
    \end{equation*}
    The factor of $z - 1$ is expected since 1 is its own preimage. The second factor has $z=-1$ as a rational preimage of $1$ for all $k$. Finally, the last factor is genus $0$, so when $k = t + t^{-1}$ the point $z = t$ maps to 1. Furthermore, this image is invariant when we replace $t$ by $\frac{1}{t}$, so in fact there will be two rational points that map to 1 for these values of $k$. Note that $t = \pm 1$ has the three rational preimages collapsing to a single preimage of the fixed point.

    Now we check if we can have additional non-periodic preimages of $1$ at the same time as non-periodic preimages of $0$ and $\infty$.
    We need values of $k$ such that $k = t+ t^{-1} = \ell^3$ for $t,\ell \in \Q$. The curve
    \begin{equation*}
        X:t^2 - t\ell^3 + 1 = 0
    \end{equation*}
    is a genus 2 hyperelliptic curve with model
    \begin{equation*}
        C:y^2 + x^3y + 1 = 0,
    \end{equation*}
    obtained by setting $x=\ell$ and $y=-t$. The curve $C$ has a (isomorphic) simplified model
    \begin{equation*}
        C':y^2 = x^6 - 4
    \end{equation*}
    whose Jacobian has rank $1$ and torsion subgroup isomorphic to $\Z/3\Z$. Magma's Chabauty method tells us that there are only two rational points on this hyperelliptic curve. However, since the map from the original curve was not an isomorphism, we still have some work to do: we need to investigate the behavior at the singular points. Fortunately, the only singular point on the projective model is $(1:0:0)$, and an initial search on the projective closure of $X$ finds the rational points $(0:1:0)$ and $(1:0:0)$, which are all we expect to have. We need to look at the places of the divisor associated to this singular point. There is only one; because it is of degree $1$, we have not missed any rational points. Thus, the only rational points are at infinity, so the affine model has none and we cannot have additional $\Q$-rational priemges of $1$ at the same time as non-periodic $\Q$-rational preimages of the fixed points $0$ and $\infty$.

    Now we look at non-periodic second preimages of the fixed point $1$. First we consider preimages of $-1$. We are looking for rational points on the curve
    \begin{equation*}
        z^4 + kz^3 + kz + 1=0.
    \end{equation*}
    This equation is linear in $k$, so we can solve as $k = \frac{-1-t^4}{t^3+t}$ for $t \in \Q \setminus \{0, \pm 1\}$. With this expression for $k$ the preimages of $-1$ satisfy
    \begin{equation*}
        (z - 1/t)(z - t)(z^2 + \frac{2t}{t^2 + 1}z + 1)=0.
    \end{equation*}
    Apply the quadratic equation to the last factor, we see
    \begin{equation*}
        z = \frac{-t \pm \sqrt{-1}}{t^2+1}
    \end{equation*}
    so there are at most $2$ rational preimages. To check whether there are one or two preimages we compute the discriminant of the curve equation with respect to $z$ to see that $k= \pm 1$ are the only rational values where the two preimages collapse to a single rational preimage.  The values $k = \pm 1$ are degenerate so this does not occur.

    Now we consider rational preimages of the preimages of $-1$. We substitute $k = \frac{-1-t^4}{t^3+t}$ and examine $f(z) = -1$ which becomes
    \begin{equation*}
        (z - 1/t)(z - t)(z^2 + (2t/(t^2 + 1))z + 1)=0.
    \end{equation*}
    Taking the last component, we need to find rational points on the (genus 1) curve
    \begin{equation*}
        C:(t^2 + 1)z^2 + 2tz + t^2 + 1=0.
    \end{equation*}
    Over $\Q(i)$ it has the rational point $(-i:0:1)$, which produces the Weierstrass model
    \begin{equation*}
        E:y^2 = x^3 + 4x^2 - 64x - 256
    \end{equation*}
    via the map
    \begin{equation*}
        (x,y,z) = (8it,-32iz,h),
    \end{equation*}
    where $h$ and $z$ are the homogenizing variables of $C$ and $E$, respectively.
    The curve $E$ is rank $0$ with torsion subgroup isomorphic to $\Z/2\Z \times \Z/4\Z$. Taking the inverse image of the eight rational points on $E$, we get the $\Q(i)$-rational points on the curve $C$ are
    \begin{equation*}
        \{(0 : 1 : 0), (i : 0 : 1), (1 : 0 : 0),(0 : -i : 1), [(0 : i : 1), (-i : 0 : 1) \}.
    \end{equation*}
    The only $\Q$-rational value of $t$ is $0$, which is not a valid parameter. So there are no second $\Q$-rational preimages of $-1$.

    There remain a few cases that were unable to be fully resolved resulting in possibly finitely many exceptional $k$ values with larger graph structures. We go through those cases now.

    Additional rational preimages of $1$ occur for $k = t + t^{-1}$ and we consider $\Q$-rational preimages of those additional preimages. We are looking for rational points on the (genus 6) curve
    \begin{equation*}
        tz^4 + (-t^3 - t)z^3 + (t^2 + 1)z - t^2=0.
    \end{equation*}
    This curve is genus 6 are was not amenable to any of the methods we tried. It has no valid points of small height.

    Non-periodic $\Q$-rational second preimages of $0$ are defined by the curves
    \begin{equation*}
        z^4 + t^4z^3 + t^3z + t = 0
    \end{equation*}
    and
    \begin{equation*}
        z^8 - t^4z^7 + t^8z^6 + 2t^3z^5 + (-t^7 - t)z^4 + 2t^5z^3 + t^6z^2 - t^4z + t^2 = 0.
    \end{equation*}
    These curves are both genus 8 are were not amenable to any of the methods we tried. They have no valid points of small height.

    Non-periodic $\Q$-rational second preimages of $\infty$ are equivalent to non-periodic $\Q$-rational second preimages of $0$ under the automorphism $z \mapsto 1/z$.

    The existence of a non-periodic $\Q$-rational preimage of $0$ and a $\Q$-rational second preimage of $1$ corresponds to rational points on the curve
    \begin{equation*}
        t_1^3(t_2^3+t_2) - (-1-t_2^4) = 0.
    \end{equation*}
    This curve is genus 6 and was not amenable to any of the methods we tried. It has no valid points of small height.

    The exists of a $\Q$-rational 2-cycle and a non-periodic $\Q$-rational preimage of a fixed point corresponds to rational points on the curve
    \begin{equation*}
        t_2^2t_1^3 - (t_2^4 + t_2^3 + t_2^2 + t_2 + 1)=0.
    \end{equation*}
    This curve is genus 4 and was not amenable to any of the methods we tried. It has no valid points of small height.

    It remains to consider non-periodic preimages of points in the $2$-cycle. The 2-cycle $\left\{t, \frac{1}{t}\right\}$ occurs for $k =t^2 + t + 1 + t^{-1} + t^{-2}$ with $t \in \Q \setminus \{0, \pm 1\}$. By the automorphism $z \mapsto 1/z$, it is equivalent to look at the preimage of either point in the cycle. Looking at preimages of $t$, we need to find rational points on the genus $6$ curve
    \begin{equation*}
        tz^3 - t(t^3 + t^2 + t + 1)z^2 - (t^3 + t^2 + t + 1)z + t^3=0.
    \end{equation*}
    This curve was not amenable to any of the method we tried. It has no valid points of small height.

    All of these unresolved cases are curves of genus at least 4 so can have at most finitely many $\Q$-rational points, due to Faltings' Theorem. These possible rational points corresponds to the possibly finitely many exceptions in the statement.

\begin{code}
\begin{lstlisting}
#second preimages of 0
R.<t>=QQ[]
P.<x,y>=ProjectiveSpace(R,1)
k=t^3
f=DynamicalSystem([x^4+k*x*y^3, k*x^3*y+y^4])
F=f.nth_iterate_map(2).dehomogenize(1)[0]
(F.numerator()).factor()

A<x,t>:=AffineSpace(Rationals(),2);
C:=Curve(A,(x^4 + t^4*x^3 + t^3*x + t));
Genus(C);

A<x,t>:=AffineSpace(Rationals(),2);
C:=Curve(A,(x^8 + (-t^4)*x^7 + t^8*x^6 + 2*t^3*x^5 + (-t^7 - t)*x^4 + 2*t^5*x^3 + t^6*x^2 + (-t^4)*x + t^2));
Genus(C);

## These are both genus 8 curves and there doesn't seem to be anything accessible for them. Only points of small height are degenerate

#second preimages of infty
R.<t>=QQ[]
P.<x,y>=ProjectiveSpace(R,1)
k=t^3
f=DynamicalSystem([x^4+k*x*y^3, k*x^3*y+y^4])
F=f.nth_iterate_map(2).dehomogenize(1)[0]
(F.denominator()).factor()


A<x,t>:=AffineSpace(Rationals(),2);
C:=Curve(A,(t*x^4 + t^3*x^3 + t^4*x + 1));
Genus(C);
Points(ProjectiveClosure(C):Bound:=1000);

A<x,t>:=AffineSpace(Rationals(),2);
C:=Curve(A,(t^2*x^8 + (-t^4)*x^7 + t^6*x^6 + 2*t^5*x^5 + (-t^7 - t)*x^4 + 2*t^3*x^3 + t^8*x^2 + (-t^4)*x + 1));
Genus(C);
Points(ProjectiveClosure(C):Bound:=1000);

## These are both genus 8 curves and there doesn't seem to be anything accessible for them. Only points of small height are degenerate


#all 3 fixed points with preimages
K := Rationals();
A<t,l> := AffineSpace(K, 2);
C := ProjectiveClosure(Curve(A, t^2 - t*l^3 + 1));
boo, D, psi := IsHyperelliptic(C);
S, phi := SimplifiedModel(D);
J := Jacobian(S);
Jratpts := RationalPoints(J:Bound:=1000);
TorsionSubgroup(J);
RankBounds(J);Jrat
Order(Jratpts[5]);
infpt := Jratpts[5];
Chabauty(infpt);
RationalPoints(C: Bound:=1000);
singpt := SingularPoints(C)[1];
p := Places(singpt)[1];
Degree(p);

#discriminant of preimage of 1
R.<k>=QQ[]
S.<z>=R[]
f=z^2-k*z+1
f.discriminant().factor()

#preimages of preimages of fixed points
#preimages of -1
R.<k>=QQ[]
P.<x,y>=ProjectiveSpace(R,1)
f=DynamicalSystem([x^4 + k*x*y^3, k*x^3*y+ y^4])
F=f.nth_iterate_map(1).dehomogenize(1)[0]
F.numerator()+F.denominator()

R.<t>=QQ[]
k=(-1-t^4)/(t^3+t)
P.<x,y>=ProjectiveSpace(R,1)
f=DynamicalSystem([x^4 + k*x*y^3, k*x^3*y+ y^4])
F=f.nth_iterate_map(1).dehomogenize(1)[0]
(F.numerator()+F.denominator()).factor()


#preimages of other preimages
R.<t>=QQ[]
k=t+1/t
P.<x,y>=ProjectiveSpace(R,1)
f=DynamicalSystem([x^4 + k*x*y^3, k*x^3*y+ y^4])
F=f.nth_iterate_map(1).dehomogenize(1)[0]
(F.numerator()-t*F.denominator())*t
A.<t,x>=AffineSpace(QQ,2)
X=A.curve(t*x^4 + (-t^3 - t)*x^3 + (t^2 + 1)*x - t^2)

A<t,x>:=AffineSpace(Rationals(),2);
X:=Curve(A,t*x^4 + (-t^3 - t)*x^3 + (t^2 + 1)*x - t^2);
Genus(X);
Points(ProjectiveClosure(X):Bound:=1000);

#looking at second preimages of 1 seems equivalently hard...
R.<k>=QQ[]
P.<x,y>=ProjectiveSpace(R,1)
f=DynamicalSystem([x^4 + k*x*y^3, k*x^3*y+ y^4])
F=f.nth_iterate_map(2).dehomogenize(1)[0]
F.numerator()-F.denominator()
A.<k,x>=AffineSpace(QQ,2)
X=A.curve(x^16 + (-k^2)*x^15 + (k^4 + 4*k)*x^13 + (-k^4 - 3*k^3 - k)*x^12 + (k^5 + 3*k^3 + 6*k^2)*x^10 + (-3*k^4 - 4*k^3 - 3*k^2)*x^9 + (3*k^4 + 4*k^3 + 3*k^2)*x^7 + (-k^5 - 3*k^3 - 6*k^2)*x^6 + (k^4 + 3*k^3 + k)*x^4 + (-k^4 - 4*k)*x^3 + k^2*x - 1)
X.irreducible_components()

A<k,x>:=AffineSpace(Rationals(),2);
F:=(x^8 + (-k^2)*x^7 + k^2*x^6 + 2*k*x^5 + (-k^3 - k)*x^4 + 2*k*x^3 + k^2*x^2 + (-k^2)*x + 1);
X:=Curve(A,F);
IsAbsolutelyIrreducible(X); //true
RationalPoints(ProjectiveClosure(X):Bound:=1000);
Genus(X); //6

//If we assume rank <= 5, then this gives all 6 rational points.
load "coleman.m";
Q:= y^8 + (-x^2)*y^7 + x^2*y^6 + 2*x*y^5 + (-x^3 - x)*y^4 + 2*x*y^3 + x^2*y^2 + (-x^2)*y + 1;
data:=coleman_data(Q,5,15);
Qpoints:=Q_points(data,1000);



#second preimages of -1
R.<t>=QQ[]
k=(-1-t^4)/(t^3+t)
P.<x,y>=ProjectiveSpace(R,1)
f=DynamicalSystem([x^4 + k*x*y^3, k*x^3*y+ y^4])
F=f.nth_iterate_map(1).dehomogenize(1)[0]
((F.numerator()+F.denominator())*(t^3 + t)).factor()

R<z> := PolynomialRing(Integers());
f := z^2 +1;
K<i> := NumberField(f);
P<t,x,z> := ProjectiveSpace(K,2);
C := Curve(P, (t^2 + z^2)*x^2 + 2*t*x*z^2 + t^2*z^2 + z^4);
p := C![-i,0,1];
E, psi := EllipticCurve(C,p);
Rank(E);
TorsionSubgroup(E);


K.<i>=QuadraticField(-1)
k=(-1-t^4)/(t^3+t)
P.<t,x,z>=ProjectiveSpace(K,2)
R=P.coordinate_ring()
f0=8*i*t*x^3 + 8*i*t*x*z^2 + 8*i*x^2*z^2 + 8*x*z^3
f1=-32*i*t*x^2*z - 16*x^2*z^2 - 32*i*t*z^3 - 32*i*x*z^3 - 32*z^4
f2=x^3*z
F=[f0,f1,f2]
for Q in E.torsion_points():
    X=P.subscheme([(t^2 + z^2)*x^2 + 2*t*x*z^2 + t^2*z^2 + z^4,Q[0]*F[1] - Q[1]*F[0], Q[0]*F[2]-Q[2]*F[0], Q[1]*F[2]-Q[2]*F[1]])
    print(X.rational_points())

# seconds preimges of -1 and preimages of 0,infty
A<t1,t2>:=AffineSpace(Rationals(),2);
X:=ProjectiveClosure(Curve(A,t1^3*(t2^3+t2) - (-1-t2^4)));
Genus(X);
Points(X:Bound:=1000);

#preimages of the 2-cycle
R.<t>=QQ[]
k=t^2+t+1+1/t+1/t^2
P.<x,y>=ProjectiveSpace(R,1)
f=DynamicalSystem([x^4 + k*x*y^3, k*x^3*y+ y^4])
F=f.nth_iterate_map(1).dehomogenize(1)[0]
((F.numerator()-t*F.denominator())).factor()

P<t,x,z> := ProjectiveSpace(Rationals(),2);
C := Curve(P, -t^4*x^2 - t^3*x^2*z - t^3*x*z^2 - t^2*x^2*z^2 + t*x^3*z^2 + t^3*z^3 - t^2*x*z^3 - t*x^2*z^3 - t*x*z^4 - x*z^5);

load "coleman.m";
Q:= -y^4*x^2 - y^3*x^2 - y^3*x - y^2*x^2 + y*x^3 + y^3 - y^2*x - y*x^2 - y*x - x;
data:=coleman_data(Q,5,15);
Qpoints:=Q_points(data,1000);


# 2-cycle and preimages of fixed points

## fixed point 0
A<t1,t2>:=AffineSpace(Rationals(),2);
X:=ProjectiveClosure(Curve(A,t2^2*t1^3 - (t2^4 + t2^3 + t2^2 + t2 + 1)));
Genus(X);
Points(X:Bound:=1000);

## additional preimages of 1
A<t1,t2>:=AffineSpace(Rationals(),2);
X:=ProjectiveClosure(Curve(A,t2^2*(t1^2+1) - t1*(t2^4 + t2^3 + t2^2 + t2 + 1)));
Genus(X);
Points(X:Bound:=1000);
boo, D, psi := IsHyperelliptic(X);
S, phi := SimplifiedModel(D);
J := Jacobian(S);
#now what???

## preimages of -1
A<t1,t2>:=AffineSpace(Rationals(),2);
X:=ProjectiveClosure(Curve(A,t2^2*(-1-t1^4) - (t1^3+t1)*(t2^4 + t2^3 + t2^2 + t2 + 1)));
Genus(X);
Points(X:Bound:=1000);


\end{lstlisting}
\end{code}
\end{proof}

\subsection{The locus $\A_4(C_3)$} \label{sect_A4_C3_preperiodic}

This family is given by $f_{k_1,k_2}(z) = \frac{z^4 + k_1z}{k_2z^3+1}$. We first examine periodic points.

\begin{prop}
    For the family $f_{k_1,k_2}(z) = \frac{z^4 + k_1z}{k_2z^3+1}$, we have the following $\Q$-rational periodic points.
    \begin{enumerate}
        \item The points $0$ and $\infty$ are fixed for all pairs $(k_1,k_2)$.
        \item For $(u,v) \in \bbA^2(\Q)$ and
            \begin{align*}
                k_1 &= u^4v - u^3 + 1\\
                k_2 &= uv,
            \end{align*}
            $f_{k_1,k_2}(z)$ has one additional rational fixed point $z=u$.
    \end{enumerate}
\end{prop}
\begin{proof}
We look at the first dynatomic polynomial. The points $0$ and $\infty$ are factors for all choices of parameters $k_1$ and $k_2$. The other component is given by
\begin{equation*}
    (k_2 - 1)z^3 - k_1 + 1=0.
\end{equation*}
This forms a rational surface that has the given parameterization. Substituting the parameterization back into $\Phi_1^{\ast}(f_{k_1,k_2})$, we see there is one additional fixed point.

\begin{code}
\begin{lstlisting}
#fixed points
R.<k1,k2>=QQ[]
P.<x,y>=ProjectiveSpace(R,1)
f=DynamicalSystem([x^4 + k1*x*y^3,k2*x^3*y+y^4])
f.dehomogenize(1).dynatomic_polynomial(1).factor()

#Surface information (run in Magma)
P3<k1,k2,h,x>:=ProjectiveSpace(Rationals(),3);
S:=Surface(P3,((k2 - h)*x^3 +(h- k1)*h^3));
P2<u,v,w>:=ProjectiveSpace(Rationals(),2);
ParametrizeProjectiveHypersurface(S,P2);

R.<u,v,w>=QQ[]
P.<x,y>=ProjectiveSpace(R,1)
k1=(u^4*v - u^3*w^2 + w^5)/(w^5)
k2=(u*v*w^3)/(w^5)
f=DynamicalSystem([x^4 + k1*x*y^3,k2*x^3*y+y^4])
f.dehomogenize(1).dynatomic_polynomial(1).factor()

\end{lstlisting}
\end{code}
\end{proof}

    As with the degree $3$ families with multiple parameters, periodic points with higher periods and (strictly) preperiodic points were difficult to study. We content ourselves with a census of $\Q$-rational preperiodic structures for parameters $k_1$ and $k_2$ with small height.

\begin{center}
Table: Preperiodic Graphs for Pairs $(k_1,k_2)$

\smallskip

\middlexcolumn
\offinterlineskip
\begin{tabularx}{0.98\textwidth}{%
  | *{4}{>{\centering\arraybackslash}X|}
}
\hline
\heading{$(0,1/2)$} & \heading{$(0,1)$} & \heading{$(0,2)$} & \heading{$(0,0)$}\\
\pre{\xygraph{
!{<0cm,0cm>;<1cm,0cm>:<0cm,1cm>::}
!{(0,1) }*+{\underset{0}{\bullet}}="a"
!{(1,1) }*+{\underset{\infty}{\bullet}}="b"
"a":@(rd,ru)"a"
"b":@(rd,ru)"b"
}}&
\pre{\xygraph{
!{<0cm,0cm>;<1cm,0cm>:<0cm,1cm>::}
!{(0,0) }*+{\bullet}="c"
!{(0,1) }*+{\underset{0}{\bullet}}="a"
!{(1,0) }*+{\underset{\infty}{\bullet}}="b"
"a":@(rd,ru)"a"
"b":@(rd,ru)"b"
"c":"b"
}}&
\pre{\xygraph{
!{<0cm,0cm>;<1cm,0cm>:<0cm,1cm>::}
!{(1,0) }*+{\bullet}="c"
!{(0,1) }*+{\underset{0}{\bullet}}="a"
!{(1,1) }*+{\underset{\infty}{\bullet}}="b"
"a":@(rd,ru)"a"
"b":@(rd,ru)"b"
"c":@(rd,ru)"c"
}}&
\pre{\xygraph{
!{<0cm,0cm>;<1cm,0cm>:<0cm,1cm>::}
!{(0,0) }*+{\bullet}="c"
!{(1,0) }*+{\bullet}="d"
!{(0,1) }*+{\underset{0}{\bullet}}="a"
!{(1,1) }*+{\underset{\infty}{\bullet}}="b"
"a":@(rd,ru)"a"
"b":@(rd,ru)"b"
"d":@(rd,ru)"d"
"c":"d"
}}\\
\hline
\end{tabularx}

\begin{tabularx}{0.98\textwidth}{%
  | *{3}{>{\centering\arraybackslash}X|}
}
\hline
\heading{$(1,-1)$} & \heading{$(1,-3)$} & \heading{$(1/2,-7/4)$}\\
\pre{\xygraph{
!{<0cm,0cm>;<1cm,0cm>:<0cm,1cm>::}
!{(0,0) }*+{\bullet}="c"
!{(1,0) }*+{\underset{0}{\bullet}}="a"
!{(0,1) }*+{\bullet}="e"
!{(1,1) }*+{\underset{\infty}{\bullet}}="b"
"a":@(rd,ru)"a"
"b":@(rd,ru)"b"
"e":"b"
"c":"a"
}}&
\pre{\xygraph{
!{<0cm,0cm>;<1cm,0cm>:<0cm,1cm>::}
!{(0,1) }*+{\bullet}="c"
!{(1,1) }*+{\bullet}="d"
!{(2,1) }*+{\underset{0}{\bullet}}="a"
!{(2,0) }*+{\underset{\infty}{\bullet}}="b"
"c":"d"
"d":"a"
"a":@(rd,ru)"a"
"b":@(rd,ru)"b"
}}&
\pre{\xygraph{
!{<0cm,0cm>;<1cm,0cm>:<0cm,1cm>::}
!{(0,0) }*+{\underset{0}{\bullet}}="a"
!{(0,1) }*+{\underset{\infty}{\bullet}}="b"
!{(1.5,0.5) }*+{\bullet}="f"
!{(3,0.5) }*+{\bullet}="e"
"a":@(rd,ru)"a"
"b":@(rd,ru)"b"
"f":@/^1pc/"e"
"e":@/^1pc/"f"
}}
\\
\hline
\end{tabularx}

\begin{tabularx}{0.98\textwidth}{%
  | *{3}{>{\centering\arraybackslash}X|}
}
\hline
\heading{$(4/3,-8)$} & \heading{$(-5/12,-25/18)$} & \heading{$(5/2,5/2)$}\\
\pre{\xygraph{
!{<0cm,0cm>;<1cm,0cm>:<0cm,1cm>::}
!{(0,1) }*+{\bullet}="c"
!{(1,1) }*+{\underset{0}{\bullet}}="a"
!{(0,0) }*+{\bullet}="d"
!{(1,0) }*+{\underset{\infty}{\bullet}}="b"
!{(2.5,0.5) }*+{\bullet}="f"
"a":@(rd,ru)"a"
"c":"a"
"b":@(rd,ru)"b"
"d":"b"
"f":@(rd,ru)"f"
}}&
\pre{\xygraph{
!{<0cm,0cm>;<1cm,0cm>:<0cm,1cm>::}
!{(0,1) }*+{\underset{0}{\bullet}}="a"
!{(0,0) }*+{\underset{\infty}{\bullet}}="b"
!{(1,0.5) }*+{\bullet}="c"
!{(2,0.5) }*+{\bullet}="f"
!{(3.5,0.5) }*+{\bullet}="g"
"a":@(rd,ru)"a"
"b":@(rd,ru)"b"
"c":"f"
"f":@/^1pc/"g"
"g":@/^1pc/"f"
}}&
\pre{\xygraph{
!{<0cm,0cm>;<1cm,0cm>:<0cm,1cm>::}
!{(0,1) }*+{\underset{0}{\bullet}}="a"
!{(0,0) }*+{\underset{\infty}{\bullet}}="b"
!{(2.5,0) }*+{\bullet}="c"
!{(1.5,0) }*+{\bullet}="d"
!{(1.5,1) }*+{\bullet}="e"
!{(2.5,1) }*+{\bullet}="f"
"a":@(rd,ru)"a"
"b":@(rd,ru)"b"
"c":@(rd,ru)"c"
"d":"c"
"e":"c"
"f":"c"
}}
\\
\hline
\end{tabularx}

\begin{tabularx}{0.98\textwidth}{%
  | *{3}{>{\centering\arraybackslash}X|}
}
\hline
\heading{$(1/8,1/8)$} & \heading{$(11/4,11/4)$} & \heading{$(17/10,17/10)$}\\
\pre{\xygraph{
!{<0cm,0cm>;<1cm,0cm>:<0cm,1cm>::}
!{(0,1) }*+{\bullet}="c"
!{(1,1) }*+{\underset{0}{\bullet}}="a"
!{(0,0) }*+{\bullet}="d"
!{(1,0) }*+{\underset{\infty}{\bullet}}="b"
!{(3,1) }*+{\bullet}="f"
!{(2,1) }*+{\bullet}="g"
"a":@(rd,ru)"a"
"c":"a"
"b":@(rd,ru)"b"
"d":"b"
"g":"f"
"f":@(rd,ru)"f"
}}&
\pre{\xygraph{
!{<0cm,0cm>;<1cm,0cm>:<0cm,1cm>::}
!{(0,0) }*+{\underset{0}{\bullet}}="a"
!{(0,1) }*+{\underset{\infty}{\bullet}}="b"
!{(1.5,1) }*+{\bullet}="f"
!{(3,1) }*+{\bullet}="e"
!{(1.5,0) }*+{\bullet}="c"
!{(2.5,0) }*+{\bullet}="d"
"a":@(rd,ru)"a"
"b":@(rd,ru)"b"
"c":"d"
"d":@(rd,ru)"d"
"f":@/^1pc/"e"
"e":@/^1pc/"f"
}}&
\pre{\xygraph{
!{<0cm,0cm>;<1cm,0cm>:<0cm,1cm>::}
!{(0,1) }*+{\underset{0}{\bullet}}="a"
!{(0,0) }*+{\underset{\infty}{\bullet}}="b"
!{(1,0) }*+{\bullet}="c"
!{(1,1) }*+{\bullet}="d"
!{(2,0.5) }*+{\bullet}="e"
!{(3,0.5) }*+{\bullet}="f"
"a":@(rd,ru)"a"
"b":@(rd,ru)"b"
"c":"e"
"d":"e"
"e":"f"
"f":@(rd,ru)"f"
}}
\\
\hline
\end{tabularx}
\end{center}

\begin{code}
\begin{lstlisting}
gr=[[0,1/2], [0,1], [0,2], [0,0], [1,-1], [1,-3], [1/2,-7/4], [4/3,-8], [-5/12,-25/18], [5/2,5/2], [1/8,1/8], [11/4,11/4], [17/10,17/10]]

set_verbose(None)
P.<x,y>=ProjectiveSpace(QQ,1)
L=[]
for A,B in gr:
    f=DynamicalSystem([x^4 + A*x*y^3, B*x^3*y + y^4])
    if f.is_morphism():
        G = f.rational_preperiodic_graph()
        found = False
        for a,b,c,g in L:
            if g.is_isomorphic(G):
                found=True
                break
        if not found:
            print(A,B,G, len(L)+1)
            L.append((A,B,G))
len(L)
\end{lstlisting}
\end{code}

\subsection{The locus $\A_4(C_2)$} \label{sect_A4_C2_preperiodic}

This family is given by $f_{k_1,k_2,k_3}(z) = \frac{z^4+k_1z^2+1}{k_2z^3+k_3z}$. The point at infinity is always a fixed point with preimage $0$. Additional fixed points are given by the first dynatomic polynomial whose vanishing defines a rational hypersurface
\begin{equation*}
    \Phi^{\ast}_1(f) = (k_2-1)z^4 - k_1z^2 + k_3z^2 - 1.
\end{equation*}

\begin{code}
\begin{lstlisting}
r.<k1,k2,k3>=QQ[]
P.<x,y>=ProjectiveSpace(r,1)
f=DynamicalSystem([x^4 + k1*x^2*y^2 + y^4, k2*x^3*y + k3*x*y^3])
f.dehomogenize(1).dynatomic_polynomial(1)
\end{lstlisting}
\end{code}

    For this family, we again content ourselves with a census of $\Q$-rational preperiodic structures for parameters $k_1$, $k_2$, and $k_3$ with small height.

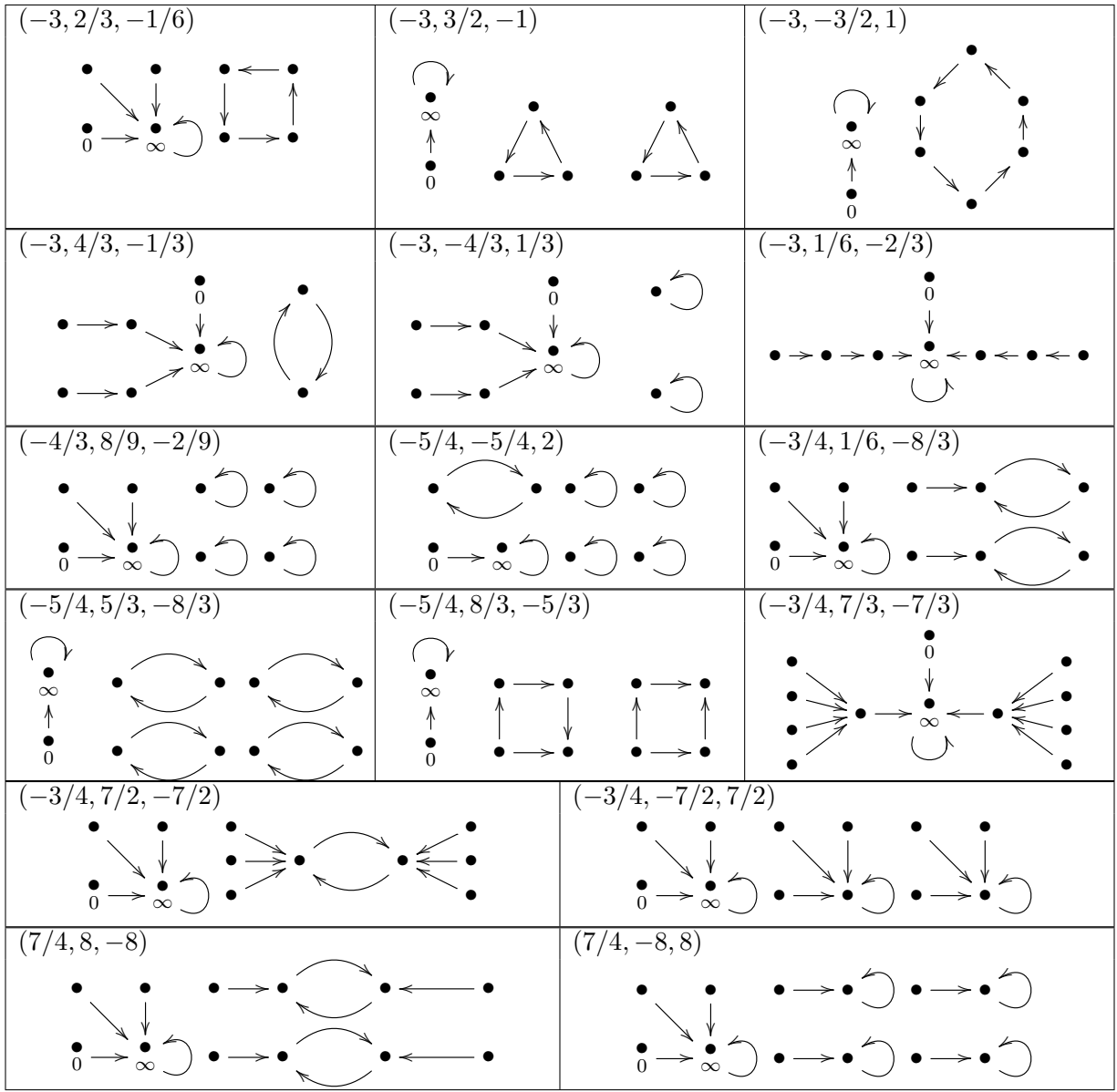
\begin{center}
Table: Preperiodic Graphs for Triples $(k_1,k_2,k_3)$

\smallskip

\middlexcolumn
\offinterlineskip
\begin{tabularx}{0.98\textwidth}{%
  | *{5}{>{\centering\arraybackslash}X|}
}
\hline
\heading{$(0,0,1)$} & \heading{$(0,0,2)$} & \heading{$(0,0,-2)$} & \heading{$(0,1,-1)$} & \heading{$(-2,0,1)$}\\
\pre{\xygraph{
!{<0cm,0cm>;<1cm,0cm>:<0cm,1cm>::}
!{(0,0) }*+{\underset{0}{\bullet}}="a"
!{(1,0) }*+{\underset{\infty}{\bullet}}="b"
"a":"b"
"b":@(rd,ru)"b"
}}&
\pre{\xygraph{
!{<0cm,0cm>;<1cm,0cm>:<0cm,1cm>::}
!{(0,0) }*+{\underset{0}{\bullet}}="a"
!{(1,0) }*+{\underset{\infty}{\bullet}}="b"
!{(0,1) }*+{\bullet}="c"
!{(1,1) }*+{\bullet}="d"
"a":"b"
"b":@(rd,ru)"b"
"c":@(rd,ru)"c"
"d":@(rd,ru)"d"
}}&
\pre{\xygraph{
!{<0cm,0cm>;<1cm,0cm>:<0cm,1cm>::}
!{(0,0) }*+{\underset{0}{\bullet}}="a"
!{(1.5,0) }*+{\underset{\infty}{\bullet}}="b"
!{(0,1) }*+{\bullet}="c"
!{(1.5,1) }*+{\bullet}="d"
"a":"b"
"b":@(rd,ru)"b"
"c":@/^1pc/"d"
"d":@/^1pc/"c"
}}&
\pre{\xygraph{
!{<0cm,0cm>;<1cm,0cm>:<0cm,1cm>::}
!{(0,0) }*+{\underset{0}{\bullet}}="a"
!{(1,0) }*+{\underset{\infty}{\bullet}}="b"
!{(0,1) }*+{\bullet}="c"
!{(1,1) }*+{\bullet}="d"
"a":"b"
"b":@(rd,ru)"b"
"c":"b"
"d":"b"
}}&
\pre{\xygraph{
!{<0cm,0cm>;<1cm,0cm>:<0cm,1cm>::}
!{(1,0.5) }*+{\underset{0}{\bullet}}="a"
!{(2,0.5) }*+{\underset{\infty}{\bullet}}="b"
!{(0,0) }*+{\bullet}="c"
!{(0,1) }*+{\bullet}="d"
"a":"b"
"b":@(rd,ru)"b"
"c":"a"
"d":"a"
}}\\
\hline
\end{tabularx}

\begin{tabularx}{0.98\textwidth}{%
  | *{4}{>{\centering\arraybackslash}X|}
}
\hline
\heading{$(0,1/3,-4/3)$} & \heading{$(0,3,-5)$} & \heading{$(0,-3,5)$} & \heading{$(0,2/3,-8/3)$}\\
\pre{\xygraph{
!{<0cm,0cm>;<1cm,0cm>:<0cm,1cm>::}
!{(2,-0.5) }*+{\underset{0}{\bullet}}="a"
!{(2,0.5) }*+{\underset{\infty}{\bullet}}="b"
!{(1,0) }*+{\bullet}="c"
!{(1,1) }*+{\bullet}="d"
!{(0,0) }*+{\bullet}="e"
!{(0,1) }*+{\bullet}="f"
"a":"b"
"b":@(rd,ru)"b"
"c":"b"
"d":"b"
"e":"c"
"f":"d"
}}&
\pre{\xygraph{
!{<0cm,0cm>;<1cm,0cm>:<0cm,1cm>::}
!{(0,0) }*+{\underset{0}{\bullet}}="a"
!{(1,0) }*+{\underset{\infty}{\bullet}}="b"
!{(1.9,0) }*+{\bullet}="c"
!{(3.4,0) }*+{\bullet}="d"
!{(1.9,1) }*+{\bullet}="e"
!{(3.4,1) }*+{\bullet}="f"
"a":"b"
"b":@(rd,ru)"b"
"c":@/^1pc/"d"
"d":@/^1pc/"c"
"e":@/^1pc/"f"
"f":@/^1pc/"e"
}}&
\pre{\xygraph{
!{<0cm,0cm>;<1cm,0cm>:<0cm,1cm>::}
!{(0,0) }*+{\underset{0}{\bullet}}="a"
!{(1,0) }*+{\underset{\infty}{\bullet}}="b"
!{(2,0) }*+{\bullet}="c"
!{(0,1) }*+{\bullet}="d"
!{(2,1) }*+{\bullet}="e"
!{(1,1) }*+{\bullet}="f"
"a":"b"
"b":@(rd,ru)"b"
"c":@(rd,ru)"c"
"d":@(rd,ru)"d"
"e":@(rd,ru)"e"
"f":@(rd,ru)"f"
}}&
\pre{\xygraph{
!{<0cm,0cm>;<1cm,0cm>:<0cm,1cm>::}
!{(0,0) }*+{\underset{0}{\bullet}}="a"
!{(1,0) }*+{\underset{\infty}{\bullet}}="b"
!{(0,1) }*+{\bullet}="c"
!{(1,1) }*+{\bullet}="d"
!{(1.75,1) }*+{\bullet}="e"
!{(3.25,1) }*+{\bullet}="f"
"a":"b"
"b":@(rd,ru)"b"
"c":"b"
"d":"b"
"e":@/^1pc/"f"
"f":@/^1pc/"e"
}}\\
\hline
\end{tabularx}

\begin{tabularx}{0.98\textwidth}{%
  | *{3}{>{\centering\arraybackslash}X|}
}
\hline
\heading{$(0,-2/3,8/3)$} & \heading{$(1,1,5)$} & \heading{$(1,-1,-5)$}\\
\pre{\xygraph{
!{<0cm,0cm>;<1cm,0cm>:<0cm,1cm>::}
!{(0,0) }*+{\underset{0}{\bullet}}="a"
!{(1,0) }*+{\underset{\infty}{\bullet}}="b"
!{(0,1) }*+{\bullet}="c"
!{(1,1) }*+{\bullet}="d"
!{(2,1) }*+{\bullet}="e"
!{(2,0) }*+{\bullet}="f"
"a":"b"
"b":@(rd,ru)"b"
"c":"b"
"d":"b"
"e":@(rd,ru)"e"
"f":@(rd,ru)"f"
}}&
\pre{\xygraph{
!{<0cm,0cm>;<1cm,0cm>:<0cm,1cm>::}
!{(0,0) }*+{\underset{0}{\bullet}}="a"
!{(1,0) }*+{\underset{\infty}{\bullet}}="b"
!{(2,0) }*+{\bullet}="c"
!{(3,0) }*+{\bullet}="d"
!{(2,1) }*+{\bullet}="e"
!{(3,1) }*+{\bullet}="f"
"a":"b"
"b":@(rd,ru)"b"
"c":"d"
"d":@(rd,ru)"d"
"e":"f"
"f":@(rd,ru)"f"
}}&
\pre{\xygraph{
!{<0cm,0cm>;<1cm,0cm>:<0cm,1cm>::}
!{(1,0) }*+{\underset{0}{\bullet}}="a"
!{(2,0) }*+{\underset{\infty}{\bullet}}="b"
!{(0,1) }*+{\bullet}="c"
!{(1,1) }*+{\bullet}="d"
!{(2.5,1) }*+{\bullet}="e"
!{(3.5,1) }*+{\bullet}="f"
"a":"b"
"b":@(rd,ru)"b"
"c":"d"
"d":@/^1pc/"e"
"e":@/^1pc/"d"
"f":"e"
}}\\
\hline
\end{tabularx}

\begin{tabularx}{0.98\textwidth}{%
  | *{3}{>{\centering\arraybackslash}X|}
}
\hline
\heading{$(-2,0,9/2)$} & \heading{$(-2,1,2)$} & \heading{$(-2,1,-1/4)$}\\
\pre{\xygraph{
!{<0cm,0cm>;<1cm,0cm>:<0cm,1cm>::}
!{(2,0.5) }*+{\underset{0}{\bullet}}="a"
!{(3,0.5) }*+{\underset{\infty}{\bullet}}="b"
!{(1,0) }*+{\bullet}="c"
!{(1,1) }*+{\bullet}="d"
!{(0,0) }*+{\bullet}="e"
!{(0,1) }*+{\bullet}="f"
"a":"b"
"b":@(rd,ru)"b"
"c":"a"
"d":"a"
"e":"c"
"f":"d"
}}&
\pre{\xygraph{
!{<0cm,0cm>;<1cm,0cm>:<0cm,1cm>::}
!{(1,0.5) }*+{\underset{0}{\bullet}}="a"
!{(2,0.5) }*+{\underset{\infty}{\bullet}}="b"
!{(0,0) }*+{\bullet}="c"
!{(0,1) }*+{\bullet}="d"
!{(3,1) }*+{\bullet}="e"
!{(3,0) }*+{\bullet}="f"
"a":"b"
"b":@(rd,ru)"b"
"c":"a"
"d":"a"
"e":@(rd,ru)"e"
"f":@(rd,ru)"f"
}}&
\pre{\xygraph{
!{<0cm,0cm>;<1cm,0cm>:<0cm,1cm>::}
!{(1,0.5) }*+{\underset{0}{\bullet}}="a"
!{(2,0.5) }*+{\underset{\infty}{\bullet}}="b"
!{(0,0) }*+{\bullet}="c"
!{(0,1) }*+{\bullet}="d"
!{(2,1.5) }*+{\bullet}="e"
!{(2,-0.5) }*+{\bullet}="f"
"a":"b"
"b":@(rd,ru)"b"
"c":"a"
"d":"a"
"e":"b"
"f":"b"
}}\\
\hline
\end{tabularx}

\begin{tabularx}{0.98\textwidth}{%
  | *{3}{>{\centering\arraybackslash}X|}
}
\hline
\heading{$(-2,1,-5/2)$} & \heading{$(1/4,9,-9/2)$} & \heading{$(1/4,9/4,-9/2)$}\\
\pre{\xygraph{
!{<0cm,0cm>;<1cm,0cm>:<0cm,1cm>::}
!{(1,0.5) }*+{\underset{0}{\bullet}}="a"
!{(2,0.5) }*+{\underset{\infty}{\bullet}}="b"
!{(0,0) }*+{\bullet}="c"
!{(0,1) }*+{\bullet}="d"
!{(2.5,1.25) }*+{\bullet}="e"
!{(4,1.25) }*+{\bullet}="f"
"a":"b"
"b":@(rd,ru)"b"
"c":"a"
"d":"a"
"e":@/^1pc/"f"
"f":@/^1pc/"e"
}}&
\pre{\xygraph{
!{<0cm,0cm>;<1cm,0cm>:<0cm,1cm>::}
!{(2,0.5) }*+{\underset{0}{\bullet}}="a"
!{(3,0.5) }*+{\underset{\infty}{\bullet}}="b"
!{(0,0) }*+{\bullet}="c"
!{(0,1) }*+{\bullet}="d"
!{(1,0) }*+{\bullet}="e"
!{(1,1) }*+{\bullet}="f"
"a":"b"
"b":@(rd,ru)"b"
"c":"d"
"c":"e"
"d":"f"
"e":"f"
}}&\pre{\xygraph{
!{<0cm,0cm>;<1cm,0cm>:<0cm,1cm>::}
!{(2.5,0.5) }*+{\underset{0}{\bullet}}="a"
!{(3.5,0.5) }*+{\underset{\infty}{\bullet}}="b"
!{(0,0) }*+{\bullet}="c"
!{(0,1) }*+{\bullet}="d"
!{(1,0) }*+{\bullet}="e"
!{(1.5,1) }*+{\bullet}="f"
"a":"b"
"b":@(rd,ru)"b"
"d":@/^1pc/"f"
"f":@/^1pc/"d"
"c":@(rd,ru)"c"
"e":@(rd,ru)"e"
}}\\
\hline
\end{tabularx}

\begin{tabularx}{0.98\textwidth}{%
  | *{3}{>{\centering\arraybackslash}X|}
}
\hline
\heading{$(1,7,-7)$} & \heading{$(1,-7,7)$} & \heading{$(1,7/2,-7/2)$}\\
\pre{\xygraph{
!{<0cm,0cm>;<1cm,0cm>:<0cm,1cm>::}
!{(1,1) }*+{\underset{0}{\bullet}}="a"
!{(2,1) }*+{\underset{\infty}{\bullet}}="b"
!{(1,2) }*+{\bullet}="c"
!{(2,2) }*+{\bullet}="d"
!{(0,0) }*+{\bullet}="e"
!{(1,0) }*+{\bullet}="f"
!{(2.5,0) }*+{\bullet}="g"
!{(3.5,0) }*+{\bullet}="h"
"a":"b"
"b":@(rd,ru)"b"
"c":"b"
"d":"b"
"e":"f"
"f":@/^1pc/"g"
"g":@/^1pc/"f"
"h":"g"
}}&
\pre{\xygraph{
!{<0cm,0cm>;<1cm,0cm>:<0cm,1cm>::}
!{(0,0) }*+{\underset{0}{\bullet}}="a"
!{(1,0) }*+{\underset{\infty}{\bullet}}="b"
!{(0,1) }*+{\bullet}="c"
!{(1,1) }*+{\bullet}="d"
!{(2,0) }*+{\bullet}="e"
!{(3,0) }*+{\bullet}="f"
!{(2,1) }*+{\bullet}="g"
!{(3,1) }*+{\bullet}="h"
"a":"b"
"b":@(rd,ru)"b"
"c":"b"
"d":"b"
"e":"f"
"f":@(rd,ru)"f"
"g":"h"
"h":@(rd,ru)"h"
}}&
\pre{\xygraph{
!{<0cm,0cm>;<1cm,0cm>:<0cm,1cm>::}
!{(2,1.75) }*+{\underset{0}{\bullet}}="a"
!{(2,0.5) }*+{\underset{\infty}{\bullet}}="b"
!{(1,0.5) }*+{\bullet}="c"
!{(3,0.5) }*+{\bullet}="d"
!{(0,0) }*+{\bullet}="e"
!{(0,1) }*+{\bullet}="f"
!{(4,0) }*+{\bullet}="g"
!{(4,1) }*+{\bullet}="h"
"a":"b"
"b":@(ld,rd)"b"
"c":"b"
"d":"b"
"e":"c"
"f":"c"
"g":"d"
"h":"d"
}}
\\
\hline
\end{tabularx}

\begin{tabularx}{0.98\textwidth}{%
  | *{3}{>{\centering\arraybackslash}X|}
}
\hline
\heading{$(-1/2,3,-9/2)$} & \heading{$(-1/2,-3,9/2)$} & \heading{$(-1/2,3/2,3/2)$}\\
\pre{\xygraph{
!{<0cm,0cm>;<1cm,0cm>:<0cm,1cm>::}
!{(0,1) }*+{\underset{0}{\bullet}}="a"
!{(1,1) }*+{\underset{\infty}{\bullet}}="b"
!{(2,1) }*+{\bullet}="c"
!{(3.5,1) }*+{\bullet}="d"
!{(0,0) }*+{\bullet}="e"
!{(1,0) }*+{\bullet}="f"
!{(2.5,0) }*+{\bullet}="g"
!{(3.5,0) }*+{\bullet}="h"
"a":"b"
"b":@(rd,ru)"b"
"c":@/^1pc/"d"
"d":@/^1pc/"c"
"e":"f"
"f":@/^1pc/"g"
"g":@/^1pc/"f"
"h":"g"
}}&
\pre{\xygraph{
!{<0cm,0cm>;<1cm,0cm>:<0cm,1cm>::}
!{(0,1) }*+{\underset{0}{\bullet}}="a"
!{(1,1) }*+{\underset{\infty}{\bullet}}="b"
!{(2,1) }*+{\bullet}="c"
!{(3,1) }*+{\bullet}="d"
!{(0,0) }*+{\bullet}="e"
!{(1,0) }*+{\bullet}="f"
!{(2,0) }*+{\bullet}="g"
!{(3,0) }*+{\bullet}="h"
"a":"b"
"b":@(rd,ru)"b"
"c":"d"
"d":@(rd,ru)"d"
"e":"f"
"f":@(rd,ru)"f"
"g":@(rd,ru)"g"
"h":@(rd,ru)"h"
}}&
\pre{\xygraph{
!{<0cm,0cm>;<1cm,0cm>:<0cm,1cm>::}
!{(0,0) }*+{\underset{0}{\bullet}}="a"
!{(0,1) }*+{\underset{\infty}{\bullet}}="b"
!{(1,0) }*+{\bullet}="c"
!{(2,0) }*+{\bullet}="d"
!{(3.5,0) }*+{\bullet}="e"
!{(1,1) }*+{\bullet}="f"
!{(2,1) }*+{\bullet}="g"
!{(3.5,1) }*+{\bullet}="h"
"a":"b"
"b":@(lu,ru)"b"
"c":"d"
"d":@/^1pc/"e"
"e":@/^1pc/"d"
"f":"g"
"g":@/^1pc/"h"
"h":@/^1pc/"g"
}}\\
\hline
\end{tabularx}

\begin{tabularx}{0.98\textwidth}{%
  | *{3}{>{\centering\arraybackslash}X|}
}
\hline
\heading{$(-1/2,3/2,-9/4)$} & \heading{$(-1/2,-3/2,9/4)$} & \heading{$(-2,0,9/4)$}\\
\pre{\xygraph{
!{<0cm,0cm>;<1cm,0cm>:<0cm,1cm>::}
!{(0,0) }*+{\underset{0}{\bullet}}="a"
!{(0,1) }*+{\underset{\infty}{\bullet}}="b"
!{(1,0) }*+{\bullet}="c"
!{(2,0) }*+{\bullet}="d"
!{(3,0) }*+{\bullet}="e"
!{(1,1) }*+{\bullet}="f"
!{(2,1) }*+{\bullet}="g"
!{(3,1) }*+{\bullet}="h"
"a":"b"
"b":@(lu,ru)"b"
"c":"d"
"d":"e"
"e":@(rd,ru)"e"
"f":"g"
"g":"h"
"h":@(rd,ru)"h"
}}&
\pre{\xygraph{
!{<0cm,0cm>;<1cm,0cm>:<0cm,1cm>::}
!{(2,0) }*+{\underset{0}{\bullet}}="a"
!{(3,0) }*+{\underset{\infty}{\bullet}}="b"
!{(1.5,1) }*+{\bullet}="c"
!{(2.75,1) }*+{\bullet}="d"
!{(0.75,1) }*+{\bullet}="e"
!{(0,1) }*+{\bullet}="f"
!{(3.5,1) }*+{\bullet}="g"
!{(4.25,1) }*+{\bullet}="h"
"a":"b"
"b":@(rd,ru)"b"
"c":@/^1pc/"d"
"d":@/^1pc/"c"
"e":"c"
"f":"e"
"g":"d"
"h":"g"
}}&
\pre{\xygraph{
!{<0cm,0cm>;<1cm,0cm>:<0cm,1cm>::}
!{(1,0.5) }*+{\underset{0}{\bullet}}="a"
!{(2,0.5) }*+{\underset{\infty}{\bullet}}="b"
!{(0,0) }*+{\bullet}="c"
!{(0,1) }*+{\bullet}="d"
!{(3,0) }*+{\bullet}="e"
!{(4,0) }*+{\bullet}="f"
!{(3,1) }*+{\bullet}="g"
!{(4,1) }*+{\bullet}="h"
"a":"b"
"b":@(rd,ru)"b"
"c":"a"
"d":"a"
"e":@(rd,ru)"e"
"f":@(rd,ru)"f"
"g":@(rd,ru)"g"
"h":@(rd,ru)"h"
}}\\
\hline
\end{tabularx}

\begin{tabularx}{0.98\textwidth}{%
  | *{3}{>{\centering\arraybackslash}X|}
}
\hline
\heading{$(-2,0,-9/4)$} & \heading{$(-2,1/4,7/2)$} & \heading{$(-2,4/5,-5)$}\\
\pre{\xygraph{
!{<0cm,0cm>;<1cm,0cm>:<0cm,1cm>::}
!{(1,0.5) }*+{\underset{0}{\bullet}}="a"
!{(2,0.5) }*+{\underset{\infty}{\bullet}}="b"
!{(0,0) }*+{\bullet}="c"
!{(0,1) }*+{\bullet}="d"
!{(3,0) }*+{\bullet}="e"
!{(4.5,0) }*+{\bullet}="f"
!{(3,1) }*+{\bullet}="g"
!{(4.5,1) }*+{\bullet}="h"
"a":"b"
"b":@(rd,ru)"b"
"c":"a"
"d":"a"
"e":@/^1pc/"f"
"f":@/^1pc/"e"
"g":@/^1pc/"h"
"h":@/^1pc/"g"
}}&
\pre{\xygraph{
!{<0cm,0cm>;<1cm,0cm>:<0cm,1cm>::}
!{(2,0.5) }*+{\underset{0}{\bullet}}="a"
!{(2,1.75) }*+{\underset{\infty}{\bullet}}="b"
!{(1,0.5) }*+{\bullet}="c"
!{(3,0.5) }*+{\bullet}="d"
!{(0,0) }*+{\bullet}="e"
!{(0,1) }*+{\bullet}="f"
!{(4,0) }*+{\bullet}="g"
!{(4,1) }*+{\bullet}="h"
"a":"b"
"b":@(lu,ru)"b"
"c":"a"
"d":"a"
"e":"c"
"f":"c"
"g":"d"
"h":"d"
}}&
\pre{\xygraph{
!{<0cm,0cm>;<1cm,0cm>:<0cm,1cm>::}
!{(2,1.25) }*+{\underset{0}{\bullet}}="a"
!{(2,0) }*+{\underset{\infty}{\bullet}}="b"
!{(1,1.25) }*+{\bullet}="c"
!{(3,1.25) }*+{\bullet}="d"
!{(1,0) }*+{\bullet}="e"
!{(3,0) }*+{\bullet}="f"
!{(0,0) }*+{\bullet}="g"
!{(4,0) }*+{\bullet}="h"
"a":"b"
"b":@(ld,rd)"b"
"c":"a"
"d":"a"
"e":"b"
"f":"b"
"g":"e"
"h":"f"
}}
\\
\hline
\end{tabularx}

\begin{tabularx}{0.98\textwidth}{%
  | *{3}{>{\centering\arraybackslash}X|}
}
\hline
\heading{$(-2,6,-3/8)$} & \heading{$(-2,7,1/2)$} & \heading{$(-2,-7,-1/2)$}\\
\pre{\xygraph{
!{<0cm,0cm>;<1cm,0cm>:<0cm,1cm>::}
!{(2,1.25) }*+{\underset{0}{\bullet}}="a"
!{(2,0) }*+{\underset{\infty}{\bullet}}="b"
!{(1,1.25) }*+{\bullet}="c"
!{(3,1.25) }*+{\bullet}="d"
!{(1,0) }*+{\bullet}="e"
!{(3,0) }*+{\bullet}="f"
!{(0,1.25) }*+{\bullet}="g"
!{(4,1.25) }*+{\bullet}="h"
"a":"b"
"b":@(ld,rd)"b"
"c":"a"
"d":"a"
"e":"b"
"f":"b"
"g":"c"
"h":"d"
}}&
\pre{\xygraph{
!{<0cm,0cm>;<1cm,0cm>:<0cm,1cm>::}
!{(1,0.5) }*+{\underset{0}{\bullet}}="a"
!{(2,0.5) }*+{\underset{\infty}{\bullet}}="b"
!{(0,0) }*+{\bullet}="c"
!{(0,1) }*+{\bullet}="d"
!{(3,0) }*+{\bullet}="e"
!{(4,0) }*+{\bullet}="f"
!{(3,1) }*+{\bullet}="g"
!{(4,1) }*+{\bullet}="h"
"a":"b"
"b":@(rd,ru)"b"
"c":"a"
"d":"a"
"e":"f"
"f":@(rd,ru)"f"
"g":"h"
"h":@(rd,ru)"h"
}}&
\pre{\xygraph{
!{<0cm,0cm>;<1cm,0cm>:<0cm,1cm>::}
!{(1,0.5) }*+{\underset{0}{\bullet}}="a"
!{(2,0.5) }*+{\underset{\infty}{\bullet}}="b"
!{(0,0) }*+{\bullet}="c"
!{(0,1) }*+{\bullet}="d"
!{(3,0) }*+{\bullet}="e"
!{(3,1.5) }*+{\bullet}="f"
!{(4,0) }*+{\bullet}="g"
!{(4,1.5) }*+{\bullet}="h"
"a":"b"
"b":@(rd,ru)"b"
"c":"a"
"d":"a"
"e":"g"
"f":"h"
"g":@/^1pc/"h"
"h":@/^1pc/"g"
}}
\\
\hline
\end{tabularx}

\begin{tabularx}{0.98\textwidth}{%
  | *{3}{>{\centering\arraybackslash}X|}
}
\hline
\heading{$(1/3,8/3,-3/2)$} & \heading{$(-3,1,1)$} & \heading{$(-3,1,-3/2)$}\\
\pre{\xygraph{
!{<0cm,0cm>;<1cm,0cm>:<0cm,1cm>::}
!{(0,0) }*+{\underset{0}{\bullet}}="a"
!{(1,0) }*+{\underset{\infty}{\bullet}}="b"
!{(0,1) }*+{\bullet}="c"
!{(1,1) }*+{\bullet}="d"
!{(2,0) }*+{\bullet}="e"
!{(3.5,0) }*+{\bullet}="f"
!{(2,1) }*+{\bullet}="g"
!{(3.5,1) }*+{\bullet}="h"
"a":"b"
"b":@(rd,ru)"b"
"c":"b"
"d":"b"
"e":@/^1pc/"f"
"f":@/^1pc/"e"
"g":@/^1pc/"h"
"h":@/^1pc/"g"
}}&
\pre{\xygraph{
!{<0cm,0cm>;<1cm,0cm>:<0cm,1cm>::}
!{(0,0) }*+{\underset{0}{\bullet}}="a"
!{(0,1.25) }*+{\underset{\infty}{\bullet}}="b"
!{(1,0) }*+{\bullet}="c"
!{(1,1) }*+{\bullet}="d"
!{(2,0) }*+{\bullet}="e"
!{(3,0) }*+{\bullet}="f"
!{(3,1) }*+{\bullet}="g"
!{(4,0) }*+{\bullet}="h"
"a":"b"
"b":@(lu,ru)"b"
"c":"e"
"d":"e"
"e":@(rd,ru)"e"
"f":"h"
"g":"h"
"h":@(rd,ru)"h"
}}&
\pre{\xygraph{
!{<0cm,0cm>;<1cm,0cm>:<0cm,1cm>::}
!{(0,0) }*+{\underset{0}{\bullet}}="a"
!{(1,0) }*+{\underset{\infty}{\bullet}}="b"
!{(0,1) }*+{\bullet}="c"
!{(1.5,1) }*+{\bullet}="d"
!{(2,0) }*+{\bullet}="e"
!{(3.5,0) }*+{\bullet}="f"
!{(2,1) }*+{\bullet}="g"
!{(3.5,1) }*+{\bullet}="h"
"a":"b"
"b":@(rd,ru)"b"
"c":@/^1pc/"d"
"d":@/^1pc/"c"
"e":@/^1pc/"f"
"f":@/^1pc/"e"
"g":@/^1pc/"h"
"h":@/^1pc/"g"
}}\\
\hline
\end{tabularx}

\begin{tabularx}{0.98\textwidth}{%
  | *{3}{>{\centering\arraybackslash}X|}
}
\hline
\heading{$(-3,-1,-1)$} & \heading{$(-3,-1,3/2)$} & \heading{$(-3,1/3,-4/3)$}\\
\pre{\xygraph{
!{<0cm,0cm>;<1cm,0cm>:<0cm,1cm>::}
!{(0,0) }*+{\underset{0}{\bullet}}="a"
!{(0,1.25) }*+{\underset{\infty}{\bullet}}="b"
!{(1,0) }*+{\bullet}="c"
!{(1,1) }*+{\bullet}="d"
!{(2,0.5) }*+{\bullet}="e"
!{(3.5,0.5) }*+{\bullet}="f"
!{(4.5,0) }*+{\bullet}="g"
!{(4.5,1) }*+{\bullet}="h"
"a":"b"
"b":@(rd,ru)"b"
"c":"e"
"d":"e"
"e":@/^1pc/"f"
"f":@/^1pc/"e"
"g":"f"
"h":"f"
}}&
\pre{\xygraph{
!{<0cm,0cm>;<1cm,0cm>:<0cm,1cm>::}
!{(0,0) }*+{\underset{0}{\bullet}}="a"
!{(1,0) }*+{\underset{\infty}{\bullet}}="b"
!{(0,1) }*+{\bullet}="c"
!{(1.5,1) }*+{\bullet}="d"
!{(2,0) }*+{\bullet}="e"
!{(3.5,0) }*+{\bullet}="f"
!{(2,1) }*+{\bullet}="g"
!{(3,1) }*+{\bullet}="h"
"a":"b"
"b":@(rd,ru)"b"
"c":@/^1pc/"d"
"d":@/^1pc/"c"
"e":@/^1pc/"f"
"f":@/^1pc/"e"
"g":@(rd,ru)"g"
"h":@(rd,ru)"h"
}}&
\pre{\xygraph{
!{<0cm,0cm>;<1cm,0cm>:<0cm,1cm>::}
!{(0,0) }*+{\underset{0}{\bullet}}="a"
!{(1,0) }*+{\underset{\infty}{\bullet}}="b"
!{(0,1) }*+{\bullet}="c"
!{(1,1) }*+{\bullet}="d"
!{(2,0) }*+{\bullet}="e"
!{(3.5,0) }*+{\bullet}="f"
!{(2,1) }*+{\bullet}="g"
!{(3,1) }*+{\bullet}="h"
"a":"b"
"b":@(rd,ru)"b"
"c":"b"
"d":"b"
"e":@/^1pc/"f"
"f":@/^1pc/"e"
"g":@(rd,ru)"g"
"h":@(rd,ru)"h"
}}\\
\hline
\end{tabularx}

\begin{tabularx}{0.98\textwidth}{%
  | *{3}{>{\centering\arraybackslash}X|}
}
\hline
\heading{$(-3,2/3,-1/6)$} & \heading{$(-3,3/2,-1)$} & \heading{$(-3,-3/2,1)$}\\
\pre{\xygraph{
!{<0cm,0cm>;<1cm,0cm>:<0cm,1cm>::}
!{(0,0) }*+{\underset{0}{\bullet}}="a"
!{(1,0) }*+{\underset{\infty}{\bullet}}="b"
!{(0,1) }*+{\bullet}="c"
!{(1,1) }*+{\bullet}="d"
!{(2,0) }*+{\bullet}="e"
!{(3,0) }*+{\bullet}="f"
!{(2,1) }*+{\bullet}="g"
!{(3,1) }*+{\bullet}="h"
"a":"b"
"b":@(rd,ru)"b"
"c":"b"
"d":"b"
"e":"f"
"f":"h"
"h":"g"
"g":"e"
}}&
\pre{\xygraph{
!{<0cm,0cm>;<1cm,0cm>:<0cm,1cm>::}
!{(0,0) }*+{\underset{0}{\bullet}}="a"
!{(0,1) }*+{\underset{\infty}{\bullet}}="b"
!{(1,0) }*+{\bullet}="c"
!{(2,0) }*+{\bullet}="d"
!{(1.5,1) }*+{\bullet}="e"
!{(3,0) }*+{\bullet}="f"
!{(4,0) }*+{\bullet}="g"
!{(3.5,1) }*+{\bullet}="h"
"a":"b"
"b":@(lu,ru)"b"
"c":"d"
"d":"e"
"e":"c"
"f":"g"
"g":"h"
"h":"f"
}}&
\pre{\xygraph{
!{<0cm,0cm>;<1cm,0cm>:<0cm,1cm>::}
!{(0,0) }*+{\underset{0}{\bullet}}="a"
!{(0,1) }*+{\underset{\infty}{\bullet}}="b"
!{(1,0.75) }*+{\bullet}="c"
!{(1.75,0) }*+{\bullet}="d"
!{(2.5,0.75) }*+{\bullet}="e"
!{(2.5,1.5) }*+{\bullet}="f"
!{(1.75,2.25) }*+{\bullet}="g"
!{(1,1.5) }*+{\bullet}="h"
"a":"b"
"b":@(lu,ru)"b"
"c":"d"
"d":"e"
"e":"f"
"f":"g"
"g":"h"
"h":"c"
}}\\
\hline
\end{tabularx}

\begin{tabularx}{0.98\textwidth}{%
  | *{3}{>{\centering\arraybackslash}X|}
}
\hline
\heading{$(-3,4/3,-1/3)$} & \heading{$(-3,-4/3,1/3)$} & \heading{$(-3,1/6,-2/3)$}\\
\pre{\xygraph{
!{<0cm,0cm>;<1cm,0cm>:<0cm,1cm>::}
!{(2,1.5) }*+{\underset{0}{\bullet}}="a"
!{(2,0.5) }*+{\underset{\infty}{\bullet}}="b"
!{(1,0) }*+{\bullet}="c"
!{(1,1) }*+{\bullet}="d"
!{(0,0) }*+{\bullet}="e"
!{(0,1) }*+{\bullet}="f"
!{(3.5,0) }*+{\bullet}="g"
!{(3.5,1.5) }*+{\bullet}="h"
"a":"b"
"b":@(rd,ru)"b"
"c":"b"
"d":"b"
"e":"c"
"f":"d"
"g":@/^1pc/"h"
"h":@/^1pc/"g"
}}&
\pre{\xygraph{
!{<0cm,0cm>;<1cm,0cm>:<0cm,1cm>::}
!{(2,1.5) }*+{\underset{0}{\bullet}}="a"
!{(2,0.5) }*+{\underset{\infty}{\bullet}}="b"
!{(1,0) }*+{\bullet}="c"
!{(1,1) }*+{\bullet}="d"
!{(0,0) }*+{\bullet}="e"
!{(0,1) }*+{\bullet}="f"
!{(3.5,0) }*+{\bullet}="g"
!{(3.5,1.5) }*+{\bullet}="h"
"a":"b"
"b":@(rd,ru)"b"
"c":"b"
"d":"b"
"e":"c"
"f":"d"
"g":@(rd,ru)"g"
"h":@(rd,ru)"h"
}}&
\pre{\xygraph{
!{<0cm,0cm>;<1cm,0cm>:<0cm,1cm>::}
!{(2.25,1) }*+{\underset{0}{\bullet}}="a"
!{(2.25,0) }*+{\underset{\infty}{\bullet}}="b"
!{(1.5,0) }*+{\bullet}="c"
!{(3,0) }*+{\bullet}="d"
!{(0.75,0) }*+{\bullet}="e"
!{(3.75,0) }*+{\bullet}="f"
!{(0,0) }*+{\bullet}="g"
!{(4.5,0) }*+{\bullet}="h"
"a":"b"
"b":@(ld,rd)"b"
"c":"b"
"d":"b"
"e":"c"
"f":"d"
"g":"e"
"h":"f"
}}
\\
\hline
\end{tabularx}

\begin{tabularx}{0.98\textwidth}{%
  | *{3}{>{\centering\arraybackslash}X|}
}
\hline
\heading{$(-4/3,8/9,-2/9)$} & \heading{$(-5/4,-5/4,2)$} & \heading{$(-3/4,1/6,-8/3)$}\\
\pre{\xygraph{
!{<0cm,0cm>;<1cm,0cm>:<0cm,1cm>::}
!{(0,0) }*+{\underset{0}{\bullet}}="a"
!{(1,0) }*+{\underset{\infty}{\bullet}}="b"
!{(0,1) }*+{\bullet}="c"
!{(1,1) }*+{\bullet}="d"
!{(2,0) }*+{\bullet}="e"
!{(3,0) }*+{\bullet}="f"
!{(2,1) }*+{\bullet}="g"
!{(3,1) }*+{\bullet}="h"
"a":"b"
"b":@(rd,ru)"b"
"c":"b"
"d":"b"
"e":@(rd,ru)"e"
"f":@(rd,ru)"f"
"g":@(rd,ru)"g"
"h":@(rd,ru)"h"
}}&
\pre{\xygraph{
!{<0cm,0cm>;<1cm,0cm>:<0cm,1cm>::}
!{(0,0) }*+{\underset{0}{\bullet}}="a"
!{(1,0) }*+{\underset{\infty}{\bullet}}="b"
!{(0,1) }*+{\bullet}="c"
!{(1.5,1) }*+{\bullet}="d"
!{(2,0) }*+{\bullet}="e"
!{(3,0) }*+{\bullet}="f"
!{(2,1) }*+{\bullet}="g"
!{(3,1) }*+{\bullet}="h"
"a":"b"
"b":@(rd,ru)"b"
"c":@/^1pc/"d"
"d":@/^1pc/"c"
"e":@(rd,ru)"e"
"f":@(rd,ru)"f"
"g":@(rd,ru)"g"
"h":@(rd,ru)"h"
}}&
\pre{\xygraph{
!{<0cm,0cm>;<1cm,0cm>:<0cm,1cm>::}
!{(0,0) }*+{\underset{0}{\bullet}}="a"
!{(1,0) }*+{\underset{\infty}{\bullet}}="b"
!{(0,1) }*+{\bullet}="c"
!{(1,1) }*+{\bullet}="d"
!{(2,0) }*+{\bullet}="e"
!{(3,0) }*+{\bullet}="f"
!{(4.5,0) }*+{\bullet}="g"
!{(2,1) }*+{\bullet}="h"
!{(3,1) }*+{\bullet}="i"
!{(4.5,1) }*+{\bullet}="j"
"a":"b"
"b":@(rd,ru)"b"
"c":"b"
"d":"b"
"e":"f"
"f":@/^1pc/"g"
"g":@/^1pc/"f"
"h":"i"
"i":@/^1pc/"j"
"j":@/^1pc/"i"
}}
\\
\hline
\end{tabularx}

\begin{tabularx}{0.98\textwidth}{%
  | *{3}{>{\centering\arraybackslash}X|}
}
\hline
\heading{$(-5/4,5/3,-8/3)$} & \heading{$(-5/4,8/3,-5/3)$} & \heading{$(-3/4,7/3,-7/3)$}\\
\pre{\xygraph{
!{<0cm,0cm>;<1cm,0cm>:<0cm,1cm>::}
!{(0,0) }*+{\underset{0}{\bullet}}="a"
!{(0,1) }*+{\underset{\infty}{\bullet}}="b"
!{(1,0) }*+{\bullet}="c"
!{(2.5,0) }*+{\bullet}="d"
!{(1,1) }*+{\bullet}="e"
!{(2.5,1) }*+{\bullet}="f"
!{(3,0) }*+{\bullet}="g"
!{(4.5,0) }*+{\bullet}="h"
!{(3,1) }*+{\bullet}="i"
!{(4.5,1) }*+{\bullet}="j"
"a":"b"
"b":@(lu,ru)"b"
"c":@/^1pc/"d"
"d":@/^1pc/"c"
"e":@/^1pc/"f"
"f":@/^1pc/"e"
"g":@/^1pc/"h"
"h":@/^1pc/"g"
"i":@/^1pc/"j"
"j":@/^1pc/"i"
}}&
\pre{\xygraph{
!{<0cm,0cm>;<1cm,0cm>:<0cm,1cm>::}
!{(0,0) }*+{\underset{0}{\bullet}}="a"
!{(0,1) }*+{\underset{\infty}{\bullet}}="b"
!{(1,0) }*+{\bullet}="c"
!{(1,1) }*+{\bullet}="d"
!{(2,0) }*+{\bullet}="e"
!{(2,1) }*+{\bullet}="f"
!{(3,0) }*+{\bullet}="g"
!{(3,1) }*+{\bullet}="h"
!{(4,0) }*+{\bullet}="i"
!{(4,1) }*+{\bullet}="j"
"a":"b"
"b":@(lu,ru)"b"
"c":"d"
"c":"e"
"d":"f"
"f":"e"
"g":"h"
"g":"i"
"i":"j"
"h":"j"
}}&
\pre{\xygraph{
!{<0cm,0cm>;<1cm,0cm>:<0cm,1cm>::}
!{(2,1.75) }*+{\underset{0}{\bullet}}="a"
!{(2,0.75) }*+{\underset{\infty}{\bullet}}="b"
!{(1,0.75) }*+{\bullet}="c"
!{(3,0.75) }*+{\bullet}="d"
!{(0,0) }*+{\bullet}="e"
!{(0,0.5) }*+{\bullet}="f"
!{(0,1) }*+{\bullet}="g"
!{(0,1.5) }*+{\bullet}="h"
!{(4,0) }*+{\bullet}="i"
!{(4,0.5) }*+{\bullet}="j"
!{(4,1) }*+{\bullet}="k"
!{(4,1.5) }*+{\bullet}="l"
"a":"b"
"b":@(ld,rd)"b"
"c":"b"
"d":"b"
"e":"c"
"f":"c"
"g":"c"
"h":"c"
"i":"d"
"j":"d"
"k":"d"
"l":"d"
}}\\
\hline
\end{tabularx}

\begin{tabularx}{0.98\textwidth}{%
  | *{2}{>{\centering\arraybackslash}X|}
}
\hline
\heading{$(-3/4,7/2,-7/2)$} & \heading{$(-3/4,-7/2,7/2)$}\\
\pre{\xygraph{
!{<0cm,0cm>;<1cm,0cm>:<0cm,1cm>::}
!{(0,0) }*+{\underset{0}{\bullet}}="a"
!{(1,0) }*+{\underset{\infty}{\bullet}}="b"
!{(0,1) }*+{\bullet}="c"
!{(1,1) }*+{\bullet}="d"
!{(2,0) }*+{\bullet}="e"
!{(2,1) }*+{\bullet}="f"
!{(2,0.5) }*+{\bullet}="g"
!{(3,0.5) }*+{\bullet}="h"
!{(4.5,0.5) }*+{\bullet}="i"
!{(5.5,0) }*+{\bullet}="j"
!{(5.5,0.5) }*+{\bullet}="k"
!{(5.5,1) }*+{\bullet}="l"
"a":"b"
"b":@(rd,ru)"b"
"c":"b"
"d":"b"
"e":"h"
"f":"h"
"g":"h"
"h":@/^1pc/"i"
"i":@/^1pc/"h"
"j":"i"
"k":"i"
"l":"i"
}}&
\pre{\xygraph{
!{<0cm,0cm>;<1cm,0cm>:<0cm,1cm>::}
!{(0,0) }*+{\underset{0}{\bullet}}="a"
!{(1,0) }*+{\underset{\infty}{\bullet}}="b"
!{(0,1) }*+{\bullet}="c"
!{(1,1) }*+{\bullet}="d"
!{(2,0) }*+{\bullet}="e"
!{(2,1) }*+{\bullet}="f"
!{(3,0) }*+{\bullet}="g"
!{(3,1) }*+{\bullet}="h"
!{(4,0) }*+{\bullet}="i"
!{(4,1) }*+{\bullet}="j"
!{(5,0) }*+{\bullet}="k"
!{(5,1) }*+{\bullet}="l"
"a":"b"
"b":@(rd,ru)"b"
"c":"b"
"d":"b"
"e":"g"
"f":"g"
"h":"g"
"g":@(rd,ru)"g"
"i":"k"
"j":"k"
"l":"k"
"k":@(rd,ru)"k"
}}\\
\hline
\end{tabularx}

\begin{tabularx}{0.98\textwidth}{%
  | *{2}{>{\centering\arraybackslash}X|}
}
\hline
\heading{$(7/4,8,-8)$} & \heading{$(7/4,-8,8)$}\\
\pre{\xygraph{
!{<0cm,0cm>;<1cm,0cm>:<0cm,1cm>::}
!{(0,0) }*+{\underset{0}{\bullet}}="a"
!{(1,0) }*+{\underset{\infty}{\bullet}}="b"
!{(0,1) }*+{\bullet}="c"
!{(1,1) }*+{\bullet}="d"
!{(2,0) }*+{\bullet}="e"
!{(3,0) }*+{\bullet}="f"
!{(4.5,0) }*+{\bullet}="g"
!{(6,0) }*+{\bullet}="h"
!{(2,1) }*+{\bullet}="i"
!{(3,1) }*+{\bullet}="j"
!{(4.5,1) }*+{\bullet}="k"
!{(6,1) }*+{\bullet}="l"
"a":"b"
"b":@(rd,ru)"b"
"c":"b"
"d":"b"
"e":"f"
"f":@/^1pc/"g"
"g":@/^1pc/"f"
"h":"g"
"i":"j"
"j":@/^1pc/"k"
"k":@/^1pc/"j"
"l":"k"
}}&
\pre{\xygraph{
!{<0cm,0cm>;<1cm,0cm>:<0cm,1cm>::}
!{(0,0) }*+{\underset{0}{\bullet}}="a"
!{(1,0) }*+{\underset{\infty}{\bullet}}="b"
!{(0,1) }*+{\bullet}="c"
!{(1,1) }*+{\bullet}="d"
!{(2,0) }*+{\bullet}="e"
!{(3,0) }*+{\bullet}="f"
!{(4,0) }*+{\bullet}="g"
!{(5,0) }*+{\bullet}="h"
!{(2,1) }*+{\bullet}="i"
!{(3,1) }*+{\bullet}="j"
!{(4,1) }*+{\bullet}="k"
!{(5,1) }*+{\bullet}="l"
"a":"b"
"b":@(rd,ru)"b"
"c":"b"
"d":"b"
"e":"f"
"f":@(rd,ru)"f"
"g":"h"
"h":@(rd,ru)"h"
"i":"j"
"j":@(rd,ru)"j"
"k":"l"
"l":@(rd,ru)"l"
}}\\
\hline
\end{tabularx}
\end{center}

\begin{code}
\begin{lstlisting}
gr=[[0, 0, 1],[0, 0, 2],[0, 0, -2],[0, 1, -1],[0, 1/3, -4/3],[0, 3, -5],[0, -3, 5],[0, 2/3, -8/3],[0, -2/3, 8/3],[1, 1, 5],[1, -1, -5],[1, 7, -7],[1, -7, 7],[1, 7/2, -7/2],[-1/2, 3, -9/2],[-1/2, -3, 9/2],[-1/2, 3/2, 3/2],[-1/2, 3/2, -9/4],[-1/2, -3/2, 9/4],[-2, 0, 1],[-2, 0, 9/2],[-2, 0, 9/4],[-2, 0, -9/4],[-2, 1, 2],[-2, 1, -1/4],[-2, 1, -5/2],[-2, 1/4, 7/2],[-2, 4/5, -5],[-2, 6, -3/8],[-2, 7, 1/2],[-2, -7, -1/2],[1/3, 8/3, -3/2],[-3, 1, 1],[-3, 1, -3/2],[-3, -1, -1],[-3, -1, 3/2],[-3, 1/3, -4/3],[-3, 2/3, -1/6],[-3, 3/2, -1],[-3, -3/2, 1],[-3, 4/3, -1/3],[-3, -4/3, 1/3],[-3, 1/6, -2/3],[1/4, 9, -9/2],[1/4, 9/4, -9/2],[-3/4, 1/6, -8/3],[-3/4, 7/2, -7/2],[-3/4, -7/2, 7/2],[-3/4, 7/3, -7/3],[-4/3, 8/9, -2/9],[-5/4, 5/3, -8/3],[-5/4, -5/4, 2],[-5/4, 8/3, -5/3],[7/4,8,-8],[7/4,-8,8]]

set_verbose(None)
P.<x,y>=ProjectiveSpace(QQ,1)
L=[]
for A,B,C in gr:
    f=DynamicalSystem([x^4+A*x^2*y^2+y^4,B*x^3*y+C*x*y^3])
    if f.is_morphism():
        G = f.rational_preperiodic_graph()
        found = False
        for a,b,c,g in L:
            if g.is_isomorphic(G):
                found=True
                break
        if not found:
            print(A,B,C,G, len(L)+1)
            L.append((A,B,C,G))
len(L)
\end{lstlisting}
\end{code}

Notice that there are $\Q$-rational periodic points of (minimal) period $\{1,2,3,4,6\}$. Are there $\Q$-rational points of (minimal) period $5$?

\section{Concluding Remarks} \label{sect_concluding}
There are a number of interesting problems that come both directly and indirectly from the results in this paper. We list a few of those here.
\begin{enumerate}
    \item Resolve the conjectures on the existence of periodic points for the various one-dimensional families.
    \item Determine all the rational points on the high genus curves that were left unresolved in $\A_4(D_3)$.
    \item Do a full analysis of rational preperiodic points for the two- and three-dimensional families.
    \item Classify $\A_d$ for $d \geq 5$.
    \item Use these families for future studies. For example, when looking at the cycle statistics of these families modulo primes, such as the average number of periodic points or the average tail length, the number of rational elements in the automorphism group appears to play a role. We have preliminary statistics on this topic that can be made available upon request.
\end{enumerate}

\providecommand\biburl[1]{\texttt{#1}}

\end{document}